\documentclass[a4paper]{amsbook}
\usepackage[utf8]{inputenc}
\usepackage{diss-kirschner}
\bibliography{diss-kirschner}

\title{Irreducible Symplectic Complex Spaces}
\author{Tim Kirschner}
\address{Lehrstuhl für Komplexe Analysis\\ Universität Bayreuth}
\email{\href{mailto:tim.kirschner@uni-bayreuth.de}{tim.kirschner@uni-bayreuth.de}}
\urladdr{\href{http://www.staff.uni-bayreuth.de/~btm107/}{http://www.staff.uni-bayreuth.de/~btm107/}}
\date{\today}

\begin{document}
\frontmatter
%\subjclass{32Cxx, 32S35}
%\keywords{}
%\thanks{}
\dedicatory{For my parents}
\maketitle
\chapter*{Preface}

This thesis is the product of a period of four years of mathematical study, research, and learning, which started in April of 2008 and lasts to the day. Looking back, I can say that it hasn't been an easy journey from the first discussions of potential research problems with my forever dedicated advisor \emph{Prof.~Dr.~Thomas Peternell} to putting the finishing touches to the text. Now, however, I am very happy with what I can offer you for a read.

There are many people to thank without whom this work could never have come into existence. First and foremost, I would like to express my cordial thanks to Thomas Peternell for seeing this project through calmly, professionally, spiritedly, free-spiritedly, and supportively from beginning to end, for meeting my partly unconventional suggestions and ideas in an open-minded and unbiased way, for letting me force my ``bureaucratic'' mathematical style upon our work, albeit this meant compromising on his own original goals, for offering me a position in his group in the first place and opening so many doors within the mathematical cosmos, for shielding me from tedious non-research duties and letting our collaboration evolve over time into a great partnership.

I would like to thank the ``Deutsche Forschungsgemeinschaft (DFG)'' for financing the vast majority of my work through the ``Forschergruppe 790 -- Classification of Algebraic Surfaces and Compact Complex Manifolds''; the importance of the DFG as a financier and promoter for the on-going production of wealthy mathematical research in Germany can by no means be underestimated. Moreover, would I like to thank the members of the Forschergruppe who constantly invest considerable shares of their time and energy into keeping our group alive.

I would like to thank \emph{Prof.~Dr.~Georg Schumacher} in Marburg who, during the winter of 2010--2011, very kindly and vigorously conversed with me about a possible $L^2$ cohomology approach towards what have become the results of Chapter \ref{ch:peri}; unfortunately, these ideas are not part of the final thesis as I could not bring them to a favorable conclusion. I would like to thank \emph{Professor Keiji Oguiso} of Osaka University for explaining several aspects of the theory of irreducible symplectic manifolds to me during a summer school in Poland back in 2008 as well as during his stays in Bayreuth and for always responding to my e-mailed questions since.

I am very much indebted to my ``early'' teachers of mathematics. Let me mention two people specifically. Firstly, it was \emph{Albrecht Kliem} who fostered and inspired me in a plethora of ways from my first participation at his ``Landeswettbewerb Mathematik Bayern'' onwards through several years of high school filled with mathematical competitions and extracurricular mathematical seminars; his encouragement and confidence were (and are) invaluable. Secondly, I am grateful to \emph{Prof.~Dr.~Gerhard Rein}, my first and at the same time most influential university teacher, whose perfect lectures in real and complex analysis had (and have) a lasting effect on my mathematical style and thinking, most certainly beyond his own knowledge.

Moving on from the professional to the more personal level, I would like to thank my dear colleagues \emph{Florian Schrack} and \emph{Tobias Dorsch}. I would like to thank both of you for discussing and debating mathematics with me, for suggesting solutions and offering advice, and for enduring my pronounced need to molest you with topics like why I feel that derived categories (among other things) do not exist. Yet, what is more, I would like to thank you for enduring my pronounced need to chat non-mathematically, for introducing me to gyokuro and amaranth (respectively), for making my work days fun, and for being friends rather than mere office mates.

Finally, I would like to thank those four people who make my life worthwhile: \emph{Julia}, \emph{Malte}, \emph{Mom}, \emph{Dad}. Jules: even though our ways have parted a certain while ago, let me put this in present tense and say that ``when I'm with you, I am calm, a pearl in your oyster; head on my chest, a silent smile, a private kind of happiness. You see giant proclamations are all very well, but our love is louder than words''\footnote{To all GuttenPlag sort of people out there: you will find this on Google \ldots (told you so)}. M: what we have is unreal---thanks for sticking around all these years! Mom \& Dad: your bringing me up in a spirit of freedom, love, stability, and unconditional support is the basis for everything. Love, hugs, and kisses to all of you.

\vspace{1cm}
April 11, 2012\par Bayreuth \hfill Tim

\tableofcontents

\mainmatter
\chapter*{Introduction}
\label{ch:intr}

\renewcommand{\thesubsection}{\arabic{subsection}}

\subsection*{Symplectic complex spaces}

Let $X$ be a complex space (\resp a finite type scheme over the field of complex numbers). Then we say that $X$ is \emph{symplectic} when $X$ is normal and there exists a symplectic structure on $X$. Here, $\sigma$ is called a \emph{symplectic structure} on $X$ when the following assertions hold:
\begin{enumeratei}
 \item \label{i:symp-cl} $\sigma$ is a closed Kähler $2$-differential on $X$ over $X_\reg$, \iev $\sigma$ is an element of $\Omega^2_X(X_\reg)$ being sent to zero by the mapping
 \[
  (\dd^2_X)_{X_\reg} \colon \Omega^2_X(X_\reg) \to \Omega^3_X(X_\reg).
 \]
 \item \label{i:symp-nondeg} The canonical image of $\sigma$ in $\Omega^2_{X_\reg}(X_\reg)$ is \emph{nondegenerate} on $X_\reg$; our preferred way of formalizing the nondegeneracy, even though uncommon, is to require that the composition of sheaf maps on $X_\reg$
 \[
  \xymatrix{
   \Theta_{X_\reg} \ar[r] & \Theta_{X_\reg} \otimes \O_{X_\reg} \ar[r]^{\id \otimes \sigma} & \Theta_{X_\reg} \otimes \Omega^2_{X_\reg} \ar[r] & \Omega^1_{X_\reg}
  },
 \]
 where the first and last arrows signify respectively the right tensor unit for the tangent sheaf on $X_\reg$ and the contraction morphism, be an isomorphism.
 \item \label{i:symp-ext} For all resolutions of singularities $f\colon \tilde X\to X$, there exists $\tilde\sigma \in \Omega^2_{\tilde X}(\tilde X)$ such that we have
 \[
  f^*(\sigma) = \tilde\sigma | f^{-1}(X_\reg) \in \Omega^2_{\tilde X}(f^{-1}(X_\reg)).
 \]
\end{enumeratei}
Conditions \eqref{i:symp-cl} and \eqref{i:symp-nondeg} are the common conditions of closedness and nondegeneracy which contain, so to speak, the heart of symplecticity. Condition \eqref{i:symp-ext} says that $\sigma$ extends to a global $2$-differential when pulled back along a resolution of singularities; this should be seen as a property moderating the nature of the singularities of $X$.

We modeled our above definition of symplecticity for complex spaces (\resp finite type $\C$-schemes) after two sources: Firstly, A.~Beauville (probably reverting to \cite{Katata82}) introduced a notion of ``symplectic singularities'' in \cite[Definition 1.1]{Be00}. In fact, Beauville's concept of a symplectic singularity is precisely the localization of our (global) notion of symplecticity, that is, one says that $X$ \emph{has a symplectic singularity} at $p$ when there exists an open neighborhood $U$ of $p$ in $X$ such that the open subspace $X|U$ of $X$ is symplectic. Secondly, in \cite{Na01}, Y.~Namikawa defines $X$ to be a \emph{projective symplectic variety} when $X$ is a normal, projective complex algebraic variety with rational Gorenstein singularities such that there exists $\sigma \in \Omega^2_X(X_\reg)$ satisfying the nondegeneracy condition \eqref{i:symp-nondeg}. Moreover, in \cite{Na01a}, Namikawa calls $X$ a \emph{symplectic variety} when $X$ is a symplectic (in our sense), compact complex space of Kähler type.

In our view, historically, the interest in either one of the mentioned forms of symplecticity for finite type $\C$-schemes or complex spaces was triggered by the interest in symplectic complex manifolds, more specifically, the interest in ``irreducible symplectic'' complex manifolds. Let us briefly review this notion. For us, a complex manifold is by definition a smooth complex space. Therefore the already defined concept of symplecticity applies. Note that in case $X$ is a complex manifold, a symplectic structure on $X$ is (in particular) a global $2$-differential on $X$, \iev an element of $\Omega^2_X(X)$, since $X_\reg = X$. Besides, when $X$ is a complex manifold, the extension condition \eqref{i:symp-ext} is fulfilled for any $\sigma \in \Omega^2_X(X)$, so that one may drop it when working exclusively with manifolds. Thus our definition of symplecticity recovers the original definition of symplecticity for complex manifolds established in the 1970s and early 1980s by F.~Bogomolov (\cf \cite{Bo74} and \cite{Bo78}) and A.~Beauville (\cf \cite[``Définition'' in \S4]{Be83}): a symplectic structure on a complex manifold is nothing but a closed, everywhere nondegenerate holomorphic $2$-form on it. Now following Beauville's terminology in \loccit, Proposition 4, a compact complex manifold of Kähler type is said to be \emph{irreducible symplectic} when it is simply connected and there exists, up to scaling, a unique symplectic structure on it. An easy argument shows that for $X$ of strictly positive dimension, \iev $\dim(X)>0$, the uniqueness condition on the symplectic structure may be replaced by requiring that
\[
 \dim_\C(\Omega^2_X(X)) = 1.
\]
In our opinion, the most compelling reason for considering irreducible symplectic manifolds as interesting or special is presented by Beauville's and Bogomolov's decomposition theorem (\cf \cite[Théorème 2]{Be83}), which exhibits irreducible symplectic manifolds as one of two nontrivial building blocks of compact Kähler manifolds with vanishing first real Chern class.

\begin{theorem}[Beauville-Bogomolov Decomposition]
 \label{t:bbdecomp}
 Let $X$ be a connected, compact, Kähler type complex manifold such that $c_1(X)_\RR = 0$ in $\H^2(X,\RR)$. Then there exists a unique natural number $k$ and, up to permutation and isomorphism, unique finite (possibly empty) tuples $(Y_1,\dots,Y_r)$ and $(Z_1,\dots,Z_s)$ of simply connected Calabi-Yau manifolds of dimension $\geq3$ and irreducible symplectic manifolds of dimension $\geq2$, respectively, such that the universal cover of $X$ is isomorphic to the product
 \[
  \C^k \times \prod_{i=1}^r Y_i \times \prod_{j=1}^s Z_j.
 \]
 Moreover, there exists a complex torus $T$ and a finite étale cover $X'\to X$ such that
 \[
  X' = T \times \prod_{i=1}^r Y_i \times \prod_{j=1}^s Z_j.
 \]
\end{theorem}

For the purposes of Theorem \ref{t:bbdecomp}, a \emph{Calabi-Yau manifold} is understood to be a compact, connected complex manifold $Y$ of Kähler type with trivial canonical bundle and $\H^0(Y',\Omega^p_{Y'}) = 0$ for all $0 < p < \dim(Y)$ and all finite étale covers $Y' \to Y$.

From the moment Theorem \ref{t:bbdecomp} had been proven, irreducible symplectic manifolds became popular objects of study, and so they remain to the day. Possibly the most striking fact about the on-going research is that the producing of ``new'' examples of irreducible symplectic manifolds appears to be the most intractable problem of all. Let us elaborate a little on this point. It is clear from the start that there exist no odd dimensional irreducible symplectic manifolds. In the lowest nontrivial dimension the picture is very clear, too. A compact, connected complex manifold of dimension $2$ admits, up to scaling, a unique symplectic structure if and only if its canonical bundle (or sheaf) is trivial. So, the irreducible symplectic manifolds of dimension $2$ are precisely the K3 surfaces. In particular, we see that any two irreducible symplectic manifolds of dimension $2$ are diffeomorphic, as a matter of fact, even deformation equivalent (\cf \cite[Theorem 13]{Ko64}). In higher dimensions, our knowledge can be subsumed as follows. For any even natural number $n>2$, Beauville constructs irreducible symplectic manifolds $H^n$ and $K^n$ of dimension $n$ (starting from a K3 surface and a $2$-dimensional torus, respectively) such that $\rb_2(H^n) = 23$ and $\rb_2(K^n) = 7$  (\cf \cite{Be83}). Due to the discrepancy in the Betti numbers, $H^n$ and $K^n$ are not homotopically equivalent, whence not homeomorphic, whence not diffeomorphic, whence not deformation equivalent. Moreover, $H^n$ and $K^n$ are not bimeromorphically equivalent as two bimeromorphically equivalent compact complex manifolds with trivial canonical bundles have the same second Betti number. K.~O'Grady constructed in \cite{OG99} and \cite{OG03} irreducible symplectic manifolds $M^{10}$ and $M^6$ of dimensions $10$ and $6$, respectively, such that $\rb_2(M^{10}) \geq 24$ and $\rb_2(M^6) = 8$. It is a standing question whether there exists an irreducible symplectic manifold (of dimension $\geq 4$) which is not deformation equivalent to any of the mentioned examples. On opposite end, examples seeming rather scarce, one might ask whether, for any given (even) natural number $n$ ($\geq 4$), there are only finitely many classes of irreducible symplectic manifolds of dimension $n$ modulo deformation equivalence. For these as well as further questions and conjectures circling around the topic of irreducible symplectic manifolds, we refer to Beauville's beautiful ``problem list'' \cite{Be11}.

Now singular symplectic complex spaces (or else $\C$-schemes) occur naturally in the constructions of irreducible symplectic manifolds. For instance, when $S$ is a K3 surface (either in the analytic or the algebraic sense) and $r$ is a natural number ($\geq 2$), then the Douady space of $0$-dimensional closed subspaces of length $r$ of $S$ (\cf \cite{Do66}) or the $r$-th punctual Hilbert scheme of $S$ (\cf \cite{SB6.221}), denoted $S^{[r]}$ or $\Hilb^r_{S/\C}$, is an irreducible symplectic manifold of dimension $2r$ by \cite[Théorème 3]{Be83}, yielding precisely the $H^{2r}$ alluded to above. In order to prove this assertion, Beauville utilizes, in \loccit, the canonical morphism $f\colon S^{[r]} \to S^{(r)}$, where $S^{(r)}$ stands for the $r$-th symmetric power of $S$. Concretely, he verifies that, for any symplectic structure $\sigma$ on $S$, a restriction of the $2$-form
\[
 \pr_0^*(\sigma) + \dots + \pr_{r-1}^*(\sigma)
\]
on the $r$-fold self-product of $S$ descends to a symplectic structure on the regular locus of $S^{(r)}$. In addition, one observes that $f$ is a resolution of singularities such that the pullback along $f$ of the symplectic structure on $(S^{(r)})_\reg$ admits an extension to a symplectic structure on all of $S^{[r]}$. From this we conclude that $S^{(r)}$ is indeed a nonsmooth symplectic complex space (or else complex algebraic variety). O'Grady's constructions in \cite{OG99} and \cite{OG03} feature similar natural occurances of singular symplectic spaces. In \cite{OG99}, for instance, O'Grady considers the moduli space $M$ of rank-$2$ Gieseker semistable, torsion-free sheaves with Chern classes $c_1=0$ and $c_2=4$ on a (suitably) polarized K3 surface. He shows that $M$ is a projective $\C$-scheme whose regular locus carries a symplectic structure and manages to construct a resolution of singularities $f\colon \tilde M \to M$ such that the pullback along $f$ of a symplectic structure on $M_\reg$ extends to a symplectic structure on $\tilde M$. Therefore, we see that $M$ is a (singular) symplectic $\C$-scheme.

Motivated by the examples of the previous paragraph, we make the following observation. When $X$ is a normal complex space (\resp a normal complex algebraic variety) and $f\colon \tilde X\to X$ is a resolution of singularities such that $\tilde X$ is a symplectic manifold, then $X$ is a symplectic complex space (\resp algebraic variety). The proof is clear. A resolution $f$ like the above is called a \emph{symplectic resolution}. In the business of trying to fabricate new irreducible symplectic manifolds, symplectic resolutions, respectively singular spaces admitting symplectic resolutions, are needless to say very desirable. Mind, however, that our definition of symplecticity allows for much more general spaces. In fact, we like to think about symplectic spaces as being spaces with well-behaved singularities whose regular loci admit symplectic structures. In the algebraic context the following result due to Namikawa makes this intuition precise (\cf \cite[Theorem 6]{Na01a}): A normal, projective complex algebraic variety $X$ is symplectic if and only if $X$ has rational Gorenstein singularities and there exists a nondegenerate (\egv in the sense of condition \eqref{i:symp-nondeg} being satisfied) $2$-form $\sigma$ on $X_\reg$. Putting it differently, on a projective variety the extension condition \eqref{i:symp-ext} and, as a result, also the closedness condition \eqref{i:symp-cl} come for free if we know a priori that the singularities of our variety are mild. For us, the chief reason for allowing spaces not resolvable to a symplectic manifold as symplectic spaces lies in the fact that it is spaces like this which play the role of irreducible symplectic manifolds in conjectural generalized versions of the Beauville-Bogomolov Decomposition Theorem (\cf \egv \cite[Open Problems, \S6]{Katata82}, \cite{GrKePe11}).

\subsection*{The local Torelli theorem}

The starting point of our research was Namikawa's paper \cite{Na01a}, and especially Theorem 8 thereof. We felt that in order to study irreducible symplectic varieties (whatever that was to mean precisely), one had to gain, in the first place, a thorough understanding of the (local) deformation theory of these varieties. Trying to write up a rigorous proof for the assertions made in Theorem 8 (3) of \loccit---unfortunately Namikawa only states that the arguments should be similar to Beauville's classical ones---, we encountered several problems. Our foremost problem was the lack of an adequate extension or analogue of Griffiths's theory of period mappings for families of compact Kähler manifolds, as developed in \cite[II.1]{Gr68a}, to either the context of families of (possibly singular) compact Kähler spaces or the context of (possibly noncompact) Kähler manifolds. When we discovered that N.~Katz and T.~Oda construct, in \cite{KaOd68}, a canonical flat connection on the relative algebraic de Rham module $\sH^n(f) := \R^nf_*(\Omega^\kdot_f)$ for any smooth morphism of $k$-schemes $f\colon X\to S$ with smooth base, $k$ being an arbitrary field and $n$ an integer, we knew we had found the right angle to tackle our problem. As a matter of fact, by transferring Katz's and Oda's ideas from \loccit~and \cite{Ka72} to the analytic category, we were able to devise a theory of period mappings of Hodge-de Rahm type for families of (not necessarily compact) complex manifolds which, in a sense, constitutes a generalization Griffiths's theory. In turn, employing our theory of period mappings, we achieved to prove a local Torelli theorem for symplectic varieties $X$ which, in opposition to Namikawa's \cite[Theorem 8 (3)]{Na01a}, does not rely on the projectivity nor the $\Q$-factoriality of $X$. We rather think that our theorem clearly exhibits the interaction of $\Q$-factoriality with the topology of the local deformations of a projective symplectic variety.

Let $g\colon \cY\to S$ be a submersive (yet not necessarily proper) morphism of complex manifolds. Fix integers $n$ and $p$ as well as an element $t\in S$. Assume that the relative algebraic de Rahm module $\sH^n(g)$ is a vector bundle, \iev a locally finite free module, on $S$ which is compatible with base change in the sense that, for all $s\in S$, the de Rham base change map
\[
 \phi^n_{g,s} \colon \C \otimes_{\O_{S,s}} (\sH^n(g))_s \to \sH^n(\cY_s)
\]
is an isomorphism of complex vector spaces. In Chapter \ref{ch:peri}, we define a map of sheaves on $S_\top$,
\[
 \nabla^n_\GM(g) \colon \sH^n(g) \to \Omega^1_S \otimes_S \sH^n(g),
\]
going by the name of \emph{Gauß-Manin connection}, in the very spirit of Katz-Oda \cite{KaOd68}. We observe that the kernel of $\nabla^n_\GM(g)$ makes up a locally constant sheaf of $\C_S$-modules on $S_\top$ whose stalks are isomorphic to the $n$-th de algebraic Rham cohomologies of the fibers of $g$ via the inclusion $H \subset \sH^n(g)$ and the de Rham base change maps $\phi^n_{g,s}$. That way, in case $S$ is simply connected, one constructs a period mapping $\cP^{p,n}_t(g)$ by transporting the Hodge filtered pieces $\F^p\sH^n(\cY_s) \subset \sH^n(\cY_s)$ to $\sH^n(\cY_t)$ along the global sections of $H$. When we require the relative Hodge filtered piece $\F^p\sH^n(g)$ to be a vector subbundle of $\sH^n(g)$ on $S$ which is compatible with base change (in the appropriate sense), the period mapping is a holomorphic map
\[
 \cP^{p,n}_t(g) \colon S \to \Gr(\sH^n(\cY_t)),
\]
where $\Gr(V)$ denotes the Grassmannian, regarded as a complex space, of a finite dimensional complex vector space $V$.

\begin{theorem}
 \label{t:froeintro}
 Let $f\colon \cX\to S$ be a proper, flat morphism of complex spaces such that $S$ is a complex manifold and, for all $s\in S$, the fiber $\cX_s$ has rational singularities, is of Kähler type, and satisfies $\codim(\Sing(\cX_s),\cX_s)\geq 4$. Define $g\colon \cY\to S$ to be the restriction of $f$ to the set of points of $\cX$ at which $f$ is submersive, and put
 \[
  I := \{(\nu,\mu) \in \Z\times\Z : \nu+\mu \leq 2\}.
 \]
 \begin{enumerate}
  \item For all $(p,q)\in I$, the Hodge module $\sH^{p,q}(g) := \R^qg_*(\Omega^p_g)$ is a locally finite free module on $S$ and compatible with base change.
  \item The Frölicher spectral sequence of $g$ degenerates in entries $I$ at sheet $1$ in $\Mod(S)$.
 \end{enumerate}
\end{theorem}

Now let $f$ and $g$ be as in Theorem \ref{t:froeintro}. Then as a corollary of the theorem, we see that, for all integers $n\leq 2$, the de Rahm module $\sH^n(g)$ is a vector bundle on $S$ and compatible with base change; moreover, for all integers $p$, the Hodge filtered piece $\F^p\sH^n(g)$ is a vector subbundle of $\sH^n(g)$ on $S$ and compatible with base change. Thus, in case $S$ is simply connected, the period mappings $\cP^{p,n}_t(g)$ are defined for all $p$ and $n$ as above and all $t\in S$. This enables us to formulate the following local Torelli theorem for irreducible symplectic spaces.

\begin{theorem}[Local Torelli, I]
 \label{t:ltintro}
 Let $X$ be a compact, symplectic complex space of Kähler type such that $\codim(\Sing(X),X) \geq 4$ and $\dim_\C(\Omega^2_X(X_\reg)) = 1$. Let $f\colon\cX\to S$ be a proper, flat morphism of complex spaces and $t\in S$ such that $X \iso \cX_t$ and $f$ is semi-universal in $t$. Assume that $S$ is a simply connected complex manifold and that the fibers of $f$ are of Kähler type, have rational singularities and singular loci of codimension $\geq 4$. Then the period mapping
 \[
  \cP^{2,2}_t(g) \colon S \to \Gr(\sH^2(\cY_t)),
 \]
 where $g\colon \cY\to S$ denotes the restriction of $f$ to the set of points of $\cX$ at which $f$ is submersive, is an immersion in codimension $1$ at $t$.
\end{theorem}

As the period mapping $\cP^{2,2}_t(g)$ which arises in Theorem \ref{t:ltintro} might not be all too tangible at first sight, we would like to enrich the theorem by drafting a supplement to it. For that matter, let $f\colon \cX\to S$ be just any proper morphism of complex spaces with fibers of Kähler type (or, more generally, fibers of Fujiki class $\sC$) and simply connected base. Let $n$ be an integer and suppose that the sheaf $\R^nf_*(\C_\cX)$ is locally constant on $S_\top$. Then, for all $s_0,s_1 \in S$, we obtain an isomorphism of complex vector spaces
\[
 \phi_{s_0,s_1} \colon \H^n(\cX_{s_0},\C) \to \H^n(\cX_{s_1},\C)
\]
by passing through the global sections of $\R^nf_*(\C_\cX)$ and the base change maps
\[
 (\R^nf_*(\C_\cX))_s \to \H^n(\cX_s,\C),
\]
which are bijective due to the properness of $f$. Further on, the cohomologies of the fibers of $f$ carry mixed Hodge structures by \cite[(1.4)]{Fu80} (see also \cite{De71}, \cite{De74}). So, for any integer $p$ and any $t\in S$, we may define a period mapping
\[
 \cP^{p,n}_t(f)_\MHS \colon S \to \Gr(\H^n(\cX_t,\C)), \quad s \mto \phi_{s,t}[\F^p\H^n(\cX_s)],
\]
where $\F^p\H^n(\cX_s) \subset \H^n(\cX_s,\C)$ denotes the $p$-th piece of the Hodge filtration of the mixed Hodge structure on the cohomology in degree $n$ of $\cX_s$.

Observe that in the situation of Theorem \ref{t:ltintro}, the sheaf $\R^2f_*(\C_\cX)$ on $S_\top$ need not be locally constant; the dimension of $\H^2(\cX_s,\C)$ might indeed jump when moving from $s=t$ to a nearby point in $S$. Therefore generally we cannot speak of $\cP^{2,2}_t(f)_\MHS$. If, however, we are lucky and $\R^2f_*(\C_\cX)$ is a locally constant sheaf on $S_\top$, the following theorem applies.

\begin{theorem}[Local Torelli, II] \label{t:lt2intro}
 Let $X$ be as in Theorem \ref{t:ltintro}. Let $f\colon \cX\to S$ be a proper, flat morphism of complex spaces and $t\in S$ such that $X\iso \cX_t$. Assume that $S$ is a simply connected complex manifold and that the fibers of $f$ are of Kähler type, have rational singularities and singular loci of codimension $\geq 4$. Define $g\colon \cY\to S$ as in Theorem \ref{t:ltintro} and assume that the tangent map
 \[
  \T_t(\cP^{2,2}_t(g)) \colon \T_S(t) \to \T_{\Gr(\sH^2(\cY_t))}(\F^2\sH^2(\cY_t))
 \]
 is an injection with $1$-dimensional cokernel. Moreover, assume that $\R^2f_*(\C_\cX)$ is a locally constant sheaf on $S_\top$.
 \begin{enumerate}
  \item The period mapping $\cP := \cP^{2,2}_t(f)_\MHS$ is a holomorphic map
  \[
   \cP \colon S \to \Gr(\H^2(\cX_t,\C)).
  \]
  \item When $Q_{\cX_t} \subset \Gr(1,\H^2(\cX_t,\C))$ denotes the zero locus of the Beauville-Bogomolov form of $\cX_t$ (see below), then $\cP$ factors uniquely through a morphism of complex spaces
  \[
   \bar\cP \colon S \to Q_{\cX_t},
  \]
  which is a biholomorphism at $t$.
  \item The mapping
  \begin{equation} \label{e:lt2intro-h2} \tag{$*$}
   \H^2(X,\C) \to \H^2(X_\reg,\C)
  \end{equation}
  induced by the inclusion $X_\reg \subset X$ is a bijection.
 \end{enumerate}
\end{theorem}

In view of Namikawa's approach towards a local Torelli theorem in \cite{Na01a}, it seems quite remarkable that in Theorem \ref{t:lt2intro} we do not presuppose the morphism of complex vector spaces \eqref{e:lt2intro-h2} to be an isomorphism, but rather derive this fact as a consequence. Turning the argument around, we conclude that given a symplectic space $X$ as in Theorem \ref{t:ltintro} or Theorem \ref{t:lt2intro} such that $\H^2(X,\C)$ does not agree with $\H^2(X_\reg,\C)$ (dimensionwise), there exists a deformation of $X$ which changes the topology of $X$. On the other hand, when $\H^2(X,\C)$ and $\H^2(X_\reg,\C)$ do agree, Namikawa's results in \cite{Na06} suggest, at least in case $X$ is projective, that any deformation of $X$ is locally topologically trivial.

\subsection*{The Fujiki relation}

Let $X$ be a compact, connected, and symplectic complex space such that $\Omega^2_X(X_\reg)$ is $1$-dimensional over the field of complex numbers. Then, generalizing Beauville's definition in \cite[p.~772]{Be83}, we introduce a complex quadratic form $q_X$ on $\H^2(X,\C)$, called the \emph{Beauville-Bogomolov form} of $X$, by first passing to a resolution of singularities $\tilde X\to X$ and then using the assignment
\[
 \tilde a \mto \frac{r}{2} \int_{\tilde X} \left(w^{r-1}\bar{w}^{r-1}\tilde a^2\right) + (r-1) \int_{\tilde X} \left(w^{r-1}\bar w^r\tilde a\right) \int_{\tilde X} \left(w^r\bar w^{r-1}\tilde a\right)
\]
for $\tilde a \in \H^2(\tilde X,\C)$, where $w \in \H^2(\tilde X,\C)$ is the class of a closed $2$-differential $\tilde\sigma \in \Omega^2_{\tilde X}(\tilde X)$ which is normed in the sense that
\[
 \int_{\tilde X} w^r\bar w^r = 1
\]
and $r$ denotes the unique natural number satisfying $2r = \dim(X)$; note that we have $r \neq 0$ in consequence.

The following result embodies an extension of A.~Fujiki's classical \cite[Theorem 4.7]{Fu87} to the context of singular symplectic spaces.

\begin{theorem}[Fujiki Relation] \label{t:frintro}
 Let $X$ be a compact, connected, and symplectic complex space of Kähler type such that $\Omega^2_X(X_\reg)$ is of dimension $1$ over the field of complex numbers and $\codim(\Sing(X),X) \geq 4$. Then, for all $a \in \H^2(X,\C)$, we have
 \[
  \int_X a^{2r} = \binom{2r}{r} \cdot (q_X(a))^r,
 \]
 where $r$ denotes half the dimension of $X$.
\end{theorem}

The validity of Theorem \ref{t:frintro} for irreducible symplectic manifolds $X$ has proven a valuable asset in a number of efforts to deduce further properties of such $X$, \egv in Matsushita's work on fiber space structures (\cf \cite{Ma99a}, \cite{Ma01a}), so that we hope Theorem \ref{t:frintro} presents a fertile ground for further research in the singular realm too.

\subsection*{Organization of the text}

Our work comprises three chapters labeled ``\ref{ch:peri}'', ``\ref{ch:froe}'', and ``\ref{ch:symp}'' as well as two supplementary chapters (or ``appendices'') labeled ``\ref{ch:foun}'' and ``\ref{ch:tool}''.  In Chapter \ref{ch:peri}, we explain our theory of period mappings of Hodge-de Rham type for families of (not necessarily compact) complex manifolds, which we have already touched upon above. In Chapter \ref{ch:froe}, we work out circumstances under which the Frölicher spectral sequence of a submersive morphism of complex manifolds $g\colon \cY\to S$ degenerates in specific entries. We address the question of the degeneration in an entry $(p,q) \in \Z \times \Z$ in close conjunction with the question of whether the corresponding Hodge module $\sH^{p,q}(g)$ is locally finite free on $S$ and base change compatible. Note that in Chapter \ref{ch:froe}, we prove the above Theorem \ref{t:froeintro}, which is essential in order to apply the results on period mappings of Chapter \ref{ch:peri} to the study of symplectic spaces in Chapter \ref{ch:symp}. In Chapter \ref{ch:symp}, we deal with symplectic complex spaces; we prove the Local Torelli of Theorem \ref{t:ltintro} together with its add-on, Theorem \ref{t:lt2intro}. Moreover, we establish the Fujiki Relation of Theorem \ref{t:frintro}.

Chapter \ref{ch:foun} shall lay the foundations for the formulation of unambiguous statements and rigorous proofs in the bulk of the text (which is made up of Chapters \ref{ch:peri}, \ref{ch:froe}, and \ref{ch:symp}). Thus, in Chapter \ref{ch:foun}, we basically fix terminology and notation; we will not prove anything there. Note that logically, Chapter \ref{ch:foun} ought to be placed in front of Chapters \ref{ch:peri}, \ref{ch:froe}, and \ref{ch:symp} rather than after. Nevertheless, have we decided to supply Chapter \ref{ch:foun} as an appendix so as not to bore readers who want to get right down to business.

In our final Chapter \ref{ch:tool}, we show, for one thing, how to conceptually construct the various base change maps that are used throughout Chapters \ref{ch:peri}, \ref{ch:froe}, and \ref{ch:symp}; for another, we apply methods from mixed Hodge theory in order to establish certain properties of complex spaces with rational singularities. Beware that, as appendices, Chapters \ref{ch:foun} and \ref{ch:tool} are not really intended to be read through in one go. We rather suggest that the reader consult the appendices upon wish or need while studying a different part of the text.

Each of Chapters \ref{ch:peri}, \ref{ch:froe}, and \ref{ch:symp} is virtually self-contained---neglecting occasional references to the appendices (and references to outside sources of course); it should be possible to read any one of these three chapters without having a particular knowledge of the other two. In fact, the sole logical (meta-)dependence between Chapters \ref{ch:peri}, \ref{ch:froe}, and \ref{ch:symp} lies in Chapters \ref{ch:peri} and \ref{ch:froe} (individually) flowing into Chapter \ref{ch:symp}. We have tried to design this logical dependence of Chapter \ref{ch:symp} on Chapters \ref{ch:peri} and \ref{ch:froe} as sharp-edged and condensed as possible, so that essentially only for the deduction of the Local Torelli one has to invoke one theorem from each of the latter two chapters.

\renewcommand{\thesubsection}{\thesection.\arabic{subsection}}

\chapter{Period mappings for families of complex manifolds}
\label{ch:peri}

\setcounter{subsection}{0}

Consider a family of compact complex manifolds $f\colon X\to S$, by which we mean that $X$ and $S$ are complex manifolds and $f$ is a proper, submersive holomorphic map between them. Then by Ehresmann's fibration theorem, $f\colon X\to S$ is a locally topologically trivial family (as a matter of fact, even a locally $\mathscr C^\infty$ trivial family). In particular, for any natural number (or else integer) $n$, we know that $\R^nf_*(\C_X)$ is a locally constant sheaf on the topological space $S_\top$ such that, for any $s\in S$, the ``base change map''
\[
 (\R^nf_*(\C_X))_s \to \H^n(X_s,\C)
\]
is a bijection. Let us assume that the complex manifold $S$ is simply connected. Then $\R^nf_*(\C_X)$ is yet a constant sheaf on $S_\top$ and, for all $s\in S$, the canonical mapping from the set of global sections $(\R^nf_*(\C_X))(S)$ to the stalk $(\R^nf_*(\C_X))_s$ is one-to-one and onto. Thus by passing through base changes and the set $(\R^nf_*(\C_X))(S)$, we obtain, for any two elements $s_0,s_1\in S$, a bijection
\[
 \phi^n_{s_0,s_1} \colon \H^n(X_{s_0},\C) \to \H^n(X_{s_1},\C).
\]
Suppose that, for all $s\in S$, the complex manifold $X_s$ is of Kähler type, and fix an element $t\in S$. We define $\cP^{p,n}_t$, for any natural number (or else integer) $p$, to be the unique function on $S$ satisfying
\[
 \cP^{p,n}_t(s) = \phi^n_{s,t}[\F^p\H^n(X_s)]
\]
for all $s\in S$, where $\F^p\H^n(X_s)$ denotes the $p$-th piece of the Hodge filtration on $n$-th cohomology of $X_s$ and we use, for sake of clarity, square brackets to denote the image of a certain set under a function. $\cP^{p,n}_t$ is called a \emph{period mapping} for the family $f$. The following result is a variant of P.~Griffiths's \cite[Theorem (1.1)]{Gr68a}\footnote{The attentive reader will notice that Griffiths's construction of the period mapping is different from ours, mainly as he directly employs a $\mathscr C^\infty$ trivialization $X_t\times S \to X$ of $f$ over $S$; moreover, several conventions of \cite{Gr68a}, e.g., regarding cohomology, do not match ours.}.

\begin{theorem}
 \label{t:pmhol}
 Under the above hypotheses, $\cP^{p,n}_t$ is a holomorphic mapping from $S$ to the Grassmannian $\Gr(\H^n(X_t,\C))$.
\end{theorem}

Note that, as we state it, Theorem \ref{t:pmhol} comprises the fact that the spaces $\F^p\H^n(X_s)$ are of a constant finite dimension when $s$ varies through $S$.

We would like to recall another theorem of Griffiths's which is closely related to Theorem \ref{t:pmhol}. To that end, put $q:=n-p$ and let
\[
 \gamma \colon \H^1(X_t,\Theta_{X_t}) \to \Hom(\H^q(X_t,\Omega^p_{X_t}),\H^{q+1}(X_t,\Omega^{p-1}_{X_t}))
\]
be the morphism of complex vector spaces which is obtained by means of tensor-hom adjunction from the composition
\[
 \H^1(X_t,\Theta_{X_t}) \otimes_\C \H^q(X_t,\Omega^p_{X_t}) \overset\cupp\to \H^{q+1}(X_t,\Theta_{X_t}\otimes _{X_t}\Omega^p_{X_t}) \to \H^{q+1}(X_t,\Omega^{p-1}_{X_t})
\]
of the evident cup product morphism and the $\H^{q+1}(X_t,-)$ of the sheaf-theoretic contraction morphism
\[
 \Theta_{X_t}\otimes_{X_t}\Omega^p_{X_t} \to \Omega^{p-1}_{X_t}.
\]
Since $X_t$ is a compact, Kähler complex manifold, the Frölicher spectral sequence of $X_t$ degenerates at sheet $1$ and we have, for any $\nu\in\Z$, an induced morphism of complex vector spaces
\[
 \psi^{\nu} \colon \F^{\nu}\H^n(X_t)/\F^{\nu+1}\H^n(X_t) \to \H^{n-\nu}(X_t,\Omega^{\nu}_{X_t}),
\]
which is in fact an isomorphism. Define $\alpha$ to be the composition of the quotient morphism
\[
 \F^p\H^n(X_t) \to \F^p\H^n(X_t)/\F^{p+1}\H^n(X_t)
\]
and $\psi^p$. Dually, define $\beta$ to be the composition of $(\psi^{p-1})^{-1}$ and the morphism
\[
 \F^{p-1}\H^n(X_t)/\F^p\H^n(X_t) \to \H^n(X_t,\C)/\F^p\H^n(X_t)
\]
which is obtained from the inclusion $\F^{p-1}\H^n(X_t) \subset \H^n(X_t,\C)$ by quotienting out $\F^p\H^n(X_t)$. Further, denote
\[
 \KS \colon \T_S(t) \to \H^1(X_t,\Theta_{X_t})
\]
the Kodaira-Spencer map for the family $f$ with basepoint $t$ (\cf Notation \ref{not:ksm}) and write
\[
 \theta \colon \T_{\Gr(\H^n(X_t,\C))}(\F^p\H^n(X_t)) \to \Hom(\F^p\H^n(X_t),\H^n(X_t,\C)/\F^p\H^n(X_t))
\]
for the isomorphism which is induced by the canonical open immersion
\[
 \Hom(\F^p\H^n(X_t),E) \to \Gr(\H^n(X_t,\C))
\]
of complex manifolds, where $E$ is a complex vector subspace of $\H^n(X_t,\C)$ such that $\H^n(X_t,\C) = \F^p\H^n(X_t) \oplus E$ (\cf Notation \ref{not:theta}). As an adaptation of \cite[Proposition (1.20) or Theorem (1.22)]{Gr68a} we formulate

\begin{theorem}
 \label{t:pmgriffiths}
 Let $f\colon X\to S$, $n$, $p$, and $t$ be as above and define $\alpha$, $\beta$, $\gamma$, $\KS$, and $\theta$ accordingly. Then the following diagram commutes in $\Mod(\C)$:
 \begin{equation} \label{e:pmgriffiths}
  \xymatrix{
   \T_S(t) \ar[r]^{\KS} \ar[dd]_{\T_t(\cP^{p,n}_t)} & \H^1(X_t,\Theta_{X_t}) \ar[d]^\gamma \\
   & \Hom(\H^q(X_t,\Omega^p_{X_t}),\H^{q+1}(X_t,\Omega^{p-1}_{X_t})) \ar[d]^{\Hom(\alpha,\beta)} \\
   \T_{\Gr(\H^n(X_t,\C))}(\F^p\H^n(X_t)) \ar[r]_-{\theta} & \Hom(\F^p\H^n(X_t),\H^n(X_t,\C)/\F^p\H^n(X_t)) 
  }
 \end{equation}
\end{theorem}

Now, the objective of this chapter is to state and prove a proposition analoguous to Theorem \ref{t:pmgriffiths}---possibly even, in a sense, generalizing Theorem \ref{t:pmgriffiths}---for families of not necessarily compact manifolds, \iev for submersive, yet not necessarily proper, morphisms of complex manifolds $f\colon X\to S$. More specifically, we are interested in submersive morphisms of complex manifolds $f\colon X\to S$ such that the relative algebraic de Rham cohomology sheaf $\sH^n(f)$ ($:= \R^nf_*(\Omega^\kdot_f)$ equipped with its canonical $\O_S$-module structure, \cf Notation \ref{not:dr}), for some fixed integer $n$, is a locally finite free module on $S$ which is compatible with base change in the sense that, for all $s\in S$, the de Rham base change map
\[
 \phi^n_{f,s} \colon \C \otimes_{\O_{S,s}} (\sH^n(f))_s \to \sH^n(X_s)
\]
is an isomorphism of complex vector spaces. We observe that the kernel $H$ of the Gauß-Manin connection
\[
 \nabla^n_\GM(f) \colon \sH^n(f) \to \Omega^1_S \otimes_S \sH^n(f),
\]
which we are going to introduce in the spirit of \cite{KaOd68} (\cf Notation \ref{not:gm}), makes up a locally constant sheaf of $\C_S$-modules on $S_\top$ whose stalks are isomorphic to the $n$-th de Rham cohomologies of the fibers of $f$ via the inclusion $H \subset \sH^n(f)$ and the de Rham base change maps. That way, in case the complex manifold $S$ is simply connected, we construct, for any integer $p$ and any basepoint $t\in S$, a period mapping $\cP^{p,n}_t(f)$ by transporting the Hodge filtered pieces $\F^p\sH^n(X_s) \subset \sH^n(X_s)$ to $\sH^n(X_t)$ along the global sections of $H$. When we require the relative Hodge filtered piece $\F^p\sH^n(f)$ to be a vector subbundle of $\sH^n(f)$ on $S$ which is compatible with base change (in an appropriate sense), the holomorphicity of the period mapping
\[
 \cP^{p,n}_t(f) \colon S \to \Gr(\sH^n(X_t))
\]
is basically automatic.

Eventually, we find that certain properties that can be expressed exclusively as degeneration properties for the Frölicher spectral sequences of $f$ and $X_t$ ensure the possibility to define morphisms $\alpha$ and $\beta$ such that a diagram similar to the one in \eqref{e:pmgriffiths}, namely the one in \eqref{e:pm}, commutes in $\Mod(\C)$. It is noteworthy that we do not assume our family $f\colon X\to S$ to be locally topologically (or $\sC^\infty$) trivial, neither do we assume $\R^nf_*(\C_X)$ to be a locally constant sheaf (which is compatible with base change).

Chapter \ref{ch:peri} is organized as follows. Our ultimate results are Theorem \ref{t:pm} as well as its hopefully more accessible corollary Theorem \ref{t:pmclassic}. The chapter's sections come in two groups: the final \oldS\S\ref{s:pm0} and \ref{s:pm} deal with the concept period mappings whereas the inital \oldS\S\ref{s:kozlambdap}--\ref{s:ksgm} don't. The upshot of \oldS\S\ref{s:kozlambdap}--\ref{s:ksgm}, besides various constructions and notation like, for instance, for the Gauß-Manin connection, is Theorem \ref{t:grgmcc}. Theorem \ref{t:grgmcc} is in turn a special case of Theorem \ref{t:grgmcc0}, whose proof \S\ref{s:cdlemma} is consecrated to. \oldS\S\ref{s:kozlambdap}--\ref{s:connhom} are preparatory for \S\ref{s:cdlemma}.

We would like to point out that the overall point of view we are adopting here is gravely inspired by N.~Katz's and T.~Oda's works \cite{KaOd68} and \cite{Ka72}, which we admittedly fancy a lot. Our \oldS\S\ref{s:kozlambdap}--\ref{s:ksgm} are actually put together along the very lines of \cite[Section 1]{Ka72}. Our view on period mappings and (relative) connections which we present in \S\ref{s:pm0} is inspired by P.~Deligne's lecture notes \cite{De70}.

\section{The \texorpdfstring{$\Lambda^p$}{Lambda-p} construction}
\label{s:kozlambdap}

\emph{For the entire section, let $X$ be a commutative ringed space.} \medskip

In this section we introduce a construction which associates to a right exact triple $t$ of modules on $X$, given some integer $p$, another right exact triple of modules on $X$ denoted $\Lambda^p_X(t)$, \cf Construction \ref{con:lambdap}. This ``$\Lambda^p$ construction'' will play a central role within Chapter \ref{ch:peri} at least up to (and including) \S\ref{s:ksgm}.

The $\Lambda^p$ construction is closely related to and even essentially based upon the following notion of a ``Koszul filtration'' (\cf \cite[(1.2.1.2)]{Ka72}).

\begin{construction}[Koszul filtration]
 \label{con:koz}
 Let $p$ be an integer. Moreover, let $\alpha \colon G\to H$ be a morphism of modules on $X$. We define a $\Z$-sequence $K$ by setting, for all $i\in\Z$:
 \begin{equation} \label{e:koz-filt}
  K^i :=
  \begin{cases}
   \im\left(\wedge^{i,p-i}(H) \circ ({\wedge^i\alpha} \otimes {\wedge^{p-i}\id_H})\right), & i\geq 0, \\
   \wedge^pH, & i<0,
  \end{cases}
 \end{equation}
 where
 \[
  \wedge^{i,p-i}(H) \colon {\wedge^iH} \otimes {\wedge^{p-i}H} \to {\wedge^pH}
 \]
 denotes the wedge product morphism. We refer to $K$ as the \emph{Koszul filtration} in degree $p$ induced by $\alpha$ on $X$.

 Let us shortly verify that $K$ is indeed a descending filtration of $\wedge^pH$ by modules on $X$. Since $K^i$ obviously is a submodule of $\wedge^pH$ on $X$ for all integers $i$, it remains to show that, for all integers $i$ and $j$ with $i\leq j$, we have $K^j\subset K^i$. In case $i<0$, this is clear as then, $K^i=\wedge^pH$. Similarly, when $j>p$, we know that $K^j$ is the zero submodule of $\wedge^pH$, so that $K^j \subset K^i$ is evident. We are left with the case where $0\leq i\leq j\leq p$. To that end, denote $\phi_i$ the composition of the following obvious morphisms in $\Mod(X)$:
 \begin{equation} \label{e:koz-1}
  \begin{aligned}
   \otimes(\overbrace{G,\dots,G}^i,\overbrace{H,\dots,H}^{p-i}) & \to \T^i(G) \otimes \T^{p-i}(H) \to {\wedge^i(G)} \otimes {\wedge^{p-i}(H)} \\
   & \to {\wedge^i(H)} \otimes {\wedge^{p-i}(H)} \to {\wedge^p(H)}.
  \end{aligned}
 \end{equation}
 Then $K^i=\im(\phi_i)$ since the first and second of the morphisms in \eqref{e:koz-1} are an isomorphism and an epimorphism in $\Mod(X)$, respectively. The same holds for $i$ replaced by $j$ given that we define $\phi_j$ accordingly. Thus our claim is implied by the easy to verify identity:
 $$\phi_j = \phi_i \circ \otimes(\underbrace{\id_G,\dots,\id_G}_{i},\underbrace{\alpha,\dots,\alpha}_{j-i},\underbrace{\id_H,\dots,\id_H}_{p-j}).$$
\end{construction}

\begin{definition}\label{d:tripex}
 Let $t\colon G\to H\to F$ be a triple of modules on $X$.
 \begin{enumerate}
  \item We call $t$ \emph{left exact} (\resp \emph{right exact}) on $X$ when
  \begin{align*}
   0 \to & G \to H \to F \\
   ( \text{\resp } \qquad & G \to H \to F \to 0 )
  \end{align*}
  is an exact sequence of modules on $X$.
  \item We call $t$ \emph{short exact} on $X$ when $t$ is both left exact on $X$ and right exact on $X$.
 \end{enumerate}
\end{definition}

The upcomping series of results is preparatory for Construction \ref{con:lambdap}.

\begin{lemma}
 \label{l:wedgep}
 Let $p$ be an integer and $t\colon G\to H\to F$ a right exact triple of modules on $X$. Then the following sequence, where the first and second arrows are given by $\wedge^{1,p-1}(H) \circ (t(0,1) \otimes \wedge^{p-1}(\id_H))$ and $\wedge^p(t(1,2))$, respectively, is exact in $\Mod(X)$:
 \begin{equation} \label{e:wedgepseq}
  G \otimes \wedge^{p-1}H \to \wedge^pH \to \wedge^pF \to 0.   
 \end{equation}
\end{lemma}

\begin{proof}
 Let $x \in X$ be arbitrary. Then there exist isomorphisms of $\O_{X,x}$-modules rendering commutative in $\Mod(\O_{X,x})$ the following diagram, where the top row is obtained from \eqref{e:wedgepseq} by applying the stalk-at-$x$ functor and the bottom row is obtained from $t_x\colon G_x\to H_x\to F_x$ the same way \eqref{e:wedgepseq} is obtained from $t\colon G\to H\to F$:
 \[
  \xymatrix{
   (G \otimes_X {\wedge^{p-1}_XH})_x \ar[r] \ar[d]_\sim & (\wedge^p_XH)_x \ar[r] \ar[d]_\sim & (\wedge^p_XF)_x \ar[r] \ar[d]^\sim & 0_x \ar[d]^\sim \\
   G_x \otimes_{\O_{X,x}} {\wedge^{p-1}_{\O_{X,x}}H_x} \ar[r] & {\wedge^p_{\O_{X,x}}H_x} \ar[r] & {\wedge^p_{\O_{X,x}}F_x} \ar[r] & 0
  }
 \]
 Now the bottom row of the diagram is exact in $\Mod(\O_{X,x})$ by \cite[Proposition A2.2, d]{Ei95}. Therefore the top row of the diagram is exact in $\Mod(\O_{X,x})$ too, whence the sequence \eqref{e:wedgepseq} is exact in $\Mod(X)$ given that $x$ was an arbitrary element of (the set underlying) $X$.
\end{proof}

\begin{proposition}
 \label{p:wedgep}
 Let $p$ be an integer and $t\colon G\to H\to F$ a right exact triple of modules on $X$. Write $K = (K^i)_{i\in\Z}$ for the Koszul filtration in degree $p$ induced by $t(0,1)\colon G\to H$. Then the following sequence is exact in $\Mod(X)$:
 \begin{equation} \label{e:wedgep-koz}
  \xymatrix@C=3pc{
   0 \ar[r] & K^1 \ar[r]^\subset & {\wedge^pH} \ar[r]^{\wedge^p(t(1,2))} & {\wedge^pF} \ar[r] & 0.
  }
 \end{equation}
\end{proposition}

\begin{proof}
 By Lemma \ref{l:wedgep}, the sequence \eqref{e:wedgepseq} is exact in $\Mod(X)$. By the definition of the Koszul filtration, \cf Construction \ref{con:koz}, the inclusion morphism $K^1 \to \wedge^pH$ is an image in $\Mod(X)$ of the morphism
 \[\wedge^{1,p-1}(H) \circ (t(0,1) \otimes \wedge^{p-1}(\id_H)) \colon G \otimes \wedge^{p-1}H \to \wedge^pH;\]
 note here that $\wedge^1_X$ equals the identity functor on $\Mod(X)$ by definition. Hence our claim follows.
\end{proof}

\begin{corollary}
 \label{c:natmor1}
 Let $p$ be an integer and $t\colon G\to H\to F$ a right exact triple of modules on $X$. Denote $K = (K^i)_{i\in\Z}$ the Koszul filtration in degree $p$ induced by $t(0,1) \colon G\to H$.
 \begin{enumerate}
  \item \label{c:natmor1-ex} There exists one, and only one, $\psi$ rendering commutative in $\Mod(X)$ the following diagram:
  \begin{equation} \label{e:natmori}
   \xymatrix{
    \wedge^pH \ar[d]_{} \ar[r]^{\wedge^p(t(1,2))} & \wedge^pF \\
    (\wedge^pH)/K^1 \ar@{.>}[ur]_{\psi}
   }  
  \end{equation}
  \item \label{c:natmor1-iso} Let $\psi$ be such that the diagram in \eqref{e:natmori} commutes in $\Mod(X)$. Then $\psi$ is an isomorphism in $\Mod(X)$.
 \end{enumerate}
\end{corollary}

\begin{proof}
 Both assertions are immediate consequences of Proposition \ref{p:wedgep}. In order to obtain a), exploit the fact that the composition of the inclusion morphism $K^1 \to \wedge^pH$ and $\wedge^p(t(1,2))$ is a zero morphism in $\Mod(X)$. In order to obtain b), make use of the fact that, by the exactness of the sequence \eqref{e:wedgep-koz}, $\wedge^p(t(1,2))$ is a cokernel in $\Mod(X)$ of the inclusion morphism $K^1 \to \wedge^pH$.
\end{proof}

\begin{proposition}
 \label{p:natmorii}
 Let $p$ be an integer and $t\colon G\to H\to F$ a right exact triple of modules on $X$. Denote $K = (K^i)_{i\in\Z}$ the Koszul filtration in degree $p$ induced by $t(0,1) \colon G\to H$.
 \begin{enumerate}
  \item \label{p:natmorii-ex} There exists a unique ordered pair $(\phi_0,\phi)$ such that the following diagram commutes in $\Mod(X)$:
  \begin{equation} \label{e:natmorii}
   \xymatrix{
    H\otimes\wedge^{p-1}H \ar[d]_{\wedge^{1,p-1}(H)} & G\otimes\wedge^{p-1}H \ar[r] \ar@{}[r]<+1ex>^{\id_G\otimes\wedge^{p-1}(t(1,2))} \ar[l] \ar@{}[l]<-1ex>_{t(0,1)\otimes\wedge^{p-1}(\id_H)} \ar@{.>}[d]_{\phi_0} & G\otimes\wedge^{p-1}F \ar@{.>}[d]^{\phi} \\
    \wedge^pH & K^1 \ar[l]^{\supset} \ar[r]_{} & K^1/K^2
   }
  \end{equation}
  \item \label{p:natmorii-epi} Let $(\phi_0,\phi)$ be an ordered pair such that the diagram in \eqref{e:natmorii} commutes in $\Mod(X)$. Then $\phi$ is an epimorphism in $\Mod(X)$.
 \end{enumerate} 
\end{proposition}

\begin{proof}
 \ref{p:natmorii-ex}). By the definition of the Koszul filtration, \cf Construction \ref{con:koz}, the inclusion morphism $K^1 \to \wedge^pH$ is an image in $\Mod(X)$ of the morphism
 $$\wedge^{1,p-1}(H) \circ (t(0,1) \otimes \wedge^{p-1}(\id_H)) \colon G \otimes \wedge^{p-1}H \to \wedge^pH.$$
 Therefore, there exists one, and only one, morphism $\phi_0$ making the left-hand square of the diagram in \eqref{e:natmorii} commute in $\Mod(X)$. By Lemma \ref{l:wedgep}, we know that the sequence \eqref{e:wedgepseq}, where we replace $p$ by $p-1$ and define the arrows as indicated in the text of the lemma, is exact in $\Mod(X)$. Tensorizing the latter sequence with $G$ on the left, we obtain yet another exact sequence in $\Mod(X)$:
 \begin{equation} \label{e:natmor2-1}
  G \otimes (G \otimes \wedge^{p-2}H) \to G \otimes \wedge^{p-1}H \to G \otimes \wedge^{p-1}F \to 0.
 \end{equation}
 The exactness of the sequence \eqref{e:natmor2-1} implies that the morphism
 \[\id_G \otimes {\wedge^{p-1}(t(1,2))} \colon G\otimes {\wedge^{p-1}H} \to G\otimes {\wedge^{p-1}F}\]
 is a cokernel in $\Mod(X)$ of the morphism given by the first arrow in \eqref{e:natmor2-1}. Besides, the definition of the Koszul filtration implies that the composition
 \[G \otimes (G\otimes \wedge^{p-2}H) \to G\otimes \wedge^{p-1}H \to K^1\]
 of the first arrow in \eqref{e:natmor2-1} with $\phi_0$ maps into $K^2 \subset K^1$, whence composing it further with the quotient morphism $K^1 \to K^1/K^2$ yields a zero morphism in $\Mod(X)$. Thus by the universal property of the cokernel, there exists a unique $\phi$ rendering commutative in $\Mod(X)$ the right-hand square in the diagram in \eqref{e:natmorii}.

 \ref{p:natmorii-epi}). Observe that by the commutativity of the left-hand square in \eqref{e:natmorii}, $\phi_0$ is a coimage of the morphism
 $$\wedge^{1,p-1}(H) \circ (t(0,1) \otimes \wedge^{p-1}(\id_H)) \colon G \otimes \wedge^{p-1}H \to \wedge^pH,$$
 whence an epimorphism in $\Mod(X)$. Moreover, the quotient morphsim $K^1\to K^1/K^2$ is an epimorphism in $\Mod(X)$. Thus the composition of $\phi_0$ and $K^1 \to K^1/K^2$ is an epimorphism in $\Mod(X)$. By the commutativity of the right-hand square in \eqref{e:natmorii}, we see that $\phi$ is an epimorphism in $\Mod(X)$.
\end{proof}

\begin{construction}[$\Lambda^p$ construction]
 \label{con:lambdap}
 Let $p$ be an integer.  Moreover, let $t\colon G\to H\to F$ be a right exact triple of modules on $X$. Write $K = (K^i)_{i\in\Z}$ for the Koszul filtration in degree $p$ induced by $t(0,1) \colon G\to H$ on $X$, \cf Construction \ref{con:koz}. Recall that $K$ is a filtration of $\wedge^pH$ by modules on $X$. We define a functor $\Lambda^p(t) \colon \mathbf3 \to \Mod(X)$ by setting, in the first place:
 \begin{align*}
  (\Lambda^p(t))(0) & := G \otimes \wedge^{p-1}F & (\Lambda^p(t))(0,0) & := \id_{G \otimes \wedge^{p-1}F} \\
  (\Lambda^p(t))(1) & := (\wedge^pH)/K^2 & (\Lambda^p(t))(1,1) & := \id_{(\wedge^pH)/K^2} \\
  (\Lambda^p(t))(2) & := \wedge^pF & (\Lambda^p(t))(2,2) & := \id_{\wedge^pF}
 \end{align*}
 Now let $\iota$ and $\pi$ be the unique morphisms such that the following diagram commutes in $\Mod(X)$:
 \begin{equation} \label{e:lambdap-1}
  \xymatrix{
   K^2 \ar[r]^{\subset} & K^1 \ar[r]^{\subset} \ar[d]_{} & \wedge^pH \ar[d]_{} \ar[dr]^{} \\ & K^1/K^2 \ar@{.>}[r]_\iota & (\wedge^pH)/K^2 \ar@{.>}[r]_\pi & (\wedge^pH)/K^1
  }
 \end{equation}
 By Proposition \ref{p:natmorii} \ref{p:natmorii-ex}), we know that there exists a unique ordered pair $(\phi_0,\phi)$ rendering commutative in $\Mod(X)$ the diagram in \eqref{e:natmorii}. Likewise, by Corollary \ref{c:natmor1} a), there exists a unique $\psi$ rendering commutative in $\Mod(X)$ the diagram in \eqref{e:natmori}. We complete our definition of $\Lambda^p(t)$ by setting:
 \begin{align*}
  (\Lambda^p(t))(0,1) & := \iota \circ \phi, & (\Lambda^p(t))(1,2) & := \psi \circ \pi, \\ (\Lambda^p(t))(0,2) & := (\psi \circ \pi) \circ (\iota \circ \phi).
 \end{align*}
 
 It is a straightforward matter to convince oneself that the so defined $\Lambda^p(t)$ is a functor from $\mathbf3$ to $\Mod(X)$, \iev a triple of modules on $X$. We claim that $\Lambda^p(t)$ is even a right exact triple of modules on $X$. To see this, observe that firstly, the bottom row of the diagram in \eqref{e:lambdap-1} makes up a short exact triple of modules on $X$, that secondly, $\psi$ is an isomorphism in $\Mod(X)$ by Corollary \ref{c:natmor1} b), and that thirdly, $\phi$ is an epimorphism in $\Mod(X)$ by Proposition \ref{p:natmorii} \ref{p:natmorii-epi}).

 Naturally, the construction of $\Lambda^p(t)$ depends on the ringed space $X$. So, whenever we feel the need to make the reference to the ringed space $X$ within the above construction visible notationally, we resort to writing $\Lambda^p_X(t)$ instead of $\Lambda^p(t)$.
\end{construction}

We show that the $\Lambda^p$ construction is nicely compatible with the restriction to open subspaces.

\begin{proposition}
 \label{p:lambdaprest}
 Let $U$ be an open subset of $X$, $p$ an integer, and $t\colon G\to H\to F$ a right exact triple of modules on $X$. Then $t|U \colon G|U \to H|U\to F|U$ is a right exact triple of modules on $X|U$ and we have $(\Lambda^p_X(t))|U = \Lambda^p_{X|U}(t|U)$\footnote{Note that in order to get a real equality here, as opposed to only a ``canonical isomorphism'', one has to work with the right sheafification functor.}.
\end{proposition}

\begin{proof}
 The fact that the triple $t|U$ is right exact on $X|U$ is clear since the restriction functor $-|U \colon \Mod(X) \to \Mod(X|U)$ is exact. Denote $K = (K^i)_{i\in\Z}$ and $K' = (K'^i)_{i\in\Z}$ the Koszul filtrations in degree $p$ induced by $t(0,1)\colon G\to H$ and $t(0,1)|U \colon G|U \to H|U$ on $X$ and $X|U$, respectively. Then by the presheaf definitions of the wedge- and tensor products, we see that $K^i|U = K'^i$ for all integers $i$. Now define $\iota$ and $\pi$ just as in Construction \ref{con:lambdap}. Similarly, define $\iota'$ and $\pi'$ using $K'$ instead of $K$ and $X|U$ instead of $X$. Then by the presheaf definition of quotient sheaves, we see that the following diagram commutes in $\Mod(X|U)$:
 \[
  \xymatrix{
   (K^1/K^2)|U \ar[r]^{\iota|U} \ar@{=}[d] & (K^0/K^2)|U \ar[r]^{\pi|U} \ar@{=}[d] & (K^0/K^1)|U \ar@{=}[d] \\
   K'^1/K'^2 \ar[r]_{\iota'} & K'^0/K'^2 \ar[r]_{\pi'} & K'^0/K'^1
  }
 \]
 Defining $\phi$ and $\psi$ just as in Construction \ref{con:lambdap} and defining $\phi'$ and $\psi'$ analogously for $t|U$ instead of $t$ and $X|U$ instead of $X$, we deduce that $\phi|U = \phi'$ and $\psi|U = \psi'$. Hence $(\Lambda^p_X(t))|U = \Lambda^p_{X|U}(t|U)$ holds according to the definitions given in Construction \ref{con:lambdap}.
\end{proof}

The remainder of this section is devoted to investigating the $\Lambda^p$ construction in the special case where the given triple $t$ is a split exact triple of modules on $X$.

\begin{definition}
 \label{d:split}
 Let $t$ a triple of modules on $X$ (for the purposes of the definition $X$ need not necessarily be commutative as a ringed space).
 \begin{enumerate}
  \item \label{d:split-se} The triple $t$ is called \emph{split exact} on $X$ when $t$ is isomorphic in the category of triples of modules on $X$ to a triple of the form
  \[
   B \to B\oplus A \to A,
  \]
  where the first and second arrows stand for the coprojection to the first summand and projection to the second summand, respectively.
  \item \label{d:split-rs} $\phi$ is called a \emph{right splitting} of $t$ on $X$ when $\phi$ is a morphism of modules on $X$, $\phi \colon t(2) \to t(1)$, such that we have $t(1,2) \circ \phi = \id_{t(2)}$ in $\Mod(X)$.
 \end{enumerate} 
\end{definition}

\begin{lemma}
 \label{l:wedgesum}
 Let $\alpha \colon G\to H$ and $\phi \colon F\to H$ be morphisms of modules on $X$ and $p\in\Z$. Assume that $\alpha \oplus \phi \colon G\oplus F\to H$ is an isomorphism in $\Mod(X)$.
 \begin{enumerate}
  \item \label{l:wedgesum-sum} The morphism
  \[
   \bigoplus_{\nu\in\Z}\left(\wedge^{\nu,p-\nu}(H) \circ (\wedge^\nu\alpha \otimes \wedge^{p-\nu}\phi)\right) \colon \bigoplus_{\nu\in\Z}(\wedge^\nu G \otimes \wedge^{p-\nu}F) \to \wedge^pH
  \]
  is an isomorphism in $\Mod(X)$.
  \item \label{l:wedgesum-koz} Let $K=(K^i)_{i\in\Z}$ be the Koszul filtration in degree $p$ induced by $\alpha$. Then, for all integers $i$, $K^i$ corresponds to $\bigoplus_{\nu\geq i}(\wedge^\nu G\otimes \wedge^{p-\nu}F)$ under the above isomorphism.
 \end{enumerate}
\end{lemma}

\begin{proof}
 \ref{l:wedgesum-sum}). By considering stalks (just like in the proof of Lemma \ref{l:wedgep}), we find that it suffices to prove the statement in case where $X$ is an ordinary ring. In that case, however, the statement follows from \cite[Proposition A2.2, c]{Ei95}.
 
 \ref{l:wedgesum-koz}). Let $i$ be an integer. Then for all integers $\nu \geq i$, the sheaf morphism
 \[
  \wedge^{\nu,p-\nu}(H) \circ (\wedge^\nu\alpha \otimes \wedge^{p-\nu}\phi) \colon {\wedge^\nu G} \otimes {\wedge^{p-\nu}F} \to \wedge^pH
 \]
 clearly maps into $K^i$. Therefore, the direct sum $\bigoplus_{\nu\geq i}(\wedge^\nu G\otimes \wedge^{p-\nu}F)$ maps into $K^i$ under the given isomorphism. Conversely, any section in $\wedge^pH$ coming from
 \[
  \wedge^{i,p-i}(H) \circ (\wedge^i(\alpha) \otimes \wedge^{p-i}(\id_H)) \colon {\wedge^iG} \otimes \wedge^{p-i}H \to \wedge^pH
 \]
 comes from $\bigoplus_{\nu\geq i}(\wedge^\nu G\otimes \wedge^{p-\nu}F)$ under the given isomorphism as one sees decomposing $\wedge^{p-i}H$ in the form
 \[
  \bigoplus_{\mu\geq0}(\wedge^\mu G \otimes \wedge^{p-i-\mu}F) \iso \wedge^{p-i}H
 \]
 according to part \ref{l:wedgesum-sum}) (with $p$ replaced by $p-i$).
\end{proof}

\begin{proposition}
 \label{p:lambdapsplit}
 Let $p$ be an integer and $t\colon G\to H\to F$ a right exact triple of modules on $X$.
 \begin{enumerate}
  \item \label{p:lambdapsplit-rs} Let $\phi$ be a right splitting of $t$ on $X$. Denote by $K=(K^i)_{i\in\Z}$ the Koszul filtration in degree $p$ induced by $t(0,1)\colon G\to H$ and write $\kappa \colon {\wedge^pH} \to (\wedge^pH)/K^2$ for the evident quotient morphism. Then the composition
  \[
   \kappa \circ \wedge^p(\phi) \colon {\wedge^pF} \to (\wedge^pH)/K^2
  \]
  is a right splitting of $\Lambda^p(t)$ on $X$.
  \item \label{p:lambdapsplit-se} When $t$ is split exact on $X$, then $\Lambda^p(t)$ is split exact on $X$.
 \end{enumerate}
\end{proposition}

\begin{proof}
 \ref{p:lambdapsplit-rs}). By Construction \ref{con:lambdap}, we see that
 $$(\Lambda^p(t))(1,2) \colon (\wedge^pH)/K^2 \to \wedge^pF$$
 is the unique morphism of modules on $X$ which precomposed with the quotient morphism $\wedge^pH \to (\wedge^pH)/K^2$ yields $\wedge^p(t(1,2)) \colon {\wedge^pH} \to \wedge^pF$. Therefore, we have:
 \begin{align*}
  (\Lambda^p(t))(1,2) \circ (\kappa \circ \wedge^p(\phi)) & = \wedge^p(t(1,2)) \circ \wedge^p(\phi) = \wedge^p(t(1,2) \circ \phi) \\ & = \wedge^p(\id_F) = \id_{\wedge^pF}.
 \end{align*}

 \ref{p:lambdapsplit-se}). Write $K=(K^i)_{i\in\Z}$ for the Koszul filtration in degree $p$ induced by
 \[
  \alpha := t(0,1)\colon G\to H.
 \]
 Then by the definition of $(\Lambda^p(t))(0,1)$ in Construction \ref{con:lambdap}, the following diagram commutes in $\Mod(X)$:
 \[
  \xymatrix{
   G \otimes \wedge^{p-1}H \ar[r] \ar@{}@<1ex>[r]^{\id_G \otimes {\wedge^{p-1}(t(1,2))}} \ar[d]_{\wedge^{1,p}(H) \circ (\alpha \otimes {\wedge^{p-1}(\id_H))}} & G \otimes \wedge^{p-1}F \ar[d]^{(\Lambda^p(t))(0,1)} \\ \wedge^pH \ar[r]_-{\kappa} & (\wedge^pH)/K^2
  }
 \]
 Since $t$ is a split exact triple of modules on $X$, there exists a right splitting $\phi$ of $t$ on $X$. Using the commutativity of the diagram, we deduce that
 $$(\Lambda^p(t))(0,1) = \kappa \circ \wedge^{1,p}(H) \circ (\alpha \otimes \wedge^{p-1}(\phi)).$$
 Hence, by Lemma \ref{l:wedgesum}, the sheaf map $(\Lambda^p(t))(0,1)$ is injective. Knowing already that the triple $\Lambda^p(t)$ is right exact on $X$ (Construction \ref{con:lambdap}), we deduce that $\Lambda^p(t)$ is short exact on $X$. Therefore $\Lambda^p(t)$ is split exact on $X$ as by \ref{p:lambdapsplit-rs}) there exists a right splitting of $\Lambda^p(t)$ on $X$.
\end{proof}

\section{Locally split exact triples and their extension classes}
\label{s:lset}

\emph{For the entire section, let $X$ be a commutative ringed space.} \medskip

Let $p$ be an integer. In what follows, we are going to examine the $\Lambda^p$ construction, \cf Construction \ref{con:lambdap}, when applied to locally split exact triples of modules on $X$. So, let $t$ be such a triple. Then, as it turns out, $\Lambda^p(t)$ is a locally split exact triple of modules on $X$, too. Now given that $t$ is in particular a short exact triple of modules on $X$, one may consider its extension class, which is an element of $\Ext^1(F,G)$ writing $t\colon G\to H\to F$. Similarly, the extension class of $\Lambda^p(t)$ is an element of $\Ext^1(\wedge^pF,G\otimes \wedge^{p-1}F)$. The decisive result of \S\ref{s:lset} will be Proposition \ref{p:lsec}, which tells how to obtain the extension class of $\Lambda^p(t)$ from the extension class of $t$ by means of an interior product morphism
$$\iota^p(F,G) \colon \sHom(F,G) \to \sHom(\wedge^pF,G\otimes \wedge^{p-1}F),$$
to be defined in the realm of Construction \ref{con:intprod}. In order to conveniently describe this relationship between the extension classes of $t$ and $\Lambda^p(t)$, we introduce the device of ``locally split extension classes'', \cf Notation \ref{not:lsec}.

First of all, however, let us state local versions of Definition \ref{d:split} and Proposition \ref{p:lambdapsplit}.

\begin{definition}
 \label{d:lset}
 Let $t$ be a triple of modules on $X$.
 \begin{enumerate}
  \item $t$ is called \emph{locally split exact} on $X$ when there exists an open cover $\cover U$ of $X_\top$ such that, for all $U\in\cover U$, the triple $t|U$ (\iev the composition of $t$ with the restriction functor $-|U \colon \Mod(X) \to \Mod(X|U)$) is a split exact triple of modules on $X|U$.
  \item $\phi$ is called a \emph{local right splitting} of $t$ on $X$ when $\phi$ is a function whose domain of definition, call it $\cover U$, is an open cover of $X_\top$ such that, for all $U\in\cover U$, $\phi(U)$ is a right splitting of $t|U$ on $X|U$.
 \end{enumerate}
\end{definition}

\begin{proposition}
 \label{p:lambdaplset}
 Let $p$ be an integer and $t\colon G\to H\to F$ a right exact triple of modules on $X$.
 \begin{enumerate}
  \item \label{p:lambdaplset-lrs} Let $\phi$ be a local right splitting of $t$ on $X$ and let $\phi'$ be a function on $\cover U := \dom(\phi)$ such that, for all $U\in\cover U$, we have
  $$\phi'(U) = \kappa|U \circ \wedge^p(\phi(U)) \colon (\wedge^pF)|U \to ((\wedge^pH)/K^2)|U,$$
  where $\kappa$ denotes the quotient morphism $\wedge^pH \to (\wedge^pH)/K^2$ and $K = (K^i)_{i\in\Z}$ denotes the Koszul filtration in degree $p$ induced by $t(0,1)\colon G\to H$. Then $\phi'$ is a local right splitting of $\Lambda^p(t)$ on $X$.
  \item \label{p:lambdaplset-lset} When $t$ is locally split exact on $X$, then $\Lambda^p(t)$ is locally split exact on $X$.
 \end{enumerate}
\end{proposition}

\begin{proof}
 a). Let $U\in\cover U$. Then $\phi(U)$ is a right splitting of $t|U$ on $X|U$. Thus by Proposition \ref{p:lambdapsplit} a), we know that
 \[
  \kappa' \circ \wedge^p(\phi(U)) \colon {\wedge^p(F|U)} \to (\wedge^p(H|U))/K'^2
 \]
 is a right splitting of $\Lambda^p_{X|U}(t|U)$ on $X|U$, where $\kappa' \colon {\wedge^p(H|U)} \to (\wedge^p(H|U))/K'^2$ denotes the quotient morphism and $K' = (K'^i)_{i\in\Z}$ the Koszul filtration in degree $p$ induced by $\alpha|U \colon G|U \to H|U$. Since $(\wedge^pH)|U = \wedge^p(H|U)$ and $K^2|U = K'^2$, we have $\kappa|U = \kappa'$. Therefore $\phi'(U)$ is a right splitting of $\Lambda^p_{X|U}(t|U)$ on $X|U$. Given that $\Lambda^p_X(t)|U = \Lambda^p_{X|U}(t|U)$, we deduce that $\phi'$ is a local right splitting of $\Lambda^p_X(t)$ on $X$.
 
 b). Since $t$ is a locally split exact triple of modules on $X$, there exists an open cover $\cover U$ of $X_\top$ such that, for all $U\in\cover U$, the triple $t|U$ is split exact on $X|U$. Therefore, by Proposition \ref{p:lambdapsplit}, the triple $\Lambda^p_{X|U}(t|U)$ is split exact on $X|U$ for all $U\in\cover U$. As $(\Lambda^p_X(t))|U = \Lambda^p_{X|U}(t|U)$ for all $U\in\cover U$, we see that $\Lambda^p_X(t)$ is a locally split exact triple of modules on $X$.
\end{proof}

\begin{notation}[Locally split extension class]
 \label{not:lsec}
 Let $t\colon G\to H\to F$ be a short exact triple of modules on $X$ with the property that the triple
 \[
  \sHom(F,t) \colon \sHom(F,G) \to \sHom(F,H) \to \sHom(F,F)
 \]
 is again a short exact triple of modules on $X$. Then we write $\xi_X(t)$ for the image of the identity sheaf morphism $\id_F \colon F\to F$ under the composition of mappings
 \[
  (\sHom(F,F))(X) \overset{\can}\to \H^0(X,\sHom(F,F)) \overset{\delta^0}\to \H^1(X,\sHom(F,G)),
 \]
 where $\delta^0$ stands for the $0$-th connecting homomorphism for the triple $\sHom(F,t)$ with respect to the right derived functor of $\Gamma(X,-) \colon \Mod(X) \to \Mod(\Z)$. Note that as $(\sHom(F,F))(X) = \Hom(F,F)$, we have $\id_F \in (\sHom(F,F))(X)$ so that the above definition makes indeed sense. We call $\xi_X(t)$ the \emph{locally split extension class} of $t$ on $X$. As usual, we will write $\xi(t)$ instead of $\xi_X(t)$ whenever we feel that the reference to the ringed space $X$ is clear from the individual context.
\end{notation}

\begin{remark}
 \label{r:lseclset}
 Let $t\colon G\to H\to F$ be a locally split exact triple of modules on $X$. Then we claim that
 $$\sHom(F,t) \colon \sHom(F,G) \to \sHom(F,H) \to \sHom(F,F)$$
 is a locally split exact triple of modules on $X$, too. In fact, let $\phi$ be a local right splitting of $t$ on $X$. Put $\cover U := \dom(\phi)$. Then, for all $U\in\cover U$, the morphism
 $$\sHom(F|U,\phi(U)) \colon \sHom(F|U,F|U) \to \sHom(F|U,H|U)$$
 is a right splitting of $\sHom(F|U,t(1,2)|U)$ by the ``functoriality'' of $\sHom(F|U,-)$. Since $\sHom(F,t(1,2))|U = \sHom(F|U,t(1,2)|U)$ for all $U\in\cover U$, the assignment $U \mto \sHom(F|U,\phi(U))$, for $U\in\cover U$, constitutes a local right splitting of $\sHom(F,t)$ on $X$. Moreover, since the functor $\sHom(F,-) \colon \Mod(X) \to \Mod(X)$ is left exact, the triple $\sHom(F,t)$ is left exact on $X$. In conclusion, we see that the triple $\sHom(F,t)$ is indeed locally split exact on $X$.
 
 Specifically, $\sHom(F,t)$ is a short exact triple of modules on $X$, which, in view of Notation \ref{not:lsec}, tells us that (gladly) any locally split exact triple of modules on $X$ possesses a locally split extension class on $X$.
\end{remark}

The following remark will explain briefly how our newly coined notion of locally split extension classes relates to the customary extension classes for short exact triples (\iev short exact sequences) on $X$. We point out that though interesting, the contents of Remark \ref{r:lsec} are dispensable for our subsequent exposition.

\begin{remark}
 \label{r:lsec}
 Let $t\colon G\to H\to F$ be a short exact triple of modules on $X$. Recall that the \emph{extension class} of $t$ on $X$ is, by definition, the image of the identity sheaf morphism $\id_F$ under the composition of mappings
 $$\Hom(F,F) \overset{\can}{\to} (\R^0\Hom(F,-))(F) \overset{\delta'^0}{\to} (\R^1\Hom(F,-))(G) = \Ext^1(F,G),$$
 where $\delta' = (\delta'^n)_{n\in\Z}$ stands for the sequence of connecting homomorphisms for the triple $t$ with respect to the right derived functor of $\Hom(F,-) \colon \Mod(X) \to \Mod(\Z)$. Observe that the commutative diagram of categories and functors
 \[
  \xymatrix{
   \Mod(X) \ar[rr]^{\sHom(F,-)} \ar[dr]_{\Hom(F,-)} && \Mod(X) \ar[ld]^{\Gamma(X,-)} \\ & \Mod(\Z)
  }
 \]
 combined with the fact that, for all injective modules $I$ on $X$, the sheaf $\sHom(F,I)$ is a flasque sheaf on $X_\top$, whence an acyclic object for the functor $\Gamma(X,-) \colon \Mod(X) \to \Mod(\Z)$, induces a sequence $\tau = (\tau^q)_{q\in\Z}$ of natural transformations
 $$\tau^q \colon \H^q(X,-) \circ \sHom(F,-) \to \Ext^q(F,-)$$
 of functors from $\Mod(X)$ to $\Mod(\Z)$. The sequence $\tau$ has the property that when $\sHom(F,t)$ is a short exact triple of modules on $X$ and $\delta = (\delta^n)_{n\in\Z}$ denotes the sequence of connecting homomorphisms for the triple $\sHom(F,t)$ with respect to the right derived functor of $\Gamma(X,-) \colon \Mod(X) \to \Mod(\Z)$, then, for any integer $q$, the following diagram commutes in $\Mod(\Z)$:
 \[
  \xymatrix@C=3pc{
   \H^q(X,\sHom(F,F)) \ar[r]^-{\tau^q(F)} \ar[d]_{\delta^q} & \Ext^q(F,F) \ar[d]^{\delta'^q} \\
   \H^{q+1}(X,\sHom(F,G)) \ar[r]_-{\tau^{q+1}(G)} & \Ext^{q+1}(F,G)
  }
 \]
 Moreover, the following diagram commutes in $\Mod(\Z)$:
 \[
  \xymatrix@C=3pc{
   (\sHom(F,F))(X) \ar@{=}[r] \ar[d]_{\can} & \Hom(F,F) \ar[d]^{\can} \\
   \H^0(X,\sHom(F,F)) \ar[r]_-{\tau^0(F)} & \Ext^0(F,F)
  }
 \]
 Hence, we see that the locally split extension class $\xi(t)$ of $t$ is mapped to the extension class of $t$ by the function $\tau^1(G)$. In addition, by means of general homological algebra (Grothendieck spectral sequence), one can show that the mapping $\tau^1(G)$ is one-to-one. Therefore, $\xi(t)$ is the unique element of $\H^1(X,\sHom(F,G))$ which is mapped to the extension class of $t$ by the function $\tau^1(G)$. We think that this observation justifies our referring to $\xi(t)$ as the ``locally split extension class'' of $t$ on $X$.
\end{remark}

The next couple of results are aimed at deriving, for $t$ a locally split exact triple of modules on $X$, from a local right splitting of $t$ a \v Cech representation of the locally split extension class $\xi(t)$. Since the definition of \v Cech cohomology tends to vary from source to source, let us settle once and for all on the following

\begin{conventions}
 \label{cvs:cech}
 Let $\cover U$ be an open cover of $X_\top$ and $n$ a natural number. An \emph{$n$-simplex} of $\cover U$ is an ordered $(n+1)$-tuple of elements of $\cover U$, \iev
 $$u = (u_0,\dots,u_n) \in \cover U \times \dots \times \cover U,$$
 such that $u_0\cap \dots \cap u_n \neq \emptyset$. Note that, by definition, a $0$-simplex of $\cover U$ is nothing but an element of $\cover U$. When $u=(u_0,\dots,u_n)$ is an $n$-simplex of $\cover U$, the intersection $u_0\cap\dots\cap u_n$ is called the \emph{support} of $u$ and denoted $|u|$.
 
 Let $F$ be a sheaf of modules on $X$. Then a \emph{\v Cech $n$-cochain} of $\cover U$ with coefficients in $F$ is a function $c$ defined on the set $S$ of $n$-simplices of $\cover U$ such that, for all $u\in S$, we have $c(u) \in F(|u|)$. We denote $\cC^n_X(\cover U,-) \colon \Mod(X) \to \Mod(\Z)$ the \v Cech $n$-cochain functor, so that $\cC^n_X(\cover U,F)$ is the set of \v Cech $n$-cochains of $\cover U$ with coefficients in $F$ equipped with the obvious addition and $\Z$-scalar multiplication. Similarly, we denote $\cC^\kdot_X(\cover U,-) \colon \Mod(X) \to \Com(\Z)$ the \v Cech complex functor, so that $\cC^\kdot(\cover U,F)$ is the habitual \v Cech complex of $\cover U$ with coefficients in $F$. We write $\cZ^n_X(\cover U,-)$, $\cB^n_X(\cover U,-)$, and $\cH^n_X(\cover U,-)$ for the functors $\Mod(X) \to \Mod(\Z)$ obtained by composing $\cC^\kdot(\cover U,-)$ with the $n$-cocycle, the $n$-coboundary, and the $n$-cohomology functors for complexes over $\Mod(\Z)$, respectively. In any of the expressions $\cC^\kdot_X$, $\cC^n_X$, $\cZ^n_X$, $\cB^n_X$, and $\cH^n_X$, we suppress the subscript ``$X$'' whenever we feel that the correct ringed space can be guessed unambiguously from the context.
 
 In Proposition \ref{p:lseccech} as well as in the proof of Proposition \ref{p:lsec}, we will make use of the familiar sequence $\tau = (\tau^q)_{q\in\Z}$ of natural transformations
 $$\tau^q \colon \cH^q(\cover U,-) \to \H^q(X,-)$$
 of functors from $\Mod(X)$ to $\Mod(\Z)$ which are obtained by considering the \v Cech resolution functors $\shcC^\kdot(\cover U,-) \colon \Mod(X) \to \Com(X)$ together with Lemma \ref{l:iversen}; in fact, the suggested construction yields a natural transformation $\cC^\kdot(\cover U,-) = \Gamma(X,-) \circ \shcC^\kdot(\cover U,-) \to \R\Gamma(X,-)$ of functors from $\Mod(X)$ to $\K^+(\Z)$, from which one derives $\tau^q$, for any $q\in\Z$, by applying the $q$-th cohomology functor $\H^q \colon \K^+(\Z) \to \Mod(\Z)$.
\end{conventions}

\begin{construction}
 \label{con:rscc}
 Let $t\colon G\to H\to F$ be a short exact triple of modules on $X$ and $\phi$ a local right splitting of $t$ on $X$. For the time being, fix a $1$-simplex $u=(u_0,u_1)$ of $\cover U:=\dom(\phi)$. Set $v := |u|$ for better readability. Then calculating in $\Mod(X|v)$, we have:
 \begin{align*}
  t(1,2)|v\circ(\phi(u_1)|v - \phi(u_0)|v) & = (t(1,2)|u_1)|v \circ \phi(u_1)|v - (t(1,2)|u_0)|v \circ \phi(u_0)|v \\
  & = (t(1,2)|u_1 \circ \phi(u_1))|v - (t(1,2)|u_0 \circ \phi(u_0))|v \\
  & = \id_{F|u_1}|v - \id_{F|u_0}|v \\
  & = \id_{F|v} - \id_{F|v} = 0.
 \end{align*}
 Since the triple $t$ is short exact on $X$, we deduce that $t(0,1)|v \colon G|v \to H|v$ is a kernel of $t(1,2)|v \colon H|v \to F|v$ in $\Mod(X|v)$. So, there exists one, and only one, morphism $c(u) \colon F|v \to G|v$ in $\Mod(X|v)$ such that
 \[
  t(0,1)|v \circ c(u) = \phi(u_1)|v - \phi(u_0)|v.
 \]
 Abandoning our fixation of $u$, we define $c$ to be the function on the set of $1$-simplices of $\cover U$ which is given by the assignment $u \mto c(u)$. We call $c$ the \emph{right splitting \v Cech $1$-cochain} of $(t,\phi)$ on $X$.

 By definition, for all $1$-simplices $u$ of $\cover U$, we know that $c(u)$ is a morphism $F||u|\to G||u|$ of modules on $X||u|$, \iev $c(u)\in(\sHom(F,G))(|u|)$. Thus,
 \[
  c\in\cC^1_X(\cover U,\sHom(F,G)).
 \]
 In this regard, we define $\bar c$ to be the residue class of $c$ in the quotient module
 \[
  \cC^1_X(\cover U,\sHom(F,G))/\cB^1_X(\cover U,\sHom(F,G)).
 \]
 We call $\bar c$ the \emph{right splitting \v Cech $1$-class} of $(t,\phi)$ on $X$.
\end{construction}

\begin{lemma}
 \label{l:cechexact}
 Let $t\colon G\to H\to F$ be a short exact triple of modules on $X$ and $\cover U$ an open cover of $X_\top$.
 \begin{enumerate}
  \item Let $n\in\N$. Then $\cC^n(\cover U,t)$ is a short exact triple in $\Mod(\Z)$ if and only if, for all $n$-simplices $u$ of $\cover U$, the mapping $t(1,2)_{|u|}\colon H(|u|)\to F(|u|)$ is surjective.
  \item $\cC^\kdot(\cover U,t)$ is a short exact triple in $\Com(\Z)$ if and only if, for all nonvoid, finite $\cover V\subset\cover U$ such that $V:=\cap\cover V \neq \emptyset$, the mapping $t(1,2)_V\colon H(V)\to F(V)$ is surjective.
  \item Let $\phi$ be a local right splitting of $t$ on $X$ such that $\cover U = \dom(\phi)$. Then $\cC^\kdot(\cover U,t)$ is a short exact triple in $\Com(\Z)$.
 \end{enumerate}
\end{lemma}

\begin{proof}
 a). We denote $S$ the set of $n$-simplices of $\cover U$ and write $\Gamma = (\Gamma_u)_{u\in S}$ for the family of section functors $\Gamma_u := \Gamma_X(|u|,-) \colon \Mod(X) \to \Mod(\Z)$. Then $\cC^n(\cover U,-)$ equals, by definition, the composition of functors
 $${\textstyle\prod \circ \prod(\Gamma)} \colon \Mod(X) \to \Mod(\Z)^S \to \Mod(\Z),$$
 where $\prod(\Gamma)$ signifies the ``external'' product of the family of functors $\Gamma$, whereas the single ``$\prod$'' symbol signifies the standard product functor for the category $\Mod(\Z)$. We formulate a sublemma: Let $\cat C$ be any category of modules and $(M_i\to N_i\to P_i)_{i\in I}$ a family of triples in $\cat C$. Then the triple $\prod M_i \to \prod N_i \to \prod P_i$ is middle exact in $\cat C$ if and only if, for all $i\in I$, the triple $M_i\to N_i\to P_i$ is middle exact $\cat C$. The proof of this is clear. Emplying the sublemma in our situation, we obtain that since, for all $u\in S$, the functor $\Gamma_u$ is left exact, the functor $\cC^n(\cover U,-)$ is left exact, too. Thus the triple $\cC^n(\cover U,t)$ is left exact. Hence, the triple $\cC^n(\cover U,t)$ is short exact if and only if $\cC^n(\cover U,H) \to \cC^n(\cover U,F) \to 0$ is exact. By the sublemma this is equivalent to saying that $\Gamma_u(H)\to\Gamma_u(F)\to0$ is exact for all $u\in S$, but $\Gamma_u(H)\to\Gamma_u(F)\to0$ is exact if and only if $H(|u|)\to F(|u|)$ is surjective.

 b). A triple of complexes of modules is short exact if and only if, for all integers $n$, the triple of modules in degree $n$ is short exact. Since the triple of complexes $\cC^\kdot(\cover U,t)$ is trivial in negative degrees, we see that $\cC^\kdot(\cover U,t)$ is a short exact triple in $\Com(\Z)$ if and only if, for all $n\in\N$, the triple $\cC^n(\cover U,t)$ is short exact in $\Mod(\Z)$, which by a) is the case if and only if, for all nonempty, finite subsets $\cover V\subset \cover U$ with $V := \cap\cover V \neq \emptyset$, the mapping $H(V) \to F(V)$ is surjective.

 c). Let $\cover V$ be a nonvoid, finite subset of $\cover U$ such that $V := \cap\cover V \neq \emptyset$. Then there exists an element $U$ in $\cover V$. By assumption, $\phi(U) \colon F|U \to H|U$ is a morphism of modules on $X|U$ such that $t(1,2)|U \circ \phi(U) = \id_{F|U}$. Thus, given that $V\subset U$, we have $t(1,2)_V \circ \phi(U)_V = \id_{F(V)}$, which entails that $t(1,2)_V \colon H(V) \to F(V)$ is surjective. Therefore, $\cC^\kdot(\cover U,t)$ is a short exact triple in $\Com(\Z)$ by means of b).
\end{proof}

\begin{proposition}
 \label{p:rscc}
 Let $t\colon G\to H\to F$ be a short exact triple of modules on $X$ and $\phi$ a local right splitting of $t$ on $X$. Put $\cover U:=\dom(\phi)$ and denote by $c$ (\resp $\bar c$) the right splitting \v Cech $1$-cochain (\resp right splitting \v Cech $1$-class) of $(t,\phi)$ on $X$. Then the following assertions hold:
 \begin{enumerate}
  \item The triple $\cC^\kdot(\cover U,\sHom(F,t))\colon$
  \begin{equation}
   \label{e:rscc-triple}
   \cC^\kdot(\cover U,\sHom(F,G))\to\cC^\kdot(\cover U,\sHom(F,H))\to\cC^\kdot(\cover U,\sHom(F,F))
  \end{equation}
  is a short exact triple in $\Com(\Z)$.
  \item We have $c\in\cZ^1(\cover U,\sHom(F,G))$ and $\bar c\in\cH^1(\cover U,\sHom(F,G))$.
  \item When $\delta=(\delta^n)_{n\in\Z}$ denotes the sequence of connecting homomorphisms associated to the triple $\cC^\kdot(\cover U,\sHom(F,t))$ of complexes over $\Mod(\Z)$ and $\bar e$ denotes the image of the identity sheaf map $\id_F$ under the canonical function $(\sHom(F,F))(X) \to \cH^0(\cover U,\sHom(F,F))$, then $\delta^0(\bar e)=\bar c$.
 \end{enumerate}
\end{proposition}

\begin{proof}
 a). By Remark \ref{r:lseclset}, the function on $\cover U$ given by the assignment
 $$\cover U \ni U \mto \sHom(F|U,\phi(U))$$
 constitutes a local right splitting of $\sHom(F,t)$ on $X$. Moreover, the triple $\sHom(F,t)$ is a short exact triple of modules on $X$. Thus $\cC^\kdot(\cover U,\sHom(F,t))$ is a short exact triple in $\Com(\Z)$ by Lemma \ref{l:cechexact} c).

 b). Observe that since $\phi$ is a local right splitting of $t$ on $X$ and $\cover U = \dom(\phi)$, we have $\phi\in\cC^0(\cover U,\sHom(F,H))$. Further on, writing $d = (d^n)_{n\in\Z}$ for the sequence of differentials of the complex $\cC^\kdot(\cover U,\sHom(F,H))$, the mapping
 \[
  \cC^1(\cover U,\sHom(F,t(0,1))) \colon \cC^1(\sHom(F,G)) \to \cC^1(\cover U,\sHom(F,H))
 \]
 sends $c$ to $d^0(\phi)$ since, for all $1$-simplices $u=(u_0,u_1)$ of $\cover U$,
 \[
  (d^0(\phi))(u) = \phi(u_1)||u|-\phi(u_0)||u|
 \]
 and $c(u)$ is, by definition, the unique morphism $F||u|\to G||u|$ such that $t(0,1)||u| \circ c(u) = \phi(u_1)||u|-\phi(u_0)||u|$. Writing $d'' = (d''^n)_{n\in\Z}$ for the sequence of differentials of the complex $\cC^\kdot(\cover U,\sHom(F,G))$, we have
 $$\cC^2(\cover U,\sHom(F,t(0,1))) \circ d''^1 = d^1 \circ \cC^1(\cover U,\sHom(F,t(0,1))).$$
 So, since the mapping $\cC^2(\cover U,\sHom(F,t(0,1)))$ is one-to-one and $d^1(d^0(\phi)) = 0$, we see that $d''^1(c) = 0$, which implies that $c\in\cZ^1(\cover U,\sHom(F,G))$ and, in turn, that $\bar c\in\cH^1(\cover U,\sHom(F,G))$.
 
 c). Write $e$ for the image of $\id_F$ under the canonical function $(\sHom(F,F))(X) \to \cC^0(\cover U,\sHom(F,F))$. Then $\phi$ is sent to $e$ by the mapping
 \[
  \cC^0(\cover U,\sHom(F,t(1,2))) \colon \cC^0(\cover U,\sHom(F,H)) \to \cC^0(\cover U,\sHom(F,F))
 \]
 since, for all $U\in\cover U$, (\iev for all $0$-simplices of $\cover U$) we have:
 \[
  \sHom(F,t(1,2))_U(\phi(U)) = t(1,2)|U\circ\phi(U) = \id_{F|U} = e(U).
 \]
 Combined with the fact that $c$ is sent to $d^0(\phi)$ by $\cC^1(\cover U,\sHom(F,t(0,1)))$, we find that $\delta^0(\bar e) = \bar c$ by the elementary definition of connecting homomorphisms for short exact triples of complexes of modules.
\end{proof}

\begin{proposition}
 \label{p:lseccech}
 Let $t\colon G\to H\to F$ be a short exact triple of modules on $X$, $\phi$ a local right splitting of $t$ on $X$, and $\bar c$ the right splitting \v Cech $1$-class of $(t,\phi)$ on $X$. Put $\cover U:=\dom(\phi)$. Then the canonical mapping
 \begin{equation}
  \label{e:lseccech}
  \cH^1(\cover U,\sHom(F,G)) \to \H^1(X,\sHom(F,G))
 \end{equation}
 sends $\bar c$ to the locally split extension class of $t$ on $X$.
\end{proposition}

\begin{proof}
 By Proposition \ref{p:rscc} a), the triple $\cC^\kdot(\cover U,\sHom(F,t)) \colon \eqref{e:rscc-triple}$ is a short exact triple in $\Com(\Z)$. So, denote $\delta = (\delta^n)_{n\in\Z}$ the associated sequence of connecting homomorphisms. Likewise, denote $\delta' = (\delta'^n)_{n\in\Z}$ the sequence of connecting homomorphisms for the triple $\sHom(F,t)$ with respect to the right derived functor of the functor $\Gamma(X,-) \colon \Mod(X) \to \Mod(\Z)$ (note that this makes sense given that $\sHom(F,t)$ is a short exact triple of modules on $X$, \cf Remark \ref{r:lseclset}). Then the following diagram commutes in $\Mod(\Z)$, where the unlabeled arrows stand for the respective canonical morphisms:
 \begin{equation}
  \label{e:lseccech-1}
  \xymatrix{
   (\sHom(F,F))(X) \ar@{=}[r] \ar[d] & \Gamma(X,\sHom(F,F)) \ar[d] \\
   \cH^0(\cover U,\sHom(F,F)) \ar[r] \ar[d]_{\delta^0} & \H^0(X,\sHom(F,F)) \ar[d]^{\delta'^0} \\
   \cH^1(\cover U,\sHom(F,G)) \ar[r] & \H^1(X,\sHom(F,G))
  }
 \end{equation}
 By Proposition \ref{p:rscc} c), the identity sheaf morphism $\id_F$ is sent to $\bar c$ by the composition of the two downwards arrows on the left in \eqref{e:lseccech-1}. Moreover, the identity sheaf morphism $\id_F$ is sent to $\xi(t)$ by the composition of the two downwards arrows on the right in \eqref{e:lseccech-1}, \cf Notation \ref{not:lsec}. Therefore, by the commutativity of the diagram in \eqref{e:lseccech-1}, $\bar c$ is sent to $\xi(t)$ by the canonical mapping \eqref{e:lseccech}.
\end{proof}

\begin{construction}[Interior product]
 \label{con:intprod}
 Let $p$ be an integer. Moreover, let $F$ and $G$ be modules on $X$. We define a morphism of modules on $X$,
 \[
  \iota^p_X(F,G) \colon \sHom(F,G) \to \sHom(\wedge^pF,G\otimes\wedge^{p-1}F),
 \]
 called \emph{interior product morphism} in degree $p$ for $F$ and $G$ on $X$, as follows: When $p\leq0$, we define $\iota^p_X(F,G)$ to be the zero morphism (note that we do not actually have a choice here). Assume $p>0$ now. Let $U$ be an open set of $X$ and $\phi$ an element of $(\sHom(F,G))(U)$, \iev a morphism $F|U\to G|U$ of modules on $X|U$. Then there is one, and only one, morphism
 \[\psi\colon(\wedge^pF)|U\to(G\otimes\wedge^{p-1}F)|U\]
 of modules on $X|U$ such that for all open sets $V$ of $X|U$ and all $p$-tuples $(x_0,\dots,x_{p-1})$ of elements of $F(V)$, we have:
 \[
  \psi_V(x_0\wedge\dots\wedge x_{p-1})=\sum_{\nu<p} (-1)^{\nu-1} \cdot \phi_V(x_\nu)\otimes (x_0\wedge\dots\wedge\hat{x_\nu}\wedge\dots\wedge x_{p-1}).
 \]
 We let $(\iota^p_X(F,G))_U$ be the function on $(\sHom(F,G))(U)$ given by the assignment $\phi \mto \psi$, where $\phi$ varies. We let $\iota^p_X(F,G)$ be the function on the set of open sets of $X$ obtained by varying $U$. Then, as one readily verifies, $\iota^p_X(F,G)$ is a morphism of modules on $X$ from $\sHom(F,G)$ to $\sHom(\wedge^pF,G\otimes\wedge^{p-1}F)$. Just as usual, we will write $\iota^p(F,G)$ instead of $\iota^p_X(F,G)$ whenever we feel the ringed space $X$ is clear from the context.
\end{construction}

\begin{proposition}
 \label{p:lsec}
 Let $t\colon G\to H\to F$ be a locally split exact triple of modules on $X$ and $p$ an integer. Then the function
 \[\H^1(X,\iota^p(F,G)) \colon \H^1(X,\sHom(F,G)) \to \H^1(X,\sHom(\wedge^pF,G\otimes \wedge^{p-1}F))\]
 sends $\xi(t)$ to $\xi(\Lambda^p(t))$.
\end{proposition}

\begin{proof}
 First of all, we note that since $t$ is a locally split exact triple of modules on $X$, the triple $t':=\Lambda^p(t)$ is a locally split exact triple of modules on $X$ by Proposition \ref{p:lambdaplset} b), whence it makes sense to speak of $\xi(t')$. When $p\leq 0$, we know that $\sHom(\wedge^pF,G \otimes \wedge^{p-1}F) \iso 0$ in $\Mod(X)$ and thus $\H^1(X,\sHom(\wedge^pF,G \otimes \wedge^{p-1}F)) \iso 0$ in $\Mod(\Z)$, so that our assertion is true in this case. So, from now on, we assume that $p$ is a natural number different from $0$. As $t$ is a locally split exact triple of modules on $X$, there exists a local right splitting $\phi$ of $t$ on $X$. Put $\cover U:=\dom(\phi)$. Let $c\in\cC^1(\cover U,\sHom(F,G))$ be the right splitting \v Cech $1$-cochain associated to $(t,\phi)$, \cf Construction \ref{con:rscc}, and denote by $K = (K^i)_{i\in\Z}$ the Koszul filtration in degree $p$ induced by $\alpha := t(0,1) \colon G\to H$. Define $\phi'$ to be the unique function on $\cover U$ such that, for all $U\in\cover U$, we have
 \[
  \phi'(U) = \kappa|U \circ \wedge^p(\phi(U)) \colon (\wedge^pF)|U \to ((\wedge^pH)/K^2)|U,
 \]
 where $\kappa$ denotes the quotient sheaf morphism $\wedge^pH \to (\wedge^pH)/K^2$. Then $\phi'$ is a local right splitting of $t'$ by Proposition \ref{p:lambdaplset} a). Write $c'$ for the right splitting \v Cech $1$-cochain associated to $(t',\phi')$ and abbreviate $\iota^p_X(F,G)$ to $\iota$. We claim that $c$ is sent to $c'$ by the mapping
 \[
  \cC^1(\cover U,\iota) \colon \cC^1(\cover U,\sHom(F,G)) \to \cC^1(\cover U,\sHom(\wedge^pF,G \otimes \wedge^{p-1}F)).
 \]
 In order to check this, let $u=(U_0,U_1)$ be a $1$-simplex of $\cover U$ and $V$ an open set of $X$ which is contained in $|u| = U_0\cap U_1$. Observe that when $h_0,\dots,h_{p-1}$ are elements of $H(V)$ and $g_0,g_1$ are elements of $G(V)$ such that $h_0 = \alpha_V(g_0)$ and $h_1 = \alpha_V(g_1)$ (\iev $p>1$), we have
 \begin{equation}
  \label{e:lsec-1}
  \kappa_V(h_0 \wedge \dots \wedge h_{p-1}) = 0
 \end{equation}
 in $((\wedge^pH)/K^2)(V)$ by the definition of the Koszul filtration. Further, observe that writing $\alpha'$ for $t'(0,1)$ and $\beta$ for $t(1,2)$, the following diagram commutes in $\Mod(X)$ by the definition of $\alpha'$ in the $\Lambda^p$ construction:
 \begin{equation}
  \label{e:lsec-2}
  \xymatrix{
   G \otimes \wedge^{p-1}H \ar[r]\ar@{}@<1ex>[r]^{\id_G \otimes \wedge^{p-1}(\beta)} \ar[d]_{\wedge^{1,p-1}(H) \circ (\alpha \otimes \wedge^{p-1}(\id_H))} & G \otimes \wedge^{p-1}F \ar[d]^{\alpha'} \\
   \wedge^pH \ar[r]_-{\kappa} & (\wedge^pH)/K^2
  }
 \end{equation}
 Let $f_0,\dots,f_{p-1}$ be elements of $F(V)$. Then, on the one hand, we have:
 \begin{align*}
  & \alpha'_V (c'(u)_V (f_0 \wedge \dots \wedge f_{p-1})) = (\alpha'||u| \circ c'(u))_V(f_0 \wedge \dots \wedge f_{p-1}) \\
  & = (\phi'(U_1)||u|-\phi'(U_0)||u|))_V(f_0 \wedge \dots \wedge f_{p-1}) \\
  & = (\phi'(U_1)_V-\phi'(U_0)_V)(f_0 \wedge \dots \wedge f_{p-1}) \\
  & = \kappa_V \bigl(\phi(U_1)_V(f_0) \wedge \dots \wedge \phi(U_1)_V(f_{p-1}) - \phi(U_0)_V(f_0) \wedge \dots \wedge \phi(U_0)_V(f_{p-1})\bigr) \\
  & \begin{aligned} = \kappa_V \Bigl(\bigl(\phi(U_0)_V(f_0) + (\phi(U_1)_V - \phi(U_0)_V)(f_0)\bigr) \wedge \cdots \\ \mathrel\wedge \bigl(\phi(U_0)_V(f_{p-1}) + (\phi(U_1)_V - \phi(U_0)_V)(f_{p-1})\bigr) \\ \mathrel- \phi(U_0)_V(f_0) \wedge \dots \wedge \phi(U_0)_V(f_{p-1})\Bigr) \end{aligned} \\
  & = \kappa_V \Bigl(\bigl(\phi(U_0)_V(f_0) + \alpha_V(c(u)_V(f_0))\bigr) \wedge \cdots \\
  & \quad \mathrel\wedge \bigl(\phi(U_0)_V(f_{p-1}) + \alpha_V(c(u)_V(f_{p-1}))\bigr) - \phi(U_0)_V(f_0) \wedge \dots \wedge \phi(U_0)_V(f_{p-1})\Bigr) \\
  & \begin{aligned} \overset{\eqref{e:lsec-1}}= \kappa_V \biggl(\sum_{i<p} (-1)^i \cdot \alpha_V(c(u)_V(f_i)) \wedge \phi(U_0)_V(f_0) \wedge \cdots \\ \mathrel\wedge \hat{\phi(U_0)_V(f_i)} \wedge \dots \wedge \phi(U_0)_V(f_{p-1})\biggr) \end{aligned} \\
  & \begin{aligned} = (\kappa \circ \wedge^{1,p-1}(H) \circ (\alpha\otimes\id_{\wedge^{p-1}H}))_V \biggl(\sum_{i<p} (-1)^i \cdot c(u)_V(f_i) \otimes (\phi(U_0)_V(f_0) \wedge \cdots \\ \mathrel\wedge \hat{\phi(U_0)_V(f_i)} \wedge \dots \wedge \phi(U_0)_V(f_{p-1}))\biggr) \end{aligned} \\
  & \overset{\eqref{e:lsec-2}}= \left(\alpha' \circ (\id_G \otimes \wedge^{p-1}(\beta))\right)_V \Bigl(\text{same as before}\Bigr) \\
  & = \alpha'_V \biggl(\sum_{i<p} (-1)^i \cdot c(u)_V(f_i) \otimes (f_0 \wedge \dots \wedge \hat{f_i} \wedge \dots \wedge f_{p-1})\biggr)
 \end{align*}
 On the other hand,
 \[
  (\cC^1(\cover U,\iota)(c))(u)=\iota_{|u|}(c(u)),
 \]
 meaning that:
 \begin{align*}
  \left((\cC^1(\cover U,\iota)(c))(u)\right)_V(f_0\wedge\dots\wedge f_{p-1}) = \left(\iota_|u|(c(u))\right)_V(f_0\wedge\dots\wedge f_{p-1}) \\
  = \sum_{i<p} (-1)^i \cdot (c(u))_V(f_i) \otimes (f_0\wedge\dots\wedge\hat{f_i}\wedge\dots\wedge f_{p-1}).
 \end{align*}
 Thus, using the the function $\alpha'_V$ is injective, we see that $(\cC^1(\cover U,\iota)(c))(u)$ and $c'(u)$ agree as sheaf morphisms
 \[
  (\wedge^pF)||u| \to (G\otimes\wedge^{p-1}F)||u|,
 \]
 whence as elements of $(\sHom(\wedge^pF,G\otimes\wedge^{p-1}F))(|u|)$. In turn, as $u$ was an arbitrary $1$-simplex of $\cover U$,
 \begin{equation}
  \label{e:lsec-cechidentity}
  (\cC^1(\cover U,\iota))(c) = c'
 \end{equation}
 as claimed. Write $t'$ as $t'\colon G'\to H'\to F'$. Then the following diagram, where the horizontal arrows altogether stand for the respective canonical morphisms, commutes in $\Mod(\Z)$:
 \begin{equation}
  \label{e:lsec-3}
  \xymatrix{
   \cZ^1(\cover U,\sHom(F,G)) \ar[r]^{} \ar[d]_{\cZ^1(\cover U,\iota)} & \cH^1(\cover U,\sHom(F,G)) \ar[r]^{} \ar[d]_{\cH^1(\cover U,\iota)} & \H^1(X,\sHom(F,G)) \ar[d]^{\H^1(X,\iota)} \\
   \cZ^1(\cover U,\sHom(F',G')) \ar[r]_{} & \cH^1(\cover U,\sHom(F',G')) \ar[r]_{} & \H^1(X,\sHom(F',G'))
  }
 \end{equation}
 By Proposition \ref{p:lseccech}, we know that $c$ (\resp $c'$) is sent to $\xi(t)$ (\resp $\xi(t')$) by the composition of arrows in the upper (\resp lower) row of \eqref{e:lsec-3}. By \eqref{e:lsec-cechidentity}, we have $(\cZ^1(\cover U,\iota))(c)=c'$. Hence, $(\H^1(X,\iota))(\xi(t))=\xi(t')$ by the commutativity of \eqref{e:lsec-1}.
\end{proof}

\section{Connecting homomorphisms}
\label{s:connhom}

Let $f\colon X\to Y$ be a morphism of commutative ringed spaces, $t$ a locally split exact triple of modules on $X$, and $p$ an integer. In what follows, we intend to employ our results of \S\ref{s:lset}, Proposition \ref{p:lsec} specifically, in order to interpret the connecting homomorphisms for the triple $\Lambda^p(t)$ (which is short exact on $X$ by means of Proposition \ref{p:lambdaplset} \ref{p:lambdaplset-lset})) with respect to the right derived functor of $f_*$. The first pivotal outcome of this section will be Proposition \ref{p:connhomrel}. Observe that Proposition \ref{p:lsec} enters the proof of Proposition \ref{p:connhomrel} via Corollary \ref{c:rellseclambdap}. The ultimate aim of \S\ref{s:connhom} however is Proposition \ref{p:connhom} which interprets the connecting homomorphisms for $\Lambda^p(t)$ in terms of a ``cup and contraction with Kodaira-Spencer class'' given the triple $t$ is of the form
$$t\colon f^*G \to H \to F,$$
where $F$ and $G$ are locally finite free modules on $X$ and $Y$, respectively. The Kodaira-Spencer class we use here, see Notation \ref{not:ks0}, presents an abstract prototype of what will later, namely in \S\ref{s:ksgm}, become the familiar Kodaira-Spencer class.

To begin with, let us introduce a relative version of the notion of a locally split extension class which we established in \S\ref{s:lset}, Notation \ref{not:lsec}.

\begin{notation}[Relative locally split extension class]
 \label{not:rellsec}
 Let $f\colon X\to Y$ be a morphism of ringed spaces and $t\colon G\to H\to F$ a short exact triple of modules on $X$ such that the triple
 $$\sHom(F,t) = \sHom(F,-) \circ t \colon \sHom(F,G) \to \sHom(F,H) \to \sHom(F,F)$$
 is again a short exact triple of modules on $X$. Write
 $$\epsilon \colon \O_Y \to f_*(\sHom(F,F))$$
 for the unique morphism of modules on $Y$ sending the $1$ of $\O_Y(Y)$ to the identity sheaf map $\id_F\colon F\to F$, which is, as one notes, an element of $\left(f_*(\sHom(F,F))\right)(Y)$ since
 $$(f_*(\sHom(F,F)))(Y) = (\sHom(F,F))(X) = \Hom(F,F).$$
 Then define $\xi_f(t)$ to be the composition of the following morphisms of modules on $Y$:
 $$\O_Y \overset{\epsilon}\to f_*(\sHom(F,F)) \overset{\can}\to \R^0f_*(\sHom(F,F)) \overset{\delta^0}\to \R^1f_*(\sHom(F,G)),$$
 where $\delta^0$ denotes the $0$-th connecting homomorphism for the triple $\sHom(F,t)$ with respect to the right derived functor of $f_*$. We call $\xi_f(t)$ the \emph{relative locally split extension class} of $t$ with respect to $f$.
\end{notation}

\begin{proposition}
 \label{p:rellseccomp}
 Let $f\colon X\to Y$ and $g\colon Y\to Z$ be morphisms of ringed spaces and $t\colon G\to H\to F$ a short exact triple of modules on $X$ such that the triple $\sHom(F,t)$ is short exact on $X$. Then, setting $h := g\circ f$, the following diagram commutes in $\Mod(Z)$:
 \begin{equation}
  \label{e:rellseccomp}
  \xymatrix@C=3.5pc{
   \O_Z \ar[r]^-{\xi_h(t)} \ar[d]_{g^\sharp} & \R^1h_*(\sHom(F,G)) \ar[d]^{\BC^1} \\ g_*(\O_Y) \ar[r]_-{g_*(\xi_f(t))} & g_*\left(\R^1f_*(\sHom(F,G))\right)
  }
 \end{equation}
\end{proposition}

\begin{proof}
 Write $\epsilon\colon \O_Y \to f_*(\sHom(F,F))$ for the unique morphism of modules on $Y$ sending the $1$ of $\O_Y(Y)$ to the identity sheaf map $\id_F \colon F\to F$. Similarly, write $\zeta \colon \O_Z \to h_*(\sHom(F,F))$ for the unique morphism of modules on $Z$ sending the $1$ of $\O_Z(Z)$ to the identity sheaf map $\id_F$. Then, clearly, the following diagram commutes in $\Mod(Z)$:
 \[
  \xymatrix{
   \O_Z \ar[r]^-{\zeta} \ar[d]_{g^\sharp} & h_*(\sHom(F,F)) \ar@{=}[d] \\ g_*\O_Y \ar[r]_-{g_*(\epsilon)} & g_*f_*(\sHom(F,F))
  }
 \]
 Denote by $\delta^0$ and $\delta'^0$ the $0$-th connecting homomorphisms for the triple $\sHom(F,t)$ with respect to the derived functors of $f_*$ and $h_*$, respectively. Then by the compatibility of base change morphisms with connecting homomorphisms and the compatibility of base change morphisms in degree $0$ with the natural transformations $h_* \to \R^0h_*$ and $f_* \to \R^0f_*$ of functors from $\Mod(X)$ to $\Mod(Z)$ and $\Mod(X)$ to $\Mod(Y)$, respectively, we see that the following diagram commutes in $\Mod(Z)$:
 \[
  \xymatrix{
   h_*(\sHom(F,F)) \ar[r]^{\can} \ar@{=}[d] & \R^0h_*(\sHom(F,F)) \ar[r]^{\delta'^0} \ar[d]_{\BC^0} & \R^1h_*(\sHom(F,G)) \ar[d]^{\BC^1} \\
   g_*f_*(\sHom(F,F)) \ar[r]_-{g_*(\can)} & g_*\R^0f_*(\sHom(F,F)) \ar[r]_{g_*(\delta^0)} & g_*\R^1f_*(\sHom(F,F))
  }
 \]
 Now, the commutativity of the diagram in \eqref{e:rellseccomp} follows readily taking into account the definitions of $\xi_f(t)$ and $\xi_h(t)$, \cf Notation \ref{not:rellsec}.
\end{proof}

\begin{corollary}
 \label{c:rellseclambdap}
 Let $f\colon X\to Y$ be a morphism of commutative ringed spaces, $t\colon G\to H\to F$ a locally split exact triple of modules on $X$, and $p$ an integer. Set $t':=\Lambda^p(t)$ and write $t'$ as $t'\colon G'\to H'\to F'$. Then the following diagram commutes in $\Mod(Y)$.
 \begin{equation}
  \label{e:rellseclambdap}
  \xymatrix@C=1.5pc{
   & \O_Y \ar[ld]_{\xi_f(t)} \ar[dr]^{\xi_f(t')} \\ \R^1f_*(\sHom(F,G)) \ar[rr]_{\R^1f_*(\iota^p(F,G))} && \R^1f_*(\sHom(F',G'))
  }
 \end{equation}
\end{corollary}

\begin{proof}
 Set $Z := \Z$ and $g := a_Y \colon Y\to \Z$. Then $h := g \circ f = a_X \colon X\to \Z$. The commutativity of the diagram in \eqref{e:rellseclambdap} follows from Proposition \ref{p:rellseccomp} (applied twice, for $t$ and $t'$) in conjuction with Proposition \ref{p:lsec}.
\end{proof}

Many results of this section rely, in their formulation and proof, on the device of the cup product for derived direct image functors. For that matter, we curtly review this concept and state several of its properties.

\begin{construction}
 \label{con:cupp}
 Let $f\colon X\to Y$ be a morphism of commutative ringed spaces and $p$ and $q$ integers. Let $F$ and $G$ be modules on $X$. Then we denote
 $$\cupp^{p,q}_f(F,G) \colon \R^pf_*(F) \otimes_Y \R^qf_*(G) \to \R^{p+q}f_*(F \otimes_X G)$$
 the \emph{cup product} in bidegree $(p,q)$ relative $f$ for $F$ and $G$. For the definition of the cup product we suggest considering the Godement resolutions $\alpha\colon F\to L$ and $\beta\colon G\to M$ of $F$ and $G$, respectively, on $X$. Besides, let $\rho_F\colon F\to I_F$, $\rho_G\colon G\to I_G$, and $\rho_{F\otimes G}\colon F\otimes G\to I_{F\otimes G}$ be the canonical injective resolutions of $F$, $G$, and $F\otimes G$, respectively, on $X$. Then by Lemma \ref{l:iversen}, there exists one, and only one, morphism $\zeta\colon L\to I_F$ (\resp $\eta\colon M\to I_G$) in $\K(X)$ such that we have $\zeta \circ \alpha = \rho_F$ (\resp $\eta \circ \beta =\rho_G$) in $\K(X)$. Since the Godement resolutions are flasque, whence acyclic for the functor $f_*$, we see that $\H^p(f_*\zeta) \colon \H^p(f_*L) \to \H^p(f_*I_F)$ and $\H^q(f_*\eta) \colon \H^q(f_*M) \to \H^q(f_*I_G)$ are isomorphisms in $\Mod(Y)$. Thus we derive an isomorphism
 \begin{equation}\label{e:cupp-1}
  \H^p(f_*L) \otimes \H^q(f_*M) \to \H^p(f_*I_F) \otimes \H^q(f_*I_G).
 \end{equation}
 Moreover, since the Godement resolutions are pointwise homotopically trivial, $\alpha\otimes \beta\colon F\otimes G\to L\otimes M$ is a resolution of $F\otimes G$ on $X$. So, again by Lemma \ref{l:iversen}, there exists one, and only one, morphism $\theta\colon L\otimes M\to I_{F\otimes G}$ in $\K(X)$ such that we have $\theta \circ (\alpha\otimes\beta) = \rho_{F\otimes G}$ in $\K(X)$. Thus by the compatibility of $f_*$ with the respective tensor products on $X$ and $Y$, we obtain the composition
 $$f_*L \otimes f_*M  \to f_*(L\otimes M) \overset{f_*\theta}\to f_*I_{F\otimes G}$$
 in $\K^+(Y)$, which in turn yields a composition
 \begin{equation}\label{e:cupp-2}
  \H^p(f_*L) \otimes \H^q(f_*M) \to \H^{p+q}(f_*L \otimes f_*M) \to \H^{p+q}(f_*I_{F\otimes G})
 \end{equation}
 in $\Mod(Y)$. Now the composition of the inverse of \eqref{e:cupp-1} and \eqref{e:cupp-2} is the cup product $\cupp^{p,q}_f(F,G)$.

 Note that the above construction is principally due to Godement \cite[II, 6.6]{Go73}, although Godement retricts himself to applying global section functors (with supports) instead of the more general direct image functors. Also note that Grothendieck \cite[(12.2.2)]{EGA3.1} defines his cup product in the relative situation $f\colon X\to Y$ by localizing Godement's construction over the base. Of course, our cup product agrees with Grothendieck's.
\end{construction}

\begin{proposition}
 \label{p:cupp}
 Let $f\colon X\to Y$ be a morphism of commutative ringed spaces and let $p$, $q$, and $r$ be integers.
 \begin{enumerate}
  \item \label{p:cupp-nat} \textup{(Naturality)} $\cupp^{p,q}_f$ is a natural transformation
  $$\cupp^{p,q}_f \colon (-\otimes_Y -) \circ (\R^pf_* \times \R^qf_*) \to \R^{p+q}f_* \circ (-\otimes_X -)$$
  of functors from $\Mod(X) \times \Mod(X)$ to $\Mod(Y)$.
  \item \label{p:cupp-connhom} \textup{(Connecting homomorphisms)} Let $t\colon F'' \to F \to F'$ be a short exact triple of modules on $X$ and $G$ a module on $X$ such that $t\otimes G \colon F''\otimes G \to F\otimes G \to F'\otimes G$ is again a short exact triple of modules on $X$. Then the following diagram commutes in $\Mod(Y)$:
  \[
   \xymatrix@C=5pc{
    \R^pf_*(F') \otimes \R^qf_*(G) \ar[r]^{\cupp^{p,q}(F',G)} \ar[d]_{\delta^p(t)\otimes \R^qf_*(G)} & \R^{p+q}f_*(F'\otimes G) \ar[d]^{\delta^{p+q}(t\otimes G)} \\
    \R^{p+1}f_*(F'') \otimes \R^qf_*(G) \ar[r]_{\cupp^{p+1,q}(F'',G)} & \R^{p+q+1}f_*(F''\otimes G)
   }
  \]
  \item \label{p:cupp-unit} \textup{(Units)} Let $G$ be a module on $X$. Then the following diagram commutes in $\Mod(Y)$, where $\phi$ denotes the canonical morphism $f_*\O_X \to \R^0f_*(\O_X)$ of sheaves on $Y_\top$:
  \[
   \xymatrix@C=5pc{
    \O_Y \otimes_Y \R^qf_*(G) \ar[r]^{\lambda_Y(\R^qf_*(G))} \ar[d]_{(\phi \circ f^\sharp) \otimes \id} & \R^qf_*(G) \\
    \R^0f_*(\O_X) \otimes_Y \R^qf_*(G) \ar[r]_{\cupp^{0,q}(\O_X,G)} & \R^qf_*(\O_X \otimes_X G) \ar[u]_{\R^qf_*(\lambda_X(G))}
   }
  \]
  \item \label{p:cupp-ass} \textup{(Associativity)} Let $F$, $G$, and $H$ be modules on $X$. Then the following diagram commutes in $\Mod(Y)$:
  \[
   \xymatrix@C=6pc{
    \R^{p+q}f_*(F\otimes G) \otimes \R^rf_*(H) \ar[r]^{\cupp^{p+q,r}(F\otimes G,H)} & \R^{p+q+r}f_*((F\otimes G)\otimes H) \ar[d]^{\R^{p+q+r}f_*(\alpha_X)} \\
    (\R^pf_*(F) \otimes \R^qf_*(G)) \otimes \R^rf_*(H) \ar[u]^{\cupp^{p+q}(F,G)\otimes \id} \ar[d]_{\alpha_Y} & \R^{p+q+r}f_*(F\otimes (G\otimes H)) \\
    \R^pf_*(F) \otimes (\R^qf_*(G) \otimes \R^rf_*(H)) \ar[r]_{\id\otimes \cupp^{q,r}(G,H)} & \R^pf_*(F) \otimes \R^{q+r}f_*(G\otimes H) \ar[u]_{\cupp^{p,q+r}(F,G\otimes H)}
   }
  \]
 \end{enumerate}
\end{proposition}

\begin{proof}
 We refrain from giving details here. Instead we refer our reader to Godement's summary of properties of the cross product \cite[II, 6.5]{Go73} and note that these properties carry over to the cup product almost word by word as pointed out in \cite[II, 6.6]{Go73}.
\end{proof}

\begin{notation}
 \label{not:eval}
 Let $X$ be a commutative ringed space. Let $F$ and $G$ be modules on $X$. Then we write
 $$\epsilon_X(F,G) \colon \sHom(F,G) \otimes F \to G$$
 for the familiar \emph{evaluation morphism}: when $U$ is an open set of $X$, $\phi \colon F|U \to G|U$ is a morphism of sheaves of modules on $X|U$, \iev an element of $(\sHom(F,G))(U)$, and $s \in F(U)$, then
 $$(\epsilon_X(F,G))_U(\phi \otimes s) = \phi_U(s).$$
 Varying $G$, we may view $\epsilon_X(F,-)$ as a function on the class of modules on $X$. That way $\epsilon_X(F,-)$ is a natural transformation
 $$\epsilon_X(F,-) \colon (-\otimes F) \circ \sHom(F,-) \to \id_{\Mod(X)}$$
 of endofunctors on $\Mod(X)$. We will write $\epsilon$ instead of $\epsilon_X$ when we feel that the ringed space $X$ is clear from the context.
\end{notation}

\begin{proposition}
 \label{p:connhomlsec}
 Let $f\colon X\to Y$ be a morphism of commutative ringed spaces, $t\colon G\to H\to F$ a locally split exact triple of modules on $X$, and $q$ an integer. Then the following diagram commutes in $\Mod(Y)$:
 \begin{equation}\label{e:connhomlsec}
  \xymatrix{
   \R^qf_*(F) \ar[r]^{\delta^q(t)} \ar[d]_{\xi_f(t)\otimes} & \R^{q+1}f_*(G) \\
   \R^1f_*(\sHom(F,G)) \otimes \R^qf_*(F) \ar[r] \ar@{}@<-1ex>[r]_{\cupp^{1,q}(\sHom(F,G),F)} & \R^{q+1}f_*(\sHom(F,G) \otimes F) \ar[u]_{\R^{q+1}f_*(\epsilon(F,G))}
  }
 \end{equation}
\end{proposition}

\begin{proof}
 Since $t$ is a locally split exact triple of modules on $X$, we know that the triples
 $$\sHom(F,t) \colon \sHom(F,G) \to \sHom(F,H) \to \sHom(F,F)$$
 as well as
 $$\sHom(F,t) \otimes F \colon \sHom(F,G)\otimes F \to \sHom(F,H)\otimes F \to \sHom(F,F)\otimes F$$
 are locally split exact triples of modules on $X$, \cf Remark \ref{r:lseclset}. Thus by Proposition \ref{p:cupp} \ref{p:cupp-connhom}), the following diagram commutes in $\Mod(Y)$:
 \begin{equation}\label{e:connhomlsec-1}
  \xymatrix@C=5.5pc{
   \R^0f_*(\sHom(F,F)) \otimes \R^qf_*(F) \ar[r]^-{\cupp^{0,q}(\sHom(F,F),F)} \ar[d]_{\delta^0(\sHom(F,t))\otimes \R^qf_*(F)} & \R^qf_*(\sHom(F,F)\otimes F) \ar[d]^{\delta^q(\sHom(F,t)\otimes F)} \\
   \R^1f_*(\sHom(F,G)) \otimes \R^qf_*(F) \ar[r]_-{\cupp^{1,q}(\sHom(F,G),F)} & \R^{q+1}f_*(\sHom(F,G)\otimes F)
  }
 \end{equation}
 By the naturality of the evaluation morphism, \cf Notation \ref{not:eval}, the composition $\epsilon(F,-) \circ t_0$ (recall that $t_0$ denotes the object function of the functor $t$) is a morphism
 $$\epsilon(F,-) \circ t_0 \colon \sHom(F,t)\otimes F \to t$$
 of triples in $\Mod(X)$ (\iev a natural transformation of functors $\mathbf3 \to \Mod(X)$). In consequence, by the naturality of $\delta^q$ the following diagram commutes in $\Mod(Y)$:
 \begin{equation}\label{e:connhomlsec-2}
  \xymatrix@C=5pc{
   \R^qf_*(\sHom(F,F) \otimes F) \ar[r]^-{\R^qf_*(\epsilon(F,F))} \ar[d]_{\delta^q(\sHom(F,t)\otimes F)} & \R^qf_*(F) \ar[d]^{\delta^q(t)} \\
   \R^{q+1}f_*(\sHom(F,G) \otimes F) \ar[r]_-{\R^{q+1}f_*(\epsilon(F,G))} & \R^{q+1}f_*(G)
  }
 \end{equation}
 Denote $\phi$ the composition
 $$\O_Y \to f_*(\sHom(F,F)) \xrightarrow{\can} \R^0f_*(\sHom(F,F))$$
 of morphisms in $\Mod(Y)$, where the first arrow stands for the unique morphism of modules on $Y$ which sends the $1$ of $\O_Y(Y)$ to the identity sheaf map $\id_F$ in
 \[
  \left(f_*(\sHom(F,F))\right)(Y) = \Hom(F,F).
 \]
 Then from the commutativity of the diagrams in \eqref{e:connhomlsec-1} and \eqref{e:connhomlsec-2} and the definition of $\xi_f(t)$, \cf Notation \ref{not:rellsec}, we deduce that
 \begin{align*}
  & \R^{q+1}f_*(\epsilon(F,G)) \circ \cupp^{1,q}(\sHom(F,G),F) \circ (\xi_f(t)\otimes) \\ & = \delta^q(t) \circ \R^qf_*(\epsilon(F,F)) \circ \cupp^{0,q}(\sHom(F,F),F) \circ (\phi\otimes \R^qf_*(F)) \circ \lambda(\R^qf_*(F))^{-1}
 \end{align*}
 in $\Mod(Y)$. In addition, using Proposition \ref{p:cupp} \ref{p:cupp-unit}), one shows that:
 $$\R^qf_*(\epsilon(F,F)) \circ \cupp^{0,q}(\sHom(F,F),F) \circ (\phi\otimes \R^qf_*(F)) = \lambda(\R^qf_*(F)).$$
 Hence, we see that the diagram in \eqref{e:connhomlsec} commutes in $\Mod(Y)$.
\end{proof}

\begin{notation}[Adjoint interior product]
 \label{not:adjint}
 Let $X$ be a commutative ringed space. Let $p$ be an integer and $F$ and $G$ modules on $X$. Then by adjunction between the functors $-\otimes\wedge^pF$ and $\sHom(\wedge^pF,-)$, both going from $\Mod(X)$ to $\Mod(X)$, the interior product
 \[
  \iota^p_X(F,G) \colon \sHom(F,G) \to \sHom(\wedge^pF,G\otimes \wedge^{p-1}F)
 \]
 of Construction \ref{con:intprod} corresponds to a morphism
 \[
  \tilde\iota^p_X(F,G) \colon \sHom(F,G) \otimes \wedge^pF \to G \otimes \wedge^{p-1}F
 \]
 of modules on $X$, which we christen the \emph{adjoint interior product} in degree $p$ for $F$ and $G$ on $X$. Explicitly, this means that $\tilde\iota^p_X(F,G)$ equals the composition
 \begin{align*}
  & \epsilon(\wedge^pF,G\otimes \wedge^{p-1}F) \circ (\iota^p_X(F,G)\otimes \id_{\wedge^pF}) \colon \\ & \sHom(F,G) \otimes \wedge^pF \to \sHom(\wedge^pF,G\otimes \wedge^{p-1}F) \otimes \wedge^pF \to G\otimes \wedge^{p-1}F.
 \end{align*}
 Just as we did with $\iota^p_X(F,G)$, among others, we omit the subscript ``$X$'' in expressions like $\tilde\iota^p_X(F,G)$ whenever we feel this is appropriate.
\end{notation}

\begin{proposition}
 \label{p:connhomrel}
 Let $f\colon X\to Y$ be a morphism of commutative ringed spaces, $t\colon G\to H\to F$ a locally split exact triple of modules on $X$, and $p$ and $q$ integers. Then the following diagram commutes in $\Mod(Y)$:
 \begin{equation}
  \label{e:connhomrel}
  \xymatrix{
   \R^qf_*(\wedge^pF) \ar[r]^-{\delta^q(\Lambda^p(t))} \ar[d]_{\xi_f(t)\otimes} & \R^{q+1}f_*(G\otimes\wedge^{p-1}F) \\
   \R^1f_*(\sHom(F,G))\otimes\R^qf_*(\wedge^pF) \ar[r] \ar@{}@<-1ex>[r]_{\cupp^{1,q}(\sHom(F,G),\wedge^pF)} & \R^{q+1}f_*(\sHom(F,G)\otimes\wedge^pF) \ar[u]_{\R^{q+1}f_*(\tilde\iota^p(F,G))}
  }  
 \end{equation}
\end{proposition}

\begin{proof}
 Set $t' := \Lambda^p(t)$ and write $t'$ as
 $$t' \colon G' \to H' \to F'.$$
 By the definition of $\tilde\iota^p(F,G)$ via tensor-hom adjunction, \cf Notation \ref{not:adjint}, we know that
 \begin{equation}\label{e:connhomrel-1}
  \tilde\iota^p(F,G) = \epsilon(F',G') \circ (\iota^p(F,G) \otimes \id_{F'})
 \end{equation}
 holds in $\Mod(X)$. Due to the naturality of the cup product relative $f$, \cf Propostion \ref{p:cupp} \ref{p:cupp-nat}), the following diagram commutes in $\Mod(Y)$:
 \begin{equation}\label{e:connhomrel-2}
  \xymatrix{
   \R^1f_*(\sHom(F,G))\otimes_Y \R^qf_*(F') \ar[r] \ar@{}@<1ex>[r]^{\cupp^{1,q}(\sHom(F,G),F')} \ar[d]_{\R^1f_*(\iota^p(F,G))\otimes \R^qf_*(\id_{F'})} & \R^{q+1}f_*(\sHom(F,G)\otimes_X F') \ar[d]^{\R^{q+1}f_*(\iota^p(F,G)\otimes \id_{F'})} \\
   \R^1f_*(\sHom(F',G'))\otimes_Y \R^qf_*(F') \ar[r] \ar@{}@<-1ex>[r]_{\cupp^{1,q}(\sHom(F',G'),F')} & \R^{q+1}f_*(\sHom(F',G')\otimes_X F')
  }
 \end{equation}
 Now, by Proposition \ref{p:lambdaplset} we know that since $t$ is a locally split exact triple of modules on $X$, also $t'$ is a locally split exact triple of modules on $X$. Thus it makes sense to speak of the relative locally split extension class of $t'$ with respect to $f$, \cf Notation \ref{not:rellsec}. By Corollary \ref{c:rellseclambdap}, the following diagram commutes in $\Mod(Y)$:
 \begin{equation}\label{e:connhomrel-3}
  \xymatrix@C=1.5pc{
   & \O_Y \ar[ld]_{\xi_f(t)} \ar[dr]^{\xi_f(t')} \\ \R^1f_*(\sHom(F,G)) \ar[rr]_{\R^1f_*(\iota^p(F,G))} && \R^1f_*(\sHom(F',G'))
  }
 \end{equation}
 By Proposition \ref{p:connhomlsec}, this next diagram commutes in $\Mod(Y)$:
 \begin{equation}\label{e:connhomrel-4}
  \xymatrix{
   \R^qf_*(F') \ar[r]^{\delta^q(t')} \ar[d]_{\xi_f(t')\otimes} & \R^{q+1}f_*(G') \\
   \R^1f_*(\sHom(F',G')) \otimes \R^qf_*(F') \ar[r] \ar@{}@<-1ex>[r]_{\cupp^{1,q}(\sHom(F',G'),F')} & \R^{q+1}f_*(\sHom(F',G') \otimes F') \ar[u]_{\R^{q+1}f_*(\epsilon(F',G'))}
  }
 \end{equation}
 All in all, we obtain:
 \begin{align*}
  \delta^q(t') & \overset{\eqref{e:connhomrel-4}}= \R^{q+1}f_*(\epsilon (F',G')) \circ \cupp^{1,q}(\sHom(F',G'),F') \circ (\xi_f(t')\otimes) \\
  & \begin{aligned} \overset{\eqref{e:connhomrel-3}}= \R^{q+1}f_*(\epsilon (F',G')) \circ \cupp^{1,q}(\sHom(F',G'),F') \\ \mathrel\circ (\R^1f_*(\iota^p(F,G))\otimes \R^qf_*(\id_{F'})) \circ (\xi_f(t)\otimes) \end{aligned} \\
  & \begin{aligned} \overset{\eqref{e:connhomrel-2}}= \R^{q+1}f_*(\epsilon (F',G')) \circ \R^{q+1}f_*(\iota^p(F,G)\otimes \id_{F'}) \\ \mathrel\circ \cupp^{1,q}(\sHom(F,G),F') \circ (\xi_f(t)\otimes) \end{aligned} \\
  & \overset{\eqref{e:connhomrel-1}}= \R^{q+1}f_*(\tilde\iota^p(F,G)) \circ \cupp^{1,q}(\sHom(F,G),F') \circ (\xi_f(t)\otimes).
 \end{align*}
 This yields precisely the commutativity in $\Mod(Y)$ of the diagram in \eqref{e:connhomrel}.
\end{proof}

\begin{notation}[Contraction morphism]
 \label{not:cont}
 Let $X$ be a commutative ringed space and $F$ a module on $X$. We set
 \[
  \gamma^p_X(F) := \lambda_X(\wedge^{p-1}F) \circ \tilde\iota^p_X(F,\O_X) \colon F^\vee \otimes \wedge^pF \to \wedge^{p-1}F,
 \]
 where we view $\O_X$ as a module on $X$. We call $\gamma^p_X(F)$ the \emph{contraction morphism} in degree $p$ for $F$ on $X$.
\end{notation}

\begin{notation}
 \label{not:mu}
 Let $X$ be a commutative ringed space. Moreover, let $F$ and $G$ be modules on $X$. We define a morphism
 \[
  \mu_X(F,G) \colon G \otimes F^\vee \to \sHom(F,G)
 \]
 of modules on $X$ by requiring that, for all open sets $U$ of $X$, all $\theta \in (F^\vee)(U)$, and all $y \in G(U)$, the function $(\mu_X(F,G))_U$ send $y \otimes \theta \in (G\otimes F^\vee)(U)$ to the composition
 \[
  \psi \circ \theta \colon F|U \to \O_X|U \to G|U
 \]
 of morphisms of modules on $X|U$, where $\psi$ denotes the unique morphism of modules on $X|U$ from $\O_X|U$ to $G|U$ mapping the $1$ of $(\O_X|U)(U)$ to $y \in G(U)$. It is an easy matter to check that one, and only one, such morphism $\mu_X(F,G)$ exists. When the ringed space $X$ is clear from the context, we shall occasionally write $\mu$ instead of $\mu_X$.
\end{notation}

\begin{proposition}
 \label{p:intcont}
 Let $X$ be a commutative ringed space, $p$ an integer, and $F$ and $G$ modules on $X$. Then the following diagram commutes in $\Mod(X)$:
 \[
  \xymatrix@C=5pc{
   (G\otimes F^\vee)\otimes\wedge^pF \ar[r]^{\mu(F,G)\otimes\id_{\wedge^pF}} \ar[d]_{\alpha(G,F^\vee,\wedge^pF)} & \sHom(F,G)\otimes\wedge^pF \ar [d]^{\tilde\iota^p(F,G)} \\
   G\otimes(F^\vee\otimes\wedge^pF) \ar[r]_{\id_G\otimes \gamma^p(F)} & G\otimes\wedge^{p-1}F
  }
 \]
\end{proposition}

\begin{proof}
 For $p\leq 0$ the assertion is clear since $G\otimes \wedge^{p-1}F \iso 0$ in $\Mod(X)$. So, assume that $p>0$. Then for all open sets $U$ of $X$, all $p$-tuples $x=(x_0,\dots,x_{p-1})$ of elements of $F(U)$, all morphisms $\theta \colon F|U \to \O_X|U$ of modules on $X|U$, \iev $\theta \in (\F^\vee)(U)$, and all $y\in G(U)$, one verifies easily, given the definitions of $\mu$, $\tilde\iota^p$, and $\gamma^p$, that $(y \otimes \theta) \otimes (x_0 \wedge \dots \wedge x_{p-1}) \in ((G\otimes F^\vee)\otimes \wedge^pF)(U)$ is mapped to one and the same element of $(G\otimes \wedge^{p-1}F)(U)$ by either of the two paths from the upper left to the lower right corner in the above diagram. Therefore, the diagram commutes in $\Mod(X)$ by the universal property of the sheaf associated to a presheaf.
\end{proof}

\begin{construction}[Projection morphism]
 \label{con:projmor}
 Let $n$ be an integer and $f\colon X\to Y$ a morphism of commutative ringed spaces. Moreover, let $F$ and $G$ be modules on $X$ and $Y$, respectively. Then we define the $n$-th \emph{projection morphism} relative $f$ for $F$ and $G$, denoted
 $$\pi^n_f(G,F) \colon G \otimes \R^nf_*(F) \to \R^nf_*(f^*G\otimes F),$$
 to be the morphism of modules on $Y$ which is obtained by first going along the composition
 $$G \to f_*(f^*G) \overset{\can}\to \R^0f_*(f^*G),$$
 where the former arrow stands for the familiar adjunction morphism for $G$ with respect to $f$, tensorized on the right with the identity of $\R^nf_*(F)$ and then applying the cup product:
 $$\cupp^{0,n}_f(f^*G,F) \colon \R^0f_*(f^*G) \otimes \R^nf_*(F) \to \R^nf_*(f^*G \otimes F).$$
 Observe that this construction is suggested by \cite[Proposition (12.2.3)]{EGA3.1}. Letting $F$ and $G$ vary, we may view $\pi^n_f$ as a function defined on the class of objects of the product category $\Mod(Y) \times \Mod(X)$. That way, it follows essentially from Proposition \ref{p:cupp} \ref{p:cupp-nat}), \iev the naturality of the cup product, that $\pi^n_f$ is a natural transformation
 $$(- \otimes_Y -) \circ (\id_{\Mod(Y)} \times \R^nf_*) \to \R^nf_* \circ (- \otimes_X -) \circ (f^* \times \id_{\Mod(X)})$$
 of functors from $\Mod(Y) \times \Mod(X)$ to $\Mod(Y)$. As usual, we will write $\pi^n$ instead of $\pi^n_f$ when we think this is appropriate.
\end{construction}

\begin{proposition}
 \label{p:projmoriso}
 Let $n$ be an integer, $f\colon X\to Y$ a morphism of commutative ringed spaces, $F$ a module on $X$, and $G$ a locally finite free module on $Y$. Then the projection morphism
 $$\pi^n_f(G,F) \colon G \otimes \R^nf_*(F) \to \R^nf_*(f^*G\otimes F)$$
 is an isomorphism in $\Mod(Y)$.
\end{proposition}

\begin{proof}
 See \cite[Proposition (12.2.3)]{EGA3.1}.
\end{proof}

\begin{proposition}
 \label{p:projmorcup}
 Let $f\colon X\to Y$ be a morphism of commutative ringed spaces. Let $q$ and $q'$ be integers, $F$ and $F'$ modules on $X$, and $G$ a module on $Y$. Then the following diagram commutes in $\Mod(Y)$:
 \[
 \xymatrix{
  \R^qf_*(f^*G\otimes F)\otimes\R^{q'}f_*(F') \ar[r] \ar@{}[r]<1ex>^{\cupp^{q,q'}_f(f^*G\otimes F,F')} & \R^{q+q'}f_*((f^*G\otimes F)\otimes F') \ar[d]^{\R^{q+q'}f_*(\alpha_X(f^*G,F,F'))} \\
  (G\otimes\R^qf_*(F)) \otimes\R^{q'}f_*(F') \ar[u]^{\pi^q_f(G,F)\otimes\id_{\R^{q'}f_*(F')}} \ar[d]_{\alpha_Y(G,\R^qf_*(F),\R^{q'}f_*(F'))} & \R^{q+q'}f_*(f^*G\otimes(F\otimes F')) \\
  G\otimes(\R^qf_*(F)\otimes\R^{q'}f_*(F')) \ar[r]_{\id_G\otimes\cupp^{q,q'}_f(F,F')} & G\otimes\R^{q+q'}f_*(F\otimes F') \ar[u]_{\pi^{q+q'}_f(G,F\otimes F')}
 }
 \]
\end{proposition}

\begin{proof}
 This follows with ease from the associativity of the cup product as stated in Proposition \ref{p:cupp} \ref{p:cupp-ass}).
\end{proof}

\begin{construction}
 \label{con:fibered}
 Let $\cat C$, $\cat D$, and $\cat E$ be categories and $S\colon\cat C\to\cat E$ and $T\colon\cat D\to\cat E$ functors. Then, the \emph{fibered product} of $\cat C$ and $\cat D$ over $\cat E$ with respect to $S$ and $T$, customarily ambiguously(!) denoted $\cat C\times_{\cat E}\cat D$, is by definition the subcategory of the ordinary product category $\cat C\times\cat D$ whose class of objects is given by those ordered pairs $(x,y)$ satisfying $Sx=Ty$; moreover, for two such ordered pairs $(x,y)$ and $(x',y')$ a morphism $(\alpha,\beta) \colon (x,y) \to (x',y')$ in $\cat C\times \cat D$ is a morphism in $\cat C\times_{\cat E}\cat D$ if and only if $S\alpha = T\beta$. Two easy observations show that, for one, for all objects $(x,y)$ of $\cat C\times_{\cat E}\cat D$, the identity $\id_{(x,y)}\colon (x,y)\to (x,y)$ in $\cat C\times\cat D$ is a morphism in $\cat C\times_{\cat E}\cat D$ and that, for another, for all objects $(x,y)$, $(x',y')$, and $(x'',y'')$ and morphisms $(\alpha,\beta) \colon (x,y)\to (x',y')$ and $(\alpha',\beta')\colon (x',y')\to (x'',y'')$ of $\cat C\times_{\cat E}\cat D$, the composition $(\alpha',\beta') \circ (\alpha,\beta) \colon (x,y) \to (x'',y'')$ in $\cat C\times \cat D$ is again a morphism in $\cat C\times_{\cat E}\cat D$.

 We apply the fibered product construction in the following situation: Let $f\colon X\to Y$ be a morphism of ringed spaces, and consider the functors
 \[
  f^* \colon \Mod(Y) \to \Mod(X) \quad \text{and} \quad p_0 \colon \Mod(X)^{\mathbf3} \to \Mod(X),
 \]
 where $p_0$ stands for the ``projection to $0$'', \iev $p_0$ takes an object $t$ of $\Mod(X)^{\mathbf3}$ to $t(0)$ and a morphism $\alpha \colon t\to t'$ in $\Mod(X)^{\mathbf3}$ to $\alpha(0)$. Then, define $\cat D_f$ to be the fibered product of $\Mod(Y)$ and $\Mod(X)^{\mathbf3}$ over $\Mod(X)$ with respect to $f^*$ and $p_0$, symbolically: $$\cat D_f := \Mod(Y) \times_{\Mod(X)} \Mod(X)^{\mathbf3}.$$
\end{construction}

\begin{notation}[Kodaira-Spencer class, I]
 \label{not:ks0}
 Let $f\colon X\to Y$ be a morphism of commutative ringed spaces and $(G,t)$ an object of $\cat D_f$ such that $t$ is a short exact triple of modules on $X$ and $F:=t2$ and $G$ are locally finite free modules on $X$ and $Y$, respectively. We associate to $(G,t)$ a morphism
 \[
  \xi_{\KS,f}(G,t) \colon \O_Y \to G\otimes \R^1f_*(F^\vee)
 \]
 in $\Mod(Y)$, which we call, from that time on, the \emph{Kodaira-Spencer class} relative $f$ of $(G,t)$. For the definition of $\xi_{\KS,f}(G,t)$, we remark, to begin with, that since $F$ is a locally finite free module on $X$, the triple $t$ is not only short exact, but locally split exact on $X$. Thus we may consider its relative locally split extension class with respect to $f$:
 \[
  \xi_f(t) \colon \O_Y \to \R^1f_*(\sHom(F,f^*G)),
 \]
 \cf Notation \ref{not:rellsec}. Set
 \[
  \mu := \mu_X(F,f^*G) \colon f^*G\otimes F^\vee \to \sHom(F,f^*G),
 \]
 \cf Notation \ref{not:mu}. Then, again by the local finite freeness of $F$ on $X$, we know that $\mu$ is an isomorphism in $\Mod(X)$. Given that $G$ is a locally finite free module on $Y$, the projection morphism
 \[
  \pi := \pi^1_f(G,F^\vee) \colon G\otimes \R^1f_*(F^\vee) \to \R^1f_*(f^*G\otimes F^\vee)
 \]
 is an isomorphism in $\Mod(Y)$ by means of Proposition \ref{p:projmoriso}. Composing, we obtain an isomorphism in $\Mod(Y)$:
 \[
  \R^1f_*(\mu) \circ \pi \colon G\otimes \R^1f_*(F^\vee) \to \R^1f_*(\sHom(F,f^*G)).
 \]
 Therefore, there exists a unique $\xi_{\KS,f}(G,t)$ rendering commutative in $\Mod(Y)$ the following diagram:
 \[
  \xymatrix{
   & \O_Y \ar[ld]_{\xi_{\KS,f}(G,t)} \ar[dr]^{\xi_f(t)} \\ G\otimes \R^1f_*(F^\vee) \ar[rr]_-{\R^1f_*(\mu) \circ \pi} && \R^1f_*(\sHom(F,f^*G))
  } 
 \]
\end{notation}

\begin{notation}[Cup and contraction, I]
 \label{not:cc0}
 We proceed in the situation of Notation \ref{not:ks0}, that is, we assume that a morphism of commutative ringed spaces $f\colon X\to Y$ as well as an object $(G,t)$ of $\cat D_f$ be given such that $t$ is a short exact triple of modules on $X$ and $F:=t2$ and $G$ are locally finite free modules on $X$ and $Y$, respectively. Additionally, let us fix two integers $p$ and $q$. Then we write $\gamma^{p,q}_{\KS,f}(G,t)$ for the composition of the following morphisms in $\Mod(Y)$:
 \begin{align*}
  \R^qf_*(\wedge^pF) & \xrightarrow{\lambda(\R^qf_*(\wedge^pF))^{-1}} && \O_Y \otimes \R^qf_*(\wedge^pF) \\
  &\xrightarrow{\xi_{\KS,f}(G,t) \otimes \id_{\R^qf_*(\wedge^pF)}} && (G \otimes \R^1f_*(F^\vee)) \otimes \R^qf_*(\wedge^pF) \\
  &\xrightarrow{\alpha(G,\R^1f_*(F^\vee),\R^qf_*(\wedge^pF))} && G \otimes (\R^1f_*(F^\vee) \otimes \R^qf_*(\wedge^pF)) \\
  &\xrightarrow{\id_G \otimes \cupp^{1,q}(F^\vee,\wedge^pF)} && G \otimes \R^{q+1}f_*(F^\vee \otimes \wedge^pF) \\
  &\xrightarrow{\id_G \otimes \R^{q+1}f_*(\gamma^p(F))} && G \otimes \R^{q+1}f_*(\wedge^{p-1}F).
 \end{align*}
 The resulting morphism of modules on $Y$,
 \[
  \gamma^{p,q}_{\KS,f}(G,t)\colon\R^qf_*(\wedge^pF)\to G\otimes\R^{q+1}f_*(\wedge^{p-1}F),
 \]
 goes by the name of \emph{cup and contraction with Kodaira-Spencer class} in bidegree $(p,q)$ relative $f$ for $(G,t)$ (the name should be self-explanatory looking at the definition of $\gamma^{p,q}_{\KS,f}(G,t)$ above: first, we tensorize with the Kodaira-Spencer class $\xi_{\KS,f}(G,t)$, \cf Notation \ref{not:ks0}, then we ``cup'', then we ``contract'').
\end{notation}

\begin{figure}[ht]
 \centerline{
  \xymatrix@C=1pc@R=3pc{
   \R^qf_*(\wedge^pF) \ar[r]^-{\delta^q(\Lambda^p(t))} \ar[d]|{\xi(t)\otimes} \ar@/_8.4pc/[ddd]|{\xi_\KS(G,t)\otimes} \ar@{}[dr]|{\diagcircled1} & \R^{q+1}f_*(f^*G\otimes\wedge^{p-1}F) \\
   \R^1f_*(\sHom(F,f^*G))\otimes\R^qf_*(\wedge^pF) \ar[r]^-{\cupp^{1,q}} \ar@{}[dr]|{\diagcircled{3}} \ar@{}@<-4.5pc>[d]|{\diagcircled2} & \R^{q+1}f_*(\sHom(F,f^*G)\otimes\wedge^pF) \ar[u]|{\R^{q+1}f_*(\tilde\iota^p(F,f^*G))} \ar@{}@<3.5pc>[d]|{\diagcircled{4}} \\
   \R^1f_*(f^*G\otimes F^\vee)\otimes\R^qf_*(\wedge^pF) \ar[r]^-{\cupp^{1,q}} \ar[u]|{\R^1f_*(\mu)\otimes \R^qf_*(\id)} \ar@{}[ddr]|{\diagcircled5} & \R^{q+1}f_*((f^*G\otimes F^\vee)\otimes\wedge^pF) \ar[u]|{\R^{q+1}f_*(\mu\otimes\id)} \ar[d]|{\R^{q+1}f_*(\alpha_X)} & G\otimes\R^{q+1}f_*(\wedge^{p-1}F) \ar@/_1.2pc/[uul]_(.4){\pi^{q+1}(G,\wedge^{p-1}F)} \\
   (G\otimes\R^1f_*(F^\vee))\otimes\R^qf_*(\wedge^pF) \ar[u]|{\pi^1(G,F^\vee)\otimes\id} \ar[d]|{\alpha_Y} & \R^{q+1}f_*(f^*G\otimes(F^\vee\otimes\wedge^pF)) \ar@/_7.3pc/[uuu]|{\R^{q+1}f_*(\id\otimes \gamma^p(F))} \ar@{}@<-2ex>[ur]|{\diagcircled{6}} \\
   G\otimes(\R^1f_*(F^\vee)\otimes\R^qf_*(\wedge^pF)) \ar[r]_-{\id\otimes\cupp^{1,q}} & G\otimes\R^{q+1}f_*(F^\vee\otimes\wedge^pF) \ar[u]|{\pi^{q+1}(G,F^\vee\otimes\wedge^pF)} \ar@/_1.2pc/[uur]_(.6){\id\otimes\R^{q+1}f_*(\gamma^p(F))} \\
  }
 }
 \caption{Diagram for the proof of Proposition \ref{p:connhom}.}
 \label{fig:ks}
\end{figure}

\begin{proposition}
 \label{p:connhom}
 Let $f\colon X\to Y$ be a morphism of commutative ringed spaces and $(G,t)$ an object of $\cat D_f$ such that $t$ is a short exact triple of modules on $X$ and $F:=t2$ and $G$ are locally finite free modules on $X$ and $Y$, respectively. Moreover, let $p$ and $q$ be integers. Then we have:
 \[
  \delta^q_f(\Lambda^p(t)) = \pi^{q+1}_f(G,\wedge^{p-1}F) \circ \gamma^{p,q}_{\KS,f}(G,t).
 \]
 In other words, the following diagram commutes in $\Mod(Y)$:
 \[
  \xymatrix@C=4pc{
   \R^qf_*(\wedge^pF) \ar[r]^-{\delta^q_f(\Lambda^p(t))} \ar@/_1.5pc/[rr]_{\gamma^{p,q}_{\KS,f}(G,t)} & \R^{q+1}f_*(f^*G\otimes\wedge^{p-1}F) & G\otimes\R^{q+1}f_*(\wedge^{p-1}F) \ar[l]_-{\pi^{q+1}_f(G,\wedge^{p-1}F)}
  }
 \]
\end{proposition}

\begin{proof}
 Set $\mu := \mu_X(F,f^*G)$, \cf Notation \ref{not:mu}, and consider the diagram in Figure \ref{fig:ks}, where we have abstained from specifying the cup products $\cupp^{1,q}$, the tensor associativity morphisms $\alpha_X$ and $\alpha_Y$, as well as the identity morphisms $\id$ further. We show that the subdiagrams labeled \diagcircled1--\diagcircled6 commute in $\Mod(Y)$ (this is indeed equivalent to saying that the diagram commutes in $\Mod(Y)$ as such, but we will merely use the commutativity of the mentioned subdiagrams afterwards). 
 We know that the triple $t$ is locally split exact on $X$. Hence the commutativity of \diagcircled1 is implied by Proposition \ref{p:connhomrel}.
 The commutativity of \diagcircled2 follows immediately from the definition of the Kodaira-Spencer class $\xi_\KS(G,t)$, \cf Notation \ref{not:ks0}.
 \diagcircled3 commutes by the naturality of the cup product, Proposition \ref{p:cupp} \ref{p:cupp-nat}).
 The commutativity of \diagcircled4 follows from Proposition \ref{p:intcont} coupled with the fact that $\R^{q+1}f_*$ is a functor from $\Mod(X)$ to $\Mod(Y)$.
 \diagcircled5 commutes due to Proposition \ref{p:projmorcup}.
 Last but not least, the commutativity of \diagcircled6 follows from the fact that $\pi^{q+1}_f$ is a natural transformation
 $$(- \otimes_Y -) \circ (\id_{\Mod(Y)} \times \R^{q+1}f_*) \to \R^{q+1}f_* \circ (- \otimes_X -) \circ (f^* \times \id_{\Mod(X)})$$
 of functors from $\Mod(Y) \times \Mod(X)$ to $\Mod(Y)$, \cf Construction \ref{con:projmor}.

 Recalling the definition of the cup and contraction with Kodaira-Spencer class, \cf Notation \ref{not:cc0}, we see that the commutativity of \diagcircled1--\diagcircled6 implies that
 $$\delta^q(\Lambda^p(t)) = \pi^{q+1}(G,\wedge^{p-1}F) \circ \gamma^{p,q}_{\KS}(G,t)$$
 holds in $\Mod(Y)$ (go through the subdiagrams one by one in the given order).
\end{proof}

\section{A framework for studying the Gauß-Manin connection}
\label{s:cdlemma}

This section makes up the technical heart of Chapter \ref{ch:peri}. In fact, Theorem \ref{t:grgmcc} of the subsequent \S\ref{s:ksgm}, which turns out to be crucial in view of our aspired study of period mappings in \S\ref{s:pm}, is a mere special case of Theorem \ref{t:grgmcc0}. One should regard this section as presenting a framework designed to study, in \S\ref{s:ksgm}, the traditional Gauß-Manin connection which is associated to a submersive morphism of complex spaces with smooth base. The results of section \S\ref{s:cdlemma} are all based upon Setup \ref{set:cdlemma}. Let us note that we are well aware of the fact that Setup \ref{set:cdlemma} might seem odd at first sight, yet looking at \S\ref{s:ksgm}, the reader will find that it is just the right thing to consider.

To begin with, we introduce the auxiliary device of what we christened ``augmented triples''.

\begin{notation}[Augmented triples]
 \label{not:augtrip}
 Let $b\colon X\to X'$ be a morphism of ringed spaces. Temporarily denote $\cat D$ the category of short exact triples of bounded below complexes of modules on $X'$; note that $\cat D$ is a full subcategory of $(\Com^+(X'))^{\mathbf3}$. Now consider the following diagram of categories and functors:
 \begin{equation} \label{e:aug-cat}
  \Com^+(X) \times \Com^+(X) \xrightarrow{b_*\times b_*} \Com^+(X') \times \Com^+(X') \xleftarrow{p_2\times p_0} \cat D,
 \end{equation}
 where $p_2$ signifies the ``projection to $2$'', \iev $p_2(t)=t(2)$ for all objects $t$ of $\cat D$ and $(p_2(t,t'))(\alpha)=\alpha(2)$ for all morphisms $\alpha\colon t\to t'$ in $\cat D$; $p_0$ signifies the ``projection to $0$'', which is declared analogously. Now, define $\cat E_b$ to be the fibered product category over the diagram \eqref{e:aug-cat}, \iev
 \[
  \cat E_b := (\Com^+(X) \times \Com^+(X)) \times_{\Com^+(X') \times \Com^+(X')} \cat D,
 \]
 \cf Construction \ref{con:fibered}. We refer to $\cat E_b$ as the \emph{category of augmented triples} with respect to $b$. An object $l_+$ of $\cat E_b$ is called an \emph{augmented triple} with respect to $b$. Note an augmented triple with respect to $b$ can always be written in the form $((K,L),l)$, where $K$ and $L$ are bounded below complexes of modules on $X$ and $l$ is a triple of bounded below complexes of modules on $X'$ such that $l(0) = b_*(L)$ and $l(2) = b_*(K)$.
\end{notation}

In this as well as in the next section we will frequently encounter the situation where two modules, say $F$ and $G$, on a ringed space $X$ together with a map $\alpha \colon F\to G$ of sheaves on $X_\top$ are given, yet $\alpha$ is not a morphism of sheaves of modules on $X$ but satisfies only a weaker linearity property, e.g., $X \to \C$ is a complex space and $\alpha$ is not $\O_X$-linear but merely $\C$-linear. For these purposes it will come in handy to establish the following

\begin{notation}
 \label{not:modrel}
 Let $f\colon X\to S$ be a morphism of ringed spaces. Then we write $\Mod(X/f)$ or else $\Mod(X/S)$ for the following large category: The class of objects of $\Mod(X/f)$ is simply the class of sheaves of modules on $X$, \iev the class of objects of $\Mod(X)$. For any ordered pair $(F,G)$ of objects of $\Mod(X/f)$, $\alpha$ is a morphism from $F$ to $G$ in $\Mod(X/f)$ if and only if $\alpha$ is a morphism of sheaves of $f^{-1}\O_S$-modules on $X_\top$ from $\bar F$ to $\bar G$, where $\bar F$ (\resp $\bar G$) stands for the sheaf of $f^{-1}\O_S$-modules on $X_\top$ which is obtained from $F$ (\resp $G$) by relaxing the scalar multiplication via the morphism $f^\sharp \colon f^{-1}\O_S \to \O_X$ of sheaves of rings on $X_\top$. The identity function of $\Mod(X/f)$ sends an object $F$ of $\Mod(X/f)$ to the identity sheaf map $\id_F$ on $X_\top$. The composition in $\Mod(X/f)$ is given by the composition of sheaf maps on $X_\top$. Note that $\Mod(X)$ is a subcategory of $\Mod(X/f)$. We omit the verification that the so defined $\Mod(X/f)$ is indeed a category.
 
 Note that $\Mod(X)$ is a subcategory of $\Mod(X/f)$. In fact, the classes of objects of $\Mod(X)$ and $\Mod(X/f)$ agree, yet the hom-sets of $\Mod(X/f)$ are generally larger than the corresponding hom-sets of $\Mod(X)$.
 
 Defining additions on the hom-sets of $\Mod(X/f)$ as usual, $\Mod(X/f)$ becomes an additive category. Thus we can speak of complexes over $\Mod(X/f)$. In this regard, we set $\Com(X/f) := \Com(\Mod(X/f))$ and $\Com^+(X/f) := \Com^+(\Mod(X/f))$.
\end{notation}

\begin{construction}
 \label{con:augconnhom}
 Suppose we are given a commutative square in the category of ringed spaces as follows:
 \begin{equation} \label{e:chom-sq}
  \xysquare{X}{X'}{S}{S'}{b}{f}{f'}{c}
 \end{equation}
 Assume that $b_\top=\id_{X_\top}$ and $c_\top=\id_{S_\top}$. Then the functors $b_*$ and $c_*$ are exact. Fix an integer $n$, and denote
 \[
  \kappa^n \colon \R^nf'_* \circ b_* \to c_* \circ \R^nf_*
 \]
 the natural transformation of functors from $\Com^+(X)$ to $\Mod(S')$ which is induced by the base change (in degree $n$) associated to the square \eqref{e:chom-sq}, \cf Construction \ref{con:bc}. Since the functors $b_*$ and $c_*$ are exact, we know that, for all $F \in \Com^+(X)$, the morphism
 \[
  \kappa^n(F) \colon \R^nf'_*(b_*(F)) \to c_*(\R^nf_*(F))
 \]
 is an isomorphism in $\Mod(S')$. That is, $\kappa^n$ is a natural equivalence between the mentioned functors.
 
 Let $l_+ = ((K,L),l)$ be an augmented triple with respect to $b$, \iev an object of $\cat E_b$, \cf  Notation \ref{not:augtrip}. We define $\delta^n_+(l_+)$ to be the composition of the following morphisms of modules on $S'$:
 \begin{equation} \label{e:chom-1}
  c_*(\R^nf_*(K)) \overset{\kappa^n(K)}\to \R^nf'_*(b_*(K)) \overset{\delta^n_{f'}(l)}\to \R^{n+1}f'_*(b_*(L)) \xrightarrow{(\kappa^{n+1}(L))^{-1}} c_*(\R^{n+1}f_*(L));
 \end{equation}
 note here that $l(2) = b_*(K)$ and $l(0) = b_*(L)$. Then $\delta^n_+(l_+)$ is a morphism
 \[
  \delta^n_+(l_+) \colon \R^nf_*(K) \to \R^{n+1}f_*(L)
 \]
 in $\Mod(S/c)$, \cf Notation \ref{not:modrel}.

 Letting $l_+$ vary, we obtain a function $\delta^n_+$ defined on the class of objects of $\cat E_b$. We call $\delta^n_+$ the $n$-th \emph{augmented connecting homomorphism} associated to the square \eqref{e:chom-sq}. Since $\kappa^n$, $\delta^n$, and $(\kappa^{n+1})^{-1}$ are altogether natural transformations (of approriate functors between appropriate categories), one infers that $\delta^n_+$ is a natural transformation
 \[
  \delta^n_+ \colon \R^nf_* \circ q_0 \to \R^{n+1}f_* \circ q_1
 \]
 of functors from $\cat E_b$ to $\Mod(S/c)$, where $q_0$ (\resp $q_1$) stands for the functor $\cat E_b \to \Com^+(X)$ which is given as the composition of the projection $\cat E_b \to \Com^+(X) \times \Com^+(X)$ to the first factor in the representation of $\cat E_b$ as a fibered product followed by the projection $\Com^+(X) \times \Com^+(X) \to \Com^+(X)$ to the first (\resp second) factor.
\end{construction}

Next, we describe the basic situation to which any of the upcoming results of this section will refer.

\begin{setup}
 \label{set:cdlemma}
 Let $f\colon X\to S$ and $g\colon S\to T$ be two morphisms of commutative ringed spaces, and put $h:=g\circ f$. Let $(G,t)$ be an object of $\cat D_f$, \cf Construction \ref{con:fibered}, such that $t$ is a short exact triple of modules on $X$ and $F:=t2$ and $G$ are locally finite free modules on $X$ and $S$, respectively. Moreover, let
 $$l\colon L \overset{\alpha}{\to} M \overset{\beta}{\to} K$$
 be a triple in $\Com(X/h)$ such that $K$ and $L$ are in fact objects of $\Com(X/f)$ (note that $\Com(X/f)$ is a subcategory of $\Com(X/h)$ since $h$ factors through $f$ in the category of ringed spaces) and, for all integers $p$, we have:
 \begin{equation} \label{e:lambdaptlp}
  l^p = \Lambda^p_X(t),
 \end{equation}
 where $l^p$ stands for the triple in $\Mod(X/h)$ obtained by extracting the degree-$p$ part from the triple of complexes $l$. Morally, requiring that \eqref{e:lambdaptlp} holds for all integers $p$ means that the only new information when passing from $(G,t)$ to $l$ lies in the presence of differentials in the complexes $L$, $M$, and $K$. We would like to impose yet another condition on the triple $l$ which relates the differentials of $K$ to the differentials of $L$. In order to do this conveniently, we need to introduce some pieces of notation first.

 As to that, put
 \begin{align*}
  \bar X & := (X_\top,f^{-1}\O_S) & u & := (\id_{|X|},f^\sharp\colon f^{-1}\O_S\to\O_X) \\
  X'     & := (X_\top,f^{-1}g^{-1}\O_T) & b & := (\id_{|X|},f^{-1}(g^\sharp)\colon f^{-1}g^{-1}\O_T\to f^{-1}\O_S) \\
  S'     & := (S_\top,g^{-1}\O_T) & c & := (\id_{|S|},g^\sharp\colon g^{-1}\O_T\to\O_S)
 \end{align*}
 and moreover:
 \begin{align*}
  \bar f & := (|f|,\eta_{\O_S}\colon \O_S \to {f_\top}_*f^{-1}\O_S) \\
  f'     & := (|f|,\eta_{g^{-1}\O_T}\colon g^{-1}\O_T \to {f_\top}_*f^{-1}g^{-1}\O_T),
 \end{align*}
 where we temporarily use $\eta$ to denote the adjunction morphism for sheaves of sets on $S_\top$ with respect to the morphism of topological spaces $f_\top$. Given the above notation, it is easy to verify that the following diagram commutes in the category of ringed spaces:
 \begin{equation}
  \label{e:gmsetup-sq}
  \xymatrix{
   X \ar[r]^u \ar[d]_f & \bar X \ar[r]^b \ar[d]_{\bar f} & X' \ar[d]^{f'} \\
   S \ar[r]_{\id_S} & S \ar[r]_c & S'
  }
 \end{equation}
 We have $b_\top=\id_{(\bar X)_\top}$ and $c_\top=\id_{S_\top}$. For any integer $n$, we let $\delta^n_+$ signify the $n$-th augmented connecting homomorphism with respect to the right square in \eqref{e:gmsetup-sq}, \cf Construction \ref{con:augconnhom}. For any integer $n$, we denote by $\delta^n$ the ordinary $n$-th connecting homomorphism for the derived functor $\R f_*$, viewed either as defined on the class of short exact triples in $\Mod(X)$ or the class of short exact triples in $\Com^+(X)$. For any integer $n$, we denote by
 $$\kappa^n \colon \R^n\bar f_*\circ u_* \to (\id_S)_*\circ\R^nf_* = \R^nf_*$$
the $n$-th base change natural transformation of functors from $\Com^+(X)$ to $\Mod(S)$ which is associated to the left square in \eqref{e:gmsetup-sq}.

 By $\bar K$ and $\bar L$ we denote the complexes of modules on $\bar X$ which are obtained by relaxing the module multiplication of the terms of the complexes $K$ and $L$, respectively, via the morphism of rings $f^\sharp\colon f^{-1}\O_S \to \O_X$ on $X_\top$. Note that just because we had required $K$ and $L$ to be objects of $\Com(X/f)$, and not only objects of $\Com(X/h)$, it is that $\bar K$ and $\bar L$ are indeed complexes of modules on $\bar X$, \iev the differentials are $f^{-1}\O_S$-linear. We define $l'$ to be the triple of complexes over $\Mod(X')$ which is obtained by relaxing the module multiplication of any of the terms in the complexes $L$, $M$, and $K$ via the composition of morphisms of rings on $X_\top$:
 $$f^\sharp \circ f^{-1}(g^\sharp) \colon f^{-1}g^{-1}(\O_T) \to f^{-1}\O_S \to \O_X.$$
 Furthermore, we set $l_+:=((\bar K,\bar L),l')$. Observe that, for all integers $p$, since \eqref{e:lambdaptlp} holds, $l^p$ is a short exact triple of modules on $X$ by Proposition \ref{p:lambdaplset} \ref{p:lambdaplset-lset}). Thus $l'$ is a short exact triple of complexes of modules on $X'$. Moreover, any of the complexes $L$, $M$, and $K$ is bounded below. Therefore, the complexes $\bar K$ and $\bar L$ are bounded below and $l'$ is a triple of bounded below complexes of modules on $X'$, whence $l_+$ is an augmented triple with respect to $b$, \iev an object of $\cat E_b$, \cf Notation \ref{not:augtrip}.
 
 For any integer $p$, we denote $\gamma^p$ the composition of the following evident morphisms in $\Mod(\bar X)$:
 \begin{align*}
  \bar f^*G \otimes_{\bar X} u_*(K^{p-1}) & \to u_*(u^*\bar f^*G) \otimes_{\bar X} u_*(K^{p-1}) \\ & \to u_*(f^*G) \otimes_{\bar X} u_*(K^{p-1}) \to u_*(f^*G\otimes_X K^{p-1}).
 \end{align*}
 We write $\gamma$ to denote the $\Z$-sequence $(\gamma^p)_{p\in\Z}$. What we want to assume is that $\gamma$ is a morphism of complexes over $\Mod(\bar X)$:
 \[
  \gamma \colon \bar f^*G\otimes_{\bar X} (\bar K[-1]) \to \bar L.
 \]
 For any integer $p$, we define
 \[
  \gamma^{\geq p} \colon \bar f^*G\otimes_{\bar X} ((\sigma^{\geq p-1}\bar K)[-1]) \to \sigma^{\geq p}\bar L
 \]
 to be the morphism in $\Com^+(\bar X)$ which is given by $\gamma^{p'}$ in degree $p'$ for all $p'\in\Z_{\geq p}$ and by the zero morphism in degrees $<p$. Similarly, for any $p\in\Z$, we define
 \[
  \gamma^{=p} \colon \bar f^*G\otimes_{\bar X}((\sigma^{=p-1}\bar K)[-1]) \to \sigma^{=p}\bar L
 \]
 to be the morphism in $\Com^+(\bar X)$ which is given by $\gamma^p$ in degree $p$ and the zero morphism in degrees $\neq p$.
\end{setup}

\begin{sidewaysfigure}
 \centerline{
  \xymatrix@R=6pc@C=5pc{
   \R^n\bar f_*(\bar K) \ar[r]^{\delta^n_+(l_+)} \ar@{}[dr]|{\diagcircled1} & \R^{n+1}\bar f_*(\bar L)  \ar@{}[dr]|{\diagcircled3} & \R^n\bar f_*(\bar f^*G\otimes\bar K) \ar[l]_{\R^{n+1}\bar f_*(\gamma)} \ar@{}[dr]|{\diagcircled5} & G\otimes\R^n\bar f_*(\bar K) \ar[l]_{\pi^n_{\bar f}(G,\bar K)} \\
   \R^n\bar f_*(\sigma^{\geq p}\bar K) \ar[r]^{\delta^n_+(\sigma^{\geq p}l_+)} \ar[u]^{\R^n\bar f_*(i^{\geq p}\bar K)} \ar[d]|{\R^n\bar f_*(j^{\leq p}(\sigma^{\geq p}\bar K))} \ar@{}[dr]|{\diagcircled2} & \R^{n+1}\bar f_*(\sigma^{\geq p}\bar L) \ar[u]|{\R^{n+1}\bar f_*(i^{\geq p}\bar L)} \ar[d]|{\R^{n+1}\bar f_*(j^{\leq p}(\sigma^{\geq p}\bar L))} \ar@{}[dr]|{\diagcircled4} & \R^n\bar f_*(\bar f^*G\otimes\sigma^{\geq p-1}\bar K) \ar[u]|{\R^n\bar f_*(\bar f^*\id_G\otimes i^{\geq p-1}\bar K)} \ar[l]_{\R^{n+1}\bar f_*(\gamma^{\geq p})} \ar[d]|{\R^n\bar f_*(\bar f^*\id_G\otimes j^{\leq p-1}(\sigma^{\geq p-1}\bar K))} \ar@{}[dr]|{\diagcircled6} & G\otimes\R^n\bar f_*(\sigma^{\geq p-1}\bar K) \ar[u]_{\id_G\otimes\R^n\bar f_*(i^{\geq p-1}\bar K)} \ar[l]_{\pi^n_{\bar f}(G,\sigma^{\geq p-1}\bar K)} \ar[d]|{\id_G\otimes\R^n\bar f_*(j^{\leq p-1}(\sigma^{\geq p-1}\bar K))} \\
   \R^n\bar f_*(\sigma^{=p}\bar K) \ar[r]^{\delta^n_+(\sigma^{=p}l_+)} \ar[d]_{\kappa^n(\sigma^{=p}K)} & \R^{n+1}\bar f_*(\sigma^{=p}\bar L) \ar[d]|{\kappa^{n+1}(\sigma^{=p}L)} & \R^n\bar f_*(\bar f^*G\otimes\sigma^{=p-1}\bar K) \ar[l]_{\R^{n+1}\bar f_*(\gamma^{=p})} & G\otimes\R^n\bar f_*(\sigma^{=p-1}\bar K) \ar[l]_{\pi^n_{\bar f}(G,\sigma^{=p-1}\bar K)} \ar[d]^{\id_G\otimes\kappa^n(\sigma^{=p-1}K)} \\
   \R^{n-p}f_*(K^p) \ar[r]_{\delta^{n-p}(l^p)} \ar@{}[ur]|{\diagcircled7} & \R^{n-p+1}f_*(L^p) \ar@{}[urr]|{\diagcircled8} && G\otimes\R^{n-p+1}f_*(K^{p-1}) \ar[ll]^{\pi^{n-p+1}_f(G,K^{p-1})}
  }
 }
 \caption{Diagram whose commutativity in $\Mod(S/g)$ is asserted in Lemma \ref{l:cdlemma}.}
 \label{fig:cdlemma}
\end{sidewaysfigure}

The following lemma is the key step towards proving Theorem \ref{t:grgmcc0} later.

\begin{lemma}
 \label{l:cdlemma}
 Suppose we are in the situation of Setup \ref{set:cdlemma}. Then, for all integers $n$ and $p$, the diagram in Figure \ref{fig:cdlemma} commutes in $\Mod(S/g)$.
\end{lemma}

\begin{proof}
 Fix $n,p\in\Z$. The commutativity of the diagram in Figure \ref{fig:cdlemma} is equivalent to the commutativity of its subdiagrams \diagcircled1--\diagcircled8. We treat the subdiagrams case by case.
 
 The subdiagrams \diagcircled1 and \diagcircled2 commute for
 $$\delta^n_+ \colon \R^nf_* \circ q_0 \to \R^{n+1}f_* \circ q_1$$
 is a natural transformation of functors from $\cat E_b$ to $\Mod(S/c)$, \cf Construction \ref{con:augconnhom}. Additionally, one should point out that the projection functors $q_0$ and $q_1$ commute with the stupid filtration functors $\sigma^{\geq p}$ and $\sigma^{\leq p}$ as well as with the natural transformations $i^{\geq p}$ and $j^{\leq p}$. In particular, we have
 \begin{gather*}
  \begin{aligned}
   q_0(\sigma^{\geq p}(l_+)) & =\sigma^{\geq p}(q_0(l_+))=\sigma^{\geq p}\bar K, & q_1(\sigma^{\geq p}l_+) & =\sigma^{\geq p}(q_1(l_+))=\sigma^{\geq p}\bar L, \\
  q_0(i^{\geq p}(l_+)) & = i^{\geq p}(q_0(l_+))=i^{\geq p}(\bar K), & q_1(i^{\geq p}(l_+)) & =i^{\geq p}(q_1(l_+))=i^{\geq p}(\bar L),
  \end{aligned} \\
  \begin{aligned}
   q_0(\sigma^{=p}(l_+)) & =q_0(\sigma^{\leq p}\sigma^{\geq p}(l_+))=\sigma^{\leq p}(q_0(\sigma^{\geq p}(l_+)))=\sigma^{\leq p}\sigma^{\geq p}(q_0(l_+))=\sigma^{=p}\bar K, \\
   q_1(\sigma^{=p}(l_+)) & =q_1(\sigma^{\leq p}\sigma^{\geq p}(l_+))=\sigma^{\leq p}(q_1(\sigma^{\geq p}(l_+)))=\sigma^{\leq p}\sigma^{\geq p}(q_1(l_+))=\sigma^{=p}\bar L,
  \end{aligned} \\
  q_0(j^{\leq p}(\sigma^{\geq p}l_+))=j^{\leq p}(q_0(\sigma^{\geq p}l_+))=j^{\leq p}(\sigma^{\geq p}(q_0(l_+)))=j^{\leq p}(\sigma^{\geq p}\bar K), \intertext{and}
  q_1(j^{\leq p}(\sigma^{\geq p}l_+))=j^{\leq p}(q_1(\sigma^{\geq p}l_+))=j^{\leq p}(\sigma^{\geq p}(q_1(l_+)))=j^{\leq p}(\sigma^{\geq p}\bar L).
 \end{gather*}

 The commutativity of \diagcircled3 follows now from the identity
 \[
  i^{\geq p}(\bar L)\circ\gamma^{\geq p}=\gamma\circ(\bar f^*\id_G\otimes i^{\geq p-1}(\bar K)[-1])
 \]
 in $\Mod(\bar X)$, which is easily checked degree-wise, and the fact that $\R^{n+1}\bar f_*$ is a functor going from $\Com^+(\bar X)$ to $\Mod(S)$. Similarly, the commutativity of \diagcircled4 follows from the identity
 \[
  j^{\leq p}(\sigma^{\geq p}(\bar L))\circ\gamma^{\geq p}=\gamma^{=p}\circ(\bar f^*\id_G\otimes j^{\leq p-1}(\sigma^{\geq p-1}(\bar K))[-1]).
 \]
 Here, we note that for all objects $F$ and all morphisms $\zeta$ of $\Com^+(\bar X)$, we have
 \[
  \R^{n+1}\bar f_*(F[-1])=\R^n\bar f_*(F) \quad \text{and} \quad \R^{n+1}\bar f_*(\zeta[-1])=\R^n\bar f_*(\zeta).
 \]

 The subdiagrams \diagcircled5 and \diagcircled6 commute for
 $$\pi^n_{\bar f} \colon (- \otimes_S -) \circ (\id_{\Mod(S)} \times \R^n\bar f_*) \to \R^n\bar f_* \circ (- \otimes_{\bar X} -) \circ (\bar f^* \times \id_{\Com^+(\bar X)})$$
 is a natural transformation of functors going from $\Mod(S) \times \Com^+(\bar X)$ to $\Mod(S)$, \cf Construction \ref{con:projmor}.

 Moving on to subdiagram \diagcircled7, we first remark that
 \[
  \R^{n-p}f_*(K^p)=\R^{n-p}f_*(K^p[0])=\R^nf_*((K^p[0])[-p])=\R^nf_*(\sigma^{=p}K)
 \]
 and, in a similar fashion,
 $$\R^{n-p+1}f_*(L^p) = \R^{n+1}f_*(\sigma^{=p}L).$$
 Moreover, $\sigma^{=p}K$ (\resp $\sigma^{=p}L$) is an object of $\Com^+(X)$ and we have $u_*(\sigma^{=p}K) = \sigma^{=p}\bar K$ (\resp $u_*(\sigma^{=p}L)  = \sigma^{=p}(\bar L)$). Thus we see that the domains and codomains which are given for the morphisms $\kappa^n(\sigma^{=p}K)$ and $\kappa^{n+1}(\sigma^{=p}L)$ in the diagram are the correct ones. We need some additional notation: Let $\kappa'^n$ and $\kappa''^n$ denote the $n$-th base change natural transformations which are associated to the right square and outer rectangle of \eqref{e:gmsetup-sq}, respectively. Furthermore, let $\delta'^n$ be the $n$-th connecting homomorphism associated to $f'$. We assert that the following diagram commutes in $\Mod(S')$:
 \begin{equation}
  \label{e:gm0-7}
  \xymatrix@C=6pt{
   \R^nf'_*(b_*(\sigma^{=p}\bar K)) \ar[rrr]^{\kappa''^n(\sigma^{=p}K)} \ar[dd]_{\delta'^n(\sigma^{=p}l')} \ar[rd]^{\kappa'^n(\sigma^{=p}\bar K)} &&& c_*(\R^nf_*(\sigma^{=p}K)) \ar[dd]^{c_*(\delta^n(\sigma^{=p}l))} \\
   & c_*(\R^n\bar f_*(\sigma^{=p}\bar K)) \ar[rru]^{c_*(\kappa^n(\sigma^{=p}K))} \ar[dd]_(0.35){c_*(\delta^n_+(\sigma^{=p}l_+))} \\   
   \R^{n+1}f'_*(b_*(\sigma^{=p}\bar L)) \ar'[r][rrr]_(0){\kappa''^{n+1}(\sigma^{=p}L)} \ar[rd]_{\kappa'^{n+1}(\sigma^{=p}\bar L)} &&& c_*(\R^{n+1}f_*(\sigma^{=p}L)) \\
   & c_*(\R^{n+1}\bar f_*(\sigma^{=p}\bar L)) \ar[rru]_{c_*(\kappa^{n+1}(\sigma^{=p}L))}
  }
 \end{equation}
 Indeed, the upper and lower triangles in \eqref{e:gm0-7} commute due to the functoriality of the construction of the base change natural transformations $\kappa$. The background rectangle in \eqref{e:gm0-7} commutes due to the compatibility of the base change $\kappa''$ with connecting homomorphisms noting that
 $$(b \circ u)_*(\sigma^{=p}l) = \sigma^{=p}l'.$$
 The left foreground rectangle (or ``parallelogram'') in \eqref{e:gm0-7} commutes by the very definition of the augmented connecting homomorphism $\delta^n_+$ noting that the triple underlying $\sigma^{=p}l_+$ is nothing but $\sigma^{=p}l'$, \cf Construction \ref{con:augconnhom}. Now since $\kappa'^n(\sigma^{=p}\bar L)$ is a monomorphism---it is in fact even an isomorphism---, we obtain the commutativity of the right foreground rectangle (or ``parallelogram'') of \eqref{e:gm0-7}. This, in turn, implies the commutativity of \diagcircled7 in $\Mod(S)$ since the functor $c_* \colon \Mod(S) \to \Mod(S')$ is faithful and we have $\sigma^{=p}l = (l^p[0])[-p]$, whence
 $$\delta^n(\sigma^{=p}l) = \delta^n((l^p[0])[-p]) = \delta^{n-p}(l^p[0]) = \delta^{n-p}(l^p).$$

 We are left with \diagcircled8. To this end, define
 \[
  \omega \colon f^*G \otimes (\sigma^{=p-1}K) \to (\sigma^{=p}L)[1]
 \]
 to be the morphism in $\Com^+(S)$ which is given by $\id_{f^*G\otimes K^{p-1}}$ in degree $p-1$ (recall that $L^p = f^*G \otimes \wedge^{p-1}F = f^*G \otimes K^{p-1}$) and by the zero morphism in degrees $\neq p-1$. In addition, define
 \[
  \psi \colon \bar f_*G \otimes (\sigma^{=p-1}\bar K) = \bar f^*G \otimes u_*(\sigma^{=p-1}K) \to u_*(f^*G \otimes (\sigma^{=p-1}K))
 \]
 the obvious base extension morphism in $\Com^+(\bar X)$. We consider the auxiliary diagram in Figure \ref{fig:aux}.
 \begin{figure}[ht]
  \centerline{
  \xymatrix@C=.5pc{
   \R^{n+1}\bar f_*(\sigma^{=p}\bar L) \ar[dd]|{\kappa^{n+1}(\sigma^{=p}L)} && \R^n\bar f_*(\bar f^*G\otimes\sigma^{=p-1}\bar K) \ar[ll]_{\R^{n+1}\bar f_*(\gamma^{=p})} \ar[dl]_{\R^n\bar f_*(\psi)} & G\otimes\R^n\bar f_*(\sigma^{=p-1}\bar K) \ar[l] \ar@{}[l]<-1ex>_{\pi^n_{\bar f}(G,\sigma^{=p-1}\bar K)} \ar[dd]|{\id_G\otimes\kappa^n(\sigma^{=p-1}K)} \\
   & \R^n\bar f_*(u_*(f^*G\otimes\sigma^{=p-1}K)) \ar[ul]_{\R^n\bar f_*(u_*(\omega))} \ar[dd]_(0.35){\kappa^n(f^*G\otimes (\sigma^{=p-1}K))} \\
   \R^{n-p+1}f_*(L^p) &&& G\otimes\R^{n-p+1}f_*(K^{p-1}) \ar'[ll]^(.75){\pi^{n-p+1}_f(G,K^{p-1})}[lll] \ar[dll]^{\pi^n_f(G,\sigma^{=p-1}K)} \\
   & \R^nf_*(f^*G\otimes\sigma^{=p-1}K) \ar[ul]^{\R^nf_*(\omega)}
  }
  }
  \caption{Auxiliary diagram for the proof of Lemma \ref{l:cdlemma}.}
  \label{fig:aux}
 \end{figure}
 By the definition of the projection morphism, \cf Construction \ref{con:projmor}, we have
 \begin{align*}
  & \pi^{n-p+1}_f(G,K^{p-1})=\R^{n-p+1}f_*(\omega[p-1])\circ\pi^{n-p+1}_f(G,K^{p-1}[0]) \\
  & = \R^nf_*(\omega)\circ\pi^n_f(G,(K^{p-1}[0])[-(p-1)])=\R^nf_*(\omega)\circ\pi^n_f(G,\sigma^{=p-1}K).
 \end{align*}
 Hence the bottom triangle of the diagram in Figure \ref{fig:aux} commutes. Taking into account the fact that $\kappa^{n+1}(\sigma^{=p}L)=\kappa^n((\sigma^{=p}L)[1])$, we see that the left foreground rectangle (or ``parallelogram'') commutes as
 $$\kappa^n \colon \R^n\bar f_*\circ u_* \to (\id_S)_*\circ\R^nf_* = \R^nf_*$$
 is a natural transformation of functors going from $\Com^+(X)$ to $\Mod(S)$. The pentagon in the right foreground commutes by compatibility of the projection morphisms with base change. The top triangle commutes since firstly, we have
 \[
  \gamma^{=p}[1]=u_*(\omega)\circ\psi,
 \]
 in $\Com^+(\bar X)$, as is easily checked degree-wise, secondly, $\R^n\bar f_*$ is a functor going from $\Com(\bar X)$ to $\Mod(S)$, and thirdly, $\R^{n+1}\bar f_*(\gamma^{=p})=\R^n\bar f_*(\gamma^{=p}[1])$. Therefore, we have established the commutativity of rectangle in the background of Figure \ref{fig:aux}, which is, however, nothing but \diagcircled8.
\end{proof}

In order to proceed further, we need to establish two more pieces of notation.

\begin{notation}
 \label{not:gm0}
 Assume we are in the situation of Setup \ref{set:cdlemma}. Then for all integers $p$, we know that
 $$\gamma^p \colon \bar f^*G\otimes_{\bar X}\bar K^{p-1} \to \bar L^p$$
 is an isomorphism in $\Mod(\bar X)$ (note here that $\bar K^{p-1} = u_*(K^{p-1})$ and $\bar L^p = u_*(L^p) = u_*(f^*G \otimes \wedge^{p-1}F) = u_*(f^*G \otimes K^{p-1})$). In turn,
 $$\gamma \colon \bar f^*G\otimes_{\bar X}(\bar K[-1]) \to \bar L$$
 is an isomorphism in $\Com^+(\bar X)$ and, for all integers $p$,
 \begin{align*}
  \gamma^{\geq p} & \colon \bar f^*G\otimes_{\bar X} ((\sigma^{\geq p-1}\bar K)[-1]) \to \sigma^{\geq p}\bar L, \\
  \gamma^{=p} & \colon \bar f^*G\otimes_{\bar X}((\sigma^{=p-1}\bar K)[-1]) \to \sigma^{=p}\bar L
 \end{align*}
 are isomorphisms in $\Com^+(\bar X)$. Furthermore, as $G$ is a locally finite free module on $S$, Proposition \ref{p:projmoriso} implies that
 $$\pi^n_{\bar f}(G,-) \colon (G \otimes -) \circ \R^n\bar f_* \to \R^n\bar f_* \circ (\bar f^*G \otimes -)$$
 is a natural equivalence of functors from $\Com^+(\bar X)$ to $\Mod(S)$ for all integers $n$. Thus, for all integers $n$ and $p$, it makes sense to set:
 \begin{align}\label{e:nabla}
  \nabla^n & := (\pi^n_{\bar f}(G,\bar K))^{-1} \circ (\R^{n+1}\bar f_*(\gamma))^{-1} \circ \delta^n_+(l_+), \\ \label{e:nablageq}
  \nabla^{\geq p,n} & := (\pi^n_{\bar f}(G,\sigma^{\geq p-1}\bar K))^{-1} \circ (\R^{n+1}\bar f_*(\gamma^{\geq p}))^{-1} \circ \delta^n_+(\sigma^{\geq p}l_+), \intertext{and} \label{e:nablaeq}
  \nabla^{=p,n} & := (\pi^n_{\bar f}(G,\sigma^{=p-1}\bar K))^{-1} \circ (\R^{n+1}\bar f_*(\gamma^{=p}))^{-1} \circ \delta^n_+(\sigma^{=p}l_+),
 \end{align}
 where we compose in $\Mod(S/g) = \Mod(S/c)$. Observe that \eqref{e:nabla}, \eqref{e:nablageq}, and \eqref{e:nablaeq} correspond to the first, second, and third left-to-right horizontal row of arrows in the diagram in Figure \ref{fig:cdlemma}, respectively.
\end{notation}

\begin{notation}
 \label{not:hodgefilt0}
 Assume we are in the situation of Setup \ref{set:cdlemma} (even though for our immediate concerns it would suffice that an arbitrary morphism $\bar f\colon \bar X\to S$ of commutative ringed spaces as well as an object $\bar K$ of $\Com^+(\bar X)$ be given). For integers $n$ and $p$ we set:
 \[
  F^{p,n} := \im(\R^n\bar f_*(i^{\geq p}\bar K) \colon \R^n\bar f_*(\sigma^{\geq p}\bar K) \to \R^n\bar f_*(\bar K));
 \]
 moreover, we write
 \[
  \iota^n(p) \colon F^{p,n} \to \R^n\bar f_*(\bar K)
 \]
 for the corresponding inclusion morphism of sheaves on $S_\top$, and we write $\lambda^n(p)$ for the unique morphism such that the following diagram commutes in $\Mod(S)$:
 \[
  \xymatrix{
   \R^n\bar f_*(\sigma^{\geq p}\bar K) \ar[rr]^{\R^n\bar f_*(i^{\geq p}\bar K)} \ar@{.>}[dr]_{\lambda^n(p)} && \R^n\bar f_*(\bar K) \\
   & F^{p,n} \ar[ru]_{\iota^n(p)}
  }
 \]
 For all $n\in\Z$, the sequence $(F^{p,n})_{p\in\Z}$ clearly constitutes a descending sequence of submodules of $\R^n\bar f_*(\bar K)$ on $S$. In more formal terms one may express this observation by saying that, for all integers $n,p,p'$ such that $p\leq p'$, there exists a unique morphism $\iota^n(p,p')$ in $\Mod(S)$ such that the following diagram commutes in $\Mod(S)$:
\[
 \xymatrix{
  & \R^n\bar f_*(\bar K) \\
  F^{p,n} \ar[ru]^{\iota^n(p)} && F^{p',n} \ar[ul]_{\iota^n(p')} \ar@{.>}[ll]^{\iota^n(p,p')}
 }
\]
\end{notation}

\begin{proposition}
 \label{p:gt0}
 Suppose we are in the situation of Setup \ref{set:cdlemma}. Let $n$ and $p$ be integers. Then there exists one, and only one, $\zeta$ such that the following diagram commutes in $\Mod(S/g)$:
 \begin{equation} \label{e:gt0}
  \xymatrix{
   \R^n\bar f_*(\bar K) \ar[r]^-{\nabla^n} & G\otimes\R^n\bar f_*(\bar K) \\
   F^{p,n} \ar@{.>}[r]_-{\zeta} \ar[u]^{\iota^n(p)} & G\otimes F^{p-1,n} \ar[u]_{\id_G\otimes\iota^n(p-1)}
  }
 \end{equation}
\end{proposition}

\begin{proof}
 Comparing \eqref{e:nabla} and \eqref{e:nablageq} with the diagram in Figure \ref{fig:cdlemma}, we find that Lemma \ref{l:cdlemma} implies the following identity in $\Mod(S/g)$:
 \begin{equation}
  \label{e:gt0-1}
  \nabla^n\circ\R^n\bar f_*(i^{\geq p}\bar K)=(\id_G\otimes\R^n\bar f_*(i^{\geq p-1}\bar K))\circ\nabla^{\geq p,n}.
 \end{equation}
 Now since $G$ is a locally finite free and hence flat module on $S$, the functor $G \otimes - \colon \Mod(S) \to \Mod(S)$ is exact and thus, in particular, transforms images into images. Specifically, $\id_G\otimes\iota^n(p-1)$ is an image in $\Mod(S)$ of $\id_G\otimes\R^n\bar f_*(i^{\geq p-1}\bar K)$. This in mind, our claim follows readily from \eqref{e:gt0-1}.
\end{proof}

\begin{proposition}
 \label{p:grgt0}
 Suppose we are in the situation of Setup \ref{set:cdlemma}. Let $n$ and $p$ be integers and let $\zeta$ be such that the diagram in \eqref{e:gt0} commutes in $\Mod(S/g)$. Then there exists one, and only one, $\bar\zeta$ rendering commutative in $\Mod(S/g)$ the following diagram:
 \begin{equation} \label{e:grgt0}
  \xymatrix{
   F^{p,n} \ar[r]^-{\zeta} \ar[d]_{\coker(\iota^n(p,p+1))} & G\otimes F^{p-1,n} \ar[d]^{\id_G\otimes\coker(\iota^n(p-1,p))} \\
   F^{p,n}/F^{p+1,n} \ar@{.>}[r]_-{\bar\zeta} & G\otimes (F^{p-1,n}/F^{p,n})
  }
 \end{equation}
\end{proposition}

\begin{proof}
 By Proposition \ref{p:gt0}, there exists $\zeta'$ such that the upper foreground trapezoid in the following diagram commutes in $\Mod(S/g)$:
 \begin{equation}
  \label{e:grgt0-aux}
  \xymatrix@R=3pc@C=12pt{
   & \R^n\bar f_*(\bar K) \ar[rr]^-{\nabla^n} && G\otimes\R^n\bar f_*(\bar K) \\
   F^{p+1,n} \ar[rrrr]^-{\zeta'} \ar[ru]^{\iota^n(p+1)} \ar[dr]_{\iota^n(p,p+1)} &&&& G\otimes F^{p,n} \ar[ul]_{\id_G \otimes \iota^n(p)} \ar[dl]^{\id_G\otimes\iota^n(p-1,p)} \\
   & F^{p,n} \ar[rr]_-{\zeta} \ar[uu]|\hole^(0.6){\iota^n(p)} && G\otimes F^{p-1,n} \ar[uu]|\hole_(0.4){\id_G\otimes\iota^n(p-1)}
  }
 \end{equation}
 We claim that the diagram in \eqref{e:grgt0-aux} commutes in $\Mod(S/g)$ as such. In fact, the left and right triangles commute by the very definitions of $\iota^n(p,p+1)$ and $\iota^n(p-1,p)$, respectively. The background square commutes by our assumption on $\zeta$, \cf \eqref{e:gt0}. The lower trapezoid commutes as a consequence of the already established commutativities taking into account that $\id_G\otimes\iota^n(p-1)$ is a monomorphism, which is due to the flatness of $G$. Using the commutativity of the lower trapezoid in \eqref{e:grgt0-aux}, we obtain:
 \begin{align*}
  & \left(\left(\id_G \otimes \coker(\iota^n(p-1,p))\right) \circ \zeta\right) \circ \iota^n(p,p+1) \\ & = \left(\id_G \otimes \coker(\iota^n(p-1,p))\right) \circ (\id_G \otimes \iota^n(p-1,p)) \circ \zeta' = 0.
 \end{align*}
 Hence, there exists one, and only one, $\bar\zeta$ rendering the diagram in \eqref{e:grgt0} commutative in $\Mod(S/g)$.
\end{proof}

\begin{proposition}
 \label{p:grgm0}
 Suppose we are in the situation of Setup \ref{set:cdlemma}. Let $n$ and $p$ be integers and $\zeta$ and $\bar\zeta$ morphisms such that the diagrams \eqref{e:gt0} and \eqref{e:grgt0} commute in $\Mod(S/g)$. Moreover, let $\psi^p$ and $\psi^{p-1}$ morphisms such that the following diagram commutes in $\Mod(S)$ for $\nu=p,p-1$:
 \begin{equation} \label{e:degen-psi}
  \xymatrix{
   F^{\nu,n} \ar[d]_{\coker(\iota^n(\nu,\nu+1))} & \R^n\bar f_*(\sigma^{\geq \nu}\bar K) \ar[l]_-{\lambda^n(\nu)} \ar[d]^{\R^n\bar f_*(j^{\leq \nu}(\sigma^{\geq \nu}\bar K))} \\
   F^{\nu,n}/F^{\nu+1,n} \ar@{.>}[r]_-{\psi^{\nu}} & \R^n\bar f_*(\sigma^{=\nu}\bar K)
  }
 \end{equation} 
 Then the following diagram commutes in $\Mod(S/g)$:
 \begin{equation} \label{e:grgm0}
  \xymatrix@C=3pc{
   F^{p,n}/F^{p+1,n} \ar[r]^-{\bar\zeta} \ar[d]_{\psi^p} & G\otimes (F^{p-1,n}/F^{p,n}) \ar[d]^{\id_G\otimes\psi^{p-1}} \\
   \R^n\bar f_*(\sigma^{=p}\bar K) \ar[r]_-{\nabla^{=p,n}} & G\otimes\R^n\bar f_*(\sigma^{=p-1}\bar K)
  }
 \end{equation}
\end{proposition}

\begin{proof}
 We proceed in three steps. In each step we derive the commutativity of a certain square-shaped (or maybe better ``trapezoid-shaped'') diagram by means of a ``prism diagram argument''. To begin with, consider the following diagram (``prism''):
 \begin{equation} \label{e:grgm0-aux1}
  \xymatrix@R=3pc@C=12pt{
   & \R^n\bar f_*(\bar K) \ar[rr]^-{\nabla^n} && G\otimes\R^n\bar f_*(\bar K) \\
   F^{p,n} \ar[rrrr]^-{\zeta} \ar[ru]^{\iota^n(p)} &&&& G\otimes F^{p-1,n} \ar[ul]_{\id_G\otimes\iota^n(p-1)} \\
   & \R^n\bar f_*(\sigma^{\geq p}\bar K) \ar[rr]_-{\nabla^{\geq p,n}} \ar[uu]|\hole|(0.65){\R^n\bar f_*(i^{\geq p}\bar K)} \ar[ul]^{\lambda^n(p)} && G\otimes\R^n\bar f_*(\sigma^{\geq p-1}\bar K) \ar[ur]_{\id_G\otimes\lambda^n(p-1)} \ar[uu]|(0.35){\id_G\otimes\R^n\bar f_*(i^{\geq p-1}\bar K)}|\hole
  }
 \end{equation}
 The above diagram is, in fact, commutative: the left and right triangles commute according to the definitions of $\lambda^n(p)$ and $\lambda^n(p-1)$, respectively; the back square (or rectangle) commutes due to Lemma \ref{l:cdlemma}; the upper trapezoid commutes by our assumption on $\zeta$, \cf \eqref{e:gt0}; therefore, the lower trapezoid commutes, too, taking into account that $\id_G\otimes\iota^n(p-1)$ is a monomorphism, which is due to the fact that $G$ is a flat module on $S$. Next, we claim that the diagram
 \begin{equation} \label{e:grgm0-aux2}
  \xymatrix@R=3pc@C=12pt{
   & \R^n\bar f_*(\sigma^{\geq p}\bar K) \ar[rr]^-{\nabla^{\geq p,n}} \ar[dd]|(0.35){\R^n\bar f_*(j^{\leq p}(\sigma^{\geq p}\bar K))}|\hole \ar[dl]_{\lambda^n(p)} && G\otimes\R^n\bar f_*(\sigma^{\geq p-1}\bar K) \ar[dr]^{\id_G\otimes\lambda^n(p-1)} \ar[dd]|\hole|(0.65){\id_G\otimes\R^n\bar f_*(j^{\leq p-1}(\sigma^{\geq p-1}\bar K))} \\
   F^{p,n} \ar[rrrr]^-{\zeta} \ar[dr]_{\phi^p} &&&& G\otimes F^{p-1,n} \ar[dl]^{\id_G\otimes\phi^{p-1}} \\
   & \R^n\bar f_*(\sigma^{=p}\bar K) \ar[rr]_-{\nabla^{=p,n}} && G\otimes\R^n\bar f_*(\sigma^{=p-1}\bar K)
  }
 \end{equation}
 commutes in $\Mod(S/g)$. The left and right triangles commute according to the definitions of $\phi^p$ and $\phi^{p-1}$, respectively. The back square commutes by means of Lemma \ref{l:cdlemma}. The upper trapezoid commutes by the commutativity of the lower trapezoid in the preceding diagram \eqref{e:grgm0-aux1}. Therefore, the lower trapezoid of the diagram in \eqref{e:grgm0-aux2} commutes taking into account the fact that $\lambda^n(p)$ is an epimorphism. Finally, we deduce the commutativity of:
 \begin{equation} \label{e:grgm0-aux3}
  \xymatrix@R=3pc@C=6pt{
   & F^{p,n} \ar[rr]^-{\zeta} \ar[dd]|(0.35){\phi^p}|\hole \ar[dl]_{\coker(\iota^n(p,p+1))} && G\otimes F^{p-1,n} \ar[dr]^{\id_G\otimes\coker(\iota^n(p,p-1))} \ar[dd]|\hole|(0.65){\id_G\otimes\phi^{p-1}} \\
   F^{p,n}/F^{p+1,n} \ar[rrrr]^-{\bar\zeta} \ar[dr]_{\psi^p} &&&& G\otimes F^{p-1,n}/F^{p,n} \ar[dl]^{\id_G\otimes\psi^{p-1}} \\
   & \R^n\bar f_*(\sigma^{=p}\bar K) \ar[rr]_-{\nabla^{=p,n}} && G\otimes\R^n\bar f_*(\sigma^{=p-1}\bar K)
  }
 \end{equation}
 Here, the left and right triangles commute by the definitions of $\psi^p$ and $\psi^{p-1}$, respectively. The back square commutes by the commutativity of the lower trapezoid of the preceding diagram in \eqref{e:grgm0-aux2}. The upper trapezoid of the diagram in \eqref{e:grgm0-aux3} commutes by our assumption on $\bar\zeta$, \cf \eqref{e:grgt0}. Therefore, the lower trapezoid, which is nothing but \eqref{e:grgm0}, commutes taking into account the fact that $\coker(\iota^n(p,p+1))$ is an epimorphism.
\end{proof}

\begin{theorem}
 \label{t:grgmcc0}
 Suppose we are in the situation of Setup \ref{set:cdlemma}. Let $n$ and $p$ be integers and $\zeta$ and $\bar\zeta$ morphisms such that the diagrams \eqref{e:gt0} and \eqref{e:grgm0} commute in $\Mod(S/g)$. Moreover, let $\psi^p$ and $\psi^{p-1}$ be morphisms such that the diagram \eqref{e:degen-psi} commutes in $\Mod(S)$ for $\nu=p,p-1$. Then the following diagram commutes in $\Mod(S/g)$:
 \begin{equation}
  \xymatrix@C=4pc{
   F^{p,n}/F^{p+1,n} \ar[r]^{\bar\zeta} \ar[d]_{\kappa^n(\sigma^{=p}K) \circ \psi^p} & G\otimes (F^{p-1,n}/F^{p,n}) \ar[d]^{\id_G \otimes (\kappa^n(\sigma^{=p-1}K) \circ \psi^{p-1})} \\
   \R^{n-p}f_*(\wedge^pF) \ar[r]_{\gamma^{p,n-p}_{\KS,f}(G,t)} & G\otimes\R^{n-p+1}f_*(\wedge^{p-1}F)
  }
 \end{equation}
\end{theorem}

\begin{proof}
 By Proposition \ref{p:grgm0}, we know that the diagram \eqref{e:grgm0} commutes in $\Mod(S/g)$, \iev
 \[
 (\id_G\otimes\psi^{p-1}) \circ \bar\zeta = \nabla^{=p,n} \circ \psi^p.
 \]
 We know too that
 \[
  \kappa^n \colon \R^n\bar f_* \circ u_* \to (\id_S)_* \circ \R^nf_* = \R^nf_*
 \]
 is a natural equivalence functors going from $\Com^+(X)$ to $\Mod(S)$. Proposition \ref{p:projmoriso} implies that $\pi^{n-p+1}_f(G,K^{p-1})$ is an isomorphism in $\Mod(S)$. Hence by Lemma \ref{l:cdlemma}, specifically the commutativity of subdiagram \diagcircled8 of the diagram in Figure \ref{fig:cdlemma}, we have:
 \[
 (\id_G \otimes \kappa^n(\sigma^{=p-1}K)) \circ \nabla^{=p,n} = (\pi^{n-p+1}_f(G,K^{p-1}))^{-1} \circ \delta^{n-p}(l^p) \circ \kappa^n(\sigma^{=p}(K)).
 \]
 By Proposition \ref{p:connhom}, the following identity holds in $\Mod(S)$ recalling that $K^{p-1}=\wedge^{p-1}F$ and $l^p=\Lambda^p_X(t)$:
 \[
  \gamma^{p,n-p}_{\KS,f}(G,t) = (\pi^{n-p+1}_f(G,K^{p-1}))^{-1} \circ \delta^{n-p}(l^p).
 \]
 Summing up, we obtain a chain of equalities:
 \begin{align*}
  \bigl(\id_G & \otimes (\kappa^n(\sigma^{=p-1}K) \circ \psi^{p-1})\bigr) \circ \bar\zeta \\
  & = (\id_G \otimes \kappa^n(\sigma^{=p-1}K)) \circ (\id_G \otimes \psi^{p-1}) \circ \bar\zeta \\
  & = (\id_G \otimes \kappa^n(\sigma^{=p-1}K)) \circ \nabla^{=p,n} \circ \psi^p \\
  & = (\pi^{n-p+1}_f(G,K^{p-1}))^{-1} \circ \delta^{n-p}(l^p) \circ \kappa^n(\sigma^{=p}(K)) \circ \psi^p \\
  & = \gamma^{p,n-p}_{\KS,f}(G,t) \circ (\kappa^n(\sigma^{=p}K) \circ \psi^p),
 \end{align*}
 which is precisely what we needed to prove.
\end{proof}

\section{The Gauß-Manin connection}
\label{s:ksgm}

In what follows, we are basically applying the results that we have established in the preceding section in a more concrete situation. Our predominant goal here is to prove Theorem \ref{t:grgmcc}, which corresponds to Theorem \ref{t:grgmcc0} of \S\ref{s:cdlemma}.

\begin{notation}
 \label{not:trippdiff}
 Let $(f,g)$ be a composable pair of morphisms of complex spaces, \iev an ordered pair of morphisms such that codomain of $f$ equals the domain of $g$. Put $h:=g\circ f$. Then we denote $\Omega^1(f,g)$ the triple of modules on $X$,
 \begin{equation} \label{e:omega1}
  f^*\Omega^1_g \overset{\alpha}{\to} \Omega^1_{h} \overset{\beta}{\to} \Omega^1_{f},
 \end{equation}
 which we have associated to $(f,g)$ according to \cite[Section 2]{SHC13.2.14}.  We call $\Omega^1(f,g)$ the \emph{triple of $1$-differentials} associated to $(f,g)$. Note that the ordered pair $(\Omega^1_g,\Omega^1(f,g))$ is an object of $\cat D_f$, \cf Construction \ref{con:fibered}.

 By \cite[Corollaire 4.5]{SHC13.2.14} we know that $\Omega^1(f,g)$ is a right exact triple of modules on $X$. Moreover, by \cite[Remarque 4.6]{SHC13.2.14}, $\Omega^1(f,g)$ is a short exact triple of modules on $X$ whenever the morphism $f$ is submersive.
\end{notation}

\begin{notation}[Kodaira-Spencer class, II]
 \label{not:kscc}
 Let $(f,g)$ be a composable pair of submersive morphisms of complex spaces. Write $f\colon X\to S$. Then $(\Omega^1_g,\Omega^1(f,g))$ is an object of $\cat D_f$ such that $\Omega^1(f,g)$ is a short exact triple of modules on $X$. Besides, $(\Omega^1(f,g))(2) = \Omega^1_f$ and $\Omega^1_g$ are locally finite free modules on $X$ and $S$, respectively. Therefore it makes sense to define
 \begin{equation} \label{e:ks}
  \xi_\KS(f,g) := \xi_{\KS,f}(\Omega^1_g,\Omega^1(f,g)),
 \end{equation}
 where the right hand side is understood in the sense of Notation \ref{not:ks0}. Observe that by definition $\xi_\KS(f,g)$ is a morphism
 $$\xi_\KS(f,g) \colon \O_S \to \Omega^1_g \otimes_S \R^1f_*(\Theta_f)$$
 of modules on $S$. Furthermore, we set $\xi_\KS(f) := \xi_\KS(f,a_S)$, where $a_S \colon S\to \C$ denotes the canonical morphism of complex spaces. We call $\xi_\KS(f,g)$ (\resp $\xi_\KS(f)$) the \emph{Kodaira-Spencer class} of $(f,g)$ (\resp $f$).
 
 Fixing in addition to $(f,g)$ two integers $p$ and $q$, we define
 \begin{equation} \label{e:cc}
  \gamma^{p,q}_\KS(f,g) := \gamma^{p,q}_{\KS,f}(\Omega^1_g,\Omega^1(f,g)),
 \end{equation} 
 where we interpret the right hand side in the sense of Notation \ref{not:cc0}. Thus, $\gamma^{p,q}_\KS(f,g)$ is a morphism
 $$\gamma^{p,q}_\KS(f,g) \colon \R^qf_*(\Omega^p_f) \to \Omega^1_g \otimes_S \R^{q+1}f_*(\Omega^{p-1}_f)$$
 of modules on $S$, which we call the \emph{cup and contraction with Kodaira-Spencer class} in bidegree $(p,q)$ for $(f,g)$. As a special case, we set $\gamma^{p,q}_\KS(f) := \gamma^{p,q}_\KS(f,a_S)$ and call this the \emph{cup and contraction with Kodaira-Spencer class} in bidegree $(p,q)$ for $f$.
\end{notation}

\begin{proposition}
 \label{p:ksgeo}
 Let $(f,g)$ be a composable pair of submersive morphisms of complex spaces and $p$ and $q$ integers. Then the following identity holds in $\Mod(S)$, where $S:=\cod(f)$:
 \[
  \delta^q_f\left(\Lambda^p(\Omega^1(f,g))\right) = \pi^{q+1}_f(\Omega^1_g,\Omega^{p-1}_f) \circ \gamma^{p,q}_\KS(f,g).
 \]
\end{proposition}

\begin{proof}
 Apply Proposition \ref{p:connhom} to the morphism of ringed spaces $f$ and the object $(\Omega^1_g,\Omega^1(f,g))$ of $\cat D_f$.
\end{proof}

\begin{notation}
 \label{not:fbar}
 Let $f\colon X\to S$ be a morphism of complex spaces (or else a morphism of ringed spaces). We recall here a notational device that we had established implicitly already in \S\ref{s:cdlemma}. Namely, we define a morphism of ringed spaces $\bar f\colon \bar X \to S$ by setting $\bar X := (X_\top,f^{-1}\O_S)$ and $\bar f := (|f|,\eta_{\O_S} \colon \O_S \to {f_\top}_*f^{-1}\O_S)$, where $\eta_{\O_S}$ denotes the evident adjunction morphism.
\end{notation}

\begin{construction}
 \label{con:tripdr}
 Let $f\colon X\to S$ and $g\colon S\to T$ be morphisms of complex spaces. Set $h:=g\circ f$. We intend to construct a functor
 \[
  \Omega^\kdot(f,g) \colon \mathbf3 \to \Com^+(X/h),
 \]
 \iev a triple of bounded below complexes over $\Mod(X/h)$, which we call the \emph{triple of de Rham complexes} associated to $(f,g)$. In order to simplify notation, we shorten $\Omega^\kdot(f,g)$ to $\Omega^\kdot$ in what follows. To begin with, we define the object function of the functor $\Omega^\kdot$. Recall the set of objects of the category $\mathbf3$ is the set $3 = \{0,1,2\}$. We define $\Omega^\kdot(0)$ to be the unique complex over $\Mod(X/f)$ such that, for all integers $p$, firstly, we have
 \[
  (\Omega^\kdot(0))^p = f^*\Omega^1_g\otimes_X\Omega^{p-1}_f
 \]
 and, secondly, the following diagram commutes in $\Mod(\bar X)$:
 \[
  \xymatrix@C=3pc{
   \bar f^*\Omega^1_g\otimes_{\bar X}\bar\Omega^{p-1}_f \ar[r]^{\id_{\bar f^*\Omega^1_g}\otimes \dd^{p-1}_f} \ar[d]_{\gamma^p} & \bar f^*\Omega^1_g\otimes_{\bar X}\bar\Omega^p_f \ar[d]^{\gamma^{p+1}} \\
   u_*(f^*\Omega^1_g\otimes_X\Omega^{p-1}_f) \ar@{.>}[r]_{\dd^p_{\Omega^\kdot(0)}} & u_*(f^*\Omega^1_g\otimes_X\Omega^p_f)
  }
 \]
 Here, $u \colon X\to \bar X$ stands for the morphism of ringed spaces which is given by $\id_{|X|}$ and $f^\sharp\colon f^{-1}\O_S \to \O_X$, and, for any integer $\nu$, $\gamma^\nu$ signifies the composition of the following morphisms in $\Mod(\bar X)$:
 \begin{align*}
  \bar f^*\Omega^1_g\otimes_{\bar X}\bar\Omega^{\nu-1}_f & \to u_*(u^*\bar f^*\Omega^1_g) \otimes_{\bar X} u_*(\Omega^{\nu-1}_f) \\ & \to u_*(u^*\bar f^*\Omega^1_g\otimes_X\Omega^{\nu-1}_f) \to u_*(f^*\Omega^1_g\otimes_X\Omega^{\nu-1}_f).
 \end{align*}
 Notice that we have $\bar\Omega^{\nu-1}_f = u_*(\Omega^{\nu-1}_f)$. In order to define $\Omega^\kdot(1)$, denote $K^p = (K^{p,i})_{i\in\Z}$ the Koszul filtration in degree $p$ which is induced by
 $$\Omega^1(f,g)|\mathbf2 \colon f^*\Omega^1_g \to \Omega^1_h,$$
 \cf Construction \ref{con:koz}. Then, for all integers $p$ and $i$, one easily verifies that the differential $\dd^p_h$ of the complex $\Omega^\kdot_h$ maps $K^{p,i}$ into $K^{p+1,i}$. Thus we may dipose of a quotient complex:
 \[
  \Omega^\kdot(1) := \Omega^\kdot_h/K^{\kdot,2}.
 \]
 Finally we set:
 \[
  \Omega^\kdot(2) := \Omega^\kdot_f.
 \]
 Moving on to the morphism function of $\Omega^\kdot$, we define, for all ordered pairs $(x,y)$ of objects of $\mathbf3$, \iev all $(x,y) \in 3\times 3$, $\Omega^\kdot(x,y)$ to be the unique function on $\hom_{\mathbf3}(x,y)$ which assigns to all morphisms $a\colon x\to y$ in $\mathbf3$ the $\Z$-sequence
 $$p \mto \left(\Lambda^p_X(\Omega^1(f,g))\right)(x,y).$$
 To verify that the so defined $\Omega^\kdot$ is a functor from $\mathbf3$ to $\Com^+(X/h)$, essentially one has to check that $\Omega^\kdot(0,1)$ (\resp $\Omega^\kdot(1,2)$) constitutes a morphism $\Omega^\kdot(0) \to \Omega^\kdot(1)$ (\resp $\Omega^\kdot(1) \to \Omega^\kdot(2)$) of complexes over $\Mod(X/h)$. This amounts to checking that the morphisms defined by the $\Lambda^p$ construction commute with the differentials of the respective complexes $\Omega^\kdot(x)$, for $x\in 3$, introduced here. In case of $\Omega^\kdot(1,2)$ the desired commutativity is rather clear since the wedge powers of the morphism
 \[
  (\Omega^1(f,g))(1,2) \colon \Omega^1_h \to \Omega^1_f
 \]
 form a morphism $\Omega^\kdot_h \to \Omega^\kdot_f$ of complexes over $\Mod(X/h)$. In case of $\Omega^\kdot(0,1)$ the compatibility is harder to establish as the definition of $(\Lambda^p(t))(0,1)$, for some right exact triple $t$ of modules on $X$, is more involved, \cf Construction \ref{con:lambdap}. Nonetheless, we dare omit this tedious task.
\end{construction}

\begin{remark}
 \label{r:identgm}
 Let $f\colon X\to S$ and $g\colon S\to T$ be submersive morphisms of complex spaces. By abuse of notation we write $f$ and $g$ also for the morphisms of ringed spaces obtained from $f$ and $g$, respectively, by applying the forgetful functor from the category of complex spaces to the category of ringed spaces. Set $G:=\Omega^1_g$ and $t:=\Omega^1(f,g)$, \cf Notation \ref{not:trippdiff}. Then $(G,t)$ clearly is an object of $\cat D_f$, \cf Construction \ref{con:fibered}. As $f$ is a submersive morphism of complex spaces, $t(2) = \Omega^1_f$ is a locally finite free module on $X$ and $t$ is a short exact triple of modules on $X$. Since $g$ is a submersive morphism of complex spaces, $G$ is a locally finite free module on $S$. Further on, set $l:=\Omega^\kdot(f,g)$, \cf Construction \ref{con:tripdr}. Then $l \colon L\to M\to K$ is a triple in $\Com^+(X/h)$, where $h:=g\circ f$, such that $K$ and $L$ are objects of $\Com^+(X/f)$. Moreover, for all integers $p$, we have $l^p = \Lambda^p_X(t)$, where $l^p$ stands for the triple in $\Mod(X/h)$ which is obtained by extracting the degree-$p$ part from the triple of complexes $l$. Define $\gamma$, $\bar f\colon \bar X\to S$, $\bar K$, and $\bar L$ just as in Setup \ref{set:cdlemma}. Then $\gamma$ is a morphism in $\Com^+(\bar X)$,
 \[
  \gamma \colon \bar f^*G\otimes_{\bar X}(\bar K[-1]) \to \bar L,
 \]
 by the very definition of the differentials of the complex $L = l(0) = (\Omega^\kdot(f,g))(0)$, \cf Construction \ref{con:tripdr}.
 Summing up, we see that with the morphisms of ringed spaces $f$ and $g$, with $(G,t)$, and with $l$, we are in the situation of Setup \ref{set:cdlemma}.
\end{remark}

\begin{notation}[Gauß-Manin connection]
 \label{not:gm} \strut
 \begin{enumerate}
  \item Let $(f,g)$ be a composable pair of submersive morphisms of complex spaces and $n$ an integer. Then we set:
  \[
   \nabla^n_\GM(f,g) := \nabla^n,
  \]
  where $\nabla^n$ on the right hand side is given by \eqref{e:nabla} and $G$, $t$, and $l$ are defined precisely as in Remark \ref{r:identgm}. Note that we may use the ``$\nabla^n$'' from Notation \ref{not:gm0} since we are in the situation of Setup \ref{set:cdlemma} as pointed out in Remark \ref{r:identgm}. We call $\nabla^n_\GM(f,g)$ the $n$-th \emph{Gauß-Manin connection} of $(f,g)$.
  \item Let $f\colon X\to S$ be a submersive morphism of complex spaces such that the complex space $S$ is smooth. Let $n$ be an integer. Then we set:
  \[
   \nabla^n_\GM(f) := \nabla^n_\GM(f,a_S),
  \]
  where $a_S \colon S\to \C$ denotes the unique morphism of complex spaces from $S$ to the distinguished one-point complex space. Observe that it makes sense to employ the terminology ``$\nabla^n_\GM(f,g)$'' of part a) since given that the complex space $S$ is smooth, the morphism of complex spaces $a_S$ is submersive.  We call $\nabla^n_\GM(f)$ the $n$-th \emph{Gauß-Manin connection} of $f$.
 \end{enumerate}
\end{notation}

\begin{notation}[Algebraic de Rham module]
 \label{not:dr}
 Let $f$ be a morphism of complex spaces. Let $n$ be an integer. Then we put:
 \[
  \sH^n(f) := \R^n\bar f_*(\bar\Omega^\kdot_f).
 \]
 We call $\sH^n(f)$ the $n$-th \emph{algebraic de Rham module} of $f$.
\end{notation}

\begin{notation}
 \label{not:hodgefilt}
 Let $f\colon X\to S$ be a morphism of complex spaces and $n$ an integer. Then for any integer $p$ we set:
 \[
  \F^p\sH^n(f) := \im(\R^n\bar f_*(i^{\geq p}\bar\Omega^\kdot_f)\colon \R^n\bar f_*(\sigma^{\geq p}\bar\Omega^\kdot_f) \to \R^n\bar f_*(\bar\Omega^\kdot_f))
 \]
 in the sense that $\F^p\sH^n(f)$ is a submodule of $\sH^n(f)$ on $S$; moreover, we write
 \[
  \iota^n_f(p) \colon \F^p\sH^n(f) \to \sH^n(f)
 \]
 for the corresponding inclusion morphism of sheaves on $S_\top$ (note that $\sH^n(f)=\R^n\bar f_*(\bar\Omega^\kdot_f)$ according to Notation \ref{not:dr}). Further on, for any integer $p$, we denote by $\lambda^n_f(p)$ the unique morphism such that the following diagram commutes in $\Mod(S)$:
 \[
  \xymatrix@C1pc{
   \R^n\bar f_*(\sigma^{\geq p}\bar\Omega^\kdot_f) \ar[rr]^{\R^n\bar f_*(i^{\geq p}\bar\Omega^\kdot_f)} \ar@{.>}[dr]_{\lambda^n_f(p)} && \R^n\bar f_*(\bar\Omega^\kdot_f) \\
   & \F^p\sH^n(f) \ar[ru]_{\iota^n_f(p)}
  }
 \]
 Obviously, the sequence $(\F^p\sH^n(f))_{p\in\Z}$ makes up a descending sequence of submodules of $\sH^n(f)$ on $S$. In more formal terms we may express this observation by saying that, for all integers $p,p'$ such that $p\leq p'$, there exists a unique morphism $\iota^n_f(p,p')$ such that the following diagram commutes in $\Mod(S)$:
 \[
  \xymatrix@C1pc{
   & \sH^n(f) \\
   \F^p\sH^n(f) \ar[ru]^{\iota^n_f(p)} && \F^{p'}\sH^n(f) \ar[ul]_{\iota^n_f(p')} \ar@{.>}[ll]^{\iota^n_f(p,p')}
  }
 \]
\end{notation}

\begin{proposition}
 \label{p:gtgrgm}
 Let $n$ and $p$ be integers and $(f,g)$ a composable pair of submersive morphisms of complex spaces. Then there exists one, and only one, ordered pair $(\zeta,\bar\zeta)$ such that abbreviating $\F^*\sH^n(f)$ to $F^*$, the following diagram commutes in $\Mod(S/g)$, where $S:=\dom(g)$:
 \begin{equation}
  \label{e:gtgrgm}
  \xymatrix@C=3pc{
   \sH^n(f) \ar[r]^{\nabla^n_\GM(f,g)} & \Omega^1_g\otimes\sH^n(f) \\
   F^p \ar@{.>}[r]^{\zeta} \ar[u]^{\iota^n_f(p)} \ar[d]_{\coker(\iota^n_f(p,p+1))} & \Omega^1_g\otimes F^{p-1} \ar[u]_{\id_{\Omega^1_g}\otimes\iota^n_f(p-1)} \ar[d]^{\id_{\Omega^1_g}\otimes\coker(\iota^n_f(p-1,p))} \\
  F^p/F^{p+1} \ar@{.>}[r]_{\bar\zeta} & \Omega^1_g\otimes (F^{p-1}/F^p)
  }
 \end{equation}
\end{proposition}

\begin{proof}
 Set $G := \Omega^1_g$, $t := \Omega^1(f,g)$, and $l := \Omega^\kdot(f,g)$ and write $f$ and $g$ also for the morphisms of ringed spaces obtained from $f$ and $g$, respectively, by applying the forgetful functor from the category of complex spaces to the category of ringed spaces. Then according to Remark \ref{r:identgm}, we are in the situation of Setup \ref{set:cdlemma}. Thus our assertion is implied by Proposition \ref{p:gt0} and Proposition \ref{p:grgt0}.
\end{proof}

\begin{notation}
 \label{not:grgm}
 Let $(f,g)$ be a composable pair of submersive morphisms of complex spaces. Then for any integers $n$ and $p$ we set:
 \[
  \bar\nabla^{p,n}_\GM(f,g) := \bar\zeta,
 \]
 where $(\zeta,\bar\zeta)$ is the unique ordered pair such that the diagram \eqref{e:gtgrgm} commutes in $\Mod(S/g)$, where we set $S:=\dom(g)$ and abbreviated $\F^*\sH^n(f)$ to $F^*$, \cf Proposition \ref{p:gtgrgm}.
\end{notation}

\begin{notation}[Hodge module]
 \label{not:hodge}
 Let $f$ be a morphism of complex spaces and $p$ and $q$ integers. Then we put:
 \[
  \sH^{p,q}(f) := \R^qf_*(\Omega^p_f).
 \]
 We call $\sH^{p,q}(f)$ the \emph{Hodge module} in bidegree $(p,q)$ of $f$.
\end{notation}

\begin{theorem}
 \label{t:grgmcc}
 Let $n$ and $p$ be integers and $(f,g)$ a composable pair of submersive morphisms of complex spaces. Let $\psi^p$ and $\psi^{p-1}$ be morphisms in $\Mod(S)$, where $S:=\dom(g)$, such that abbreviating $\F^*\sH^n(f)$ to $F^*$, the following diagram commutes in $\Mod(S)$ for $\nu=p,p-1$:
 \begin{equation}
  \label{e:grgmcc-psi}
  \xymatrix{
    \R^n\bar f_*(\sigma^{\geq\nu}{\bar\Omega}^\kdot_f) \ar[r]^-{\lambda^n_f(\nu)} \ar[d]_{\R^n\bar f_*(j^{\leq\nu}(\sigma^{\geq\nu}\bar\Omega^\kdot_f))} & F^\nu \ar[d]^{\coker(\iota^n_f(\nu,\nu+1))} \\ \R^n\bar f_*(\sigma^{=\nu}\bar\Omega^\kdot_f) & F^\nu/F^{\nu+1} \ar@{.>}[l]^-{\psi^\nu} 
  }
 \end{equation}
 Then the following diagram commutes in $\Mod(S)$:
 \begin{equation}
  \label{e:grgmcc}
  \xymatrix@C4pc{
   F^p/F^{p+1} \ar[r]^-{\bar\nabla^{p,n}_\GM(f,g)} \ar[d]_{\kappa^n_f(\sigma^{=p}\Omega^\kdot_f)\circ\psi^p} & \Omega^1_g \otimes (F^{p-1}/F^p) \ar[d]^{\id_{\Omega^1_g}\otimes(\kappa^n_f(\sigma^{=p-1}\Omega^\kdot_f)\circ\psi^{p-1})} \\
   \sH^{p,n-p}(f) \ar[r]_-{\gamma^{p,n-p}_\KS(f,g)} & \Omega^1_g\otimes\sH^{p-1,n-p+1}(f)
  }
 \end{equation}
\end{theorem}

\begin{proof}
 Set $G := \Omega^1_g$, $t := \Omega^1(f,g)$, and $l := \Omega^\kdot(f,g)$, and write $f$ and $g$ also for the morphisms of ringed spaces obtained from $f$ and $g$, respectively, by applying the forgetful functor from the category of complex spaces to the category of ringed spaces. Then according to Remark \ref{r:identgm}, we are in the situation of Setup \ref{set:cdlemma}. By Proposition \ref{p:gtgrgm} there exists an ordered pair $(\zeta,\bar\zeta)$ such that the diagram \eqref{e:gtgrgm} commutes in $\Mod(S/g)$. Therefore, Theorem \ref{t:grgmcc0} implies that the diagram in \eqref{e:grgmcc} commutes in $\Mod(S)$ since $\bar\nabla^{p,n}_\GM(f,g) = \bar\zeta$ by Notation \ref{not:grgm}.
\end{proof}

\section{Generalities on period mappings}
\label{s:pm0}

This and the next section are devoted to the study of certain ``period mappings''. The common basis for any sort of period mapping that we consider in our exposition is captured by Construction \ref{con:pmrep}. Observe that we like the point of view of defining period mappings in the situation where a representation
\[
 \rho \colon \Pi(X) \to \Mod(A)
\]
of the fundamental groupoid of some topological space $X$ is given ($A$ being some ring), as opposed to the situation where an $A$-local system on $X$, \iev a certain locally constant sheaf of $A_X$-modules on $X$, is given. The reason for this is of technical nature: When working with local systems in the sense of sheaves, one is bound to use the stalks of the given sheaf as reference spaces for the period mappings. When working with representations of the fundamental groupoid on the other hand, one has the liberty of choosing these reference spaces freely (freely meaning up to isomorphism of course). The more familiar setting of working with local systems becomes a special case of the representation setting by means of Construction \ref{con:locsysrep} and Remark \ref{r:locsysrep}. Eventually, we are interested predominantly in ``holomorphic period mappings'' arising from Construction \ref{con:pmhol}, where the local system comes about as the module of horizontal sections associated to a flat vector bundle. Lemma \ref{l:pm0} will give a preliminary, conceptual interpretation of the tangent morphism of such a period mapping. This interpretation will be exploited in the subsequent \S\ref{s:pm} in order to derive the concluding theorems of Chapter \ref{ch:peri} from Theorem \ref{t:grgmcc}.

\begin{notation}
 \label{not:fundgrpd}
 Let $X$ be a topological space. Then we denote by $\Pi(X)$ the \emph{fundamental groupoid} of $X$, \cf \cite[Chapter 2, \S5]{Ma99} for instance.
\end{notation}

\begin{definition}
 \label{d:repdist}
 Let $A$ be a ring and $G$ a groupoid (or just any category for that matter).
 \begin{enumerate}
  \item \label{d:repdist-rep} We say that $\rho$ is an \emph{$A$-representation} of $G$ when $\rho$ is a functor from $G$ to $\Mod(A)$.
  \item \label{d:repdist-dist} Let $\rho$ be an $A$-representation of $G$. Then $F$ is called an \emph{$A$-distribution} in $\rho$ when $F$ is a function whose domain of definition equals $\dom(\rho_0)$ (which in turn equals $G_0$, \iev the set of objects of the category $G$) such that, for all $s\in\dom(\rho_0)$, $F(s)$ is an $A$-submodule of $\rho_0(s)$.
 \end{enumerate}
\end{definition}

\begin{construction}
 \label{con:pmrep}
 Let $A$ be a ring, $S$ a simply connected topological space, $\rho$ an $A$-representation of $\Pi(S)$, $F$ an $A$-distribution in $\rho$, and $t\in S$. Since $S$ is simply connected, we know that for all $s\in S$ there exists a unique morphism $a_{s,t}$ from $s$ to $t$ in $\Pi(S)$, \iev $a_{s,t}$ is the unique element of $(\Pi(S))_1(s,t)$. We define $\cP^A_t(S,\rho,F)$ to be the unique function on $|S|$ such that, for all $s\in S$, we have:
 \[
  (\cP^A_t(S,\rho,F))(s) = \left((\rho_1(s,t))(a_{s,t})\right)[F(t)],
 \]
 where we use square brackets to emphasize that we are referring to the image of a set under a given function. $\cP^A_t(S,\rho,F)$ is called the $A$-\emph{period mapping} on $S$ with basepoint $t$ associated to $\rho$ and $F$.
\end{construction}

\begin{construction}
 \label{con:locsysrep}
 Let $A$ be a ring and $X$ a connected topological space. Let $F$ be a constant sheaf of $A_X$-modules on $X$. We define a functor
 \[
  \rho \colon \Pi(X) \to \Mod(A)
 \]
 as follows: In the first place, we let $\rho_0$ be the unique function on $(\Pi(X))_0$ ($=|X|$) such that, for all $x\in X$, we have:
 \[
  \rho_0(x) = F_x,
 \]
 where the stalk $F_x$ is understood to be equipped with its canonical $A$-module structure. In the second place, we observe that, for all $x\in X$, the evident function
 \[
  \theta_x \colon F(X) \to F_x
 \]
 is a bijection since $F$ is a constant sheaf on $X$ and the topological space $X$ is connected. For all ordered pairs $(x,y)$ of elements of $|X|$, we define $\rho_1(x,y)$ to be the constant function on $\Pi(X)_1(x,y)$ with value $\theta_y\circ(\theta_x)^{-1}$, \iev for all morphisms $a\colon x\to y$ in $\Pi(X)$, we have:
 \[
  (\rho_1(x,y))(a) = \theta_y \circ (\theta_x)^{-1}.
 \]
 Finally, set $\rho := (\rho_0,\rho_1)$. It is an easy matter to verify that the so defined $\rho$ is actually a functor from $\Pi(X)$ to $\Mod(A)$.
\end{construction}

\begin{remark}
 \label{r:locsysrep}
 Construction \ref{con:locsysrep} is in fact a special case of a construction which allows to associate---given a ring $A$ and an arbitrary topological space $X$---to a locally constant sheaf $F$ of $A_X$-modules on $X$ an $A$-representation $\rho$ of the fundamental groupoid of $X$. We briefly sketch how this can be achieved: The object function of $\rho$ is defined just as before, that is, we set $\rho_0(x) := F_x$ for all $x\in X$. However, the morphism function of $\rho$ is harder to define when $F$ is not a constant sheaf but only a locally constant sheaf on $X$. Let $x,y\in X$ and $a\in \Pi(X)_1(x,y)$. Let $\gamma \in a$, \iev $\gamma\colon I \to X$ is a path in $X$ representing $a$, where $I$ stands for the unit interval topologized by the Euclidean topology of $\RR$. Then $\gamma^*F$ is a constant sheaf on $I$. Therefore, one obtains a mapping $(\gamma^*F)_0 \to (\gamma^*F)_1$ in the same fashion as in Construction \ref{con:locsysrep} by passing through the set of global sections of $\gamma^*F$ on $I$. Working in the canonical bijections $(\gamma^*F)_0 \to F_x$ and $(\gamma^*F)_1 \to F_y$, we arrive at a function $F_x \to F_y$. After checking that the latter function $F_x\to F_y$ is independent of the choice $\gamma$ in $a$, we may define $(\rho_1(x,y))(a)$ accordingly. As the reader might imagine, verifying that $(\rho_1(x,y))(a)$ is independent of $\gamma$ is a little tedious, hence we omit it. Next, one has to verify that the $\rho$ defined here is a functor from $\Pi(X)$ to $\Mod(A)$, which again turns out to be a little less obvious than in the ``baby case'' of Construction \ref{con:locsysrep}. Finally, one should convince oneself that in case $F$ is a constant sheaf on $X$ and $X$ is a connected topological space the $\rho$ defined here agrees with the $\rho$ of Construction \ref{con:locsysrep}.
\end{remark}

\begin{definition}
 \label{d:conn}
 Let $S$ be a complex space and $\sH$ a module on $S$.
 \begin{enumerate}
  \item \label{d:conn-rel} Let $g\colon S\to T$ be a morphism of complex spaces. Then $\nabla$ is called a \emph{$g$-connection} on $\sH$ when $\nabla$ is a morphism in $\Mod(S/g)$
  \[
   \nabla \colon \sH \to \Omega^1_g\otimes_S\sH
  \]
  such that for all open sets $U$ of $S$, all $\lambda\in\O_S(U)$, and all $\sigma\in\sH(U)$ Leibniz's rule holds:
  \[
   \nabla_U(\lambda\cdot\sigma) = (\dd_g)_U(\lambda)\otimes\sigma + \lambda\cdot\nabla_U(\sigma).
  \]
  \item \label{d:conn-abs} $\nabla$ is called an \emph{$S$-connection} on $\sH$ when $\nabla$ is a $a_S$-connection on $\sH$ in the sense of part a), where $a_S\colon S\to \C$ denotes the unique morphism of complex spaces from $S$ to the distinguished one-point complex space.
 \end{enumerate}
\end{definition}

\begin{notation}[Module of horizontal sections]
 \label{not:hor}
 Let $g\colon S\to T$ be a morphism of complex spaces, $\sH$ a module on $S$, and $\nabla$ a $g$-connection on $\sH$. Put $S':=(S_\top,g^{-1}\O_T)$ and let
 \[
  c \colon S \to S'
 \]
 be the morphism of ringed spaces given by:
 \[
  (\id_{|S|},g^\sharp \colon g^{-1}\O_T \to \O_S).
 \]
 Then $\nabla$ is a morphism of modules on $S'$ from $c_*(\sH)$ to $c_*(\Omega^1_g\otimes_S \sH)$. Thus it makes sense to set:
 \[
  \Hor_g(\sH,\nabla) := \ker_{S'}(\nabla \colon c_*(\sH) \to c_*(\Omega^1_g\otimes_S \sH)).
 \]
 Note that by definition $\Hor_g(\sH,\nabla)$ is a module on $S'$. We call $\Hor_g(\sH,\nabla)$ the \emph{module of horizontal sections} of $(\sH,\nabla)$ relative $g$. When instead of $g\colon S\to T$ merely a single complex space $S$ is given and $\nabla$ is an $S$-connection on $\sH$, we set $\Hor_S(\sH,\nabla) := \Hor_{a_S}(\sH,\nabla)$, where the right hand side is understood in the already defined sense. $\Hor_S(\sH,\nabla)$ is then called the \emph{module of horizontal sections} of $(\sH,\nabla)$ on $S$.
\end{notation}

\begin{definition}
 \label{d:flat}
 Let $g\colon S\to T$ be a morphism of complex spaces, $\sH$ a module on $S$, and $\nabla$ a $g$-connection on $\sH$. Let $p$ be a natural number. Then there exists a unique morphism
 \[
  \nabla^p \colon \Omega^p_g\otimes\sH \to \Omega^{p+1}_g \otimes \sH
 \]
 in $\Mod(S/g)$ such that for all open sets $U$ of $S$, all $\alpha \in \Omega^p_g(U)$, and all $\sigma \in \sH(U)$, we have:
 \[
  (\nabla^p)_U(\alpha\otimes\sigma) = (\dd^p_g)_U(\alpha) \otimes \sigma + (-1)^p\Lambda_U(\alpha \otimes \nabla_U(\sigma)),
 \]
 where $\Lambda$ stands for the composition of the following morphisms in $\Mod(S)$:
 \[
  \Omega^p_g \otimes (\Omega^1_g \otimes \sH) \to (\Omega^p_g \otimes \Omega^1_g) \otimes \sH \to \Omega^{p+1}_g \otimes \sH.
 \]
 The existence of $\nabla^p$ is in fact not completely obvious, \cf \cite[2.10]{De70}, yet we take it for granted here. We say that $\nabla$ is \emph{flat} (as a $g$-connection on $\sH$) when the composition
 \[
  \nabla^1\circ\nabla \colon \sH \to \Omega^2_g \otimes \sH
 \]
 is the zero morphism in $\Mod(S/g)$.
\end{definition}

\begin{definition}
 \label{d:bundles}
 Let $S$ be a complex space.
 \begin{enumerate}
  \item By a \emph{vector bundle} on $S$ we understand a locally finite free module on $S$.
  \item A \emph{flat vector bundle} on $S$ is an ordered pair $(\sH,\nabla)$ such that $\sH$ is a vector bundle on $S$ and $\nabla$ is a flat $S$-connection on $\sH$.
  \item Let $\sH$ be a vector bundle on $S$. Then $\sh F$ is a \emph{vector subbundle} of $\sH$ on $S$ when $\sh F$ is a locally finite free submodule of $\sH$ on $S$ such that for all $s\in S$ the function
  \[
   \iota(s) \colon \sh F(s) \to \sH(s)
  \]
  is one-to-one, where $\iota \colon \sh F \to \sH$ denotes the inclusion morphism.
 \end{enumerate}
\end{definition}

\begin{proposition}
 \label{p:rhc}
 Let $S$ be a complex manifold and $(\sH,\nabla)$ a flat vector bundle on $S$. Then:
 \begin{enumerate}
  \item \label{p:rhc-locsys} $H:=\Hor_S(\sH,\nabla)$ is a $\C$-local system on $S_\top$.
  \item \label{p:rhc-iso} The sheaf map
  \[
   \O_S \otimes_{\C_S} H \to \O_S \otimes_{\C_S} \sH \to \sH
  \]
  induced by the inclusion $H \subset \sH$ and the $\O_S$-scalar multiplication of $\sH$ is an isomorphism of modules on $S$.
 \end{enumerate}
\end{proposition}

\begin{proof}
 This is implied by \cite[Théorème 2.17]{De70}.
\end{proof}

\begin{construction}
 \label{con:pmhol}
 Let $S$ be a simply connected complex manifold, $(\sH,\nabla)$ a flat vector bundle on $S$, $\sh F$ a submodule of $\sH$ on $S$, and $t\in S$. Put $H:=\Hor_S(\sH,\nabla)$. Then by Proposition \ref{p:rhc}, $H$ is a locally constant sheaf of $\C_S$-modules on $S_\top$. As the topological space $S_\top$ is simply connected, $H$ is even a constant sheaf of $\C_S$-modules on $S_\top$. Thus by means of Construction \ref{con:locsysrep}, we obtain a $\C$-representation $\rho$ of $\Pi(S)$:
 \[
  \rho \colon \Pi(S) \to \Mod(\C).
 \]
 Now for all $s\in S$, we set $\sH(s) := \C\otimes_{\O_{S,s}}\sH_s$ (considered a $\C$-module) and let
 \[
  \psi_s \colon H_s \to \sH(s)
 \]
 denote the evident morphism of $\C$-modules. We define a new functor
 \[
  \rho' \colon \Pi(S) \to \Mod(\C)
 \]
 by composing $\rho$ with the family $(\psi_s)_{s\in S}$; explicitly, that is, we set
 \[
  \rho'_0(s) := \sH(s)
 \]
 for all $s\in S$ and
 \[
  (\rho'_1(x,y))(a) := \psi_y \circ (\rho_1(x,y))(a) \circ (\psi_x)^{-1}
 \]
 for all $x,y\in S$ and all morphisms $a\colon x\to y$ in $\Pi(S)$. One checks without effort that the so declared $\rho'$ is in fact a functor from $\Pi(S)$ to $\Mod(\C)$. Next, define $F$ to be the unique function on $|S|$ such that, for all $s\in S$, we have
 \[
  F(s) = \im(\iota(s) \colon \sh F(s) \to \sH(s)),
 \]
 where $\sh F(s) := \C\otimes_{\O_{S,s}}\sh F_s$ and $\iota(s)$ stands for the morphism derived from the inclusion morphism $\sh F\to \sH$. Then clearly $F$ is a $\C$-distribution in $\rho'$, \cf Definition \ref{d:repdist} b). Therefore it makes sense to set:
 \[
  \cP_t(S,(\sH,\nabla),\sh F) := \cP^\C_t(S_\top,\rho',F)
 \]
 where the right hand side is to be understood in the sense of Construction \ref{con:pmrep}.
\end{construction}

\begin{notation}[Grassmannians]
 \label{not:gr}
 Let $V$ be a finite dimensional $\C$-vector space. Then by $\Gr(V)$ we denote the \emph{Grassmannian} of $V$ regarded as a complex space. Let us elaborate a little on this terminology. First of all, set-theoretically $\Gr(V)$ is plainly the set of all $\C$-vector subspaces of $V$, \iev
 \[
  |\Gr(V)| = \{W: W \text{ is a $\C$-linear subset of } V\}.
 \]
 Note that, in contrast to us, many authors only look at subspaces $W$ of $V$ which are of a certain prescribed dimension $d$. Second of all, a topology as well as a complex structure (in the sense of the theory of complex manifolds, \iev via charts) is defined on $|\Gr(V)|$ by means of \cite[Proposition 10.5]{Vo02}. Lastly, we transform the thus obtained object into a complex space by means of the standard procedure: $\O_{\Gr(V)}$ is the sheaf of holomorphic functions; the morphism of ringed spaces $\Gr(V) \to \C$ is the canonical one.
\end{notation}

\begin{proposition}
 \label{p:pmhol}
 Let $S$ be a simply connected complex manifold, $(\sH,\nabla)$ a flat vector bundle on $S$, $\sh F$ a vector subbundle of $\sH$ on $S$, and $t\in S$. Then $\cP := \cP_t(S,(\sH,\nabla),\sh F)$ is a holomorphic map from $S$ to $\Gr(\sH(t))$.
\end{proposition}

\begin{proof}
 Set $H := \Hor_S(\sH,\nabla)$. Then $H$ is a $\C$-local system on $S$ by Proposition \ref{p:rhc} a). Since $S$ is simply connected, there exists thus a natural number $r$ as well as an isomorphism $(\C_S)^{\oplus r} \to H$ of $\C_S$-modules on $S_\top$. Denote by $e = (e_0,\dots,e_{r-1})$ the thereby induced ordered $\C$-basis of $H(S)$. Let $s_0$ be an arbitrary element of $S$. Then, as $\sh F$ is a locally finite free module on $S$, there exist an open neighborhood $U$ of $s_0$ in $S$, a natural number $d$, as well as an isomorphism
 $$\phi \colon (\O_S|U)^{\oplus d} \to \sh F|U$$
 of modules on $S|U$. Denote, for any $j\in d$, by $\sigma_j$ the image of the $j$-th unit vector in $((\O_S|U)^{\oplus d})(U) = (\O_S(U))^{\oplus d}$ under the function $\phi_U$. Then exploiting the fact that by Proposition \ref{p:rhc} b) the canonical sheaf map
 $$\O_S \otimes_{\C_S} H \to \sH$$
 is an isomorphism of modules on $S$, we see that there exists an $r\times d$-matrix $\lambda = (\lambda_{ij})$ with values in $\O_S(U)$ such that, for all $j\in d$, we have:
 $$\sigma_j = \sum_{i\in r}\lambda_{ij}\cdot (e_i|U),$$
 where we add and multiply in the $\O_S(U)$-module $\sH(U)$. Clearly, for all $s\in U$, the $d$-tuple $(\sigma_0(s),\dots,\sigma_{d-1}(s))$ makes up an ordered $\C$-basis for $\sh F(s)$. Since $\sh F$ is a vector subbundle of $\sH$ on $S$, we know that, for all $s\in S$, the map
 $$\iota(s) \colon \sh F(s) \to \sH(s)$$
 is one-to-one, where $\iota \colon \sh F \to \sH$ stands for the inclusion morphism. Thus, for all $s\in U$, the $d$-tuple given by the association
 $$j \mto \sum_{i\in r}\lambda_{ij}(s) \cdot e_i(s)$$
 constitutes a $\C$-basis of
 $$F(s) := \im(\iota(s) \colon \sh F(s) \to \sH(s)).$$
 Define:
 $$L \colon U \to \C^{r\times d}, \quad L(s) = (\lambda_{ij}(s))_{i\in r,j\in c}.$$
 Then, for all $s\in U$, the columns of the matrix $L(s)$ are linearly independent. In particular, without loss of generality, we may assume that the matrix $L(s_0)|d\times d$ is invertible. Since the functions $s \mto \lambda_{ij}(s)$ are altogether continuous (from $U$ to $\C$), the set $U'$ of elements $s$ of $U$ such that $L(s)|d\times d$ is invertible, is an open neighborhood of $s_0$ in $S$. We set:
 $$L' \colon U' \to \C^{r\times d}, \quad L'(s) = L(s) \cdot (L(s)|d\times d)^{-1}.$$
 Then, for all $s\in U'$, $\cP(s)$ is the linear span in $\sH(t)$ of the following elements:
 $$e_j(t) + \sum_{d\leq i<r} (L'(s))_{ij} \cdot e_i(t),$$
 where $j$ varies through $d$. In other words, setting $c := r - d$ and
 $$L'' \colon U' \to \C^{c\times d}, \quad (L''(s))_{ij} = (L'(s))_{i+d,j},$$
 and letting $h$ signify the mapping which associates to a matrix $M \in \C^{c\times d}$ the linear span in $\sH(t)$ of the elements
 $$e_j(t) + \sum_{i<c} M_{ij} \cdot e_{i+d}(t),$$
 where $j$ varies through $d$, the following diagram commutes in $\Set$:
 \[
  \xymatrix{
   & U' \ar[ld]_{L''} \ar[dr]^{\cP|U'} \\ \C^{c\times d} \ar[rr]_h && \Gr(\sH(t))
  }
 \]
 Since the tuple $(e_0(t),\dots,e_{r-1}(t))$ forms a $\C$-basis of $\sH(t)$, we see that $h$ is one-to-one and $h^{-1}$ composed with the canonical function $\C^{c\times d} \to \C^{cd}$ is a holomorphic chart on the Grassmannian $\Gr(\sH(t))$, \cf Notation \ref{not:gr}. Moreover, the components of $L''$ are holomorphic functions on $S|U'$. This shows that $\cP|U'$ is a holomorphic map from $S|U'$ to $\Gr(\sH(t))$. Since $s_0$ was an arbitrary element of $S$, we infer that $\cP$ is a holomorphic map from $S$ to $\Gr(\sH(t))$.
\end{proof}

\begin{notation}
 \label{not:pmhol}
 Let $S$ be a simply connected complex manifold, $(\sH,\nabla)$ a flat vector bundle on $S$, $\sh F$ a vector subbundle of $\sH$ on $S$, and $t\in S$. Then Proposition \ref{p:pmhol} implies that $\cP_t(S,(\sH,\nabla),\sh F)$ is a holomorphic map from $S$ to $\Gr(\sH(t))$. Therefore, since the complex space $S$ is reduced, there exists one, and only one, morphism of complex spaces
 \[
  \cP^+ \colon S \to \Gr(\sH(t))
 \]
 such that the function underlying $\cP^+$ is precisely $\cP_t(S,(\sH,\nabla),\sh F)$. We agree on setting:
 \[
  \cP_t(S,(\sH,\nabla),\sh F) := \cP^+.
 \]
 Observe that in view of Construction \ref{con:pmhol} this notation is slightly ambiguous. In fact, now $\cP_t(S,(\sH,\nabla),\sh F)$ may refer to a morphism of complex spaces as well as to the function underlying this morphism of complex spaces. Nonetheless, we are confident that this ambiguity will not irritate our readers.
\end{notation}

\begin{construction}
 \label{con:A}
 Let $S$ be a complex space and $t\in S$. Moreover, let $F$ and $H$ be two modules on $S$. We intend to fabricate a mapping
 \[
  \A_{S,t}(F,H) \colon \Hom_S(F,\Omega^1_S\otimes H) \to \Hom_\C(\T_S(t),\Hom(F(t),H(t))).
 \]
 For that matter, let
 $$\phi \colon F \to \Omega^1_S \otimes H$$
 be a morphism of modules on $S$. Then we obtain a morphism
 $$\Theta_S \otimes F \to H$$
 as the composition:
 \begin{gather*}
  \Theta_S \otimes F \to \Theta_S \otimes (\Omega^1_S \otimes H) \to (\Theta_S \otimes \Omega^1_S) \otimes H \to \O_S \otimes H \to H.
 \end{gather*}
 By means of tensor-hom adjunction on $S$ (with respect to the modules $\Theta_S$, $F$, and $H$) the latter morphism corresponds to a morphism
 $$\Theta_S \to \sHom(F,H).$$
 Now evaluating at $t$ and composing with the canonical map
 $$(\sHom(F,H))(t) \to \Hom(F(t),H(t))$$
 yields a morphism in $\Mod(\C)$:
 $$\Theta_S(t) \to \Hom(F(t),H(t)).$$
 Finally, precomposing with the inverse of the canonical isomorphism
 $$\Theta_S(t) \to \T_S(t),$$
 we end up with a morphism
 $$\T_S(t) \to \Hom(F(t),H(t))$$
 in $\Mod(\C)$, which we define to be the image of $\phi$ under $\A_{S,t}(F,H)$. That way we obtain our desired function $\A_{S,t}(F,H)$. Observing that $\A_{S,t}(F,H)$ is a set, we may let $(F,H)$ vary and view $\A_{S,t}$ itself as a function defined on the class of ordered pairs of modules on $S$.
\end{construction}

\begin{proposition}
 \label{p:A}
 Let $S$ be a complex space and $t\in S$. Then $\A_{S,t}$ is a natural transformation of functors from $\Mod(S)^\op \times \Mod(S)$ to $\Set$:
 \[
  \A_{S,t} \colon \Hom_S(-,\Omega^1_S\otimes-) \to \Hom_\C(\T_S(t),\Hom(-(t),-(t))).
 \]
\end{proposition}

\begin{proof}
 One needs to verify that the individual steps applied in Construction \ref{con:A} are altogether natural transformations of appropriate functors from $\Mod(S)^\op \times \Mod(S)$ to $\Set$. We dare omit the details here.
\end{proof}

\begin{remark}
 \label{r:A}
 We would like to give a more down-to-earth interpretation of Proposition \ref{p:A}. So, let $S$ be a complex space and $t\in S$. Let $(F,H)$ and $(F',H')$ be two ordered pairs of modules on $S$ and let
 \[
  (\alpha,\gamma) \colon (F,H) \to (F',H')
 \]
 be a morphism in $\Mod(S)^\op \times \Mod(S)$, \iev $\alpha \colon F' \to F$ and $\gamma \colon H \to H'$ are morphisms in $\Mod(S)$. Moreover, let $\phi$ and $\phi'$ be such that the following diagram commutes in $\Mod(S)$:
 \[
  \xymatrix{
   F \ar[r]^-\phi & \Omega^1_S\otimes H \ar[d]^{\id_{\Omega^1_S}\otimes\gamma} \\
   F' \ar[r]_-{\phi'} \ar[u]^\alpha & \Omega^1_S\otimes H'
  }
 \]
 Then $\phi'$ is the image of $\phi$ under the function:
 \[
  \Hom_S(\alpha,\id_{\Omega^1_S}\otimes\gamma) \colon \Hom_S(F,\Omega^1_S\otimes H) \to \Hom_S(F',\Omega^1_S\otimes H').
 \]
 Therefore, by Proposition \ref{p:A}, $(\A_{S,t}(F',H'))(\phi')$ is the image of $(\A_{S,t}(F,H))(\phi)$ under the function:
 \begin{align*}
  & \Hom_\C(\T_S(t),\Hom(\alpha(t),\gamma(t))) \colon \\ & \Hom_\C(\T_S(t),\Hom(F(t),H(t))) \to \Hom_\C(\T_S(t),\Hom(F'(t),H'(t))),
 \end{align*}
 which translates as the commutativity in $\Mod(\C)$ of the following diagram:
 \[
  \xymatrix@C=1pc{
   & \T_S(t) \ar[ld]_{(\A_{S,t}(F,H))(\phi)} \ar[dr]^{(\A_{S,t}(F',H'))(\phi')} \\
   \Hom(F(t),H(t)) \ar[rr]_{\Hom(\alpha(t),\gamma(t))} && \Hom(F'(t),H'(t))
  }
 \]
 This line of reasoning will be exploited heavily in the proof of Proposition \ref{p:pm'} in the subsequent \S\ref{s:pm}.
\end{remark}

\begin{proposition}
 \label{p:nablaiota}
 Let $S$ be a complex space, $\iota \colon \sh F \to \sH$ a morphism of modules on $S$, and $\nabla$ an $S$-connection on $\sH$. Then
 \begin{equation}
  \label{e:nablaiota}
  \nabla_\iota := (\id_{\Omega^1_S}\otimes\coker(\iota)) \circ \nabla \circ \iota
 \end{equation}
 is a morphism
 $$\nabla_\iota \colon \sh F \to \Omega^1_S \otimes (\sH/\sh F)$$
 in $\Mod(S)$.
\end{proposition}

\begin{proof}
 A priori $\nabla_\iota$ is a morphism from $\sh F$ to $\Omega^1_S \otimes (\sH/\sh F)$ in $\Mod(S/\C)$. That $\nabla_\iota$ is a morphism in $\Mod(S)$ is equivalent to saying that it is compatible with the $\O_S$-scalar multiplications of $\sh F$ and $\Omega^1_S \otimes (\sH/\sh F)$. Let $U$ be an open set of $S$, $\sigma\in\sh F(U)$, and $\lambda\in\O_S(U)$. Then we have:
\[
 (\nabla\circ\iota)_U(\lambda\cdot\sigma) = \nabla_U(\lambda\cdot\iota_U(\sigma)) = \dd_U(\lambda)\otimes\iota_U(\sigma)+\lambda\cdot\nabla_U(\iota_U(\sigma)),
\]
where $\dd$ is short for the differential $\dd^0_S \colon \O_S \to \Omega^1_S$. Thus,
\[
 (\nabla_\iota)_U(\lambda\cdot\sigma) = \lambda\cdot(\nabla_\iota)_U(\sigma),
\]
which proves our claim.
\end{proof}

\begin{notation}
 \label{not:theta}
 Let $V$ be a finite dimensional $\C$-vector space and $F$ a $\C$-linear subset of $V$, \iev $F \in \Gr(V)$. Then we write
 \[
  \theta(V,F) \colon \T_{\Gr(V)}(F) \to \Hom(F,V/F)
 \]
 for the function which is introduced in \cite[Lemme 10.7]{Vo02}. Let us briefly indicate how one defines $\theta(V,F)$. For that matter, be $E$ a $\C$-vector subspace of $V$ such that $V = F \oplus E$. Then one has a morphism of complex spaces
 $$g_E \colon \Hom(F,E) \to \Gr(V)$$
 such that $g_E$ sends a homomorphism $\alpha \colon F \to E$ to the image of the function $\alpha' \colon F \to V$, $\alpha'(x) := x + \alpha(x)$. In fact, $g_E$ is an open immersion which maps the $0$ of $\Hom(F,E)$ to $F$ in $\Gr(V)$. Hence, the tangential map
 $$\T_0(g_E) \colon \T_{\Hom(F,E)}(0) \to \T_{\Gr(V)}(F)$$
 is an isomorphism in $\Mod(\C)$. Moreover, we have a canonical isomorphism
 $$\T_{\Hom(F,E)}(0) \to \Hom(F,E)$$
 as well as the morphism
 $$\Hom(F,E) \to \Hom(F,V/F)$$
 which is induced by the restriction of the quotient mapping $V \to V/F$ to $E$.
\end{notation}

\begin{lemma}
 \label{l:pm0}
 Let $S$ be a simply connected complex manifold, $(\sH,\nabla)$ a flat vector bundle on $S$, $\sh F$ a vector subbundle of $\sH$ on $S$, and $t\in S$. Set
 \[
  \cP := \cP_t(S,(\sH,\nabla),\sh F)
 \]
 and define $\nabla_\iota$ by \eqref{e:nablaiota}, where $\iota \colon \sh F \to \sH$ denotes the inclusion morphism. Moreover, set
 \[
  F(t) := \im(\iota(t) \colon \sh F(t) \to \sH(t))
 \]
 and write
 \begin{align*}
  \sigma & \colon \sh F(t) \to F(t) \\
  \tau & \colon \sH(t)/F(t) \to (\sH/\sh F)(t)
 \end{align*}
 for the obvious mappings. Then $\sigma$ and $\tau$ are isomorphisms in $\Mod(\C)$ and the following diagram commutes in $\Mod(\C)$:
 \begin{equation} \label{e:pm0}
  \xymatrix@C=5pc{
   \T_S(t) \ar[r]^-{\A_{S,t}(\sh F,\sH/\sh F)(\nabla_\iota)} \ar[d]_{\T_t(\cP)} & \Hom(\sh F(t),(\sH/\sh F)(t)) \\
   \T_{\Gr(\sH(t))}(F(t)) \ar[r]_-{\theta(\sH(t),F(t))} & \Hom(F(t),\sH(t)/F(t)) \ar[u]_{\Hom(\sigma,\tau)}
  }
 \end{equation}
\end{lemma}

\begin{proof}
 The fact that $\sigma$ and $\tau$ are isomorphisms is pretty obvious: since $\sh F$ is a vector subbundle of $\sH$ on $S$, $\sigma$ is injective; $\sigma$ is surjective by the definition of $F(t)$. $\tau$ is an isomorphism since both $\sH(t) \to \sH(t)/F(t)$ and $\sH(t) \to (\sH/\sh F)(t)$ are cokernels in $\Mod(\C)$ of $\iota(t) \colon \sh F(t) \to \sH(t)$ taking into account particularly the right exactness of the evaluation functor ``$-(t)$''.

 Set $H := \Hor_S(\sH,\nabla)$. Then by the same method as in the proof of Proposition \ref{p:pmhol} we deduce that there exist an ordered $\C$-basis $e' = (e'_0,\dots,e'_{r-1})$ for $H(S)$, an open neighborhood $U$ of $t$ in $S$, as well as a $c\times d$-matrix $\lambda'$ with values in $\O_S(U)$ such that the tuple
 $$\alpha = (\alpha_0,\dots,\alpha_{d-1})$$
 with
 $$\alpha_j = e'_j|U + \sum_{i<c} \lambda'_{ij} \cdot (e'_{d+i}|U)$$
 for all $j\in d$, where we add in multiply within the $\O_S(U)$-module $\sH(U)$, trivializes the module $\sh F$ on $S$ over $U$, \iev the unique morphism $(\O_S|U)^{\oplus d} \to \sh F|U$ of modules on $S|U$ which sends the standard basis of $(\O_S(U))^{\oplus d}$ to $\alpha$ is an isomorphism. Define a $c$-tuple $e$ and a $d$-tuple $f$ by setting
 \begin{align*}
  e_i & := e'_{d+i}, \\
  f_j & := e'_j + \sum_{i<c} \lambda'_{ij}(t) \cdot e'_{d+i}
 \end{align*}
 for all $i\in c$ and $j\in d$, respectively. Moreover, define a $c\times d$-matrix $\lambda$ by setting for $(i,j) \in c\times d$:
 $$\lambda_{ij} := \lambda'_{ij} - \lambda'_{ij}(t).$$
 Then the concatenated tuple
 $$(f_0,\dots,f_{d-1},e_0,\dots,e_{c-1})$$
 is an ordered $\C$-basis for $H(S)$, and
 $$\alpha_j = f_j|U + \sum_{i<c} \lambda_{ij} \cdot (e_i|U)$$
 for all $j\in d$. Thus, for all $s\in U$, $\cP(s)$ equals the $\C$-linear span of the elements
 $$f_j(t) + \sum_{i<c} \lambda_{ij}(s) \cdot e_i(t),$$
 $j$ ranging through $d$, in $\sH(t)$. Specifically, as $\lambda_{ij}(t) = 0$ for all $(i,j) \in c\times d$, we see that $F(t)$ equals the $\C$-linear span of $f_0(t),\dots,f_{d-1}(t)$ in $\sH(t)$. Define $E$ to be the $\C$-linear span of $e_0(t),\dots,e_{c-1}(t)$ in $\sH(t)$. Let
 $$g \colon \Hom(F(t),E) \to \Gr(\sH(t))$$
 be the morphism of complex spaces which sends an element $\phi \in \Hom(F(t),E)$ to the range of the homomorphism $\id_{F(t)} + \phi \colon F(t) \to \sH(t)$. Let
 $$\bar\cP \colon S|U \to \Hom(F(t),E)$$
 be the morphism of complex spaces which sends $s$ to the homomorphism $F(t) \to E$ which is represented by the $c\times d$-matrix
 $$(i,j) \mto \lambda_{ij}(s)$$
 with respect to the bases $(f_0(t),\dots,f_{d-1}(t))$ and $(e_0(t),\dots,e_{c-1}(t))$ of $F(t)$ and $E$, respectively. Then the following diagram commutes in the category of complex spaces:
 \[
  \xymatrix@C=0pc{
   S|U \ar[rr]^-{\cP|U} \ar[dr]_{\bar\cP} && \Gr(\sH(t)) \\ & \Hom(F(t),E) \ar[ru]_{g}
  }
 \]
 Let $v$ be an arbitrary element of $\T_S(t)$. Then by the explicit description of $\bar\cP$, we see that the image of $v$ under the composition
 $$\can \circ \T_t(\bar\cP) \colon \T_S(t) \to \T_{\Hom(F(t),E)}(0) \to \Hom(F(t),E)$$
 is represented by the matrix
 \begin{equation}
  \label{e:pm0-diffmat}
  c\times d \ni (i,j) \mto v \cont (\dd_S)_U(\lambda_{ij})
 \end{equation}
 with respect to the bases $(f_0(t),\dots,f_{d-1}(t))$ and $(e_0(t),\dots,e_{c-1}(t))$, where, for any $\omega \in \Omega^1_S(U)$,
 $$v \cont \omega := v(\omega(t))$$
 (note that $v$ is a linear functional $v \colon \Omega^1_S(t) = \C \otimes_{\O_{S,s}} \Omega^1_{S,s} \to \C$). By the definition of $\theta(-,-)$ in Notation \ref{not:theta}, when $\pi \colon \sH(t) \to \sH(t)/F(t)$ denotes the residue class mapping, the following diagram commutes in $\Mod(\C)$:
 \[
  \xymatrix@C=4pc{
   \T_{\Hom(F(t),E)}(0) \ar[r]^-{\can} \ar[d]_{\T_0(g)} & \Hom(F(t),E) \ar[d]^{\Hom(\id_{F(t)},\pi|E)} \\
   \T_{\Gr(\sH(t))}(F(t)) \ar[r]_-{\theta(\sH(t),F(t))} & \Hom(F(t),\sH(t)/F(t))
  }
 \]
 Hence, the image of $v$ under the composition $$\theta(\sH(t),F(t)) \circ \T_t(\cP)$$ is represented by the matrix \eqref{e:pm0-diffmat} with respect to the bases $(f_0(t),\dots,f_{d-1}(t))$ and $(\pi(e_0(t)),\dots,\pi(e_{c-1}(t)))$. On the other hand, we have for all $j\in d$:
 \[
  \nabla_U(\alpha_j) = \sum_{i<c} (\dd_S)_U(\lambda_{ij}) \otimes (e_i|U),
 \]
 whence
 \[
  (\nabla_\iota)_U(\alpha_j) = \sum_{i<c} (\dd_S)_U(\lambda_{ij}) \otimes \bar{e_i},
 \]
 where $\bar{e_i}$ denotes the image of $e_i|U$ under the mapping $\sH(U) \to (\sH/\sh F)(U)$. Put $A := \A_{S,t}(\sh F,\sH/\sh F)(\nabla_\iota)$. Then by Contruction \ref{con:A}, we have:
 $$
  Av(\alpha_j(t)) = \sum_{i<c} (v \cont (\dd_S)_U(\lambda_{ij})) \cdot \bar{e_i}(t).
 $$
 Evidently, for all $j\in d$, the mapping $\iota(t) \colon \sh F(t) \to \sH(t)$ sends $\alpha_j(t)$ (evaluation in $\sh F$) to
 $$\alpha_j(t) = f_j(t) + \sum_{i<c} \lambda_{ij}(t) \cdot e_i(t) = f_j(t)$$
 (evaluation in $\sH$); thus $\sigma(\alpha_j(t)) = f_j(t)$. Likewise, for all $i\in c$, the mapping $(\coker\iota)(t) \colon \sH(t) \to (\sH/\sh F)(t)$ sends $e_i(t)$ to $\bar{e_i}(t)$; thus $\tau(\pi(e_i(t))) = \bar{e_i}(t)$. This proves the commutativity of \eqref{e:pm0}.
\end{proof}

\section{Period mappings of Hodge-de Rham type}
\label{s:pm}

After the ground laying work of the previous \S, we are now in the position to analyze period mappings of ``Hodge-de Rham type''; the concept will be made precise in the realm of Notation \ref{not:pm} \ref{not:pm-bc}) below. As a preparation we need the following result.

\begin{proposition}
 \label{p:gmconn}
 Let $n$ be an integer.
 \begin{enumerate}
  \item \label{p:gmconn-fg} Let $(f,g)$ be a composable pair of submersive morphisms of complex spaces. Then $\nabla^n_\GM(f,g)$ is a flat $g$-connection on $\sH^n(f)$.
  \item \label{p:gmconn-f} Let $f\colon X\to S$ be a submersive morphism of complex spaces such that $S$ is a complex manifold. Then $\nabla^n_\GM(f)$ is a flat $S$-connection on $\sH^n(f)$.
  \end{enumerate}
\end{proposition}

\begin{proof}
 Clearly, assertion \ref{p:gmconn-f}) follows from assertion \ref{p:gmconn-fg}) letting $g = a_S$. For the verification of Leibniz's rule and the flatness condition for connections we refer our reader to \cite[Section 2]{KaOd68}.
\end{proof}

\begin{notation}
 \label{not:pm}
 Let $f\colon X\to S$ be a submersive morphism of complex spaces such that $S$ is a simply connected complex manifold. Let $n$ and $p$ be integers and $t\in S$. Assume that $\sH^n(f)$ is a locally finite free module on $S$ and $\F^p\sH^n(f)$ is a vector subbundle of $\sH^n(f)$ on $S$.
 \begin{enumerate}
  \item \label{not:pm-wobc} We put:
  \begin{equation} \label{e:pm'def}
   \cP'^{p,n}_t(f) := \cP_t(S,(\sH^n(f),\nabla^n_\GM(f)),\F^p\sH^n(f)),
  \end{equation}
  where the right hand side is to be interpreted in the sense of Construction \ref{con:pmhol}. Note that \eqref{e:pm'def} makes sense in particular because by Proposition \ref{p:gmconn} b) $\nabla^n_\GM(f)$ is a flat $S$-connection on $\sH^n(f)$, whence $(\sH^n(f),\nabla^n_\GM(f))$ is a flat vector bundle on $S$. Note that by means of Proposition \ref{p:pmhol} we may regard $\cP'^{p,n}_t(f)$ as a morphism of complex spaces:
  \[
   \cP'^{p,n}_t(f) \colon S \to \Gr((\sH^n(f))(t)).
  \]
  \item \label{not:pm-bc} Assume that for all $s\in S$ the base change maps
  \begin{align*}
   \phi^n_{f,s} & \colon (\sH^n(f))(s) \to \sH^n(X_s) \\
   \phi^{p,n}_{f,s} & \colon (\F^p\sH^n(f))(s) \to \F^p\sH^n(X_s)
  \end{align*}
  are isomorphisms in $\Mod(\C)$. Write
  \[
   \rho' \colon \Pi(S) \to \Mod(\C)
  \]
  for the $\C$-representation of the fundamental groupoid of $S$ which is defined for $(\sH,\nabla) := (\sH^n(f),\nabla^n_\GM(f))$ in Construction \ref{con:pmhol}. Let
  \[
   \rho \colon \Pi(S) \to \Mod(\C)
  \]
  be the functor which is obtained by ``composing'' $\rho'$ with the family of isomorphisms $\phi := (\phi^n_{f,s})_{s\in S}$. Define $F$ to be the unique function on $S$ such that, for all $s\in S$, we have:
  \[
   F(s) = \F^p\sH^n(X_s).
  \]
  Then clearly $F$ is a $\C$-distribution in $\rho$. We set:
  \[
   \cP^{p,n}_t(f) := \cP^\C_t(S,\rho,F),
  \]
  \cf Construction \ref{con:pmrep}. Note that $\phi$ is an isomorphism of functors from $\Pi(S)$ to $\Mod(\C)$ from $\rho'$ to $\rho$. Moreover, when $F'$ denotes the unique function on $S$ such that, for all $s\in S$, we have
  \[
   F'(s) = \im((\iota^n_f(p))(s) \colon (\F^p\sH^n(f))(s) \to (\sH^n(f))(s)),
  \]
  then
  \[
   \phi^n_{f,s}[F'(s)] = F(s)
  \]
  for all $s\in S$. Therefore the following diagram commutes in $\Set$:
  \[
   \xymatrix{
    & S \ar[ld]_{\cP'^{p,n}_t(f)} \ar[dr]^{\cP^{p,n}_t(f)} \\
    \Gr((\sH^n(f))(t)) \ar[rr]_-{\Gr(\phi^n_{f,t})} && \Gr(\sH^n(X_t))
   }
  \]
  Since $\cP'^{p,n}_t(f)$ is a holomorphic map from $S$ to $\Gr((\sH^n(f))(t))$ by Proposition \ref{p:pmhol} and $\Gr(\phi^n_{f,t})$ is an isomorphism of complex spaces (as $\phi^n_{f,t}$ is an isomorphisms of $\C$-vector spaces), we may view $\cP^{p,n}_t(f)$ as a morphism of complex spaces from $S$ to $\Gr(\sH^n(X_t))$. We call $\cP^{p,n}_t(f)$ the \emph{(Hodge-de Rham) period mapping} in bidegree $(p,n)$ of $f$ with basepoint $t$.
 \end{enumerate}
\end{notation}

Next, we introduce the classical concept of Kodaira-Spencer maps. Our definition shows how to construct these maps out of the Kodaira-Spencer class given by Notation \ref{not:kscc} and Notation \ref{not:ks0}. As an auxiliary means, we also introduce ``Kodaira-Spencer maps without base change''.

\begin{notation}[Kodaira-Spencer maps]
 \label{not:ksm}
 Let $f\colon X\to S$ be a submersive morphism of complex spaces such that $S$ is a complex manifold. Then, by means of Notation \ref{not:kscc}, we may speak of the Kodaira-Spencer class of $f$, written $\xi_\KS(f)$, which is a morphism
 $$\xi_\KS(f) \colon \O_S \to \Omega^1_S \otimes \R^1f_*(\Theta_f)$$
 of modules on $S$. We write $\KS_f$ for the composition of the following morphisms in $\Mod(S)$:
 \begin{align*}
  \Theta_S & \xrightarrow{\rho(\Theta_S)^{-1}} \Theta_S \otimes \O_S \xrightarrow{\id_{\Theta_S} \otimes \xi_\KS(f)} \Theta_S \otimes (\Omega^1_S \otimes \R^1f_*(\Theta_f)) \\ & \xrightarrow{\alpha(\Theta_S,\Omega^1_S,\R^1f_*(\Theta_f))^{-1}} (\Theta_S \otimes \Omega^1_S) \otimes \R^1f_*(\Theta_f) \\ & \xrightarrow{\gamma^1(\Omega^1_S) \otimes \id_{\R^1f_*(\Theta_f)}} \O_S \otimes \R^1f_*(\Theta_f) \xrightarrow{\lambda(\R^1f_*(\Theta_f))} \R^1f_*(\Theta_f).
 \end{align*}
 Let $t\in S$. Then define
 \[\KS'_{f,t} \colon \T_S(t) \to (\R^1f_*(\Theta_f))(t)\]
 to be the composition of the inverse of the canonical isomorphism $\Theta_S(t)\to\T_S(t)$ with $\KS_f(t) \colon \Theta_S(t) \to (\R^1f_*(\Theta_f))(t)$. We call $\KS'_{f,t}$ the \emph{Kodaira-Spencer map without base change} of $f$ at $t$. Furthermore, define
 \[
  \KS_{f,t} \colon \T_S(t) \to \H^1(X_t,\Theta_{X_t})
 \]
 to be the composition of $\KS'_{f,t}$ with the evident base change morphism:
 \[
  \beta^1_{f,t} \colon (\R^1f_*(\Theta_f))(t) \to \H^1(X_t,\Theta_{X_t}).
 \]
 We call $\KS_{f,t}$ the \emph{Kodaira-Spencer map} of $f$ at $t$.
\end{notation}

\begin{notation}[Cup and contraction, II]
 \label{not:cc'}
 Let $f\colon X\to S$ be an arbitrary morphism of complex spaces. Let $p$ and $q$ be integers. We define
 $$\gamma^{p,q}_f \colon \R^1f_*(\Theta_f) \otimes \sH^{p,q}(f) \to \sH^{p-1,q+1}(f)$$
 to be the composition in $\Mod(S)$ of the cup product morphism
 $$\cupp^{1,q}_f(\Theta_f,\Omega^p_f) \colon  \R^1f_*(\Theta_f) \otimes \R^qf_*(\Omega^p_f) \to \R^{q+1}f_*(\Theta_f\otimes\Omega^p_f)$$
 and the $\R^{q+1}f_*(-)$ of the contraction morphism:
 $$\gamma^p_X(\Omega^1_f) \colon \Theta_f \otimes \Omega^p_f \to \Omega^{p-1}_f,$$
 \cf Notation \ref{not:cont}. $\gamma^{p,q}_f$ is called the \emph{cup and contraction} in bidegree $(p,q)$ for $f$. As a shorthand, we write $\gamma^{p,q}_X$ for $\gamma^{p,q}_{a_X}$.
 
 By means of tensor-hom adjunction on $S$ (with respect to the modules $\R^1f_*(\Theta_f)$, $\sH^{p,q}(f)$, and $\sH^{p-1,q+1}(f)$), the morphism $\gamma^{p,q}_f$ corresponds to a morphism
 $$\R^1f_*(\Theta_f) \to \sHom_S(\sH^{p,q}(f),\sH^{p-1,q+1}(f))$$
 in $\Mod(S)$. Let $t\in S$. Then evaluating the latter morphism at $t$ and composing the result in $\Mod(\C)$ with the canonical morphism
 $$(\sHom_S(\sH^{p,q}(f),\sH^{p-1,q+1}(f)))(t) \to \Hom\left((\sH^{p,q}(f))(t),(\sH^{p-1,q+1}(f))(t)\right),$$
 yields:
 $$\gamma'^{p,q}_{f,t} \colon (\R^1f_*(\Theta_f))(t) \to \Hom\left((\sH^{p,q}(f))(t),(\sH^{p-1,q+1}(f))(t)\right).$$
 We refer to $\gamma'^{p,q}_{f,t}$ as the \emph{cup and contraction without base change} in bidegree $(p,q)$ of $f$ at $t$.
\end{notation}

The following two easy lemmata pave the way for the first essential statement of \S\ref{s:pm}, which is Proposition \ref{p:pm'}.

\begin{lemma}
 \label{l:kscc}
 Let $f\colon X\to S$ be a submersive morphism of complex spaces such that $S$ is a complex manifold. Let $p$ and $q$ be integers and $t\in S$. Then the following identity holds in $\Mod(\C)$:
 \begin{equation}
  \label{e:kscc}
  \A_{S,t}(\sH^{p,q}(f),\sH^{p-1,q+1}(f))(\gamma^{p,q}_\KS(f)) = \gamma'^{p,q}_{f,t} \circ \KS'_{f,t}.
 \end{equation}
\end{lemma}

\begin{proof}
 We argue in several steps. To begin with, observe that the following diagram commutes in $\Mod(S)$:
 \[
  \xymatrix@C=0pc{
   & \Theta_S \otimes \sH^{p,q}(f) \ar[ld]_{\KS_f \otimes \id_{}} \ar[dr]^{\gamma^{p,q}_\KS(f)} \\ \R^1f_*(\Theta_f) \otimes \sH^{p,q}(f) \ar[rr]_-{\gamma^{p,q}_f} && \sH^{p-1,q+1}(f)
  }
 \]
 Thus, by the naturality of tensor-hom adjunction, the next diagram commutes in $\Mod(S)$, too:
 \[
  \xymatrix@C=0pc{
   & \Theta_S \ar[ld]_{\KS_f} \ar[dr] \\ \R^1f_*(\Theta_f) \ar[rr]_{} && \sHom(\sH^{p,q}(f),\sH^{p-1,q+1}(f))
  }
 \]
 From this we deduce by means of evaluation at $t$ that the diagram
 \[
  \xymatrix@C=0pc{
   & \Theta_S(t) \ar[ld]_{\KS_f(t)} \ar[dr] \\ (\R^1f_*(\Theta_f))(t) \ar[rr]_-{\gamma'^{p,q}_{f,t}} && \Hom((\sH^{p,q}(f))(t),(\sH^{p-1,q+1}(f))(t))
  }
 \]
 commutes in $\Mod(\C)$. Plugging in the inverse of the canonical isomorphism
 $$\Theta_S(t) \to \T_S(t)$$
 and taking into account the definitions of $\A_{S,t}$ and $\KS'_{f,t}$, we deduce the validity of \eqref{e:kscc}.
\end{proof}

\begin{lemma}
 \label{l:grgm}
 Let $n$ and $p$ be integers and $f\colon X\to S$ be a submersive morphism of complex spaces such that $S$ is a complex manifold. Put $\sH:=\sH^n(f)$ and, for any integer $\nu$, $\sh F^\nu := \F^\nu\sH^n(f)$. Denote by
 \[
  \bar\iota \colon \sh F^{p-1}/\sh F^p \to \sH/\sh F^p
 \]
 the morphism in $\Mod(S)$ obtained from $\iota^n_f(p-1)\colon\sh F^{p-1}\to\sH$ by quotienting out $\sh F^p$. Moreover, set:
 \[
  \nabla_\iota := (\id_{\Omega^1_S}\otimes\coker(\iota^n_f(p))) \circ \nabla^n_\GM(f) \circ \iota^n_f(p).
 \]
  Then the following diagram commutes in $\Mod(S)$:
 \begin{equation} \label{e:grgm}
  \xymatrix@C=3pc{
   \sh F^p \ar[r]^-{\nabla_\iota} \ar[d]_{\coker(\iota^n_f(p,p+1))} & \Omega^1_S\otimes\sH/\sh F^p \\
   \sh F^p/\sh F^{p+1} \ar[r]_-{\bar\nabla^{p,n}_\GM(f)} & \Omega^1_S\otimes\sh F^{p-1}/\sh F^p \ar[u]_{\id_{\Omega^1_S}\otimes\bar\iota}
  }
 \end{equation}
\end{lemma}

\begin{proof}
 By Proposition \ref{p:gtgrgm}, there exists an ordered pair $(\zeta,\bar\zeta)$ of morphisms in $\Mod(S/\C)$ such that the following two identities hold in $\Mod(S/\C)$:
 \begin{gather*}
  \nabla^n_\GM(f) \circ \iota^n_f(p) = (\id_{\Omega^1_S}\otimes\iota^n_f(p-1))\circ\zeta, \\
  \left(\id_{\Omega^1_S}\otimes\coker\left(\iota^n_f(p-1,p)\right)\right) \circ \zeta = \bar\zeta\circ\coker(\iota^n_f(p,p+1)).
 \end{gather*}
 From this we deduce:
 \begin{align*}
  \nabla_\iota & = \left(\id_{\Omega^1_S}\otimes\coker\left(\iota^n_f(p)\right)\right)\circ\nabla^n_\GM(f) \circ\iota^n_f(p) \\
  & = \left(\id_{\Omega^1_S}\otimes\coker\left(\iota^n_f(p)\right)\right)\circ(\id_{\Omega^1_S}\otimes\iota^n_f(p-1))\circ\zeta \\
  & = \left(\id_{\Omega^1_S}\otimes\left(\coker(\iota^n_f(p))\circ\iota^n_f(p-1)\right)\right)\circ\zeta \\
  & = \left(\id_{\Omega^1_S}\otimes\left(\bar\iota\circ\coker(\iota^n_f(p-1,p))\right)\right)\circ\zeta \\
  & = (\id_{\Omega^1_S}\otimes\bar\iota)\circ\left(\id_{\Omega^1_S}\otimes\coker\left(\iota^n_f(p-1,p)\right)\right)\circ\zeta \\
  & = (\id_{\Omega^1_S}\otimes\bar\iota)\circ\bar\zeta\circ\coker(\iota^n_f(p,p+1)).
 \end{align*}
 Taking into account that $\bar\nabla^{p,n}_\GM(f)=\bar\zeta$ by Notation \ref{not:grgm}, we are finished.
\end{proof}

\begin{proposition}
 \label{p:pm'}
 Let $f\colon X\to S$ be a submersive morphism of complex spaces such that $S$ is a simply connected complex manifold. Let $n$ and $p$ be integers and $t\in S$. In addition, let $\psi^p$ and $\psi^{p-1}$ be such that the following diagram commutes in $\Mod(S)$ for $\nu=p,p-1$:
 \begin{equation} \label{e:froedegen}
  \xymatrix{
   \R^n\bar{f}_*(\sigma^{\geq \nu}\bar\Omega^\kdot_{f}) \ar[r]^-{\lambda^n_{f}(\nu)} \ar[d]_{\R^n\bar{f}_*(j^{\leq \nu}(\sigma^{\geq \nu}\bar\Omega^\kdot_f))} & \F^\nu\sH^n(f) \ar[d]^{\coker(\iota^n_f(\nu,\nu+1))} \\
   \R^n\bar{f}_*(\sigma^{= \nu}\bar\Omega^\kdot_{f}) & \F^\nu\sH^n(f)/\F^{\nu+1}\sH^n(f) \ar@{.>}[l]^-{\psi^\nu}
  }
 \end{equation}
 Let $\omega^{p-1}$ be a left inverse of $\psi^{p-1}$ in $\Mod(S)$. Assume that $\sH^n(f)$ is a locally finite free module on $S$ and $\F^p\sH^n(f)$ is a vector subbundle of $\sH^n(f)$ on $S$. Put
 \begin{align*}
  \alpha' & := \kappa^n_{f}(\sigma^{=p}\Omega^\kdot_{f}) \circ \psi^p \circ \coker(\iota^n_{f}(p,p+1)), \\
  \beta' & := (\iota^n_{f}(p-1)/\F^p\sH^n(f)) \circ \omega^{p-1} \circ (\kappa^n_{f}(\sigma^{=p-1}\Omega^\kdot_{f}))^{-1}.
 \end{align*}
 Moreover, set
 $$F'(t) := \im((\iota^n_f(p))(t) \colon (\F^p\sH^n(f))(t) \to (\sH^n(f))(t))$$
 and write
 \begin{align*}
  \sigma & \colon (\F^p\sH^n(f))(t) \to F'(t), \\
  \tau & \colon (\sH^n(f))(t)/F'(t) \to (\sH^n(f)/\F^p\sH^n(f))(t)
 \end{align*}
 for the evident morphisms. Then $\sigma$ and $\tau$ are isomorphisms in $\Mod(\C)$ and the following diagram commutes in $\Mod(\C)$:
 \begin{equation} \label{e:pm'}
  \xymatrix{
   \T_S(t) \ar[r]^-{\KS'_{f,t}} \ar[dd]_{\T_t(\cP'^{p,n}_t(f))} & (\R^1f_*(\Theta_f))(t) \ar[d]^{\gamma'^{p,n-p}_{f,t}} \\
   & \Hom((\sH^{p,n-p}(f))(t),(\sH^{p-1,n-p+1}(f))(t)) \ar[d]^{\Hom(\alpha'(t)\circ\sigma^{-1},\tau^{-1}\circ\beta'(t))} \\
   \T_{\Gr((\sH^n(f))(t))}(F'(t)) \ar[r] \ar@{}@<-1ex>[r]_-{\theta((\sH^n(f))(t),F'(t))} & \Hom(F'(t),(\sH^n(f))(t)/F'(t)) 
  }
 \end{equation}
\end{proposition}

\begin{proof}
 Introduce the following notational shorthands:
 \begin{align*}
  \sH & := \sH^n(f), & \theta & := \theta(\sH(t),F'(t)), \\ \sh F^* & := \F^*\sH^n(f).
 \end{align*}
 Furthermore, set:
 \[
  \nabla_\iota := (\id_{\Omega^1_S}\otimes\coker(\iota^n_f(p))) \circ \nabla^n_\GM(f) \circ \iota^n_f(p).
 \]
 Then by Lemma \ref{l:pm0}, $\sigma$ and $\tau$ are isomorphisms in $\Mod(\C)$ and the following diagram commutes in $\Mod(\C)$:
 \[
  \xymatrix@C=5pc{
   \T_S(t) \ar[r]^-{\A_{S,t}(\sh F^p,\sH/\sh F^p)(\nabla_\iota)} \ar[d]_{\T_t(\cP'^{p,n}_t(f))} & \Hom(\sh F^p(t),(\sH/\sh F^p)(t)) \ar[d]^{\Hom(\sigma^{-1},\tau^{-1})} \\
   \T_{\Gr(\sH(t))}(F'(t)) \ar[r]_-{\theta} & \Hom(F'(t),\sH(t)/F'(t))
  }
 \]
 By Lemma \ref{l:grgm}, setting
 $$\bar\iota := \iota^n_f(p-1)/\sh F^p \colon \sh F^{p-1}/\sh F^p \to \sH/\sh F^p,$$
 the diagram in \eqref{e:grgm} commutes in $\Mod(S)$. Therefore, with
 $$\bar{\bar\iota} := \coker(\iota^n_f(p,p+1)) \colon \sh F^p \to \sh F^p/\sh F^{p+1}$$
 we have according to Remark \ref{r:A}:
 $$\A_{S,t}(\sh F^p,\sH/\sh F^p)(\nabla_\iota) = \Hom(\bar{\bar\iota}(t),\bar\iota(t)) \circ \A_{S,t}(\sh F^p/\sh F^{p+1},\sh F^{p-1}/\sh F^p)(\bar\nabla^{p,n}_\GM(f)).$$
 Hence the following diagram commutes in $\Mod(\C)$:
 \[
  \xymatrix@C=5pc{
   \T_S(t) \ar[r] \ar@{}@<1ex>[r]^-{\A_{S,t}(\sh F^p/\sh F^{p+1},\sh F^{p-1}/\sh F^p)(\bar\nabla^{p,n}_\GM(f))} \ar[d]_{\T_t(\cP'^{p,n}_t(f))} & \Hom((\sh F^p/\sh F^{p+1})(t),(\sh F^{p-1}/\sh F^p)(t)) \ar[d]^{\Hom(\bar{\bar\iota}(t) \circ \sigma^{-1},\tau^{-1} \circ \bar\iota(t))} \\
   \T_{\Gr(\sH(t))}(F'(t)) \ar[r]_-{\theta} & \Hom(F'(t),(\sH^n(f))(t)/F'(t))
  }
 \]
 By Theorem \ref{t:grgmcc}, the following diagram commutes in $\Mod(S)$:
 \[
  \xymatrix@C=5pc{
   \sh F^p/\sh F^{p+1} \ar[r]^-{\bar\nabla^{p,n}_\GM(f)} \ar[d]_{\kappa^n_f(\sigma^{=p}\Omega^\kdot_f)\circ\psi^p} & \Omega^1_S\otimes\sh F^{p-1}/\sh F^p \ar[d]^{\id_{\Omega^1_S}\otimes(\kappa^n_f(\sigma^{=p-1}\Omega^\kdot_f)\circ\psi^{p-1})} \\
   \sH^{p,n-p} \ar[r]_-{\gamma^{p,n-p}_\KS(f)} & \Omega^1_S\otimes\sH^{p-1,n-p+1}
  }
 \]
 Thus making use of Remark \ref{r:A} again, we obtain:
 \begin{align*}
  & \A_{S,t}(\sh F^p/\sh F^{p+1},\sh F^{p-1}/\sh F^p)(\bar\nabla^{p,n}_\GM(f)) \\ & \begin{aligned} = \Hom\left((\kappa^n_f(\sigma^{=p}\Omega^\kdot_f) \circ \psi^p)(t),(\omega^{p-1} \circ (\kappa^n_f(\sigma^{=p-1}\Omega^\kdot_f))^{-1})(t)\right) \\ \circ \A_{S,t}(\sH^{p,n-p},\sH^{p-1,n-p+1})(\gamma^{p,n-p}_\KS(f)).\end{aligned}
 \end{align*}
 Hence, this next diagram commutes in $\Mod(\C)$:
 \[
  \xymatrix@C=5pc{
   \T_S(t) \ar[r] \ar@{}@<1ex>[r]^-{\A_{S,t}(\sH^{p,n-p},\sH^{p-1,n-p+1})(\gamma^{p,n-p}_\KS(f))} \ar[d]_{\T_t(\cP'^{p,n}_t(f))} & \Hom(\sH^{p,n-p}(t),\sH^{p-1,n-p+1}(t)) \ar[d]^{\Hom(\alpha(t)\circ\sigma^{-1},\tau^{-1}\circ\beta(t))} \\
   \T_{\Gr(\sH(t))}(F'(t)) \ar[r]_-{\theta} & \Hom(F'(t),(\sH^n(f))(t)/F'(t))
  }
 \]
 Employing Lemma \ref{l:kscc}, we infer the commutativity of \eqref{e:pm'}.
\end{proof}

The next result is a variant of the previous Proposition \ref{p:pm'} incorporating base changes.

\begin{theorem}
 \label{t:pm}
 Let $f\colon X\to S$ be a submersive morphism of complex spaces such that $S$ is a simply connected complex manifold. Let $n$ and $p$ be integers and $t\in S$. Let $\psi^p_{X_t}$ and $\psi^{p-1}_{X_t}$ be such that the following diagram commutes in $\Mod(\C)$ for $\nu=p,p-1$:
 \begin{equation} \label{e:froedegenfiber}
  \xymatrix{
   \R^n\bar{a_{X_t}}_*(\sigma^{\geq \nu}\bar\Omega^\kdot_{X_t}) \ar[r]^-{\lambda^n_{X_t}(\nu)} \ar[d]_{\R^n\bar{a_{X_t}}_*(j^{\leq\nu}(\sigma^{\geq\nu}\bar\Omega^\kdot_{X_t}))} & \F^\nu\sH^n(X_t) \ar[d]^{\coker(\iota^n_{X_t}(\nu,\nu+1))} \\
   \R^n\bar{a_{X_t}}_*(\sigma^{= \nu}\bar\Omega^\kdot_{X_t}) & \F^\nu\sH^n(X_t)/\F^{\nu+1}\sH^n(X_t) \ar@{.>}[l]^-{\psi^\nu_{X_t}}
  }
 \end{equation}
 Let $\omega^{p-1}_{X_t}$ be a left inverse of $\psi^{p-1}_{X_t}$ in $\Mod(\C)$. Assume that $\sH^n(f)$ is a locally finite free module on $S$, $\F^p\sH^n(f)$ is a vector subbundle of $\sH^n(f)$ on $S$, and the base change morphisms
 \begin{align*}
  \phi^n_{f,s} & \colon (\sH^n(f))(s) \to \sH^n(X_s), \\
  \phi^{p,n}_{f,s} & \colon (\F^p\sH^n(f))(s) \to \F^p\sH^n(X_s)
 \end{align*}
 are isomorphisms in $\Mod(\C)$ for all $s\in S$. Assume that there exist $\psi^p$, $\psi^{p-1}$, and $\omega^{p-1}$ such that firstly, the diagram in \eqref{e:froedegen} commutes in $\Mod(S)$ for $\nu=p,p-1$ and secondly, $\omega^{p-1}$ is a left inverse of $\psi^{p-1}$ in $\Mod(S)$. Moreover, assume that the Hodge base change map
 \[
  \beta^{p,n-p}_{f,t} \colon (\sH^{p,n-p}(f))(t) \to \sH^{p,n-p}(X_t)
 \]
 is an isomorphism in $\Mod(\C)$. Then, setting
 \begin{align*}
  \alpha & := \kappa^n_{X_t}(\sigma^{=p}\Omega^\kdot_{X_t}) \circ \psi^p_{X_t} \circ \coker(\iota^n_{X_t}(p,p+1)), \\
  \beta & := (\iota^n_{X_t}(p-1)/\F^p\sH^n(X_t)) \circ \omega^{p-1}_{X_t} \circ (\kappa^n_{X_t}(\sigma^{=p-1}\Omega^\kdot_{X_t}))^{-1},
 \end{align*}
 the following diagram commutes in $\Mod(\C)$:
 \begin{equation} \label{e:pm}
  \xymatrix{
   \T_S(t) \ar[r]^{\KS_{f,t}} \ar[dd]_{\T_t(\cP^{p,n}_t(f))} & \H^1(X_t,\Theta_{X_t}) \ar[d]^{\gamma^{p,n-p}_{X_t}} \\
   & \Hom(\sH^{p,n-p}(X_t),\sH^{p-1,n-p+1}(X_t)) \ar[d]^{\Hom(\alpha,\beta)} \\
   \T_{\Gr(\sH^n(X_t))}(\F^p\sH^n(X_t)) \ar[r] \ar@{}@<-1ex>[r]_{\theta(\sH^n(X_t),\F^p\sH^n(X_t))} & \Hom(\F^p\sH^n(X_t),\sH^n(X_t)/\F^p\sH^n(X_t)) 
  }
 \end{equation}
\end{theorem}

\begin{proof}
 We set
 \[
  \phi := \phi^n_{f,t} \colon (\sH^n(f))(t) \to \sH^n(X_t).
 \]
 Then by Notation \ref{not:pm}, the following diagram commutes in the category of complex spaces:
 \[
  \xymatrix{
   & S \ar[ld]_{\cP'^{p,n}_t(f)} \ar[dr]^{\cP^{p,n}_t(f)} \\
   \Gr((\sH^n(f))(t)) \ar[rr]_{\Gr(\phi)} && \Gr(\sH^n(X_t))
  }
 \]
 In consequence, letting
 \[
  F'(t) := \im((\iota^n_f(p))(t) \colon (\F^p\sH^n(f))(t) \to (\sH^n(f))(t)),
 \]
 the following diagram commutes in $\Mod(\C)$:
 \begin{equation} \label{e:pm-1}
  \xymatrix@C=1pc{
   & \T_t(S) \ar[ld]_{\T_t(\cP'^{p,n}_t(f))} \ar[dr]^{\T_t(\cP^{p,n}_t(f))} \\
   \T_{\Gr((\sH^n(f))(t))}(F'(t)) \ar[rr]_-{\T_{F'(t)}(\Gr(\phi))} && \T_{\Gr(\sH^n(X_t))}(\F^p\sH^n(X_t))
  }
 \end{equation}
 Note that
 $$(\Gr(\phi))(F'(t)) = \phi[F'(t)] = \F^p\sH^n(X_t)$$
 due to the commutativity of
 \[
  \xysquare{(\F^p\sH^n(f))(t)}{\F^p\sH^n(X_t)}{(\sH^n(f))(t)}{\sH^n(X_t)}{\phi^{p,n}_{f,t}}{(\iota^n_f(p))(t)}{\iota^n_{X_t}(p)}{\phi}
 \]
 in $\Mod(\C)$ and the fact that $\phi^{p,n}_{f,t}$ is an isomorphism. So, when we denote by
 $$\bar\phi \colon (\sH^n(f))(t)/F'(t) \to \sH^n(X_t)/\F^p\sH^n(X_t)$$
 the morphism which is induced by $\phi$ the obvious way, the following diagram commutes in $\Mod(\C)$ by means of the the naturality of $\theta(-,-)$, \cf Notation \ref{not:theta}:
 \begin{equation} \label{e:pm-2}
  \xymatrix{
   \T_{\Gr((\sH^n(f))(t))}(F'(t)) \ar[r]^{\T_{F'(t)}(\Gr(\phi))} \ar[d]_{\theta((\sH^n(f))(t),F'(t))} & \T_{\Gr(\sH^n(X_t))}(\F^p\sH^n(X_t)) \ar[d]^{\theta(\sH^n(X_t),\F^p\sH^n(X_t))} \\ \Hom(F'(t),(\sH^n(f))(t)/F'(t)) \ar[r] \ar@{}@<-1ex>[r]_{\Hom((\phi|F'(t))^{-1},\bar\phi)} & \Hom(\F^p\sH^n(X_t),\sH^n(X_t)/\F^p\sH^n(X_t))
  }
 \end{equation}
 Define
 \begin{align*}
  \alpha' & := \kappa^n_{f}(\sigma^{=p}\Omega^\kdot_{f}) \circ \psi^p \circ \coker(\iota^n_{f}(p,p+1)), \\
  \beta' & := (\iota^n_{f}(p-1)/\F^p\sH^n(f)) \circ \omega^{p-1} \circ (\kappa^n_{f}(\sigma^{=p-1}\Omega^\kdot_{f}))^{-1}
 \end{align*}
 and introduce the evident morphisms:
 \begin{align*}
  & \sigma \colon (\F^p\sH^n(f))(t) \to F'(t), \\
  & \tau \colon (\sH^n(f))(t)/F'(t) \to (\sH^n(f)/\F^p\sH^n(f))(t).
 \end{align*}
 Then by Proposition \ref{p:pm'}, $\sigma$ and $\tau$ are isomorphisms in $\Mod(\C)$ and the diagram in \eqref{e:pm'} commutes in $\Mod(\C)$. Let
 $$\bar\phi_1 \colon (\sH^n(f)/\F^p\sH^n(f))(t) \to \sH^n(X_t)/\F^p\sH^n(X_t)$$
 be the morphism which is naturally induced by $\phi$---similar to $\bar\phi$ above---using the fact that
 $$\left(\coker(\iota^n_f(p))\right)(t) \colon (\sH^n(f))(t) \to (\sH^n(f)/\F^p\sH^n(f))(t)$$
 is a cokernel for $(\iota^n_f(p))(t)$ in $\Mod(\C)$ as the evaluation functor ``$-(t)$'' is a right exact functor from $\Mod(S)$ to $\Mod(\C)$. Then it is a straightforward matter to verify these identities:
 \begin{equation}
  \label{e:pm-3}
  \begin{split}
   \phi^{p,n}_{f,t} & = (\phi|F'(t)) \circ \sigma, \\ \bar\phi & = \bar\phi_1 \circ \tau.
  \end{split}
 \end{equation} 
 Moreover, comparing $\alpha'$ and $\alpha$ and setting $q := n-p$, we see that this next diagram commutes in $\Mod(\C)$:
 \begin{equation}
  \label{e:pm-4a}
  \xysquare{(\F^p\sH^n(f))(t)}{\F^p\sH^n(X_t)}{(\sH^{p,q}(f))(t)}{\sH^{p,q}(X_t)}{\phi^{p,n}_{f,t}}{\alpha'(t)}{\alpha}{\beta^{p,q}_{f,t}}
 \end{equation}
 Similarly, comparing $\beta'$ and $\beta$, we see that
 \begin{equation}
  \label{e:pm-4b}
  \xysquare{(\sH^{p-1,q+1}(f))(t)}{\sH^{p-1,q+1}(X_t)}{(\sH^n(f)/\F^p\sH^n(f))(t)}{\sH^n(X_t)/\F^p\sH^n(X_t)}{\beta^{p-1,q+1}_{f,t}}{\beta'(t)}{\beta}{\bar\phi_1}
 \end{equation}
 commutes in $\Mod(\C)$. Let us write
 $$\beta^1_{f,t} \colon (\R^1f_*(\Theta_f))(t) \to \H^1(X_t,\Theta_{X_t})$$
 for the evident base change morphism. Then, since the cup product morphisms $\cupp^{1,q}$ as well as the contraction morphisms $\gamma^p$ are compatible with base change, the following diagram commutes in $\Mod(\C)$:
 \[
  \xymatrix{
   (\R^1f_*(\Theta_f) \otimes_S \sH^{p,q}(f))(t) \ar[r]^-{\gamma^{p,q}_f(t)} \ar[d]_{\can} & (\sH^{p-1,q+1}(f))(t) \ar[dd]^{\beta^{p-1,q+1}_{f,t}} \\
   (\R^1f_*(\Theta_f))(t) \otimes_\C (\sH^{p,q}(f))(t) \ar[d]_{\beta^1_{f,t}\otimes \beta^{p,q}_{f,t}} \\
   \H^1(X_t,\Theta_{X_t}) \otimes_\C \sH^{p,q}(X_t) \ar[r]_-{\gamma^{p,q}_{X_t}} & \sH^{p-1,q+1}(X_t)
  }
 \]
 Therefore, given that $\beta^{p,q}_{f,t}$ is an isomorphism by assumption, the following diagram commutes in $\Mod(\C)$ also:
 \begin{equation}
  \label{e:pm-5}
  \xymatrix{
   (\R^1f_*(\Theta_f))(t) \ar[r]^{\beta^1_{f,t}} \ar[d]_{\gamma'^{p,q}_{f,t}} & \H^1(X_t,\Theta_{X_t}) \ar[d]^{\gamma^{p,q}_{X_t}} \\ \Hom((\sH^{p,q}(f))(t),(\sH^{p-1,q+1}(f))(t)) \ar[r] \ar@{}@<-1ex>[r]_{\Hom((\beta^{p,q}_{f,t})^{-1},\beta^{p-1,q+1}_{f,t})} & \Hom(\sH^{p,q}(X_t),\sH^{p-1,q+1}(X_t))
  }
 \end{equation}
 According to the definition of the Kodaira-Spencer map $\KS_{f,t}$ in Notation \ref{not:ksm}, the following diagram commutes in $\Mod(\C)$:
 \begin{equation}
  \label{e:pm-6}
  \xymatrix@C=1pc{
   & \T_t(S) \ar[ld]_{\KS'_{f,t}} \ar[dr]^{\KS_{f,t}} \\
   (\R^1f_*(\Theta_f))(t) \ar[rr]_-{\beta^1_{f,t}} && \H^1(X_t,\Theta_{X_t})
  }
 \end{equation}
 Taking all our previous considerations into account, we obtain:
 \begin{align*}
  & \theta(\sH^n(X_t),\F^p\sH^n(X_t)) \circ \T_t(\cP^{p,n}_t(f)) \\
  & \overset{\eqref{e:pm-1}}{=} \theta(\sH^n(X_t),\F^p\sH^n(X_t)) \circ \T_{F'(t)}(\Gr(\phi)) \circ \T_t(\cP'^{p,n}_t(f)) \\
  & \overset{\eqref{e:pm-2}}{=} \Hom((\phi|F'(t))^{-1},\bar\phi) \circ \theta((\sH^n(f))(t),F'(t)) \circ \T_t(\cP'^{p,n}_t(f)) \\
  & \overset{\eqref{e:pm'}}{=} \Hom((\phi|F'(t))^{-1},\bar\phi) \circ \Hom(\alpha'(t)\circ\sigma^{-1},\tau^{-1}\circ\beta'(t)) \circ \gamma'^{p,n-p}_{f,t} \circ \KS'_{f,t} \\
  & \overset{\eqref{e:pm-3}}{=} \Hom(\alpha'(t) \circ (\phi^{p,n}_{f,t})^{-1},\bar\phi_1 \circ \beta'(t)) \circ \gamma'^{p,n-p}_{f,t} \circ \KS'_{f,t} \\
  & \overset{\substack{\eqref{e:pm-4a}\\ \eqref{e:pm-4b}}}{=} \Hom(\alpha,\beta) \circ \Hom((\beta^{p,q}_{f,t})^{-1},\beta^{p-1,q+1}_{f,t}) \circ \gamma'^{p,q}_{f,t} \circ \KS'_{f,t} \\
  & \overset{\eqref{e:pm-5}}{=} \Hom(\alpha,\beta) \circ \gamma^{p,q}_{X_t} \circ \beta^1_{f,t} \circ \KS'_{f,t} \\
  & \overset{\eqref{e:pm-6}}{=} \Hom(\alpha,\beta) \circ \gamma^{p,n-p}_{X_t} \circ \KS_{f,t},
 \end{align*}
 which implies precisely the commutativity of \eqref{e:pm}.
\end{proof}

When it comes to applying Theorem \ref{t:pm}, one is faced with the problem of deciding whether there exist morphisms $\psi^\nu$ (\resp $\psi^\nu_{X_t}$) rendering commutative in $\Mod(S)$ (\resp $\Mod(\C)$) the diagram in \eqref{e:froedegen} (\resp \eqref{e:froedegenfiber}). Let us formulate two (hopefully) tangible criteria.

\begin{proposition}
 \label{p:froedegen}
 Let $n$ and $\nu$ be integers and $f\colon X\to S$ an arbitrary morphism of complex spaces. Denote by $E$ the Frölicher spectral sequence of $f$.
 \begin{enumerate}
  \item \label{p:froedegen-behind} The following are equivalent:
   \begin{enumerate}
    \item $E$ degenerates from behind in $(\nu,n-\nu)$ at sheet $1$ in $\Mod(S)$;
    \item there exists $\psi^\nu$ rendering commutative in $\Mod(S)$ the diagram in \eqref{e:froedegen}.
   \end{enumerate}
  \item \label{p:froedegen-total} The following are equivalent:
   \begin{enumerate}
    \item $E$ degenerates in $(\nu,n-\nu)$ at sheet $1$ in $\Mod(S)$;
    \item there exists an isomorphism $\psi^\nu$ rendering commutative in $\Mod(S)$ the diagram in \eqref{e:froedegen}.
   \end{enumerate}
 \end{enumerate}
\end{proposition}

\begin{proof}
 The above statements are special cases of standard interpretations of the degeneration of a spectral sequence associated to a filtered complex. We refer our readers to Deligne's treatment \cite[\S1]{De71}.
\end{proof}

\begin{theorem}
 \label{t:pmclassic}
 Let $n$ be an integer and $f\colon X\to S$ a submersive morphism of complex spaces such that $S$ is a simply connected complex manifold. Assume that:
 \begin{enumeratei}
  \item \label{t:pmclassic-froe} the Frölicher spectral sequence of $f$ degenerates in entries
  \[
   I := \{(p,q) \in \Z\times\Z : p + q = n\}
  \]
  at sheet $1$ in $\Mod(S)$;
  \item \label{t:pmclassic-flf} for all $(p,q)\in I$, $\sH^{p,q}(f)$ is a locally finite free module on $S$;
  \item \label{t:pmclassic-froefiber} for all $s\in S$, the Frölicher spectral sequence of $X_s$ degenerates in entries $I$ at sheet $1$ in $\Mod(\C)$;
  \item \label{t:pmclassic-bc} for all $s\in S$ and all $(p,q)\in I$, the Hodge base change map
  \[
   \beta^{p,q}_{f,s} \colon (\sH^{p,q}(f))(s) \to \sH^{p,q}(X_s)
  \]
  is an isomorphism in $\Mod(\C)$.
 \end{enumeratei}
 Let $t\in S$. Then there exists a sequence $(\tilde\psi^\nu)_{\nu\in\Z}$ of isomorphisms in $\Mod(\C)$,
 \[
  \tilde\psi^\nu \colon \F^\nu\sH^n(X_t)/\F^{\nu+1}\sH^n(X_t) \to \sH^{\nu,n-\nu}(X_t),
 \]
 such that, for all $p\in\Z$, the diagram in \eqref{e:pm} commutes in $\Mod(\C)$, where we set:
 \begin{equation} \label{e:pmclassic-ab}
  \begin{split}
  \alpha & := \tilde\psi^p \circ \coker(\iota^n_{X_t}(p,p+1)), \\ \beta & := (\iota^n_{X_t}(p-1)/\F^p\sH^n(X_t)) \circ (\tilde\psi^{p-1})^{-1}.
  \end{split}
 \end{equation}
\end{theorem}

\begin{proof}
 By assumption \eqref{t:pmclassic-froe} we know using Proposition \ref{p:froedegen} that, for all integers $\nu$, there exists one, and only one, $\psi^\nu$ such that the diagram in \eqref{e:froedegen} commutes in $\Mod(S)$ (note that the uniqueness of $\psi^\nu$ follows from the fact that both $\lambda^n_f(\nu)$ and $\coker(\iota^n_f(\nu,\nu+1))$, and whence their composition, are epimorphisms in $\Mod(S)$); moreover, $\psi^\nu$ is an isomorphism. Further on, for all integers $\nu$,
 $$\kappa^n_f(\sigma^{=\nu}\Omega^\kdot_f) \colon \R^n\bar{f}_*(\sigma^{=\nu}\bar\Omega^\kdot_f) \to \R^nf_*(\sigma^{=\nu}\Omega^\kdot_f) = \R^{n-p}f_*(\Omega^p_f)$$
 is an isomorphism. Thus for all $\nu\in\Z$, there exists an isomorphism
 $$\F^\nu\sH^n(f)/\F^{\nu+1}\sH^n(f) \to \sH^{\nu,n-\nu}(f)$$
 in $\Mod(S)$. Now since, for all $\nu\in\Z_{\geq n+1}$, $\F^\nu\sH^n(f)$ is a zero module on $S$ and in particular locally finite free, we conclude by descending induction on $\nu$ starting at $\nu=n+1$ that, for all $\nu\in\Z$, $\F^\nu\sH^n(f)$ is a locally finite free module on $S$; along the way we make use of \eqref{t:pmclassic-flf}. Specifically, since $\F^0\sH^n(f) = \sH^n(f)$, we see that $\sH^n(f)$ is a locally finite free module, \iev in the terminology of Definition \ref{d:bundles}, a vector bundle, on $S$. Furthermore, for all integers $\mu$ and $\nu$ such that $\mu \leq \nu$, there exists a short exact sequence
 $$0 \to \F^\mu/\F^\nu \to \F^{\mu-1}/\F^\nu \to \F^{\mu-1}/\F^\mu \to 0,$$
 where we write $\F^*$ as a shorthand for $\F^*\sH^n(f)$. Therefore, we see, using descending induction on $\mu$, that for all integers $\nu$ and all integers $\mu$ such that $\mu \leq \nu$ the quotient $\F^\mu/\F^\nu$ is a locally finite free module on $S$. Specifically, we see that for all integers $\nu$, the quotient $\sH^n(f)/\F^\nu\sH^n(f)$ is a locally finite free module on $S$. Thus we conclude that, for all integers $\nu$, $\F^\nu\sH^n(f)$ is a vector subbundle of $\sH^n(f)$ on $S$.
 
 For the time being, fix an arbitrary element $s$ of $S$. Then by assumption \eqref{t:pmclassic-froefiber} and Proposition \ref{p:froedegen} we deduce that, for all integers $\nu$, there exists a (unique) isomorphism $\psi^\nu_{X_s}$ such that the diagram in \eqref{e:froedegenfiber}, where we replace $t$ by $s$, commutes in $\Mod(\C)$. As base change commutes with taking stupid filtrations, the following diagram has exact rows and commutes in $\Mod(\C)$ for all integers $\nu$:
 \[
  \xymatrix@C=3pc{
   0 \ar[r] & \F^\nu(s) \ar[r]^-{(\iota^n_f(\nu-1,\nu))(s)} \ar[d]_{\phi^{\nu,n}_{f,s}} & \F^{\nu-1}(s) \ar[r] \ar[d]_{\phi^{\nu-1,n}_{f,s}} & \sH^{\nu,n-\nu}(s) \ar[r] \ar[d]^{\beta^{\nu,n-p}_{f,s}} & 0 \\
   0 \ar[r] & \F^\nu\sH^n(X_s) \ar[r]_-{\iota^n_{X_s}(\nu-1,\nu)} & \F^{\nu-1}\sH^n(X_s) \ar[r] & \sH^{\nu,n-\nu}(X_s) \ar[r] & 0
  }
 \]
 Therefore, using a descending induction on $\nu$ starting at $\nu=n+1$ together with assumption \eqref{t:pmclassic-bc} and the ``short five lemma'', we infer that, for all $\nu\in\Z$, the base change map $\phi^{\nu,n}_{f,s}$ is an isomorphism in $\Mod(\C)$. Specifically, since $\phi^{0,n}_{f,s} = \phi^n_{f,s}$, we see that the de Rham base change map $\phi^n_{f,s}$ is an isomorphism in $\Mod(\C)$.
 
 Abandon the fixation of $s$ and define a $\Z$-sequence $\tilde\psi$ by putting, for any $\nu\in\Z$:
 $$\tilde\psi^\nu := \kappa^n_{X_t}(\sigma^{=\nu}\Omega^\kdot_{X_t}) \circ \psi^\nu_{X_t}.$$
 Let $p$ be an integer. Then defining $\alpha$ and $\beta$ according to \eqref{e:pmclassic-ab}, the commutativity of the diagram in \eqref{e:pm} is implied by Theorem \ref{t:pm}.
\end{proof}

\chapter{Degeneration of the Frölicher spectral sequence}
\label{ch:froe}

\setcounter{subsection}{0}

Later, in Chpater \ref{ch:symp}, we would like to make use of Theorem \ref{t:pmclassic} in order to establish a local Torelli theorem for certain compact, symplectic complex spaces of Kähler type, \cf Theorem \ref{t:lt}. The basic idea in the application of Theorem \ref{t:pmclassic} thereby is the following. Consider a family of compact complex spaces over a smooth base, \iev a proper, flat morphism of complex spaces $f\colon X\to S$ such that $S$ is a complex manifold. Wanting to talk about period mappings, we assume additionally that the space $S$ be simply connected, although this assumption is not essential for the problems of Chapter \ref{ch:froe}. Now define $g\colon Y\to S$ to be the ``submersive share'' of $f$, by which we mean that $Y$ is the open complex subspace of $X$ induced on set of points of $X$ in which the morphism $f$ is submersive and $g$ is the composition of the inclusion $Y \to X$ and $f$. Then $g$ is certainly a submersive morphism of complex spaces with smooth and simply connected base, so that we might think of applying Theorem \ref{t:pmclassic} to $g$ (in place of $f$, as in the formulation of the theorem) and an integer $n$. This leads us to the task of determining circumstances, in terms of $f$ and $n$, under which the following assertions hold---observe that these correspond to conditions (i)--(iv) of Theorem \ref{t:pmclassic}:
\begin{enumerate}
 \item \label{i:froe-degen} the Frölicher spectral sequence of $g$ degenerates in entries
 \[
  I := \{(p,q)\in\Z\times\Z:p+q=n\}
 \]
 at sheet $1$ in $\Mod(S)$;
 \item \label{i:froe-lff} for all $(p,q)\in I$, the Hodge module $\sh H^{p,q}(g)$ is a locally finite free module on $S$;
 \item \label{i:froe-degenfiber} for all $s\in S$, the Frölicher spectral sequence of $Y_s$ degenerates in entries $I$ at sheet $1$ in $\Mod(\C)$;
 \item \label{i:froe-bc} for all $s\in S$ and all $(p,q)\in I$, the Hodge base change map
 \[
  \beta^{p,q}_{g,s} \colon (\sh H^{p,q}(g))(s) \to \sh H^{p,q}(Y_s)
 \]
 is an isomorphism in $\Mod(\C)$.
\end{enumerate}

In view of assertion \ref{i:froe-degenfiber}), we remind the reader of a result of T.~Ohsawa (\cf \cite[Theorem 1]{Oh87}):

\begin{theorem}
 \label{t:ohsawaintro}
 Let $X$ be a compact, pure dimensional complex space of Kähler type and $A$ a closed analytic subset of $X$ such that $\Sing(X)\subset A$. Then the Frölicher spectral sequence of $X\setminus A$ degenerates in entries
 \[
  \{(p,q)\in\Z\times\Z : p+q+2\leq\codim(A,X)\}
 \]
 at sheet $1$ in $\Mod(\C)$.
\end{theorem}

So, suppose that the fibers of $f$ are altogether pure dimensional (\egv normal and connected) and of Kähler type. Then Theorem \ref{t:ohsawaintro} guarantees the validity of assertion \ref{i:froe-degenfiber}) when we have
\begin{equation} \label{e:froeintro-codim} \tag{$*$}
 n+2 \leq \codim(\Sing(X_s),X_s) \quad \text{for all } s\in S;
\end{equation}
note that due to the flatness of $f$, we know that, for all $s\in S$, the complex spaces $Y_s$ and $(X_s)_\reg = X_s \setminus \Sing(X_s)$ are isomorphic. Inspired by this observation, we pose the following

\begin{question}
 \label{q:froeintro}
 Let $f$ and $n$ be as above; in particular, we assume that \eqref{e:froeintro-codim} holds. Define $g$ to be the submersive share of $f$. Which of the assertions \ref{i:froe-lff}), \ref{i:froe-bc}), and \ref{i:froe-degen}) are then fulfilled?
\end{question}

We put forward another question, which is wider in scope.

\begin{question}
 \label{q:froeintro2}
 Let $f\colon X\to S$ be a proper (and flat) morphism of complex spaces (such that $S$ is a complex manifold), $A$ a closed analytic subset of $X$ such that the restriction $g \colon Y := X\setminus A \to S$ of $f$ is submersive, $n$ an integer such that
 \[
  n+2 \leq \codim(A\cap X_s,X_s) \quad \text{for all } s\in S.
 \]
 Assume that $f$ is
 \begin{enumeratei}
  \item locally equidimensional, \iev the function $x \mto \dim_x(X_{f(x)})$ is locally constant on $X$, and
  \item weakly Kähler (\cf \cite[(5.1)]{Bi83}).
 \end{enumeratei}
 Do assertions \ref{i:froe-lff}), \ref{i:froe-bc}), and \ref{i:froe-degen}) hold then?
\end{question}

Our goal in this chapter is to give several positive answers in the direction of Question \ref{q:froeintro2} and Question \ref{q:froeintro}---unfortunately we do not manage to answer either of the proposed questions in its entirety.

In \S\ref{s:coh}, we investigate the coherence of the Hodge modules $\sH^{p,q}(g)$ by means of standard techniques of local cohomology as a first step towards the local finite freeness stated in \ref{i:froe-lff}). In \S\ref{s:in}, we study the degeneration behavior of the Frölicher spectral sequence when passing from one infinitesimal neighborhood of a fiber of $g\colon Y\to S$ to the next. In \S\ref{s:formal}, we invoke a comparison theorem between formal and ordinary higher direct image sheaves due to C.\ B\u{a}nic\u{a} and O.\ St\u{a}n\u{a}\c{s}il\u{a} (\cf \cite{BaSt76}) in order to establish \ref{i:froe-lff}) and \ref{i:froe-bc}). Finally, in \S\ref{s:ohsawa}, we draw conclusions for the degeneration of the the Frölicher spectral sequence of $g$.

\section{Coherence of direct image sheaves}
\label{s:coh}

\begin{notation}[Depth]
 \label{d:prof}
 Let $A$ be a commutative ring, $I$ an ideal of $A$, and $M$ an $A$-module. Then we define:
 \begin{equation}
  \label{e:prof}
  \prof_A(I,M):=\sup\{n\in\N:(\forall N\in T)(\forall i\in n)\Ext^i_A(N,M)=0\},
 \end{equation}
 where $T$ denotes the class of all finite type $A$-modules for which there exists a natural number $m$ such that $I^mN=0$. Note that the set in \eqref{e:prof} over which the supremum is taken certainly contains $0$, whence is nonempty. $\prof_A(I,M)$ is called the $I$-\emph{depth} of $M$ over $A$.
 
 When $A$ is a commutative local ring, we define:
 \[
  \prof_A(M) := \prof_A(\fm(A),M).
 \]
 
  Let $X$ be an analytic space (or else a commutative locally ringed space). Let $x$ be an element of the set underlying $X$ and $F$ a module on $X$. Then we set:
  \[
   \prof_{X,x}(F) := \prof_{\O_{X,x}}(F_x).
  \]
\end{notation}

\begin{proposition}
 \label{p:profsum}
 Let $A$ be a commutative ring and $I$ an ideal of $A$.
 \begin{enumerate}
  \item For all $A$-modules $M$ and $M'$, we have:
  \[
   \prof_A(I,M\oplus M')=\min(\prof_A(I,M),\prof_A(I,M')).
  \]
  \item For all $A$-modules $M$ and all $r\in\N$, we have $\prof_A(I,M^{\oplus r})=\prof_A(I,M)$.
 \end{enumerate}
\end{proposition}

\begin{proof}
 Follows from the fact that, for all $A$-modules $N$ and all $i\in\N$ (\resp $i\in\Z$), $\Ext^i_A(N,-)$ is an additive functor from $\Mod(A)$ to $\Mod(A)$ hence commutes with the formation of finite sums.
\end{proof}

\begin{notation}
 \label{not:singsh}
 Let $X$ be a complex space (or else a commutative locally ringed space), $F$ a module on $X$, and $m$ an integer. Then we define:
 \[
  \rS_m(X,F) := \{x\in X : \prof_{X,x}(F)\leq m\}.
 \]
 $\rS_m(X,F)$ is called the $m$-th \emph{singular set} of $F$ on $X$. 
\end{notation}

\begin{notation}[Sheaves of local cohomology]
 \label{not:lc}
 Let $X$ be a topological space (respectively a ringed space or complex space), $A$ a closed subset of $X$. We denote
 \[
  \sGamma_A(X,-) \colon \Ab(X) \to \Ab(X) \quad (\text{\resp } \Mod(X) \to \Mod(X))
 \]
 the \emph{sheaf of sections on $X$ with supports in $A$} functor. That is, for any abelian sheaf (\resp sheaf of modules) $F$ on $X$, we define $\sGamma_A(X,F)$ to be the abelian subsheaf (\resp subsheaf of modules) of $F$ on $X$ such that, for all open subsets $U$ of $X$, the set $(\sGamma_A(X,F))(U)$ comprises precisely those elements of $F(U)$ which are sent to the zero of $F(U\setminus A)$ by the restriction mapping $F(U) \to F(U\setminus A)$. Note that the functor $\sGamma_A(X,-)$ is additive as well as left-exact. For any integer $n$ we write
 \[
  \uH^n_A(X,-) \colon \Ab(X) \to \Ab(X) \quad (\text{\resp } \Mod(X) \to \Mod(X))
 \]
 for the $n$-th right derived functor of $\sGamma_A(X,-)$, \cf \S\ref{s:der}.
\end{notation}

\begin{proposition}
 \label{p:lcgap}
 Let $X$ be a topological space (\resp a ringed space or complex space), $A$ a closed subset of $X$. Write $i\colon X\setminus A \to X$ for the inclusion morphism. Then, for all abelian sheaves (\resp sheaves of modules) $F$ on $X$, there exists an exact sequence in $\Ab(X)$ (\resp $Mod(X)$),
 \[
  0 \to \uH^0_A(X,F) \to F \to \R^0i_*(i^*(F)) \to \uH^1_A(X,F) \to 0,
 \]
 and, for all integers $q\geq 1$, we have
 \[
  \R^qi_*(i^*(F)) \iso \uH^{q+1}_A(X,F).
 \]
\end{proposition}

\begin{proof}
 The topological space case is \cite[Chapter II, Corollary 1.10]{BaSt76}; the ringed space case is proven along the very same lines.
\end{proof}

\begin{lemma}
 \label{l:3coh}
 Let $X$ be a commutative ringed space.
 \begin{enumerate}
  \item For all morphisms $\phi\colon F\to G$ in $\Mod(X)$ such that $F$ and $G$ are coherent on $X$, both $\ker(\phi)$ and $\coker(\phi)$ are coherent on $X$.
  \item For all short exact sequences
  \[
   0\to F\to G\to H\to 0
  \]
  in $\Mod(X)$, when $F$ and $H$ are coherent on $X$, then $G$ is coherent on $X$.
 \end{enumerate}
\end{lemma}

\begin{proof}
 See \cite[I, \S2, Théorème 1 and Théorème 2]{FAC}.
\end{proof}

\begin{corollary}
 \label{c:lcgap}
 Let $X$ be a ringed space, $A$ a closed subset of $X$, and $F$ a coherent module on $X$. Then, for all $n\in\Z$, the following are equivalent:
 \begin{enumeratei}
  \item For all $q\in\Z$ such that $q\leq n$, the module $\R^qi_*(i^*(F))$ is coherent on $X$.
 \item For all $q\in\Z$ such that $q\leq n+1$, the module $\uH^q_A(X,F)$ is coherent on $X$.
 \end{enumeratei}
\end{corollary}

\begin{proof}
 This is clear from Proposition \ref{p:lcgap} and Lemma \ref{l:3coh}.
\end{proof}

\begin{theorem}
 \label{t:finiteness}
 Let $X$ be a complex space, $A$ a closed analytic subset of $X$, and $F$ a coherent module on $X$. Denote $i \colon X\setminus A \to X$ the inclusion morphism. Then, for all natural numbers $n$, the following are equivalent:
 \begin{enumeratei}
  \item \label{t:finiteness-dim} For all integers $k$, we have
  \begin{equation} \label{e:finiteness-dim}
   \dim(A \cap \bar{\rS_{k+n+1}(X\setminus A,i^*(F))})\leq k,
  \end{equation}
  where the bar refers to taking the closure in $X_\top$. Note that imposing \eqref{e:finiteness-dim} hold for all integers $k<0$ is equivalent to requiring that
  \[
   A \cap \bar{\rS_n(X\setminus A,i^*(F))} = \emptyset.
  \]
  \item \label{t:finiteness-lccoh} For all $q\in\N$ (or $q\in\Z$) such that $q\leq n$, the module $\uH^q_A(X,F)$ is coherent on $X$.
 \end{enumeratei}
\end{theorem}

\begin{proof}
 This is \cite[Chapter II, Theorem 4.1]{BaSt76}.
\end{proof}

\begin{proposition}
 \label{p:gapcoh}
 Let $X$ be a locally pure dimensional complex space, $A$ a closed analytic subset of $X$, and $F$ a coherent module on $X$. Assume that,
 \begin{equation} \label{e:gapcoh-prof}
  \text{for all } x\in X\setminus A, \quad \prof_{X,x}(F)=\dim_x(X).
 \end{equation}
 Denote by $i$ the inclusion morphism of complex spaces from $X\setminus A$ to $X$. Then, for all $q\in\Z$ such that $q+2\leq\codim(A,X)$, the module $\R^qi_*(i^*(F))$ is coherent on $X$.
\end{proposition}

\begin{proof}
 Assume that $X$ is pure dimensional. When $A=\emptyset$, the morphism $i$ is the identity on $X$, hence we have $\R^0i_*(i^*(F))\iso F$ and $\R^q i_*(i^*(F))\iso 0$ for all integers $q\neq0$ in $\Mod(X)$. Thus, for all integers $q$, the module $\R^qi_*(i^*(F))$ is coherent on $X$.
 
 Now assume that $A\neq\emptyset$ and put $n:=\codim(A,X)-1$. Write $S_m$ as a shorthand for $\rS_m(X\setminus A,i^*(F))$. For all $x\in X\setminus A$, we have:
 \[
  \prof_{X\setminus A,x}(i^*(F))=\prof_{X,x}(F)=\dim_x(X)=\dim(X).
 \]
 Therefore, $S_m=\emptyset$ for all integers $m$ such that $m<\dim(X)$. Let $k$ be an arbitrary integer. When $\dim(A) \leq k$, then
 \[
  \dim(A\cap \bar{S_{k+n+1}})\leq \dim(A)\leq k.
 \]
 For all $x\in A$, we have
 \[
  \dim_x(A) = \dim_x(X) - \codim_x(A,X) \leq \dim(X) - \codim(A,X),
 \]
 which implies that
 \[
  \dim(A) \leq \dim(X) - \codim(A,X).
 \]
 In turn, when $k < \dim(A)$, we have
 \[
  k+n+1 = k+\codim(A,X) < \dim(X),
 \]
 whence $S_{k+n+1} = \emptyset$. Thus, $A \cap \bar{S_{k+n+1}} = \emptyset$ and so again,
 \[
  \dim(A \cap \bar{S_{k+n+1}}) \leq k.
 \]
 We see that assertion \eqref{t:finiteness-dim} of Theorem \ref{t:finiteness} holds. Hence by Theorem \ref{t:finiteness}, assertion \eqref{t:finiteness-lccoh} holds too, so that, for all integers $q$ with $q\leq n$, the module $\uH^q_A(X,F)$ is a coherent module on $X$. Corollary \ref{c:lcgap} implies that, for all integers $q$ with $q+2 \leq \codim(A,X)$, \iev $q \leq n-1$, the module $\R^qi_*(i^*(F))$ is coherent on $X$.
 
 Abandon the assumption that $X$ is pure dimensional. Let $q$ be an integer such that $q+2\leq\codim(A,X)$. Let $x\in X$ be any point. Then, since $X$ is locally pure dimensional, there exists an open neighborhood $U$ of $x$ in $X$ such that the open complex subspace of $X$ induced on $U$ is pure dimensional. Put $Y:=X|U$, $B:=A\cap U$, and $G:=F|U$. By what we have already proven, and the fact that
 \[
  q+2 \leq \codim(A,X) \leq \codim(B,Y),
 \]
 we infer that $\R^qj_*(j^*(G))$ is coherent on $Y$, where $j$ stands for the canonical morphism from $Y\setminus B$ to $Y$. Since $\R^qi_*(i^*(F))|U$ is isomorphic to $\R^qj_*(j^*(G))$ in $\Mod(Y)$, we see that $\R^qi_*(i^*(F))$ is coherent on $X$ in $x$. As $x$ was an arbitrary point of $X$, the module $\R^qi_*(i^*(F))$ is coherent on $X$.
\end{proof}

\begin{corollary}
 \label{c:gapcoh}
 Let $X$ be a locally pure dimensional complex space, $A$ a closed analytic subset of $X$, and $F$ a coherent module on $X$. Assume that $X$ is Cohen-Macaulay in $X\setminus A$ and $F$ is locally finite free on $X$ in $X\setminus A$. Then $\R^qi_*(i^*(F))$ is a coherent module on $X$ for all integers $q$ satisfying $q+2\leq\codim(A,X)$, where $i$ denotes the canonical immersion from $X\setminus A$ to $X$.
\end{corollary}

\begin{proof}
  Let $x\in X\setminus A$. As $F$ is locally finite free on $X$ in $x$, there exists a natural number $r$ such that $F_x$ is isomorphic to $(\O_{X,x})^{\oplus r}$ in the category of $\O_{X,x}$-modules, hence using Proposition \ref{p:profsum}, we obtain:
 \begin{align*}
  \prof_{X,x}(F) & =\prof_{\O_{X,x}}(F_x)=\prof_{\O_{X,x}}((\O_{X,x})^{\oplus r})=\prof_{\O_{X,x}}(\O_{X,x}) \\ & =\dim(\O_{X,x})=\dim_x(X).
 \end{align*}
 As $x$ was an arbitrary element of $X\setminus A$, we see that \eqref{e:gapcoh-prof} holds. Thus our claim follows readily from Proposition \ref{p:gapcoh}.
\end{proof}

\begin{theorem}[Grauert's direct image theorem]
 \label{t:propercoh}
 Let $f\colon X\to S$ be a proper morphism of complex spaces. Then, for all coherent modules $F$ on $X$ and all integers $q$, the module $\R^qf_*(F)$ is coherent on $S$.
\end{theorem}

\begin{proof}
 This is \cite[``Hauptsatz I'', p.~59]{Gr60}. See also \cite[Chapter III, Theorem 2.1]{BaSt76} as well as \cite[Chapter III, \S4]{GrPeRe94} and references there.
\end{proof}

\begin{proposition}
 \label{p:cohcomp}
 Let $f\colon X\to S$ and $g\colon Y\to X$ be morphisms of complex spaces. Let $G$ be a module on $Y$ and $n$ an integer. Put $h:=f\circ g$. Suppose that $f$ is proper and, for all integers $q \leq n$, the module $\R^qg_*(G)$ is coherent on $X$. Then, for all integers $k \leq n$, the module $\R^kh_*(G)$ is coherent on $S$.
\end{proposition}

\begin{proof}
 We employ the following fact: There exists a spectral sequence $E$ with values in $\Mod(S)$ such that:
 \begin{enumeratei}
  \item for all $p,q\in\Z$, $E_2^{p,q}\iso \R^pf_*(\R^qg_*(G))$;
  \item for all $k\in\Z$, there exists a filtration $F$ on $\R^kh_*(G)$ such that $F^0=\R^kh_*(G)$, $F^{k+1}\iso0$, and for all $p\in\Z$ there exists $r\in\N_{\geq2}$ such that $F^p/F^{p+1}\iso E_r^{p,k-p}$.
 \end{enumeratei}
 Given $E$, we claim: For all $r\in\N_{\geq 2}$ and all $p,q\in\Z$, $E_r^{p,q}$ is coherent on $S$ when any of the following conditions is satisfied:
 \begin{enumerate}
  \item $p+q\leq n$;
  \item there exists $\Delta\in\N_{\geq 1}$ such that $p+q=n+\Delta$, and
  \begin{equation} \label{e:cohcomp-q}
   q\leq n - \left(\sum_{\nu=2}^{\min(\Delta,r-1)} (r-\nu)\right),
  \end{equation}
  where the sum appearing on the right hand side in \eqref{e:cohcomp-q} is defined to equal $0$ in case $\min(\Delta,r-1)<2$.
 \end{enumerate}
 The claim is proven by means of induction. By Theorem \ref{t:propercoh} and property (i) we know that, for all $p,q\in\Z$ such that $q\leq n$ or $p<0$, $E_2^{p,q}$ is coherent on $S$; note that in case $p<0$, (i) implies that $E_2^{p,q}\iso 0$. Let $p,q\in\Z$ be arbitrary. When $p+q\leq n$, then either $p<0$ or $q\leq n$. On the contrary, when there exists $\Delta\in\N_{\geq 1}$ such that \eqref{e:cohcomp-q} holds for $r=2$, then $q\leq n$ as the value of the sum in \eqref{e:cohcomp-q} is always $\geq 0$. Hence our claim holds in case $r=2$.
 
 Now let $r\in\N_{\geq 2}$ and assume that, for all $p,q\in\Z$, when either a) or b) holds, then $E_r^{p,q}$ is coherent on $S$. Since $E$ is a spectral sequence, we know that for all $p,q\in\Z$,
 \[
  E_{r+1}^{p,q} \iso \H(\xymatrix@C3.5pc{E_r^{p-r,q+r-1} \ar[r]^-{d_r^{p-r,q+r-1}} & E_r^{p,q} \ar[r]^-{d_r^{p,q}} & E_r^{p+r,q-r+1}}).
 \]
 In particular, Lemma \ref{l:3coh} implies that, for all $p,q\in\Z$, when $E_r^{p-r,q+r-1}$, $E_r^{p,q}$, and $E_r^{p+r,q-r+1}$ are coherent on $S$, then $E_{r+1}^{p,q}$ is coherent on $S$. Let $p,q\in\Z$ be arbitrary. When $p+q\leq n-1$, then
 \[
  (p-r)+(p+r-1)\leq p+q\leq (p+r)+(q-r+1)=p+q+1\leq n
 \]
 so that $E_r^{p-r,q+r-1}$, $E_r^{p,q}$, and $E_r^{p+r,q-r+1}$ are coherent on $S$ and consequently $E_{r+1}^{p,q}$ is coherent on $S$ by means of the preceding argument. Assume that $p+q=n$. In case $p<0$, $E_r^{p,q}\iso0$, whence $E_{r+1}^{p,q}\iso0$; in particular, $E_{r+1}^{p,q}$ is coherent on $S$. When $p\geq 0$, in addition to $E_r^{p-r,q+r-1}$ and $E_r^{p,q}$, $E_r^{p+r,q-r+1}$ is coherent on $S$ since $(p+r)+(q-r+1)=p+q+1=n+1$ and $q-r+1\leq q\leq n$ which means that \eqref{e:cohcomp-q} holds for $\Delta=1$. So, $E_{r+1}^{p,q}$ is coherent on $S$. Assume that there exists $\Delta\in\N$ such that $p+q=n+\Delta$, and $\Delta\geq 1$, and:
 \[
  q \leq n - \left(\sum_{\nu=2}^{\min(\Delta,(r+1)-1)} ((r+1)-\nu)\right).
 \]
 When $\Delta=1$,
 \[
  (p-r) + (q+r-1) = p+q-1 = n+\Delta-1 = n,
 \]
 whence $E_r^{p-r,q+r-1}$ is coherent on $S$. When $\Delta\geq2$, we have $(p-r)+(q+r-1)=n+\Delta'$ with $\Delta':=\Delta-1$. Moreover, $\Delta'\geq 1$ and
 \[
  q+r-1\leq n-\left(\sum_{\nu=1}^{\min(\Delta,r)-1}(r-\nu)\right)+(r-1)=n-\left(\sum_{\nu=2}^{\min(\Delta',r-1)}(r-\nu)\right),
 \]
 whence again, $E_r^{p-r,q+r-1}$ is coherent on $S$. Similarly, as $p+q=n+\Delta$, and $\Delta\geq 1$, and
 \[
  q \leq n - \left(\sum_{\nu=2}^{\min(\Delta,r)}(r+1-\nu)\right) \leq n - \left(\sum_{\nu=2}^{\min(\Delta,r-1)}(r-\nu)\right),
 \]
 $E_r^{p,q}$ is coherent on $S$. Finally, as firstly $(p+r)+(q-r+1)=p+q+1=n+\Delta''$, where $\Delta'':=\Delta+1$, secondly $\Delta''\geq 1$, and thirdly, setting $m:=\min(\Delta,r)$,
 \begin{equation*}
  \begin{split}
  q-r+1 & \leq n - \left(\sum_{\nu=1}^{m-1}(r-\nu)\right)-(r-1)\leq n-\sum_{\nu=2}^{m+1}(r-\nu) \\ & \leq n - \left(\sum_{\nu=2}^{\min(\Delta'',r-1)}(r-\nu)\right),
  \end{split}
 \end{equation*}
 $E_r^{p+r,q-r+1}$ is coherent on $S$. Therefore, $E_{r+1}^{p,q}$ is coherent on $S$. This finishes the proof of the claim.
 
 Now let $k\in\Z$ such that $k\leq n$. Let $F$ be a filtration of $\R^kh_*(G)$ as in (ii). We claim that, for all $p\in\Z_{\leq k+1}$, $F^p$ is coherent on $S$. Indeed, $F^{k+1}$ is isomorphic to zero in $\Mod(S)$, whence coherent on $S$. Let $p\in\Z_{\geq k}$ such that $F^{p+1}$ is coherent on $S$. Then there exists $r\in\N_{\geq2}$ such that $F^p/F^{p+1}\iso E_r^{p,k-p}$. Thus there exists an exact sequence
 \[
  0 \to F^{p+1} \to F^p \to E_r^{p,k-p} \to 0
 \]
 in $\Mod(S)$. As $p+(k-p)=k\leq n$, $E_r^{p,k-p}$ is coherent on $S$ by our preliminary claim. In turn by Lemma \ref{l:3coh}, $F^p$ is coherent on $S$.
\end{proof}

\begin{notation}
 \label{not:singf}
 Let $f\colon X\to S$ be a morphism of complex spaces.
 \begin{enumerate}
  \item \label{not:singf-singf} We set
  \[
   \Sing(f):=\{x\in X:\text{$f$ is not submersive in $x$}\}
  \]
  and call $\Sing(f)$ the \emph{singular locus} of $f$. Note that one has to regard $f$ as an absolute morphism in $\An$ in order to obtain a clean notation, \cf \S\ref{s:cat}.
  \item \label{not:singf-ss} Evidently, $\Sing(f)$ is a closed subset of $X$. Hence $X\setminus \Sing(f)$ is an open complex subspace of $X$. The composition of the canonical morphism $X\setminus \Sing(f) \to X$ and $f$ will be referred to as the \emph{submersive share} of $f$.
 \end{enumerate}
\end{notation}

\begin{proposition}
 \label{p:cohhdi}
 Let $f\colon X\to S$ be a proper morphism of complex spaces and $A$ a closed analytic subset of $X$ such that $\Sing(f)\subset A$. Suppose that $S$ is Cohen-Macaulay and $X$ is locally pure dimensional. Denote $g$ the restriction of $f$ to $X\setminus A$. Then, for all integers $p$ and $q$ such that $q+2\leq\codim(A,X)$, the Hodge module $\sH^{p,q}(g) := \R^qg_*(\Omega^p_g)$ is a coherent module on $S$.
\end{proposition}

\begin{proof}
 In case $A=\emptyset$, $g=f$, hence $\sH^{p,q}(g)$ is coherent on $S$ for all integers $p$ and $q$ by Theorem \ref{t:propercoh} as $f$ is proper and $\Omega^p_f$ is coherent on $X$. So, assume that $A\neq\emptyset$. Let $p\in\Z$ be arbitrary and put $n:=\codim(A,X)-2$, which makes sense now given $\codim(A,X)\in\N$. Write $i$ for the canonical immersion from $X\setminus A$ to $X$. Since, for all $x\in X\setminus A$, $f$ is submersive in $x$ and $S$ is nonsingular (\resp Cohen-Macaulay) in $f(x)$, we see that $X$ is nonsingular (\resp Cohen-Macaulay) in $X\setminus A$. Moreover, $\Omega^p_f$ is coherent on $X$ and locally finite free on $X$ in $X\setminus A$. Thus by Corollary \ref{c:gapcoh}, $\R^qi_*(i^*(\Omega^p_f))$ is coherent on $X$ for all integers $q$ such that $q\leq n$. As $i^*(\Omega^p_f)$ is isomorphic to $\Omega^p_g$ in $\Mod(Y)$, it follows that $\R^qi_*(\Omega^p_g)$ is coherent on $X$ for all integers $q$ such that $q\leq n$. Therefore, our claim is implied by Proposition \ref{p:cohcomp}.
\end{proof}

\section{Infinitesimal neighborhoods}
\label{s:in}

In this section we will consider a morphism of complex spaces $f\colon X\to S$ together with a distinguished ``basepoint'' $t\in S$. Most of the time, we shall assume $f$ to be submersive in the points of $f^{-1}(\{t\})$. We are interested in the degeneration behavior of the Frölicher spectral sequence for the infinitesimal neighborhoods of the morphism $f$ with respect to the closed analytic subsets $f^{-1}(\{t\})$ and $\{t\}$ of $X$ and $S$, respectively. Our main result, namely Theorem \ref{t:froeinf}, of which we have found no account in the literature, asserts that if the Frölicher spectral sequence of the zeroth infinitesimal neighborhood of $f$ degenerates in entries of a certain total degree $n\in\Z$, then the Frölicher spectral sequence of any infinitesimal neighborhood of $f$ degenerates in entries of total degree $n$. The proof of Theorem \ref{t:froeinf} proceeds by induction on the order of the infinitesimal neighborhood, which is why we call this technique the ``infinitesimal lifting of degeneration''.

\begin{setup}
 \label{set:inbc}
 Let $f\colon X\to S$ be a morphism of complex spaces and $t\in S$. We set $S':=\{t\}$ and $X':=f^{-1}(S')$ (set-theoretically, for now) and write $\sI$ and $\sJ$ for the ideals of $X'$ and $S'$ on $X$ and $S$, respectively. For any natural number $m$, we define $X_m$ (\resp $S_m$) to be the $m$-th infinitesimal neighborhood of $X'$ (\resp $S'$) in $X$ (\resp $S$), so that $X_m$ (\resp $S_m$) is the closed complex subspace of $X$ (\resp $S$) defined by the ideal $\sI^{m+1}$ (\resp $\sJ^{m+1}$), \cf \cite[p.~32]{GrPeRe94}. We write $i_m\colon X_m\to X$ (\resp $b_m\colon S_m\to S$) for the so induced canonical morphism of complex spaces. Moreover, we let $f_m\colon X_m\to S_m$ signify the unique morphism of complex spaces satisfying $f \circ i_m = b_m \circ f_m$, that is, making the diagram
 \begin{equation} \label{e:inbc-m}
  \xysquare{X_m}{X}{S_m}{S}{i_m}{f_m}{f}{b_m}
 \end{equation}
 commute in the category of complex spaces.
 
 When $l$ is a natural number such that $l\leq m$, then we denote $i_{l,m} \colon X_l \to X_m$ the unique morphism of complex spaces satisfying $i_m \circ i_{l,m} = i_l$. Similarly, we denote $b_{l,m} \colon S_l \to S_m$ the unique morphism of complex spaces which satisfies $b_m \circ b_{l,m} = b_l$. Given this notation, it follows that the diagram
 \begin{equation} \label{e:inbc-lm}
  \xysquare{X_l}{X_m}{S_l}{S_m}{i_{l,m}}{f_l}{f_m}{b_{l,m}}
 \end{equation}
 commutes in the category of complex spaces.
 
 When $l$ and $m$ are as above and $n$, $p$, and $q$ are integers, we denote by
 \begin{align*}
  \phi^n_m & \colon \sH^n(f_m) \to \sH^n(f), & \phi^n_{l,m} & \colon \sH^n(f_m) \to \sH^n(f_l), \\
  \phi^{p,n}_m & \colon \F^p\sH^n(f_m) \to \F^p\sH^n(f), & \phi^{p,n}_{l,m} & \colon \F^p\sH^n(f_m) \to \F^p\sH^n(f_l), \\
  \beta^{p,q}_m & \colon \sH^{p,q}(f_m) \to \sH^{p,q}(f), & \beta^{p,q}_{l,m} & \colon \sH^{p,q}(f_m) \to \sH^{p,q}(f_l)
 \end{align*}
 the de Rham base change maps in degree $n$, the filtered de Rham base change maps in bidegree $(p,n)$, and the Hodge base change maps in bidegree $(p,q)$ associated respectively to the commutative squares in \eqref{e:inbc-m} and \eqref{e:inbc-lm}, \cf \S\ref{s:bc}.
\end{setup}

In order to prove Theorem \ref{t:froeinf}, we observe in the first place that the algebraic de Rham modules $\sH^n(f_m)$ ($n\in\Z$, $m\in\N$) of our infinitesimal neighborhoods $f_m\colon X_m \to S_m$ (here we speak in terms of Setup \ref{set:inbc}) are altogether free and ``compatible with base change''. Note that even though one actually gains information about the algebraic de Rham modules $\sH^n(f_m)$ as a corollary of Theorem \ref{t:froeinf}, it is crucial to establish their mentioned properties a priori. The key is the following sort of ``universal coefficient theorem''/``topological base change theorem''.

\begin{lemma}
 \label{l:topbc}
 Let $f\colon X\to S$ be a morphism of topological spaces and $\theta \colon B \to A$ a morphism of commutative sheaves of rings on $S$. Let
 \begin{align*}
  f^A \colon (X,f^*A) \to (S,A) \quad (\text{\resp } f^B \colon (X,f^*B) \to (S,B))
 \end{align*}
 be the morphism of ringed spaces given by $f$ and the adjunction morphism $A \to f_*f^*A$ (\resp $B \to f_*f^*B$). Moreover, let
 \[
  u \colon (X,f^*A) \to (X,f^*B) \quad (\text{\resp } w \colon (S,A) \to (S,B))
 \]
 be the morphism of ringed spaces given by $\id_{|X|}$ and $f^*(\theta)$ (\resp $\id_{|S|}$ and $\theta$). Then the following diagram commutes in the category of ringed spaces:
 \begin{equation} \label{e:topbc-sq}
  \xysquare{(X,f^*A)}{(X,f^*B)}{(S,A)}{(S,B)}{u}{f^A}{f^B}{w}
 \end{equation}
 Furthermore, when $\theta$ makes $A$ into a locally finite free $B$-module on $S$, then, for all integers $n$, the morphism
 \[
  \beta^n \colon w^*(\R^nf^B_*(f^*B)) \to \R^nf^A_*(f^*A)
 \]
 which is obtained from $f^*(\theta) \colon f^*B \to f^*A$ by means of Construction \ref{con:derfib} with respect to the square in \eqref{e:topbc-sq} is an isomorphism of $A$-modules on $S$.
\end{lemma}

\begin{proof}
 It is clear that the diagram in \eqref{e:topbc-sq} commutes in the category of commutative ringed spaces. Now fix an integer $n$. Consider the $n$-th projection morphism relative $f^B$, denoted $\pi^n_{f^B}$, which is a natural transformation between certain functors going from $\Mod(S,B) \times \Mod(X,f^*B)$ to $\Mod(S,B)$, \cf Construction \ref{con:projmor}. By Proposition \ref{p:projmoriso}, we know that, for all $f^*B$-modules $F$ on $X$, the projection morphism
 $$\pi^n_{f^B}(A,F) \colon A \otimes_{(S,B)} \R^nf^B_*(F) \to \R^nf^B_*((f^B)^*A \otimes_{(X,f^*B)} F)$$
is an isomorphism of $B$-modules on $S$ given that $A$ is a locally finite free module on $(S,B)$. Therefore, writing
 $$\rho \colon (f^B)^*A \otimes_{(X,f^*B)} f^*B = f^*A \otimes_{(X,f^*B)} f^*B \to f^*A$$
 for the canonical isomorphism of modules on $(X,f^*B)$ (which is nothing but the right tensor unit for $f^*A$ on $(X,f^*B)$), we see that
 $$\R^nf^B_*(\rho) \circ \pi^n_{f^B}(A,f^*B) \colon A \otimes_{(S,B)} \R^nf^B_*(f^*B) \to \R^nf^B_*(f^*A)$$
 is an isomorphism of modules on $(S,B)$, too. Composing the latter morphism further with the base change
 $$\R^nf^B_*(f^*A) \to w_*(\R^nf^A_*(f^*A))$$
 yields yet another isomorphism of modules on $(S,B)$,
 $$A \otimes_{(S,B)} \R^nf^B_*(f^*B) \to w_*(\R^nf^A_*(f^*A)),$$
 which can be seen to equal $w_*(\beta^n)$ (we omit the verification of this very last assertion). Since the functor $w_* \colon \Mod(S,A) \to \Mod(S,B)$ is faithful, we deduce that $\beta^n$ is an isomorphism of modules on $(S,A)$.
\end{proof}

\begin{proposition}
 \label{p:drbc}
 Let $f\colon X\to S$ be a morphism of complex spaces and $t\in S$ such that $f$ is submersive in $f^{-1}(\{t\})$. Adopt the notation of Setup \ref{set:inbc} and let $n$ and $m$ be an integer and a natural number, respectively.
 \begin{enumerate}
  \item \label{p:drbc-free} The module $\sH^n(f_m)$ is free on $S_m$.
  \item \label{p:drbc-bc} For all $l\in\N$ such that $l\leq m$, the de Rham base change map
  $$\phi^n_{l,m} \colon b_{l,m}^*(\sH^n(f_m)) \to \sH^n(f_l)$$
  is an isomorphism in $\Mod(S_l)$.
 \end{enumerate}
\end{proposition}

\begin{proof}
 \ref{p:drbc-free})
 By abuse of notation, we write $X'$ (\resp $S'$) for the topological space induced on $X'$ (\resp $S'$) by $X_\top$ (\resp $S_\top$). We write $f'\colon X' \to S'$ for the corresponding morphism of topological spaces. Fix some natural number $k$. We set $B:=\C_{S'}$ and $A_k:=\O_{S_k}$ and write $\theta_k \colon B \to A_k$ for the morphism of sheaves of rings on $S'$ which is induced by the structural morphism $S_k \to \C$ of the complex space $S_k$. Define morphisms of ringed spaces
 \begin{align*}
  f'^{A_k} \colon (X',f'^*A_k) & \to (S',A_k), & f'^B \colon (X',f'^*B) & \to (S',B), \\
  u_k \colon (X',f'^*A_k) & \to (X',f'^*B), & w_k \colon (S',A_k) & \to (S',B)
 \end{align*}
 just as in Lemma \ref{l:topbc} (for $f'\colon X'\to S'$ in place of $f\colon X\to S$ and $\theta_k\colon B\to A_k$ in place of $\theta\colon B\to A$). Observe that $f' = (f_k)_\top$ and thus $f'^* = f_k^{-1}$ and $f'^{A_k} = \bar{f_k}$. Therefore, by Lemma \ref{l:topbc}, as $\theta_k$ makes $A_k$ into a finite free $B$-module on $S'$, there exists an isomorphism
 $$\beta_k \colon w_k^*(\R^nf'^B_*(f'^*B)) \to \R^nf'^{A_k}_*(f'^*A_k) = \R^n\bar{f_k}_*(f_k^{-1}\O_{S_k})$$
 of $A_k$-modules on $S'$. Since the morphism of complex spaces $f\colon X\to S$ is submersive in $f^{-1}(\{t\})$, we see that the morphism of complex spaces $f_k \colon X_k \to S_k$ is submersive. In consequence, the canonical morphism $f_k^{-1}\O_{S_k} \to \bar\Omega^\kdot_{f_k}$ in $\K^+(\bar{X_k})$ is a quasiisomorphism. Hence the induced morphism
 $$\R^n\bar{f_k}_*(f_k^{-1}\O_{S_k}) \to \R^n\bar{f_k}_*(\bar\Omega^\kdot_{f_k}) = \sH^n(f_k)$$
 is an isomorphism in $\Mod(S_k)$. Since $\R^nf'^B_*(f'^*B)$ is clearly free on $(S',B)$, whence $w_k^*(\R^nf'^B_*(f'^*B))$ is free on $(S',A_k)$, this shows that $\sH^n(f_k)$ is free on $S_k$.
 
 \ref{p:drbc-bc})
 Let $l\in\N$ such that $l\leq m$. Define
 \[
  \bar{i_{l,m}} \colon (X',f'^*A_l) \to (X',f'^*A_m)
 \]
 to be the morphism of ringed spaces given by $\id_{X'}$ and the image of the canonical morphism $A_m \to A_l$ of sheaves of rings on $S'$ under the functor $f'^*$. Then evidently, the diagram
 \begin{equation} \label{e:drbc-1}
  \xymatrix{
   (X',f'^*A_l) \ar[r]^{\bar{i_{l,m}}} \ar[d]_{\bar{f_l}} & (X',f'^*A_m) \ar[r]^{u_m} \ar[d]_{\bar{f_m}} & (X',f'^*B) \ar[d]^{f'^B} \\
   (S',A_l) \ar[r]_{b_{l,m}} & (S',A_m) \ar[r]_{w_m} & (S',B)
  }
 \end{equation}
 commutes in the category of ringed spaces, and we have
 \[
  u_m \circ \bar{i_{l,m}} = u_l \quad \text{and} \quad w_m \circ b_{l,m} = w_l.
 \]
 Define
 \[
  \beta_{l,m} \colon b_{l,m}^*(\R^n\bar{f_m}_*(f_m^{-1}\O_{S_m})) \to \R^n\bar{f_l}_*(f_l^{-1}\O_{S_l})
 \]
 to be the base change in degree $n$ with respect to the left-hand square in \eqref{e:drbc-1} induced by the canonical morphism $f'^*A_m \to f'^*A_l$ of sheaves on $X'$, \cf Construction \ref{con:derfib}. Then, by the associativity of base changes, the following diagram commutes in $\Mod(S_l)$:
 \[
  \xymatrix{
   b_{l,m}^*(w_m^*(\R^nf'^B_*(f'^*B))) \ar[r]^-{\sim} \ar[d]_{b_{l,m}^*(\beta_m)} & w_l^*(\R^nf'^B_*(f'^*B)) \ar[d]^{\beta_l} \\ b_{l,m}^*(\R^n\bar{f_m}_*(f_m^{-1}\O_{S_m})) \ar[r]_-{\beta_{l,m}} & \R^n\bar{f_l}_*(f_l^{-1}\O_{S_l})
  }
 \]
 Since $\beta_l$ and $\beta_m$ are isomorphisms (by Lemma \ref{l:topbc}, see above), it follows that $\beta_{l,m}$ is an isomorphism. By the functoriality of the morphisms $f^{-1}\O_{S} \to \bar\Omega^\kdot_{f}$ in terms of $f$, we know that the following diagram of complexes of modules commutes:
 \[
  \xysquare{f_m^{-1}\O_{S_m}}{f_l^{-1}\O_{S_l}}{\bar\Omega^\kdot_{f_m}}{\bar\Omega^\kdot_{f_l}}{}{}{}{}
 \]
 Therefore, due to the functoriality of Construction \ref{con:bc} (with respect to a fixed square), the following diagram commutes in $\Mod(S_l)$:
 \[
  \xymatrix{
   b_{l,m}^*(\R^n\bar{f_m}_*(f_m^{-1}\O_{S_m})) \ar[r]^-{\beta_{l,m}} \ar[d]_{\sim} & \R^n\bar{f_l}_*(f_l^{-1}\O_{S_l}) \ar[d]^{\sim} \\ b_{l,m}^*(\R^n\bar{f_m}_*(\bar\Omega^\kdot_{f_m})) \ar[r] & \R^n\bar{f_l}_*(\bar\Omega^\kdot_{f_l})
  }
 \]
 Here, the lower horizontal arrow is, by definition, nothing but the de Rham base change map
 $$\phi^n_{l,m} \colon b_{l,m}^*(\sH^n(f_m)) \to \sH^n(f_l),$$
 which we hence see to be an isomorphism.
\end{proof}

The upcoming series of results paves the way for the proof of Theorem \ref{t:froeinf}. Proposition \ref{p:grmod} and Lemma \ref{l:grbc} are rather general (we include them here for lack of good references), whereas Proposition \ref{p:infext} and Lemma \ref{l:hdginfext} are more specific and adapted to our Setup \ref{set:inbc} of infinitesimal neighborhoods. Lemma \ref{l:lang} recalls a result which is closely related to Nakayama's Lemma.

\begin{proposition}
 \label{p:grmod}
 Let $X$ be a commutative ringed space, $I$ an ideal on $X$, and $m$ a natural number. Then, for all locally finite free modules $F$ on $X$, the canonical morphism of sheaves
 \begin{equation} \label{e:grmod}
  I^m/I^{m+1} \otimes_X F \to I^mF/I^{m+1}F
 \end{equation}
 is an isomorphism in $\Mod(X)$ (or else in $\Mod(X_\top,\O_X/I)$ adjusting the scalar multiplications of the modules in \eqref{e:grmod} appropriately).
\end{proposition}

\begin{proof}
 Frist of all, one observes that the associations
 $$F \mto I^m/I^{m+1} \otimes_X F \quad \text{and} \quad F \mto I^mF/I^{m+1}F,$$
 where $F$ runs through the modules on $X$, are the object functions of certain additive functors from $\Mod(X)$ to $\Mod(X)$ (or else $\Mod(X)$ to $\Mod(X_\top,\O_X/I)$). The morphism \eqref{e:grmod} may be defined for all modules $F$ on $X$ and as such makes up a natural transformation between the mentioned functors. When $F = \O_X$, \eqref{e:grmod} clearly is an isomorphism in $\Mod(X)$ (or else $\Mod(X_\top,\O_X/I)$). Thus, by means of ``abstract nonsense'', \eqref{e:grmod} is an isomorphism for all finite free modules $F$ on $X$. When $F$ is only locally finite free on $X$, one concludes by restricting $X$, $I$, and $F$ to open subsets $U$ of $X$ over which $F$ is finite free; this works as \eqref{e:grmod} is in the obvious way compatible with restriction to open subspaces.
\end{proof}

\begin{lemma}
 \label{l:grbc}
 Let
 \begin{equation} \label{e:grbc-sq}
  \xysquare{Y}{X}{T}{S}{u}{g}{f}{w}
 \end{equation}
 be a pullback square in the category of commutative ringed spaces such that $f$ is flat in points coming from $u$ and $w^\sharp \colon w^{-1}\O_S \to \O_T$ is a surjective morphism of sheaves of rings on $T_\top$. Denote $\sI$ (\resp $\sJ$) the ideal sheaf which is the kernel of the morphism of rings $u^\sharp \colon u^{-1}\O_X \to \O_Y$ on $Y_\top$ (\resp $w^\sharp \colon w^{-1}\O_S \to \O_T$ on $T_\top$). Then, for all $m\in\N$, the canonical morphism
 \begin{equation} \label{e:grbc}
  (u^{-1}\O_X/\sI) \otimes_{(Y_\top,g^{-1}(w^{-1}\O_S/\sJ))} g^{-1}(\sJ^m/\sJ^{m+1}) \to \sI^m/\sI^{m+1}
 \end{equation}
 is an isomorphism of $u^{-1}\O_X/\sI$-modules on $Y_\top$.
\end{lemma}

\begin{proof}
 We formulate a sublemma. Let
 \[
  \xysquare{B}{B'}{A}{A'}{\theta}{\phi}{\phi'}{\eta}
 \]
 be a pushout square in the category of commutative rings such that $\theta \colon B \to B'$ is surjective and $\phi$ makes $A$ into a flat $B$-module. Denote $I$ (\resp $J$) the kernel of $\eta \colon A \to A'$ (\resp $\theta \colon B\to B'$). Then, for all $m\in\N$, the canonical map
 $$A/I \otimes_{B/J} J^m/J^{m+1} \to I^m/I^{m+1}$$
 is a bijection. This assertion can be proven by verifying inductively that, for all $m\in\N$, the canonical mappings $A \otimes_B J^m \to I^m$ and $A \otimes_B (J^m/J^{m+1}) \to I^m/I^{m+1}$ are bijective; we omit the details.
 
 In order to prove the actual lemma, let $y$ be an arbitrary element of $Y$ and $m$ a natural number. Then the image of the morphism \eqref{e:grbc} under the stalk-at-$y$ functor on $Y_\top$ is easily seen to be isomorphic, over the canonical isomorphism of rings $(u^{-1}\O_X/\sI)_y \to \O_{X,x}/I$, to the canonical map
 \begin{equation} \label{e:grbc-1}
  \O_{X,x}/I \otimes_{\O_{S,s}/J} J^m/J^{m+1} \to I^m/I^{m+1},
 \end{equation}
 where $x := u(y)$, $t := g(y)$, $s := f(x) = w(t)$, and $I$ (\resp $J$) denotes the kernel of $u^\sharp_y \colon \O_{X,x} \to \O_{Y,y}$ (\resp $w^\sharp_t \colon \O_{S,s} \to \O_{T,t}$). Now since the diagram in \eqref{e:grbc-sq} is a pullback square in the category of commutative ringed spaces, we know that
 \[
  \xysquare{\O_{S,s}}{\O_{T,t}}{\O_{X,x}}{\O_{Y,y}}{w^\sharp_t}{f^\sharp_x}{g^\sharp_y}{u^\sharp_y}
 \]
 is a pushout square in the category of commutative rings. Therefore, the map \eqref{e:grbc-1} is bijective according to the sublemma. As $y\in Y$ was arbitrary, we conclude that \eqref{e:grbc} is an isomorphism of sheaves on $Y_\top$.
\end{proof}

\begin{proposition}
 \label{p:infext}
 Assume we are in the situation of Setup \ref{set:inbc}. Let $F'$ be a locally finite free module on $(X',i^*\O_X)$. Assume that $f$ is flat along $f^{-1}(\{t\})$ and denote by $H$ the Hilbert function of the local ring $\O_{S,t}$. Define $I$ to be the ideal on $(X',i^*\O_X)$ which is given as the kernel of the morphism
 $$(i_0)^\sharp \colon i^*\O_X = i_0^{-1}\O_X \to \O_{X_0}$$
 of sheaves of rings on $X' = (X_0)_\top$. Then, for all $m\in\N$, we have
 \[
  I^mF'/I^{m+1}F' \iso (F'/IF')^{\oplus H(m)}
 \]
 as $(i^*\O_X)/I$-modules on $X'$.
\end{proposition}

\begin{proof}
 Let $m$ be a natural number. As $F'$ is a locally finite free module on the ringed space $(X',i^*\O_X)$, the canonical morphism
 \[
  I^m/I^{m+1} \otimes_{(X',i^*\O_X)} F' \to I^mF'/I^{m+1}F'
 \]
 of $(i^*\O_X)/I$-modules on $X'$ is an isomorphism by Proposition \ref{p:grmod}. Set
 \[
  b := (b_0)_\top \colon S' = (S_0)_\top \to S_\top
 \]
 and define $J$ to be the ideal on $(S',b^*\O_S)$ which is the kernel of the morphism
 $$(b_0)^\sharp \colon b^*\O_S = b_0^{-1}\O_S \to \O_{S_0}.$$
 Then by Proposition \ref{l:grbc}, since $f$ was assumed to be flat in $f^{-1}(\{t\})$, the canonical morphism
 \begin{equation} \label{e:infext-1}
  (i^*\O_X/I) \otimes_{(X',f'^*(b^*\O_S/J))} f'^*(J^m/J^{m+1}) \to I^m/I^{m+1}
 \end{equation}
 is an isomorphism of $i^*\O_X/I$-modules on $X'$. Since the local (or ``idealized'') rings $(b^*\O_S,J)$ and $(\O_{S,t},\fm)$ are isomorphic, $J^m/J^{m+1}$ is isomorphic to $(b^*\O_S/J)^{\oplus H(m)}$ as a module on $(S',b^*\O_S/J)$. Therefore, by means of the isomorphism \eqref{e:infext-1}, $I^m/I^{m+1}$ is isomorphic to $(i^*\O_X/I)^{\oplus H(m)}$ as a module on $(X',i^*\O_X/I)$. This in turn implies that we have
 \[
  I^m/I^{m+1} \otimes_{(X',i^*\O_X)} F' \iso ((i^*\O_X/I) \otimes_{(X',i^*\O_X)} F')^{\oplus H(m)} \iso (F'/IF')^{\oplus H(m)}
 \]
 in the category of $i^*\O_X/I$-modules on $X'$.
\end{proof}

\begin{lemma}
 \label{l:hdginfext}
 Assume we are in the situation of Setup \ref{set:inbc} with $f$ submersive in $f^{-1}(\{t\})$. Denote by $H$ the Hilbert function of the local ring $\O_{S,t}$. Then, for all $p,q\in\Z$ and all $m\in\N$, there exists $\alpha^{p,q}_m$ such that
 \begin{equation} \label{e:hdginfext}
  \xymatrix@C3pc{
   (\sH^{p,q}(f_0))^{\oplus H(m+1)} \ar[r]^-{\alpha^{p,q}_m} & \sH^{p,q}(f_{m+1}) \ar[r]^-{\beta^{p,q}_{m,m+1}} & \sH^{p,q}(f_m)
  }
 \end{equation}
 is an exact sequence in $\Mod(S_{m+1})$, where the first and the last of the modules in \eqref{e:hdginfext} need to be regarded as modules on $S_{m+1}$ via the canonical morphisms of rings $\O_{S_{m+1}} \to \O_{S_0}$ and $\O_{S_{m+1}} \to \O_{S_m}$ on $S'$, respectively.
\end{lemma}

\begin{proof}
 Let $p,q\in\Z$ and $m\in\N$. Moreover, let $k\in\N$. Then the diagram
 \begin{equation} \label{e:hdginfext-sq}
  \xysquare{X_k}{X}{S_k}{S}{i_k}{f_k}{f}{b_k}
 \end{equation}
 commutes in the category of complex spaces, \cf Setup \ref{set:inbc}, and thus induces a morphism $i_k^*(\Omega^p_f) \to \Omega^p_{f_k}$ of modules on $X_k$, which is nothing but the $i_k^*$-${i_k}_*$-adjoint of the usual pullback of $p$-differentials $\Omega^p_f \to {i_k}_*(\Omega^p_{f_k})$. Since \eqref{e:hdginfext-sq} is a pullback square, the mentioned morphism $i_k^*(\Omega^p_f) \to \Omega^p_{f_k}$ is, in fact, an isomorphism in $\Mod(X_k)$. Now define $I$ as in Proposition \ref{p:infext}. Then $(i_k)^\sharp \colon i^*\O_X \to \O_{X_k}$ factors uniquely through the quotient morphism $i^*\O_X \to (i^*\O_X)/I^{k+1}$ to yield an isomorphism $(i^*\O_X)/I^{k+1} \to \O_{S_k}$ of rings on $X'$. In turn, we obtain an isomorphism of sheaves on $X'$,
 \[
  ((i^*\O_X)/I^{k+1}) \otimes_{(X',i^*\O_X)} i^*(\Omega^p_f) \to \O_{X_k}\otimes_{(X',i^*\O_X)} i^*(\Omega^p_f) = i_k^*(\Omega^p_f).
 \]
 Precomposing with the canonical morphism
 \[
  F'/I^{k+1}F' \to (i^*\O_X/I^{k+1}) \otimes_{(X',i^*\O_X)} F'
 \]
 of $i^*\O_X/I^{k+1}$-modules on $X'$ for $F' := i^*(\Omega^p_f)$ and composing with the already mentioned $i_k^*(\Omega^p_f) \to \Omega^p_{f_k}$, we arrive at an isomorphism of modules on $X'$,
 \[
  \alpha_k \colon F'/I^{k+1}F' \to \Omega^p_{f_k},
 \]
 over the isomorphism of rings $i^*\O_X/I^{k+1} \to \O_{X_k}$.
 
 Since $f$ is submersive in $f^{-1}(\{t\}) = X'$, the module $\Omega^p_f$ is locally finite free on $X$ in $X'$. Consequently, $F'$ is a locally finite free module on $(X',i^*\O_X)$. Hence by Proposition \ref{p:infext}, there exists an isomorphism
 $$\psi \colon (F'/IF')^{\oplus H(m+1)} \to I^{m+1}F'/I^{m+2}F'$$
 of $i^*\O_X/I$-modules on $X'$. Consider the following sequence of morphisms of sheaves on $X'$:
 \begin{equation} \label{e:hdginfext-1}
  \begin{split}
  (\Omega^p_{f_0})^{\oplus H(m+1)} & \xrightarrow{(\alpha_0^{-1})^{\oplus H(m+1)}} (F'/IF')^{\oplus H(m+1)} \overset{\psi}\to I^{m+1}F'/I^{m+2}F' \\ & \to F'/I^{m+2}F' \xrightarrow{\alpha_{m+1}} \Omega^p_{f_{m+1}},
  \end{split}
 \end{equation}
 where the unlabeled arrow stands for the morphism obtained from the inclusion morphism $I^{m+1}F'\to F'$ by quotienting out $I^{m+2}F'$. Then, when $\Omega^p_{f_0}$ is regarded as an $\O_{X_{m+1}}$-module on $X'$ via the canonical morphism of sheaves of rings $\O_{X_{m+1}} \to \O_{X_0}$, the composition of sheaf maps \eqref{e:hdginfext-1} is a morphism of sheaves of $\O_{X_{m+1}}$-modules on $X'$. This is due to the fact that the following diagram of canonical morphisms of sheaves of rings on $X'$ commutes:
 \[
  \xysquare{i^*\O_X/I^{m+2}}{\O_{X_{m+1}}}{i^*\O_X/I}{\O_{X_0}}{}{}{}{}
 \]
 We define $\alpha^{p,q}_m$ to be the composition of the following sheaf maps on $S'$:
 \begin{align*}
  (\sH^{p,q}(f_0))^{\oplus H(m+1)} & = (\R^q{f_0}_*(\Omega^p_{f_0}))^{\oplus H(m+1)} \to \R^q{f_0}_*((\Omega^p_{f_0})^{\oplus H(m+1)}) \\ & \to \R^q{f_{m+1}}_*((\Omega^p_{f_0})^{\oplus H(m+1)}) \to \R^q{f_{m+1}}_*(\Omega^p_{f_{m+1}}),
 \end{align*}
 where the last arrow signifies the image of the composition \eqref{e:hdginfext-1} under the functor $\R^q{f_{m+1}}_*$ (viewed as a functor from $\Mod(X_{m+1})$ to $\Mod(S_{m+1})$).
 
 We show that, given this definition of $\alpha^{p,q}_m$, the sequence \eqref{e:hdginfext} constitutes an exact sequence in $\Mod(S_{m+1})$. For that matter, consider the pullback of $p$-differentials $\Omega^p_{f_{m+1}} \to \Omega^p_{f_m}$ associated to the square in \eqref{e:inbc-lm} with $l$ and $m$ replaced by $m$ and $m+1$, respectively. Then the following diagram commutes in $\Mod(X_{m+1})$:
 \begin{equation}
  \label{e:hdginfext-2}
  \xysquare{F'/I^{m+2}F'}{F'/I^{m+1}F'}{\Omega^p_{f_{m+1}}}{\Omega^p_{f_m}}{}{\alpha_{m+1}}{\alpha_m}{}
 \end{equation}
 In addition, the following sequence is exact in $\Mod(X_{m+1})$:
 \[
  0\to I^{m+1}F'/I^{m+2}F'\to F'/I^{m+2}F'\to F'/I^{m+1}F'\to 0.
 \]
 Thus, as the diagram in \eqref{e:hdginfext-2} commutes in $\Mod(X_{m+1})$, the vertical arrows being isomorphisms, the sequence
 \[
  0 \to I^{m+1}F'/I^{m+2}F' \to \Omega^p_{f_{m+1}} \to \Omega^p_{f_m}\to 0,
 \]
 where the second arrow denotes the composition of the canonical morphism
 \[
  I^{m+1}F'/I^{m+2}F' \to F'/I^{m+2}F'
 \]
 and $\alpha_{m+1}$, is exact in $\Mod(X_{m+1})$. Therefore, applying the functor $\R^q{f_{m+1}}_*$, we obtain an exact sequence in $\Mod(S_{m+1})$,
 \begin{equation} \label{e:hdginfext-3}
  \R^q{f_{m+1}}_*(I^{m+1}F'/I^{m+2}F')\to \sH^{p,q}(f_{m+1})\to \R^q{f_{m+1}}_*(\Omega^p_{f_m}).
 \end{equation}
 Now by the definition of $\alpha^{p,q}_m$, there exists an isomorphism
 $$(\sH^{p,q}(f_0))^{\oplus H(m+1)} \to \R^q{f_{m+1}}_*(I^{m+1}F'/I^{m+2}F')$$
 of modules on $S_{m+1}$ such that $\alpha^{p,q}_m$ equals the composition of it and the first arrow in \eqref{e:hdginfext-3}. Likewise, by the definition of the Hodge base change map, there exists an isomorphism
 \[
  \R^q{f_{m+1}}_*(\Omega^p_{f_m})\to\sH^{p,q}(f_m)
 \]
 (over $b_{m,m+1}$) such that $\beta^{p,q}_{m,m+1}$ equals the composition of the second arrow in \eqref{e:hdginfext-3} and this. Hence the sequence \eqref{e:hdginfext} is exact.
\end{proof}

\begin{lemma}
 \label{l:lang}
 Let $A$ be a commutative local ring with maximal ideal $\fm$. Let $E$ be a finite, projective $A$-module, $r$ a natural number, and $x$ an $r$-tuple of elements of $E$ such that $\kappa \circ x$ is an ordered $A/\fm$-basis for $E/\fm E$, where $\kappa \colon E \to E/\fm E$ denotes the residue class map. Then $x$ is an ordered $A$-basis for $E$.
\end{lemma}

\begin{proof}
 See \cite[Chapter X, Theorem 4.4]{La02}.
\end{proof}

\begin{theorem}
 \label{t:froeinf}
 Let $f\colon X\to S$ be a morphism of complex spaces and $t\in S$ such that $f$ is submersive in $f^{-1}(\{t\})$. Adopt the notation of Setup \ref{set:inbc}. Let $n$ be an integer and write $D_n$ for the $n$-diagonal in $\Z\times\Z$. Assume that the Frölicher spectral sequence of $X_0$ degenerates in entries $D_n$ at sheet $1$ in $\Mod(\C)$ and that, for all $(p,q)\in D_n$, the module $\sH^{p,q}(X_0)$ is of finite type over $\C$. Set
 $$h^{p,q} := \dim_\C(\sH^{p,q}(X_0)) \quad \text{and} \quad b^p := \dim_\C(\F^p\sH^n(X_0)).$$
 Then, for all $p\in\Z$, we have
 \begin{equation} \label{e:froeinf-dim}
  b^p = \sum_{\nu = p}^n h^{\nu,n-\nu}.
 \end{equation}
 Moreover, for all $m\in\N$, the following assertions hold:
 \begin{enumerate}
  \item[a$_m$)] \label{t:froeinf-degen} The Frölicher spectral sequence of $f_m$ degenerates in $D_n$ at sheet $1$ in $\Mod(S_m)$.
  \item[b$_m$)] \label{t:froeinf-free} For all $p\in\Z$, the modules $\sH^{p,n-p}(f_m)$ and $\F^p\sH^n(f_m)$ are free of ranks $h^{p,n-p}$ and $b^p$ on $S_m$, respectively.
  \item[c$_m$)] \label{t:froeinf-bc} For all $p\in\Z$ and all $l\in\N$ such that $l\leq m$, the base change maps
  $$\beta^{p,n-p}_{l,m} \colon b_{l,m}^*(\sH^{p,n-p}(f_m)) \to \sH^{p,n-p}(f_l)$$
  and
  $$\phi^{p,n}_{l,m} \colon b_{l,m}^*(\F^p\sH^n(f_m)) \to \F^p\sH^n(f_l)$$
  are isomorphisms in $\Mod(S_l)$.
 \end{enumerate}
\end{theorem}

\begin{proof}
 Let $p$ be an integer. When $p\geq n+1$, we have $\F^p\sH^n(X_0) \iso 0$ in $\Mod(\C)$ and thus $b^p = 0$, so that \eqref{e:froeinf-dim} holds. For arbitrary $p$, \eqref{e:froeinf-dim} now follows by means of descending induction on $p$ (starting at $p=n+1$) exploiting the fact that $\F^p\sH^n(X_0)/\F^{p+1}\sH^n(X_0)$ is isomorphic to $\sH^{p,n-p}(X_0)$ in $\Mod(\C)$ since the Frölicher spectral sequence of $X_0$ degenerates in $(p,n-p)$ at sheet $1$ in $\Mod(\C)$.
 
 In order to prove the second part of the theorem, we use induction on $m$. Denote $\C$ the distinguished terminal complex space and $a_{X_0}$ (\resp $a_{S_0}$) the unique morphism of complex spaces from $X_0$ (\resp $S_0$) to $\C$. Then the following diagram commutes in $\An$, the horizontal arrows being isomorphisms:
 \[
  \xysquare{X_0}{X_0}{S_0}{\C}{\id_{X_0}}{f_0}{a_{X_0}}{a_{S_0}}
 \]
 Hence the Frölicher spectral sequence of $a_{X_0}$, which is by definition the Frölicher spectral sequence of $X_0$, is isomorphic to the Frölicher spectral sequence of $f_0$ (over $a_{S_0}$). So, as the Frölicher spectral sequence of $X_0$ degenerates in entries $D_n$ at sheet $1$ in $\Mod(\C)$ by hypothesis, the Frölicher spectral sequence of $f_0$ degenerates in entries $D_n$ at sheet $1$ in $\Mod(S_0)$. This is a$_0$). By the functoriality of the conceptions `$\sH^{p,q}$' and `$\F^p\sH^n$', we see that, for all integers $p$ and $q$, the module $\sH^{p,q}(X_0)$ (\resp $\F^p\sH^n(X_0)$) is isomorphic to $\sH^{p,q}(f_0)$ (\resp $\F^p\sH^n(f_0)$) over $a_{S_0}$. Thus for all integers $p$, the module $\sH^{p,n-p}(f_0)$ (\resp $\F^p\sH^n(f_0)$) is free of rank $h^{p,n-p}$ (\resp $b^p$) on $S_0$, which proves b$_0$). Assertion c$_0$) is trivially fulfilled since $l\in\N$ together with $l\leq0$ implies that $l=0$; moreover, we have $i_{0,0}=\id_{X_0}$ and $b_{0,0}=\id_{S_0}$, whence, for all $p\in\Z$, by the functoriality of the base changes, $\beta^{p,n-p}_{0,0}$ and $\phi^{p,n}_{0,0}$ are the identities on $\sH^{p,n-p}(f_0)$ and $\F^p\sH^n(f_0)$ in $\Mod(S_0)$, respectively.

 Now, let $m\in\N$ be arbitrary and assume that a$_m$), b$_m$), and c$_m$) hold. For the time being, fix $(p,q)\in D_n$. Denote by $H$ the Hilbert function of the local ring $\O_{S,t}$. Then by Lemma \ref{l:hdginfext}, there exists a morphism of modules on $S_{m+1}$,
 $$\alpha^{p,q}_m \colon (\sH^{p,q}(f_0))^{\oplus H(m+1)} \to \sH^{p,q}(f_{m+1}),$$
 such that the three-term sequence \eqref{e:hdginfext} is exact in $\Mod(S_{m+1})$. By hypothesis b$_m$), $\sH^{p,q}(f_m)$ is isomorphic to $(\O_{S_m})^{\oplus h^{p,q}}$ in $\Mod(S_m)$. Therefore, $\sH^{p,q}(f_m)$ is of finite type over $\C$ and we have:
 $$\dim_\C(\sH^{p,q}(f_m)) = h^{p,q} \cdot \dim_\C(\O_{S_m}).$$
 Moreover, since $a_{S_0} \colon S_0 \to \C$ is an isomorphism, a$_0$) yields:
 $$\dim_\C(\sH^{p,q}(f_0)) = \dim_{S_0}(\sH^{p,q}(f_0)) = h^{p,q}.$$
 The exactness of the sequence \eqref{e:hdginfext} implies that $\sH^{p,q}(f_{m+1})$ is of finite type over $\C$ with
 \begin{equation} \label{e:froeinf-1}
  \begin{split}
  \dim_\C(\sH^{p,q}(f_{m+1})) & \leq H(m+1) \cdot h^{p,q} + h^{p,q} \cdot \dim_\C(\O_{S_m}) \\ & = h^{p,q} \cdot (H(m+1) + \dim_\C(\O_{S_m})) = h^{p,q} \cdot \dim_\C(\O_{S_{m+1}});
  \end{split}
 \end{equation}
 here we made use of the fact that, for all natural numbers $k$, firstly, $\O_{S_k}$ is isomorphic to $\O_{S,t}/\fm^{k+1}$ as a $\C$-algebra and secondly, $\dim_\C(\O_{S,t}/\fm^{k+1}) = \sum_{\nu=0}^kH(\nu)$.
 
 Let $E$ denote the Frölicher spectral sequence of $f_{m+1}$ and $d_r$ the differential of $E_r$. Then $E_1^{p,q}$ is isomorphic to $\sH^{p,q}(f_{m+1})$ in $\Mod(S_{m+1})$. In particular, we see that $E_1^{p,q}$ is of finite type over $\C$. Using induction, we deduce that $E_r^{p,q}$ is of finite type over $\C$ for all $r\in\N_{\geq1}$. Moreover, for all $r\in\N_{\geq1}$, we have
 $$\dim_\C(E_{r+1}^{p,q}) \leq \dim_\C(E_r^{p,q})$$
 with equality holding if and only if both $d_r^{p-r,q+r-1}$ and $d_r^{p,q}$ are zero morphisms. Now write $\F=(\F^\nu)_{\nu\in\Z}$ for the Hodge filtration on the algebraic de Rham module $\sH^n(f_{m+1})$, so that for all $\nu\in\Z$ we have $\F^\nu = \F^\nu\sH^n(f_{m+1})$. Then by the definition of the Frölicher spectral sequence, there exists $r\in\N_{\geq1}$ (take, \egv $r = \max(1,q+2)$) such that $E_r^{p,q}\iso \F^p/\F^{p+1}$ in $\Mod(S_{m+1})$ and $E$ degenerates in $(p,q)$ at sheet $r$ in $\Mod(S_{m+1})$. Hence we conclude that
 \begin{equation} \label{e:froeinf-2}
  \dim_\C(\F^p/\F^{p+1}) \leq \dim_\C(\sH^{p,q}(f_{m+1}))
 \end{equation}
 with equality holding if and only if $E$ degenerates in $(p,q)$ at sheet $1$ in $\Mod(S_{m+1})$.
 
 We abandon our fixation of $(p,q)$. By Proposition \ref{p:drbc} \ref{p:drbc-free}), $\sH^n(f_{m+1})$ is a finite free module on $S_{m+1}$ of rank $\dim_\C(\sH^n(X_0))$. Since $\sH^n(X_0) = \F^0\sH^n(X_0)$, we have $\dim_\C(\sH^n(X_0)) = b^0$. Thus,
 $$\dim_\C(\sH^n(f_{m+1})) = b^0 \cdot \dim_\C(\O_{S_{m+1}}).$$ 
 Besides, since $\F$ is a filtration on $\sH^n(f_{m+1})$ (by modules on $S_{m+1}$) satisfying $\F^{n+1}=0$ as well as $\F^0=\sH^n(f_{m+1})$, we have:
 $$\dim_\C(\sH^n(f_{m+1})) = \sum_{p=0}^n\dim_\C(\F^p/\F^{p+1}).$$
 Taking all of the above into account, we obtain:
 \begin{equation} \label{e:froeinf-3}
  \begin{split}
  & b^0 \cdot \dim_\C(\O_{S_{m+1}}) = \dim_\C(\sH^n(f_{m+1})) = \sum_{p=0}^n\dim_\C(\F^p/\F^{p+1}) \\ & \leq \sum_{p=0}^n\dim_\C(\sH^{p,n-p}(f_{m+1})) \leq \sum_{p=0}^n(h^{p,n-p} \cdot \dim_\C(\O_{S_{m+1}})) \\ & = \left(\sum_{p=0}^n h^{p,n-p}\right) \cdot \dim_\C(\O_{S_{m+1}}) \overset{\eqref{e:froeinf-dim}}= b^0 \cdot \dim_\C(\O_{S_{m+1}}).
  \end{split}
 \end{equation}
 By the antisymmetry of `$\leq$', equality holds everywhere in \eqref{e:froeinf-3}. By the strict monotony of finite sums, it follows that, for all $(p,q) \in D_n$ with $0\leq p\leq n$, we have equality in \eqref{e:froeinf-2}; yet this is possible only if $E$ degenerates in $(p,q)$ at sheet $1$ in $\Mod(S_{m+1})$. Observing that, for all $(p,q) \in D_n$ with $p<0$ or $n<p$, $E$ degenerates in $(p,q)$ at sheet $1$ in $\Mod(S_{m+1})$ as, in that case, $E_1^{p,q} \iso \sH^{p,q}(f_{m+1}) \iso 0$ in $\Mod(S_{m+1})$, we see we have proven a$_{m+1}$).
 
 As we have equality everywhere in \eqref{e:froeinf-3}, we deduce that, for all $(p,q)\in D_n$, we have equality in \eqref{e:froeinf-1}; note that when $p<0$ or $n<p$, this is clear a priori. Equality in \eqref{e:froeinf-1}, however, is possible only if \eqref{e:hdginfext} constitutes a short exact triple in $\Mod(\C)$ (or equivalently, in $\Mod(S_{m+1})$). We claim that, for all $p\in\Z$, the base change map
 $$\phi^{p,n}_{m,m+1} \colon \F^p\sH^n(f_{m+1}) \to \F^p\sH^n(f_m)$$
 is an epimorphism in $\Mod(S_{m+1})$. When $p\geq n+1$, this is obvious since then, $\F^p\sH^n(f_m) \iso 0$. Now let $p\in\Z$ be arbitrary and assume that $\phi^{p+1,n}_{m,m+1}$ is an epimorphism. By a$_m$) we know that the Frölicher spectral sequence of $f_m$ degenerates in $(p,n-p)$ at sheet $1$ in $\Mod(S_m)$, whence by Proposition \ref{p:froedegen} \ref{p:froedegen-behind}) there exists one, and only one, $\psi^p_m$ such that the following diagram commutes in $\Mod(S_m)$ (concerning notation, we refer the reader to Chapter \ref{ch:peri}):
 \begin{equation} \label{e:froeinf-degen}
  \xymatrix{
   \R^n\bar{f_m}_*(\sigma^{\geq p}\bar\Omega^\kdot_{f_m}) \ar[r] \ar@{}@<1ex>[r]^{\R^n\bar{f_m}_*(j^{\leq p}(\sigma^{\geq p}\bar\Omega^\kdot_{f_m}))} \ar[d]_{\lambda^n_{f_m}(p)} & \R^n\bar{f_m}_*(\sigma^{=p}\bar\Omega^\kdot_{f_m}) \ar[d]^{\kappa^n_{f_m}(\sigma^{=p}\Omega^\kdot_{f_m})} \\ \F^p\sH^n(f_m) \ar@{.>}[r]_{\psi^p_m} & \sH^{p,n-p}(f_m)
  }
 \end{equation}
 Furthermore, Proposition \ref{p:froedegen} \ref{p:froedegen-total}) tells that $\psi^p_m$ is isomorphic to the quotient morphism
 \[
  \F^p\sH^n(f_m) \to \F^p\sH^n(f_m)/\F^{p+1}\sH^n(f_m)
 \]
 (note that in the above diagram $\kappa^n_{f_m}(\sigma^{=p}\Omega^\kdot_{f_m})$ is an isomorphism). Similarly, by means of a$_{m+1}$), we dispose of a morphism $\psi^p_{m+1}$. Since any of the solid arrows in \eqref{e:froeinf-degen} is compatible with base change (in the obvious sense), one infers that the following diagram commutes in $\Mod(S_{m+1})$, the rows being exact:
 \[
  \xymatrix{
   0 \ar[r] & \F^{p+1}\sH^n(f_{m+1}) \ar[r]^\subset \ar[d]_{\phi^{p+1,n}_{m,m+1}} & \F^p\sH^n(f_{m+1}) \ar[r]^{\psi^p_{m+1}} \ar[d]_{\phi^{p,n}_{m,m+1}} & \sH^{p,q}(f_{m+1}) \ar[r] \ar[d]^{\beta^{p,q}_{m,m+1}} & 0 \\
   0 \ar[r] & \F^{p+1}\sH^n(f_m) \ar[r]_\subset & \F^p\sH^n(f_m) \ar[r]_{\psi^p_m} & \sH^{p,q}(f_m) \ar[r] & 0   
  }
 \]
 Here $\phi^{p+1,n}_{m,m+1}$ is an epimorphism by assumption, and $\beta^{p,q}_{m,m+1}$ is an epimorphism as \eqref{e:hdginfext} is a short exact triple. Hence the Five Lemma implies that $\phi^{p,n}_{m,m+1}$ is an epimorphism. Thus, our claim follows by descending induction on $p$ starting at, \egv $p=n+1$.
 
 By c$_m$), for all integers $p$, the morphism
 $$\phi^{p,n}_{0,m} \colon \F^p\sH^n(f_m) \to \F^p\sH^n(f_0)$$
 is an epimorphism in $\Mod(S_m)$. By the associativity of the base change construction we have, for all integers $p$:
 $$\phi^{p,n}_{0,m+1} = \phi^{p,n}_{0,m} \circ \phi^{p,n}_{m,m+1}.$$
 So, we see that $\phi^{p,n}_{0,m+1}$ is an epimorphism in $\Mod(S_{m+1})$ for all $p\in\Z$. Since the ring $\O_{S_0}$ is a field---in fact, the structural map $\C \to \O_{S_0}$ is bijective---and the numbers $h^{n,0},h^{n-1,1},\dots,h^{0,n}$ are altogether finite, there exists an ordered $\O_{S_0}$-basis (equivalently, $\C$-basis) $e=(e_\nu)_{\nu\in b^0}$ for $\sH^n(f_0)$ such that, for all $p\in\N$ with $p\leq n$, the restricted tuple $e|b^p = (e_0,\dots,e_{b^p-1})$ makes up an $\O_{S_0}$-basis for $\F^p\sH^n(f_0)$. By the surjectivity of the maps $\phi^{p,n}_{0,m+1}$ (for varying $p$), there exists a $b^0$-tuple $x = (x_\nu)$ of elements of $\sH^n(f_{m+1})$ such that, for all $p\in\N$ with $p\leq n$ and all $\nu\in b^p$, we have $x_\nu \in \F^p$ and $\phi^n_{0,m+1}(x_\nu) = e_\nu$. As pointed out before, $\sH^n(f_{m+1})$ is a finite free $\O_{S_{m+1}}$-module (by Proposition \ref{p:drbc} \ref{p:drbc-free})). Write $\fm$ for the unique maximal ideal of $\O_{S_{m+1}}$ and denote by
 $$\bar\phi^n_{0,m+1} \colon \sH^n(f_{m+1})/\fm\sH^n(f_{m+1}) \to \sH^n(f_0)$$
 the unique map which factors
 $$\phi^n_{0,m+1} \colon \sH^n(f_{m+1}) \to \sH^n(f_0)$$
 through the evident residue map. Then by Proposition \ref{p:drbc} \ref{p:drbc-bc}), $\bar\phi^n_{0,m+1}$ is an isomorphism of modules over the isomorphism of rings $\O_{S_{m+1}}/\fm \to \O_{S_0}$ which is induced by the canonical map $\O_{S_{m+1}} \to \O_{S_0}$, \iev by $b_{0,m+1} \colon S_0 \to S_{m+1}$. In particular, since the tuple $\bar x = (\bar{x_\nu})$ of residue classes obtained from $x$ is sent to $e$ by $\bar\phi^n_{0,m+1}$, we see that $\bar x$ constitutes an $\O_{S_{m+1}}/\fm$-basis for $\sH^n(f_{m+1})/\fm\sH^n(f_{m+1})$. Hence by Lemma \ref{l:lang}, $x$ constitutes an $\O_{S_{m+1}}$-basis for $\sH^n(f_{m+1})$. In consequence, the tuple $x$, and whence any restriction of it, is linearly independent over $\O_{S_{m+1}}$. Thus for all integers $p$, there exists an injective morphism
 $$(\O_{S_{m+1}})^{\oplus b^p} \to \F^p\sH^n(f_{m+1})$$
 of $\O_{S_{m+1}}$-modules. From a$_{m+1}$) we deduce, using a descending induction on $p\in\Z_{\leq n+1}$ as before, that
 $$\dim_\C(\F^p\sH^n(f_{m+1})) = \sum_{\nu=p}^n \dim_\C(\sH^{\nu,n-\nu}(f_{m+1}))$$
 for all $p\in\Z$. We already noted that, for all $(p,q)\in D_n$, equality holds in \eqref{e:froeinf-1}. As a result, we obtain:
 $$\dim_\C(\F^p\sH^n(f_{m+1})) = \sum_{\nu=p}^n (h^{\nu,n-\nu} \cdot \dim_\C(\O_{S_{m+1}})) = \dim_\C((\O_{S_{m+1}})^{\oplus b^p})$$
 for all integers $p$. It follows that, for all $p\in\Z$, any injective morphism of $\O_{S_{m+1}}$-modules (or yet merely $\C$-modules) from $(\O_{S_{m+1}})^{\oplus b^p}$ to $\F^p$ is indeed bijective. We deduce that, for all integers $p$, the module $\F^p$ is free of rank $b^p$ on $S_{m+1}$. Furthermore, for all $p\in\Z$, $\F^p$ equals the $\O_{S_{m+1}}$-span of $x_0,\dots,x_{b^p-1}$ in $\sH^n(f_{m+1})$. Consequently, for all $p\in\Z$, the module $\F^p/\F^{p+1}$ is free of rank $b^p-b^{p+1} = h^{p,n-p}$ on $S_{m+1}$. As $\F^p/\F^{p+1} \iso \sH^{p,n-p}(f_{m+1})$ in $\Mod(S_{m+1})$ according to a$_{m+1}$), this proves b$_{m+1}$).
 
 It remains to prove c$_{m+1}$). Let $p$ be an arbitrary integer. Put $q:=n-p$. Then, as already established above, the Hodge base change map 
 $$\beta^{p,q}_{m,m+1} \colon b_{m,m+1}^*(\sH^{p,q}(f_{m+1})) \to \sH^{p,q}(f_m)$$
 is surjective. Since by b$_{m+1}$), the module $\sH^{p,q}(f_{m+1})$ is free of rank $h^{p,q}$ on $S_{m+1}$, we see that $b_{m,m+1}^*(\sH^{p,q}(f_{m+1}))$ is free of rank $h^{p,q}$ on $S_m$. By b$_m$), the module $\sH^{p,q}(f_m)$ is free of rank $h^{p,q}$ on $S_m$, too. Therefore, the surjection $\beta^{p,q}_{m,m+1}$ is in fact a bijection. Analogously, one shows that the filtered de Rham base change map
  $$\phi^{p,n}_{m,m+1} \colon b_{m,m+1}^*(\F^p\sH^n(f_{m+1})) \to \F^p\sH^n(f_m)$$
 is an isomorphism in $\Mod(S_m)$. So, we have proven c$_{m+1}$) in case $l=m$. In case $l=m+1$ the assertion is trivial. In case $l<m$, the assertion follows from the assertion for $l=m$ combined with c$_m$) and the associativity of base changes.
\end{proof}

\section{Formal completions of complex spaces}
\label{s:formal}

\begin{setup}
 \label{set:comp}
 Let $f\colon X\to S$ be a morphism of complex spaces, $\sI$ and $\sJ$ finite type ideals on $X$ and $S$, respectively, such that $f^\sharp \colon \O_S \to f_*(\O_X)$ maps $\sJ$ into $f_*(\sI)$. Let $q$ be an integer. Let $F$ be a module on $X$. For any natural number $m$ set
 \[
  F^{(m)} := F/\sI^{m+1}F
 \]
 and write
 \[
  \alpha_m \colon F \to F^{(m)}
 \]
 for the evident quotient morphism. For natural numbers $l$ and $m$ such that $l\leq m$, denote
 \[
  \alpha_{lm} \colon F^{(m)} \to F^{(l)}
 \]
 the unique morphism such that $\alpha_l = \alpha_{lm} \circ \alpha_m$. Note that the data of the $F^{(m)}$'s and the $\alpha_{lm}$'s makes up an inverse system of modules on $X$; more formally, this data can be collected into a functor $\N^\op \to \Mod(X)$, which we refer to as $\alpha$.
 
 Set
 \[
  G := \R^qf_*(F)
 \]
 and define $G^{(m)}$, $\beta_m$, $\beta_{lm}$, and $\beta$ in analogy to $F^{(m)}$, $\alpha_m$, $\alpha_{lm}$, and $\alpha$ above. Observe that, for all natural numbers $m$, the morphism
 \[
  \R^qf_*(\alpha_m) \colon \R^qf_*(F) \to \R^qf_*(F^{(m)})
 \]
 factors uniquely through a morphism
 \[
  \tau_m \colon G^{(m)} = \R^qf_*(F)/\sJ^{m+1}\R^qf_*(F) \to \R^qf_*(F^{(m)})
 \]
 (via $\beta_m \colon G \to G^{(m)}$). Moreover, observe that the sequence $\tau = (\tau_m)_{m\in\N}$ constitutes a morphism of functors from $\N^\op$ to $\Mod(S)$,
 \[
  \tau \colon \beta \to \R^qf_* \circ \alpha,
 \]
 which in turn induces a morphism
 \begin{equation} \label{e:comp-1}
  \lim(\tau) \colon \lim(\beta) \to \lim(\R^qf_* \circ \alpha).
 \end{equation}
\end{setup}

\begin{theorem}
 \label{t:banica}
 In the situation of Setup \ref{set:comp} assume that $S$ is a complex manifold, $f$ is submersive, $\sJ$ is the ideal of a point $t\in S$, $\sI$ is the ideal of $f^{-1}(\{t\}) \subset X$, and $F$ is locally finite free on $X$. Furthermore, assume that $\R^qf_*(F)$ and $\R^{q+1}f_*(F)$ are coherent on $S$. Then the morphism \eqref{e:comp-1} is an isomorphism.
\end{theorem}

\begin{proof}
 This follows from \cite[VI, Theorem 4.1 (i)]{BaSt76} using a small modification of \loccit, VI, Proposition 4.2. Note that the morphism \eqref{e:comp-1} is introduced on p.~218 in \loccit and denoted $\phi_q$.
\end{proof}

\begin{lemma}
 \label{l:comfflat}
 For all Noetherian commutative local rings $A$, the canonical morphism of rings $A\to \hat A$ makes $\hat A$ into a faithfully flat $A$-module.
\end{lemma}

\begin{proof}
 Put $I:=\fm(A)$. As $A$ is local, $I$ is the only maximal ideal of $A$ and consequently the Jacobson radical of $A$ equals $I$. In particular, $I$ is a subset of the Jacobson radical of $A$. Hence we are finished taking into account that (1) implies (3) in \cite[Theorem 8.14]{Ma89}. Note that ``ring'' means commutative ring in \loccit
\end{proof}

\begin{proposition}
 \label{p:fflat}
 Let $\phi\colon A\to B$ be a morphism of commutative rings making $B$ into a faithfully flat $A$-module. Then, for all $A$-modules $M$, $M$ is flat over $A$ if and only if $B\otimes_A M$ is flat over $B$.
\end{proposition}

\begin{proof}
 See \cite[Exercises to \S7, 7.1]{Ma89}.
\end{proof}

\begin{lemma}
 \label{l:compfree}
 Let $A$ be a Noetherian commutative local ring and $M$ a finitely generated $A$-module. Then $M$ is (finite) free over $A$ if and only if $\hat A\otimes_A M$ is (finite) free over $\hat A$.
\end{lemma}

\begin{proof}
 Clearly, when $M$ is free over $A$, then $M$ is finite free over $A$ and thus $\hat A\otimes_A M$ is finite free over $\hat A$ since the tensor product distributes over direct sums and $\hat A\otimes_A A\iso \hat A$ as $\hat A$-modules.
 
 Conversely, assume that $N:=\hat A\otimes_A M$ is free over $\hat A$. Then $N$ is certainly flat over $\hat A$. By Lemma \ref{l:comfflat}, the canonical morphism of rings $A\to\hat A$ makes $\hat A$ into a faithfully flat $A$-module. Thus by Proposition \ref{p:fflat}, the fact that $N$ is flat over $\hat A$ implies that $M$ is flat over $A$. Therefore, since $M$ is a finitely presented $A$-module, $M$ is projective over $A$. Since $A$ is local, $M$ is free over $A$.
\end{proof}

\begin{lemma}
 \label{l:limitfree}
 Let $F$ be a functor from $\N^\op$ to the category of modules. For $k,l,m\in\N$ such that $l\leq m$, write $F_0(k)=(A_k,M_k)$ and $F_1(m,l)=(\theta_{l,m},\phi_{l,m})$. Assume that, for all $k\in\N$, the ring $A_k$ is commutative local and the module $M_k$ is finite free over $A_k$. Moreover, assume that, for all $l,m\in\N$ with $l\leq m$, $\theta_{l,m}$ is a surjection from $A_m$ onto $A_l$ mapping the unique maximal ideal of $A_m$ into the unique maximal ideal of $A_l$ and the morphism of $A_l$-modules $A_l\otimes_{A_m} M_m\to M_l$ induced by $\phi_{l,m}$ is an isomorphism.
 \begin{enumerate}
  \item The limit $(A_\infty,M_\infty)$ of $F$ with respect to $\N^\op$ and the category of modules is finite free.
  \item For all $l\in\N$, the morphism of $A_l$-modules $A_l\otimes_{A_\infty} M_\infty\to M_l$ induced by the canonical projections $A_\infty\to A_l$ and $M_\infty\to M_l$ is an isomorphism.
 \end{enumerate}
\end{lemma}

\begin{proof}
 By hypothesis there exists $r\in\N$ as well as an $r$-tuple $e^{(0)}$ with values in $M_0$ such that $e^{(0)}$ is an ordered $A_0$-basis of $M_0$. As for all $l\in\N$, $\phi_{l,l+1}$ is a surjection from $M_{l+1}$ onto $M_l$, there exists, for all $i\in r$, a function $e_i$ with domain of definition $\N$ such that $e_i(0)=(e^{(0)})_i$ and, for all $k\in\N$, $e_i(k)\in M_k$, and, for all $l\in\N$, $\phi_{l,l+1}(e_i(l+1))=e_i(l)$. It follows that, for all $i\in r$, $e_i\in M_\infty$. We claim that the $r$-tuple $e$ given by the $e_i$, $i\in r$, is an ordered $A_\infty$-basis for $M_\infty$. As an intermediate step we show that, for all $k\in\N$, the $r$-tuple $e^{(k)}$ given by the $e_i(k)$, $i\in r$, is an ordered $A_k$-basis for $M_k$. In fact, by hypothesis, $\phi_{0,k}$ factors to yield an isomorphism of modules from $M_k/\fm(A_k)M_k$ to $M_0/\fm(A_0)M_0$ over the isomorphism of rings $A_k/\fm(A_k)\to A_0/\fm(A_0)$ induced by $\theta_{0,k}$. As $e^{(0)}$ is a basis for $M_0$ over $A_0$, the residue class tuple $\bar e^{(0)}$ is a basis for $M_0/\fm(A_0)M_0$ over $A_0/\fm(A_0)$. Therefore, the residue class tuple $\bar e^{(k)}$ of $e^{(k)}$ is a basis for $M_k/\fm(A_k)M_k$ over $A_k/\fm(A_k)$. Thus $e^{(k)}$ is a basis for $M_k$ over $A_k$ by Lemma \ref{l:lang}.
 
 Now let $x\in M_\infty$ be arbitrary. Then, for all $k\in\N$, $x(k)\in M_k$, hence there exists a unique vector $\lambda^{(k)}\in (A_k)^r$ such that $\sum_{i\in r}(\lambda^{(k)})_ie_i(k)=x(k)$. Therefore, for all $l,m\in\N$ with $l\leq m$ and all $i\in r$, we have $\theta_{l,m}((\lambda^{(m)})_i)=(\lambda^{(l)})_i$. Thus, for all $i\in r$, the function $\lambda_i$ given by $\lambda_i(k)=(\lambda^{(k)})_i$ is an element of $A_\infty$. Moreover, $\sum_{i\in r}\lambda_ie_i=x$. Hence the tuple $e$ generates $M_\infty$ over $A_\infty$. Apart from that, assume we have given an element $\lambda\in (A_\infty)^r$ such that $\sum_{i\in r}\lambda_ie_i=0$ in $M_\infty$. Then, for all $k\in\N$, $\sum_{i\in r}\lambda_i(k)e_i(k)=0$ in $M_k$, whence $\lambda_i(k)=0$ in $A_k$ for all $k\in\N$ and $i\in r$. That is, for all $i\in r$, $\lambda_i=0$ in $A_\infty$. So we have proven a).
 
 Let $l\in\N$. In order to prove b), it suffices to note that as the $e_i$, $i\in r$, make up an $A_\infty$-basis for $M_\infty$, the elements $1\otimes e_i$ make up an $A_l$-basis for $A_l\otimes_{A_\infty} M_\infty$. In addition, the canonical morphism $A_l\otimes_{A_\infty} M_\infty\to M_l$ maps, for all $i\in r$, $1\otimes e_i$ to $e_i(l)=(e^{(l)})_i$ and the tuple $e^{(l)}$ is an $A_l$-basis for $M_l$, cf.~above.
\end{proof}

\begin{proposition}
 \label{p:hdglffbc0}
 Let $f\colon X\to S$ be a submersive morphism of complex spaces with smooth base, $t\in S$, $p$ and $q$ integers. We adopt the notation of Setup \ref{set:inbc}. Moreover, we set $n := p+q$ and write $D_n$ for the $n$-diagonal in $\Z\times\Z$. Assume that
 \begin{enumeratei}
  \item the Frölicher spectral sequence of $X_0$ degenerates in $D_n$ at sheet $1$ in $\Mod(\C)$;
  \item for all $(\nu,\mu)\in D_n$, the module $\sH^{\nu,\mu}(X_0)$ is of finite type over $\C$;
  \item both $\sH^{p,q}(f)$ and $\sH^{p,q+1}(f)$ are coherent modules on $S$.
 \end{enumeratei}
 Then the following assertions hold:
 \begin{enumerate}
  \item \label{p:hdglffbc0-lff} $\sH^{p,q}(f)$ is locally finite free on $S$ in $t$.
  \item \label{p:hdglffbc0-bc} For all $m\in\N$, the Hodge base change map
  \begin{equation} \label{e:hdglffbc0-bc}
   \beta^{p,q}_{m} \colon b_m^*(\sH^{p,q}(f)) \to \sH^{p,q}(f_m)
  \end{equation}
  is an isomorphism in $\Mod(S_m)$.
 \end{enumerate}
\end{proposition}

\begin{proof}
 a). Set $F := \Omega^p_f$. Consider the inverse system of modules given by
 \[
  \R^q{f_m}_*(i_m^*(F)) \to \R^q{f_l}_*(i_l^*(F))
 \]
 for natural numbers $l$ and $m$ with $l\leq m$. By Theorem \ref{t:froeinf} and Lemma \ref{l:limitfree}, we know that the limit of this inverse system is a finite free $\hat{\O_{S,t}}$-module (of rank equal to the dimension of $\sH^{p,q}(X_0)$ over $\C$). By Theorem \ref{t:banica}, the latter limit is, in the category of $\hat{\O_{S,t}}$-modules, isomorphic to the completion of the stalk
 \[
  (\R^qf_*(F))_t = (\sH^{p,q}(f))_t
 \]
 with respect to the maximal ideal of $\O_{S,t}$, and as the stalk $(\sH^{p,q}(f))_t$ is of finite type over $\O_{S,t}$, its completion is isomorphic to
 \[
  \hat{\O_{S,t}} \otimes_{\O_{S,t}} (\sH^{p,q}(f))_t.
 \]
 Thus by Lemma \ref{l:compfree}, we see that $(\sH^{p,q}(f))_t$ is a finite free module over $\O_{S,t}$. Since $\sH^{p,q}(f)$ is coherent on $S$ by assumption, we deduce that $\sH^{p,q}(f)$ is locally finite free on $S$ in $t$.
 
 b). Let $m$ be a natural number. Then the following diagram commutes:
 \[
  \xymatrix{
   \lim_k \left(b_k^*(\R^qf_*(F))\right) \ar[r]^\sim \ar[d]_\pr & \lim_k \left(\R^q{f_k}_*(i_k^*(F))\right) \ar[d]^\pr \\
   b_m^*(\R^qf_*(F)) \ar[r]_{\BC} & \R^q{f_m}_*(i_m^*(F))
  }
 \]
 Tensoring the morphism in the upper row with $\O_{S_m}$ over $\hat{\O_{S,t}}$, the vertical maps become isomorphisms. Thus the arrow in the lower row is an isomorphism. Hence the Hodge base change map \eqref{e:hdglffbc0-bc} is an isomorphism too.
\end{proof}

\section{Compactifiable submersive morphisms}
\label{s:ohsawa}

\begin{theorem}[Ohsawa's criterion]
 \label{t:ohsawa}
 Let $X$ be a compact, locally pure dimensional complex space of Kähler type and $A$ a closed analytic subset of $X$ such that $\Sing(X)\subset A$. Then the Frölicher spectral sequence of $X\setminus A$ degenerates in entries
 \[
  I := \{(p,q) \in \Z\times\Z : p+q+2 \leq \codim(A,X)\}
 \]
 at sheet $1$ in $\Mod(\C)$. Moreover, for all $(p,q)\in I$, we have
 \[
  \sH^{p,q}(X\setminus A) \iso \sH^{q,p}(X\setminus A)
 \]
 in $\Mod(\C)$.
\end{theorem}

\begin{proof}
 When $X$ is pure dimensional, the assertions follow from \cite[Theorem 1]{Oh87}. When $X$ is only locally pure dimensional, $X$ has finitely many connected components $X_0,\dots,X_{b-1}$. For all $\nu<b$, we know that $X_\nu$ is pure dimensional, so that the Frölicher spectral sequence of $Y_\nu := X_\nu \setminus (A\cap X_\nu)$ degenerates in entries $I$ at sheet $1$. Therefore, since the Frölicher spectral sequence of $X \setminus A$ is isomorphic to the direct sum of the Frölicher spectral sequences of the $Y_\nu$, wheres $\nu$ ranges through $b$, it degenerates in entries $I$ at sheet $1$ also. Similarly, the Hodge symmetry is traced back to the pure dimensional case.
\end{proof}

\begin{question}
 Let $X$ and $A$ be as in Theorem \ref{t:ohsawa}, $n$ an integer such that
 \[
  n+2 \leq \codim(A,X).
 \]
 Does the filtration $(\F^p\sH^n(X\setminus A))_{p\in\Z}$ on $\sH^n(X \setminus A)$ coincide with the Hodge filtration of the mixed Hodge structure of $n$-th cohomology associated to the compactification $X\setminus A \to X$ in the spirit of Deligne and Fujiki (\cf \cite{Fu80,De71,De74})?
\end{question}

\begin{proposition}
 \label{p:equidim}
 Let $f\colon X\to S$ be a flat morphism of complex spaces with locally pure dimensional base $S$. Then the following are equivalent:
 \begin{enumeratei}
  \item \label{i:equidim-f} $f$ is locally equidimensional.
  \item \label{i:equidim-x} $X$ is locally pure dimensional.
 \end{enumeratei}
\end{proposition}

\begin{proof}
 Since the morphism $f$ is flat, Theorem \ref{t:dimformula} tells that
 \begin{equation} \label{e:equidim-1}
  \dim_x(X) = \dim_x(X_{f(x)}) + \dim_{f(x)}(S)
 \end{equation}
 for all $x\in X$. By assumption, the complex space $S$ is locally pure dimensional so that the function given by the assignment
 \[
  s \mto \dim_s(S)
 \]
 is locally constant on $S$. Thus by the continuity of $f$, the function given by the assignment
 \[
  x \mto \dim_{f(x)}(S)
 \]
 is a locally constant function on $X$. Therefore, by \eqref{e:equidim-1} we see that $\dim_x(X_{f(x)})$ is a locally constant function of $x$ on $X$ if and only if $\dim_x(X)$ is a locally constant function of $x$ on $X$. This implies that \eqref{i:equidim-f} holds if and only if \eqref{i:equidim-x} holds.
\end{proof}

\begin{proposition}
 \label{p:dimrk}
 Let $f\colon X\to Y$ be a morphism of complex spaces and $p\in X$. Then
 \[
  \dim_p(X) - \dim_p(X_{f(p)}) \leq \dim_{f(p)}(S).
 \]
\end{proposition}

\begin{proof}
 See \cite[3.9, Proposition ($*$)]{Fi76}.
\end{proof}

\begin{definition}
 \label{d:relcodim}
 Let $f\colon X\to S$ be a morphism of complex spaces and $A$ a closed analytic subset of $X$. Then we define:
 \begin{equation} \label{e:relcodim}
  \codim(A,f) := \inf\{\codim(A\cap |X_s|,X_s):s\in S\},
 \end{equation}
 where $X_s$ denotes the fiber of $f$ over $s$. We call $\codim(A,f)$ the \emph{relative codimension} of $A$ with respect to $f$. Note that the set in \eqref{e:relcodim} over which the infimum is taken is a subset of $\N\cup\{\omega\}$. We agree on taking the infimum with respect to the canonical (strict) partial order on $\N\cup\{\omega\}$ given by the $\in$-relation.
\end{definition}

\begin{proposition}
 \label{p:codim}
 Let $f\colon X\to S$ be a morphism of complex spaces, $A$ a closed analytic subset of $X$.
 \begin{enumerate}
  \item \label{p:codim-point} For all $p\in A$, when $f$ is flat in $p$, we have
  \[
   \codim_p(A \cap |X_{f(p)}|,X_{f(p)}) \leq \codim_p(A,X).
  \]
  \item \label{p:codim-global} When $f$ is flat, we have
  \[
   \codim(A,f) \leq \codim(A,X).
  \]
 \end{enumerate}
\end{proposition}

\begin{proof}
 \ref{p:codim-point}). Let $p\in A$ and assume that $f$ is flat in $p$. Put $t := f(p)$. Then we have
 \[
  \dim_p(X_t) + \dim_t(S) = \dim_p(X)
 \]
 by Theorem \ref{t:dimformula}. By abuse of notation we denote $A$ also the closed analytic subspace of $X$ induced on $A$. Write $i\colon A\to X$ for the corresponding inclusion morphism and set $g := f \circ i$. Then Proposition \ref{p:dimrk}, applied to $g$, implies that
 \[
  \dim_p(A) - \dim_p(A_t) \leq \dim_t(S).
 \]
 On the whole, we obtain:
 \begin{align*}
  \codim_p(A \cap |X_t|,X_t) & = \dim_p(X_t) - \dim_p(A_t) \\
  & \leq (\dim_p(X) - \dim_t(S)) + (\dim_t(S) - \dim_p(A)) \\ & = \codim_p(A,X).
 \end{align*}
 Note that implicitly we have used that the closed complex subspace of $X_t$ induced on $A \cap |X_t|$ is canonically isomorphic to $A_t$, which is by definition the closed complex subspace of $A$ induced on $g^{-1}(\{t\})$.
 
 \ref{p:codim-global}). Noticing that
 \[
  \codim(A,f) = \inf\{\codim_p(A \cap X_{f(p)},X_{f(p)}) : p\in A\},
 \]
 the desired inequality follows immediately from part \ref{p:codim-point}).
\end{proof}

\begin{proposition}
 \label{p:hdglffbc}
 Let $f\colon X\to S$ be a proper, flat morphism of complex spaces, $A$ a closed analytic subset of $X$ such that $\Sing(f)\subset A$, and $t\in S$. Assume that $S$ is smooth, $f$ is locally equidimensional, and $X_t$ is of Kähler type. Define $g$ to be the composition of the canonical morphism $Y := X\setminus A\to X$ and $f$. Then, for all ordered pairs of integers $(p,q)$ such that
 \begin{equation} \label{e:hdglffbc-codim}
  p+q+2 \leq c := \codim(A,f),
 \end{equation}
 yet not $(p,q+2)=(0,c)$, the following assertions hold:
 \begin{enumerate}
  \item \label{p:hdglffbc-lff} The module $\sH^{p,q}(g)$ is locally finite free on $S$ in $t$.
  \item \label{p:hdglffbc-bc} The Hodge base change map
  \[
   (\sH^{p,q}(g))(t) \to \sH^{p,q}(Y_t)
  \]
  is an isomorphism in $\Mod(\C)$.
 \end{enumerate}
\end{proposition}

\begin{proof}
 Let $(p,q)$ be an ordered pair of integers as above. When $p<0$, assertions \ref{p:hdglffbc-lff}) and \ref{p:hdglffbc-bc}) are trivially fulfilled. So, assume that both $p\geq 0$. We claim that
 \begin{equation} \label{e:hdglffbc-1}
  (q+1)+2 \leq c = \codim(A,f).
 \end{equation}
 Indeed, when $p=0$, we have $q+2 \leq c$ by \eqref{e:hdglffbc-codim}, but also $q+2 \neq c$, so that $q+2 < c$. When $p>0$, \eqref{e:hdglffbc-1} follows directly from \eqref{e:hdglffbc-codim}.
 
 Using \eqref{e:hdglffbc-1} in conjunction with Proposition \ref{p:codim} \ref{p:codim-global}), we obtain:
 \[
  q+2 \leq (q+1)+2 \leq \codim(A,X).
 \] 
 Moreover, by Proposition \ref{p:equidim}, the complex space $X$ is locally pure dimensional. Thus by Proposition \ref{p:cohhdi}, we see that $\sH^{p,q+1}(g)$ and $\sH^{p,q}(g)$ are coherent modules on $S$.
 
 Next, we claim that, for all integers $\nu$ and $\mu$ such that $\nu+\mu = p+q$, the Hodge module $\sH^{\nu,\mu}(Y_t)$ is of finite type over $\C$. When $\nu<0$, this is clear. When $\nu \geq 0$, we have
 \[
  \mu+2 \leq \nu+\mu+2 = p+q+2 \leq c \leq \codim(A \cap |X_t|,X_t).
 \]
 In addition, the complex space $X_t$ is compact and locally pure dimensional. Thus $\sH^{\nu,\mu}(X_t \setminus (A \cap |X_t|))$ is of finite type over $\C$ by Proposition \ref{p:cohhdi}, whence $\sH^{\nu,\mu}(Y_t)$ is of finite type over $\C$ given that
 \[
  Y_t \iso X_t \setminus (A \cap |X_t|)
 \]
 as complex spaces.
 
 Theorem \ref{t:ohsawa} implies that for all integers $\nu$ and $\mu$ such that $\nu+\mu = p+q$ the Frölicher spectral sequence of $Y_t$ degenerates in $(\nu,\mu)$ at sheet $1$. Therefore, we deduce \ref{p:hdglffbc-lff}) and \ref{p:hdglffbc-bc}) from \ref{p:hdglffbc0-lff}) and \ref{p:hdglffbc0-bc}) of Proposition \ref{p:hdglffbc0}, respectively.
\end{proof}

\begin{question}
 Let $f$, $A$, and $t$ be as in Proposition \ref{p:hdglffbc}. Define $g$ and $c$ accordingly, and assume that $A \cap |X_t| \neq \emptyset$ (which implies that $c\in\N$). Do \ref{p:hdglffbc-lff}) and \ref{p:hdglffbc-bc}) of Proposition \ref{p:hdglffbc} hold for $(p,q)=(0,c-2)$?
\end{question}

\begin{proposition}
 \label{p:vanishing}
 Let $X$ be a Cohen-Macaulay complex space, $A$ a closed analytic subset of $X$, $F$ a locally finite free module on $X$.
 \begin{enumerate}
  \item \label{p:vanishing-lc} For all integers $q$ such that $q+1 \leq \codim(A,X)$, we have $\uH^q_A(X,F) \iso 0$.
  \item \label{p:vanishing-gap} Denote $j\colon X\setminus A\to X$ the inclusion morphism. When $2 \leq \codim(A,X)$, the canonical morphism
  \begin{equation} \label{e:vanishing-0}
   F \to \R^0j_*(j^*(F))
  \end{equation}
  of modules on $X$ is an isomorphism. Moreover, for all integers $q \neq 0$ such that $q+2 \leq \codim(A,X)$, we have $\R^qj_*(j^*(F)) \iso 0$.
 \end{enumerate}
\end{proposition}

\begin{proof}
 a). When $q$ is an integer such that $q+1 \leq \codim(A,X)$, then we have
 \[
  q+1+\dim_x(A) \leq \dim_x(X) = \dim(\O_{X,x}) = \prof_{X,x}(F)
 \]
 for all $x\in A$. Thus $\uH^q_A(X,F) \iso 0$ by \cite[II, Theorem 3.6, (b) $\Rightarrow$ (c)]{BaSt76}; see also II, Corollary 3.9 in \loccit
 
 b). When $2 \leq \codim(A,X)$, we deduce that $\uH^0_A(X,F) \iso 0$ and $\uH^1_A(X,F) \iso 0$ from part \ref{p:vanishing-lc}). Besides, by Proposition \ref{p:lcgap}, there exists an exact sequence
 \[
  0 \to \uH^0_A(X,F) \to F \xrightarrow{\can} \R^0j_*(j^*(F)) \to \uH^1_A(X,F) \to 0
 \]
 of sheaves of modules on $X$. Thus the canonical morphism \eqref{e:vanishing-0} is an isomorphism.
 
 Let $q$ be an integer $\neq 0$ such that $q+2 \leq \codim(A,X)$. When $q<0$, we have $\R^qj_*(j^*(F)) \iso 0$ in $\Mod(X)$ trivially. When $q>0$, we have
 \[
  R^qj_*(j^*(F)) \iso \uH^{q+1}_A(X,F)
 \]
 according to Proposition \ref{p:lcgap}, yet $\uH^{q+1}_A(X,F) \iso 0$ by means of \ref{p:vanishing-lc}).
\end{proof}

\begin{proposition}
 \label{p:gaprel}
 Let $f\colon X\to S$ be a morphism of complex spaces, $A$ a closed analytic subset of $X$ and $F$ a locally finite free module on $X$. Assume that $X$ is Cohen-Macaulay. Denote $j\colon X\setminus A\to X$ the inclusion morphism and set $g := f\circ j$. Then, for all integers $q$ such that $q+2 \leq \codim(A,X)$, the canonical morphism
 \begin{equation} \label{e:gaprel}
  \R^qf_*(F) \to \R^qg_*(j^*(F))
 \end{equation}
 of modules on $S$ is an isomorphism.
\end{proposition}

\begin{proof}
 Let $q$ be an integer as above. When $q<0$, our assertion is clear. So assume that $q\geq 0$. Denote $E$ the Grothendieck spectral sequence associated to the triple
 \[
  \xymatrix{
   \Mod(X\setminus A) \ar[r]^-{j_*} & \Mod(X) \ar[r]^-{f_*} & \Mod(S)
  }
 \]
 of categories and functors and the object $j^*(F)$ of $\Mod(X\setminus A)$. Then, for all integers $\nu$ and $\mu$, we have
 \[
  E_2^{\nu,\mu} \iso \R^{\nu}f_*(\R^{\mu}j_*(j^*(F))).
 \]
 In particular, Proposition \ref{p:vanishing} \ref{p:vanishing-gap}) implies that $E_2^{\nu,\mu} \iso 0$ whenever $\mu\neq 0$ and $\mu+2 \leq \codim(A,X)$. Of course, we also have $E_2^{\nu,\mu} \iso 0$ whenever $\nu<0$. Therefore, $E$ degenerates in entries of total degree $q$ at sheet $1$, and the edge morphism
 \[
  \R^qf_*(\R^0j_*(j^*(F))) \to \R^q(f_*\circ j_*)(j^*(F))
 \]
 is an isomorphism in $\Mod(S)$. Now $f_*\circ j_* = g_*$, and the edge morphism is the canonical one, so that the canonical morphism \eqref{e:gaprel} factors through the edge morphism via the morphism
 \[
  \R^qf_*(F) \to \R^qf_*(\R^0j_*(j^*(F)))
 \]
 which is obtained from \eqref{e:vanishing-0} by applying the functor $\R^qf_*$. By Proposition \ref{p:vanishing} \ref{p:vanishing-gap}), the morphism \eqref{e:vanishing-0} is an isomorphism in $\Mod(X)$ since
 \[
  2 \leq q+2 \leq \codim(A,X).
 \]
 Hence we conclude that \eqref{e:gaprel} is an isomorphism in $\Mod(S)$.
\end{proof}

\begin{proposition}
 \label{p:hdg00fllbc}
 Let $f\colon X\to S$ be a proper, flat morphism of complex spaces and $t\in S$ such that $X_t$ is a reduced complex space. Then the module $\R^0f_*(\O_X)$ is locally finite free on $S$ in $t$ and the base change map
 \[
  \C \otimes_{\O_{S,t}} (\R^0f_*(\O_X))_t \to \H^0(X_t,\O_{X_t})
 \]
 is an isomorphism of complex vector spaces.
\end{proposition}

\begin{proof}
 See \cite[III, Proposition 3.12]{BaSt76}.
\end{proof}

\begin{proposition}
 \label{p:hdglffbccm}
 Let $f\colon X\to S$ be a proper, flat morphism of complex spaces and $A$ a closed analytic subset of $X$ such that $\Sing(f)\subset A$. Assume that $S$ is smooth and the fibers of $f$ are Cohen-Macaulay and of Kähler type. Define $g$ to be the composition of the canonical morphism $Y := X\setminus A \to X$ and $f$. Then for all integers $p$ and $q$ such that
 \begin{equation} \label{e:hdglffbc-codim'}
  p+q+2 \leq c := \codim(A,f)
 \end{equation}
 the following assertions hold:
 \begin{enumerate}
  \item \label{p:hdglffbccm-lff} The module $\sH^{p,q}(g)$ is locally finite free on $S$.
  \item \label{p:hdglffbccm-bc} For all $s\in S$, the Hodge base change map
  \[
   (\sH^{p,q}(g))(s) \to \sH^{p,q}(Y_s)
  \]
  is an isomorphism in $\Mod(\C)$.
 \end{enumerate}
\end{proposition}

\begin{proof}
 Let $p$ and $q$ be integers satisfying \eqref{e:hdglffbc-codim'}. When $p<0$ or $q<0$, assertions \ref{p:hdglffbccm-lff}) and \ref{p:hdglffbccm-bc}) are trivially fulfilled. So, assume that both $p$ and $q$ are $\geq 0$. Then,
 \[
  2 \leq p+q+2 \leq \codim(A,f) \leq \codim(A,X) \leq \codim(\Sing(X),X),
 \]
 where we have used that
 \[
  \Sing(X) \subset \Sing(f) \subset A.
 \]
 Moreover, given that $f\colon X\to S$ is flat and fiberwise Cohen-Macaulay with a smooth base $S$, the complex space $X$ is Cohen-Macaulay. Thus $X$ is normal, whence locally pure dimensional by Proposition \ref{p:normalpuredim}. By Proposition \ref{p:equidim}, the morphism $f$ is locally equidimensional. Therefore, as long as $(p,q+2) \neq (0,c)$, assertions \ref{p:hdglffbccm-lff}) and \ref{p:hdglffbccm-bc}) are implied by the corresponding assertions of Proposition \ref{p:hdglffbc}.
 
 Now assume that $(p,q+2) = (0,c)$ (and still $q\geq 0$). We claim that $\R^qf_*(\O_X)$ is locally finite free on $S$ and compatible with base change. In case $q=0$, this follows from Proposition \ref{p:hdg00fllbc}. So, be $q>0$. Let $s\in S$ be arbitrary. Then $X_s$ is a Cohen-Macaulay complex space and $A \cap |X_s|$ is a closed analytic subset of $X_s$ such that
 \[
  q+2 \leq \codim(A,f) \leq \codim(A \cap |X_s|,X_s).
 \]
 Hence by Proposition \ref{p:gaprel} (applied to $a_{X_s} \colon X_s \to \C$), the canonical morphism
 \[
  \H^q(X_s,\O_{X_s}) \to \H^q(Y_s,\O_{Y_s}) = \sH^{0,q}(Y_s)
 \]
 is an isomorphism in $\Mod(\C)$. By Theorem \ref{t:ohsawa}, we have:
 \[
  \sH^{0,q}(Y_s) \iso \sH^{q,0}(Y_s).
 \]
 By what we have already proven, the module $\sH^{q,0}(g)$ is locally finite free on $S$ and compatible with base change, so that, in particular, (abandoning our fixation of $s$) the function
 \[
  s \mto \dim_\C(\sH^{q,0}(Y_s))
 \]
 is locally constant on $S$. As a consequence, the function
 \[
  s \mto \dim_\C(\H^q(X_s,\O_{X_s}))
 \]
 is a locally constant function on $S$ too, so that Grauert's continuity theorem yields our claim.
 
 We allow $q=0$ again. By Proposition \ref{p:gaprel}, the canonical morphism
 \[
  \R^qf_*(\O_X) \to \R^qg_*(\O_Y)
 \]
 is an isomorphism in $\Mod(S)$. Therefore $\R^qg_*(\O_Y) = \sH^{p,q}(g)$ is locally finite free on $S$. Furthermore, by the functoriality of base changes, the diagram
 \[
  \xymatrix{
   (\R^qf_*(\O_X))(s) \ar[r]^\sim \ar[d]_\BC & (\R^qg_*(\O_Y))(s) \ar[d]^\BC \\
   \H^q(X_s,\O_{X_s}) \ar[r]_\sim & \H^q(Y_s,\O_{Y_s})
  }
 \]
 commutes in the category of complex vector spaces for all $s\in S$. As the left base change in the diagram is an isomorphism, the right one is as well.
\end{proof}

\begin{theorem}
 \label{t:ohsawarel}
 Let $f\colon X\to S$ be a proper, flat morphism of complex spaces and $A$ a closed analytic subset of $X$ such that $\Sing(f)\subset A$. Assume that $S$ is smooth and the fibers of $f$ are Cohen-Macaulay as well as of Kähler type. Define $g$ to be the composition of the inclusion $Y:=X\setminus A\to X$ and $f$ and denote $E$ the Frölicher spectral sequence of $g$.
 \begin{enumerate}
  \item \label{t:ohsawarel-behind} $E$ degenerates from behind in entries
  \[
   I := \{(p,q) \in \Z\times\Z : p+q+2 \leq \codim(A,f)\}
  \]
  at sheet $1$ in $\Mod(S)$.
  \item \label{t:ohsawarel-total} $E$ degenerates in entries $I$ at sheet $1$ in $\Mod(S)$ if and only if either $A=\emptyset$ or $A\neq\emptyset$ and the canonical morphism
  \begin{equation} \label{e:ohsawarel-0}
   \sH^n(g) \to \sH^{0,n}(g)
  \end{equation}
  of modules on $S$, where we set $n := \codim(A,f) - 2$, is an epimorphism.
 \end{enumerate}
\end{theorem}

\begin{proof}
 \ref{t:ohsawarel-behind}). Let $(p,q) \in I$ and put $n := p+q$. When $n<0$, we know that $E_1^{p,q} \iso 0$ in $\Mod(S)$ so that $E$ certainly degenerates in $(p,q)$ at sheet $1$. So, we assume $n\geq 0$. Define $K$ to be the kernel of the morphism of modules on $S$,
 \[
  \R^ng_*(i^{\geq p}\Omega^\kdot_g) \colon \R^ng_*(\sigma^{\geq p}\Omega^\kdot_g) \to \R^ng_*(\Omega^\kdot_g).
 \]
 Let $t\in S$ be arbitrary. Adopt the notation of Setup \ref{set:inbc} for infinitesimal neighborhoods (for $g\colon Y\to S$ in place of $f\colon X\to S$). Fix a natural number $m$ and write $K_m$ for the kernel of the morphism
 \[
  \R^n{g_m}_*(i^{\geq p}\Omega^\kdot_{g_m}) \colon \R^n{g_m}_*(\sigma^{\geq p}\Omega^\kdot_{g_m}) \to \R^n{g_m}_*(\Omega^\kdot_{g_m})
 \]
 of modules on $S_m$. Then by the functoriality of base change maps the diagram
 \[
  \xymatrix{
   b_m^*(\R^ng_*(\sigma^{\geq p}\Omega^\kdot_g)) \ar[r] \ar[d]_\BC & b_m^*(\R^ng_*(\Omega^\kdot_g)) \ar[d]^\BC \\
   \R^n{g_m}_*(\sigma^{\geq p}\Omega^\kdot_{g_m}) \ar[r] & \R^n{g_m}_*(\Omega^\kdot_{g_m})
  }
 \]
 commutes in $\Mod(S_m)$. Thus, there exists one, and only one, $\alpha_m$ rendering commutative in $\Mod(S_m)$ the following diagram:
 \[
  \xymatrix@C3pc{
   b_m^*(K) \ar[r]^-{b_m^*(\subset)} \ar@{.>}[d]_{\alpha_m} & b_m^*(\R^ng_*(\sigma^{\geq p}\Omega^\kdot_g)) \ar[d]^\BC \\
   K_m \ar[r]_-\subset & \R^n{g_m}_*(\sigma^{\geq p}\Omega^\kdot_{g_m})
  }
 \]
 By Theorem \ref{t:ohsawa}, the Frölicher spectral sequence of $Y_0 = Y_t$ degenerates in entries of total degree $n$ at sheet $1$, so that the Frölicher spectral sequence of $Y_m$ degenerates in entries of total degree $n$ at sheet $1$ by means of Theorem \ref{t:froeinf} a$_m$). In particular, the Frölicher spectral sequence of $Y_m$ degenerates from behind in $(p,q)$ at sheet $1$. Therefore, by Proposition \ref{p:froedegen} \ref{p:froedegen-behind}), there exists $\psi_m$ such that the diagram
 \[
  \xymatrix{
   \R^n{g_m}_*(\sigma^{\geq p}\Omega^\kdot_{g_m}) \ar[rr] \ar[dr] && \R^n{g_m}_*(\sigma^{=p}\Omega^\kdot_{g_m}) \\ & \F^p\sH^n(g_m) \ar@{.>}[ru]_{\psi_m}
  }
 \]
 commutes in $\Mod(S_m)$. Thus the composition of the two arrows in the bottom row of the following diagram is zero:
 \begin{equation} \label{e:ohsawarel-1}
  \xymatrix{
   b_m^*(K) \ar[r] \ar[d]_{\alpha_m} & b_m^*(\R^ng_*(\sigma^{\geq p}\Omega^\kdot_g)) \ar[r] \ar[d]_\BC & b_m^*(\R^ng_*(\sigma^{=p}\Omega^\kdot_g)) \ar[d]^\BC \\
   K_m \ar[r] & \R^n{g_m}_*(\sigma^{\geq p}\Omega^\kdot_{g_m}) \ar[r] & \R^n{g_m}_*(\sigma^{=p}\Omega^\kdot_{g_m})
  }
 \end{equation}
 By Proposition \ref{p:hdglffbc0} \ref{p:hdglffbc0-bc}), the Hodge base change map which is the rightmost vertical arrow in the diagram in \eqref{e:ohsawarel-1} is an isomorphism. Moreover, by the functoriality of base change maps and the definition of $\alpha_m$, the diagram in \eqref{e:ohsawarel-1} commutes. In consequence, we see that the composition of the two arrows in the top row of \eqref{e:ohsawarel-1} equals zero. Since $m$ was an arbitrary natural naumber, we deduce that the composition of stalk maps
 \[
  K_t \to (\R^ng_*(\sigma^{\geq p}\Omega^\kdot_g))_t \to (\R^ng_*(\sigma^{=p}\Omega^\kdot_g))_t
 \]
 equals zero. As $t$ was an arbitrary element of $S$, we deduce further that the composition
 \[
  K \to \R^ng_*(\sigma^{\geq p}\Omega^\kdot_g) \to \R^ng_*(\sigma^{=p}\Omega^\kdot_g)
 \]
 of morphisms of sheaves of modules on $S$ equals zero. Thus there exists $\psi$ such that the diagram
 \[
  \xymatrix{
   \R^ng_*(\sigma^{\geq p}\Omega^\kdot_g) \ar[rr] \ar[dr] && \R^ng_*(\sigma^{=p}\Omega^\kdot_g) \\ & \F^p\sH^n(g) \ar@{.>}[ru]_{\psi}
  }
 \]
 commutes in $\Mod(S)$ and, in turn, Proposition \ref{p:froedegen} \ref{p:froedegen-behind}) tells that the Frölicher spectral sequence of $g$ degenerates from behind in $(p,q)$ at sheet $1$.
 
 \ref{t:ohsawarel-total}). Not that by \ref{t:ohsawarel-behind}), the spectral sequence $E$ certainly degenerates in entries
 \[
  \{(p,q) \in \Z\times\Z : p+q+3 \leq \codim(A,f)\}
 \]
 at sheet $1$. Therefore, our assertions holds in case $A = \emptyset$. In case $A \neq \emptyset$, our assertion follows readily from Proposition \ref{p:froedegen} \ref{p:froedegen-total}).
\end{proof}

\begin{proposition}
 \label{p:ohsawarel+}
 Let $f\colon X\to S$ and $A$ be as in Theorem \ref{t:ohsawarel}. Define $g\colon Y\to S$ accordingly. Assume that $A\neq\emptyset$ and put $n:=\codim(A,f)-2$. Moreover, assume that, for all $s\in S$, the canonical mapping
 \begin{equation} \label{e:ohsawarel+}
  \H^n(X_s,\C) \to \sH^{0,n}(X_s)
 \end{equation}
 is a surjection. Then the canonical morphism \eqref{e:ohsawarel-0} is an epimorphism in $\Mod(S)$.
\end{proposition}

\begin{proof}
 Let $s\in S$ be arbitrary. As $f$ is proper, the base change map
 \[
  (\R^nf_*(\C_X))_s \to \H^n(X_s,\C)
 \]
 is a bijection. By assumption, the canonical map \eqref{e:ohsawarel+} is a surjection. As $X_s$ is Cohen-Macaulay and $A_s := A \cap |X_s|$ is a closed analytic subset of $X_s$ such that
 \[
  n+2 = \codim(A,f) \leq \codim(A_s,X_s),
 \]
 the morphism
 \[
  \sH^{0,n}(X_s) \to \sH^{0,n}(X_s\setminus A_s)
 \]
 induced by the inclusion $X_s\setminus A_s \to X_s$ of complex spaces is a bijection. Now the morphism $Y_s\to X_s$ of complex spaces which is induced on fibers over $s$ by the inclusion $Y \to X$ is isomorphic to $X_s\setminus A_s \to X_s$ in the overcategory $\An_{/X_s}$, hence by the functoriality of $\sH^{0,n}$, the morphism
 \[
  \sH^{0,n}(X_s)\to \sH^{0,n}(Y_s)
 \]
 induced by $Y_s \to X_s$ is a bijection. Combining these results, one infers that the composition of functions
 \begin{equation} \label{e:defrtlsing-0}
  (\R^nf_*(\C_X))_s \to \H^n(X_s,\C) \to \sH^{0,n}(X_s) \to \sH^{0,n}(Y_s)
 \end{equation}
 is a surjection.

 The diagram
 \begin{equation} \label{e:defrtlsing-1}
  \xymatrix{
   (\R^nf_*(\C_X))_s \ar[r] \ar[d] & (\R^ng_*(\C_Y))_s \ar[r] \ar[d] & (\sH^n(g))(s) \ar[d] \ar@/^4pc/[dd] \\
   \H^n(X_s,\C) \ar[r] \ar[d] & \H^n(Y_s,\C) \ar[r] \ar[d] & \sH^n(Y_s) \ar[ld] \\
   \sH^{0,n}(X_s) \ar[r] & \sH^{0,n}(Y_s) & (\sH^{0,n}(g))(s) \ar[l]
  }
 \end{equation}
 commutes in $\Mod(\C)$ by the functoriality of the various base changes appearing in it. By the commutativity of the diagram in \eqref{e:defrtlsing-1} and the surjectivity of the composition \eqref{e:defrtlsing-0}, we deduce that the composition
 \[
  (\sH^n(g))(s) \to (\sH^{0,n}(g))(s) \to \sH^{0,n}(Y_s)
 \]
 is a surjection. As the base change map
 \[
  (\sH^{0,n}(g))(s) \to \sH^{0,n}(Y_s)
 \]
 is a bijection by Proposition \ref{p:hdglffbccm} \ref{p:hdglffbccm-bc}), we see that
 \[
  (\sH^n(g))(s) \to (\sH^{0,n}(g))(s)
 \]
 is a surjection. By Proposition \ref{p:hdglffbccm} \ref{p:hdglffbccm-lff}), the module $(\sH^{0,n}(g))_s$ is finite free over $\O_{S,s}$, whence Nakayama's lemma tells that
 \[
  (\sH^n(g))_s \to (\sH^{0,n}(g))_s
 \]
 is a surjection. Taking into account that while conducting this argument, $s$ was an arbitrary point of $S$, we conclude that the canonical morphism \eqref{e:ohsawarel-0} of sheaves of modules on $S$ is an epimorphism.
\end{proof}

\begin{corollary}
 \label{c:froertlsing}
 Let $f\colon X\to S$ be a proper, flat morphism of complex spaces such that $S$ is smooth and the fibers of $f$ have rational singularities, are of Kähler type, and have singular loci of codimension $\geq 4$. Define $g\colon Y\to S$ to be the submersive share of $f$ (\cf Notation \ref{not:singf}). Set
 \[
  I := \{(\nu,\mu) \in \Z\times\Z : \nu+\mu \leq 2\}.
 \]
 \begin{enumerate}
  \item \label{c:froertlsing-lff} For all $(p,q)\in I$, the module $\sH^{p,q}(g)$ is locally finite free on $S$.
  \item \label{c:froertlsing-bc} For all $(p,q)\in I$ and all $s\in S$, the Hodge base change map
  \begin{equation} \label{e:froertlsing-bc}
   (\sH^{p,q}(g))(s) \to \sH^{p,q}(Y_s)
  \end{equation}
  is an isomorphism in $\Mod(\C)$.
  \item \label{c:froertlsing-degen} The Frölicher spectral sequence of $g$ degenerates in entries $I$ at sheet $1$ in $\Mod(S)$.
 \end{enumerate}
\end{corollary}

\begin{proof}
 Set $A:=\Sing(f)$. Then $A$ is a closed analytic subset of $X$ which contains $\Sing(f)$. Moreover, according to Notation \ref{not:singf}, we have $Y = X\setminus A$ and $g$ equals the composition of the canonical morphism $Y\to X$ and $f$. Since any complex space which has rational singularities is Cohen-Macaulay, the fibers of $f$ are Cohen-Macaulay. Let $(p,q)\in I$. Then
 \[
  p+q+2 \leq 4 \leq \codim(A,f).
 \]
 Thus by Proposition \ref{p:hdglffbccm}, the Hodge module $\sH^{p,q}(g)$ is locally finite free on $S$ and, for all $s\in S$, the Hodge base change map \eqref{e:froertlsing-bc} is an isomorphism in $\Mod(\C)$. As $(p,q)$ was an arbitrary element of $I$, this proves \ref{c:froertlsing-lff}) and \ref{c:froertlsing-bc}).
 
 For proving \ref{c:froertlsing-degen}), we distinguish two cases: Firstly, suppose that $\codim(A,f) \geq 5$. Then clearly
 \[
  I \subset \{(p,q) \in \Z\times\Z : p+q+3 \leq \codim(A,f)\}.
 \]
 So, \ref{c:froertlsing-degen}) is implied by \ref{t:ohsawarel-behind}) of Theorem \ref{t:ohsawarel}. Secondly, suppose that $\codim(A,f)<5$. Then as $\codim(A,f)\geq 4$, we have $\codim(A,f)=4$. In particular, $A\neq \emptyset$. For all $s\in S$ the complex space $X_s$ has rational singularities, whence the canonical mapping
 \[
  \H^2(X_s,\C) \to \sH^{0,2}(X_s)
 \]
 is a surjection. Therefore, \ref{c:froertlsing-degen}) is implied by Proposition \ref{p:ohsawarel+} in conjunction with \ref{t:ohsawarel-total}) of Theorem \ref{t:ohsawarel}.
\end{proof}

\chapter{Symplectic complex spaces}
\label{ch:symp}

In this chapter we study symplectic complex spaces. The main results are Theorem \ref{t:lt} (Local Torelli) and Theorem \ref{t:fr} (Fujiki Relation).

In \S\ref{s:symp}, we define what we mean by a symplectic complex space; moreover, we establish certain fundamental properties of such spaces. In \S\ref{s:bbform}, we generalize the notion of the Beauville-Bogomolov form, which has proven a pivotal tool in the study of irreducible symplectic manifolds, to the context of compact, connected, symplectic complex spaces $X$ whose $\Omega^2_X(X_\reg)$ is $1$\hyphen dimensional over the field of complex numbers. \S\ref{s:defsymp} is consecrated to the (local) deformation theory of compact, Kähler, symplectic complex spaces whose singular loci have codimension $\geq 4$. In \S\ref{s:lt}, we derive our version of a local Torelli theorem for symplectic complex spaces as well as several corollaries thereof. Eventually, in \S\ref{s:fr}, we prove that compact, Kähler, symplectic complex spaces $X$ with $1$\hyphen dimensional $\Omega^2_X(X_\reg)$ and singular locus of codimension $\geq 4$ satisfy the Fujiki relation.

\section{Symplectic structures on complex spaces}
\label{s:symp}

We intend to define a notion of symplecticity for complex spaces. Our definition (\cf Definition \ref{d:symp} below) is inspired by Y.\ Namikawa's notions of a ``projective symplectic variety'' and a ``symplectic variety'' in \cite{Na01} and \cite{Na01a}, respectively (note that the two definitions differ slightly), as well as by A.\ Beauville's notion of ``symplectic singularities'', \cf \cite[Definition 1.1]{Be00}. These concepts rely themselves on the concept of symplectic structures on complex manifolds. For the origins of symplectic structures on complex manifolds, we refer our readers to the works \cite{Bo74} and \cite{Bo78} of F.\ Bogomolov, where the term ``Hamiltonian'' is used instead of ``symplectic'', as well as to Beauville's \cite[``Définition'' in \S4]{Be83}.

In Definition \ref{d:symp0}, we coin the new term of a ``generically symplectic structure'' on a complex manifold---not with the intention of specifically studying the geometry of spaces possessing these structures, but merely as a tool to define and study the Beauville-Bogomolov form for (possibly nonsmooth) complex spaces in \S\ref{s:bbform}. Furthermore, we introduce notions of ``symplectic classes'', which seem new in the literature too. Apart from giving definitions, we state several easy or well-known consequences of the fact that a complex space is symplectic. Proposition \ref{p:ratgor} (mildness of singularities) and Proposition \ref{p:symph2} (purity of the mixed Hodge structure $\H^2(X)$) are of particular importance and fundamental for the theory developed in subsequent sections.

Let us point out that our view on symplectic structures is purely algebraic in the sense that, to begin with, our candidates for symplectic structures are elements of the sheaf of Kähler $2$\hyphen differentials on a complex space. That way, our ideas and terminology may be translated effortlessly into the framework of, say, (relative) schemes, even though we refrain from realizing this translation here. We start by defining nondegeneracy for a global $2$\hyphen differential on a complex manifold. The definition might appear a little unusual, yet we like it for its algebraic nature.

\begin{definition}[Nondegeneracy]
 \label{d:nondeg}
 Let $X$ be a complex manifold and $\sigma\in\Omega^2_X(X)$. Define $\phi$ to be the composition of the following morphisms in $\Mod(X)$:
 \begin{equation} \label{e:nondeg}
  \Theta_X \to \Theta_X\otimes\O_X \xrightarrow{\id\otimes\sigma} \Theta_X\otimes\Omega^2_X \to \Omega^1_X.
 \end{equation}
 Here, the first and last arrows stand for the inverse of the right tensor unit for $\Theta_X$ on $X$ and the contraction morphism $\gamma^2_X(\Omega^1_X)$, \cf Notation \ref{not:cont}, respectively, and
 \[
  \sigma \colon \O_X \to \Omega^2_X
 \]
 denotes, by abuse of notation, the unique morphism of modules on $X$ sending the $1$ of $\O_X(X)$ to the actual $\sigma \in \Omega^2_X(X)$.
 \begin{enumerate}
  \item \label{d:nondeg-pt} Let $p\in X$. Then $\sigma$ is called \emph{nondegenerate} on $X$ at $p$ when $\phi$ is an isomorphism of modules on $X$ at $p$, \iev when there exists an open neighborhood $U$ of $p$ in $X$ such that $i^*(\phi)$ is an isomorphism in $\Mod(X|U)$, where $i \colon X|U \to X$ denotes the evident inclusion morphism.
  \item \label{d:nondeg-global} $\sigma$ is called \emph{nondegenerate} on $X$ when $\sigma$ is nondegenerate on $X$ at $p$ for all $p\in X$.
  \item \label{d:nondeg-gen} $\sigma$ is called \emph{generically nondegenerate} on $X$ when there exists a thin subset $A$ of $X$ such that $\sigma$ is nondegenerate on $X$ at $p$ for all $p \in X \setminus A$.
 \end{enumerate}
\end{definition}

Observe that allowing $X$ to be an arbitrary complex space in Definition \ref{d:nondeg} would not cause any problems. However, as we do not see whether the thereby obtained more general concept possesses intriguing meaning---especially at the singular points of $X$---, we desisted from admitting the further generality.

\begin{remarks}
 \label{r:nondeg}
 Let $X$ and $\sigma$ be as in Definition \ref{d:nondeg}. Define $\phi$ accordingly.
 \begin{enumerate}
  \item \label{r:nondeg-global} By general sheaf theory, we see that $\sigma$ is nondegenerate on $X$ if and only if
  \[
   \phi \colon \Theta_X \to \Omega^1_X
  \]
  is an isomorphism in $\Mod(X)$.
  \item \label{r:nondeg-pt} Let $p \in X$. Then $\sigma$ is nondegenerate on $X$ at $p$ if and only if there exists an open neighbohood $U$ of $p$ in $X$ such that the image of $\sigma$ under the canonical mapping $\Omega^2_X(X) \to \Omega^2_{X|U}(U)$ is nondegenerate on $X|U$.
  \item \label{r:nondeg-deg} Define
  \[
   D := \{p \in X: \text{$\sigma$ is not nondegenerate on $X$ at $p$}\}
  \]
  to be the \emph{degeneracy locus} of $\sigma$ on $X$. Then $D$ is a closed analytic subset of $X$. Moreover, $\sigma$ is generically nondegenerate on $X$ if and only if $D$ is thin in $X$; $\sigma$ is nondegenerate on $X$ if and only if $D = \emptyset$.
 \end{enumerate}
\end{remarks}

We briefly digress in order to establish, for later use, a typical characterization of the nondegeneracy of a Kähler $2$\hyphen differential $\sigma$ on a complex manifold $X$ employing the wedge powers of $\sigma$, \cf Proposition \ref{p:nondeg}. The quick reader may well skip this discussion and head on to Definition \ref{d:symp0} immediately.

\begin{proposition}
 \label{p:nondegchart}
 Let $X$ be a complex manifold, $\sigma \in \Omega^2_X(X)$, and $p\in X$. Let $n$ be a natural number and $z\colon U \to \C^n$ an $n$\hyphen dimensional (holomorphic) chart on $X$ at $p$. Then the following are equivalent:
 \begin{enumeratei}
  \item $\sigma$ is nondegenerate on $X$ at $p$;
  \item $\phi_p\colon \Theta_{X,p} \to \Omega^1_{X,p}$ is an isomorphism in $\Mod(\O_{X,p})$, where $\phi$ denotes the composition \eqref{e:nondeg} in $\Mod(X)$ (just as in Definition \ref{d:nondeg});
  \item when $A$ is the unique alternating $n\times n$-matrix with values in the ring $\O_X(U)$ such that
  \[
   \sigma|U = \sum_{i<j}A_{ij}\cdot\dd z^i\wedge\dd z^j,
  \]
  then $A(p) \in \GL_n(\C)$, where $A(p)$ denotes the composition of $A$ with the evaluation of sections in $\O_X$ over $U$ at $p$.
 \end{enumeratei}
\end{proposition}

\begin{proof}
 (i) implies (ii) since for any open neighborhood $V$ of $p$ in $X$ the stalk-at-$p$ functor $\Mod(X) \to \Mod(\O_{X,p})$ on $X$ factors over $i^*\colon \Mod(X) \to \Mod(X|V)$, where $i\colon X|V\to X$ denotes the obvious inclusion morphism. That (ii) implies (i) is due to the fact that $\Theta_X$ and $\Omega^1_X$ both are coherent modules on $X$. Now let $A$ be a matrix as in (iii). Then, essentially by the definition of the contraction morphism, the matrix associated with the morphism of $\O_X(U)$-modules $\phi(U) \colon \Theta_X(U) \to \Omega^1_X(U)$ relative to the bases $(\partial_{z^0},\dots,\partial_{z^{n-1}})$ and $(\dd z^0,\dots,\dd z^{n-1})$ is the transpose $A^\transp$ of $A$. Thus the matrix associated with the morphism of $\O_{X,p}$-modules $\phi_p\colon \Theta_{X,p}\to \Omega^1_{X,p}$ relative to the bases $((\partial_{z^0})_p,\dots,(\partial_{z^{n-1}})_p)$ and $((\dd z^0)_p,\dots,(\dd z^{n-1})_p)$ is $(A^\transp)_p$, by which we mean the matrix of germs of $A^\transp$. Hence we have (ii) if and only if $(A^\transp)_p \in \GL_n(\O_{X,p})$. Yet $(A^\transp)_p \in \GL_n(\O_{X,p})$ if and only if $(A^\transp)(p) \in \GL_n(\C)$ if and only if $A(p) \in \GL_n(\C)$, the latter equivalence being true for $(A^\transp)(p)=(A(p))^\transp$.
\end{proof}

\begin{corollary}
 \label{c:nondegdim}
 Let $X$ be a complex manifold, $\sigma \in \Omega^2_X(X)$, and $p\in X$ such that $\sigma$ is nondegenerate on $X$ at $p$. Then $\dim_p(X)$ is even.
\end{corollary}

\begin{proof}
 Set $n:=\dim_p(X)$. Then there exists an $n$\hyphen dimensional holomorphic chart $z\colon U\to \C^n$ on $X$ at $p$. There exists an alternating $n\times n$-matrix $A$ with values in the ring $\O_X(U)$ such that $\sigma|U = \sum_{i<j}A_{ij}\cdot\dd z^i\wedge\dd z^j$. By Proposition \ref{p:nondegchart}, as $\sigma$ is nondegenerate on $X$ at $p$, we have $A(p) \in \GL_n(\C)$. Yet the existence of an invertible, alternating complex $n\times n$-matrix implies that $n$ is even (see, \egv \cite[XV, Theorem 8.1]{La02}).
\end{proof}

\begin{remark}[Pfaffians]
 \label{r:pfaffian}
 Let $R$ be a commutative ring. We define the $R$-\emph{Pfaffian} $\Pf_R$ as a function on the set of alternating (quadratic) matrices of arbitrary size over $R$; the function $\Pf_R$ is to take values in $R$. Concretely, when $A$ is an alternating $n\times n$-matrix over $R$, where $n$ is some natural number, we set $\Pf_R(A):=0$ in case $n$ is odd; in case $n$ is even, we set
 \[
  \Pf_R(A) := \sum_{\pi \in \Pi} \sgn(\pi) A_{\pi(0),\pi(1)} A_{\pi(2),\pi(3)} \cdot \ldots \cdot A_{\pi(n-2),\pi(n-1)},
 \]
 where $\Pi$ denotes the set of all permutations $\pi$ of $n$ such that we have
 \[
  \pi(0) < \pi(1), \quad \pi(2) < \pi(3), \quad \dots, \quad \pi(n-2) < \pi(n-1)
 \]
 and
 \[
  \pi(0) < \pi(2) < \dots < \pi(n-2).
 \]
 The $R$-Pfaffian enjoys property that, for any alternating $n\times n$-matrix over $R$, we have (\cf \egv \cite{Mui60}):
 \begin{equation} \label{e:pfaffiandet}
  \sideset{}{_R}\det(A) = (\Pf_R(A))^2.
 \end{equation}
\end{remark}

\begin{proposition}
 \label{p:pfaffian}
 Let $R$ be a commutative ring, $r\in\N$, $M$ an $R$-module, $v$ an ordered $R$-basis of length $2r$ for $M$, and $A$ an alternating $2r\times 2r$-matrix with values in $R$. Set $\sigma := \sum_{i<j} A_{ij}\cdot v_i\wedge v_j$.
 \begin{enumerate}
  \item We have $\sigma^{\wedge r}= r!\Pf_R(A)\cdot v_0\wedge\dots\wedge v_{2r-1}$.
  \item Assume that $R$ is a field of characteristic zero. Then the following are equivalent:
  \begin{enumeratei}
   \item $A \in \GL_{2r}(R)$;
   \item $\sigma^{\wedge r} \neq 0$ in $\wedge^{2r}_R(M)$.
  \end{enumeratei}
 \end{enumerate}
\end{proposition}

\begin{proof}
 We omit the calculation leading to assertion a). As to b): When $A \in \GL_{2r}(R)$, we have $\det_R(A)\neq 0$, hence $\Pf_R(A)\neq 0$ by \eqref{e:pfaffiandet}. Thus as $\ch(R)=0$, we have $r!\Pf_R(A)\neq 0$, so that (ii) follows from a). Conversely, when (ii) holds, a) implies that $\Pf_R(A)\neq 0$. Thus $\det_R(A)\neq 0$ by \eqref{e:pfaffiandet}, which implies (i).
\end{proof}

\begin{proposition}
 \label{p:nondeg}
 Let $X$ be a complex manifold, $\sigma \in \Omega^2_X(X)$, $p \in X$, and $r \in \N$ such that $\dim_p(X) = 2r$. Then the following are equivalent:
 \begin{enumeratei}
  \item $\sigma$ is nondegenerate on $X$ at $p$;
  \item $(\sigma^{\wedge r})(p) \neq 0$ in the complex vector space $(\Omega^{2r}_X)(p)$;
  \item $(\sigma'^{\wedge r})(p) \neq 0$ in $\wedge^{2r}_\C(\T^*_{\C,p}(X))$, where $\sigma'$ denotes the image of $\sigma$ under the canonical mapping $\Omega^2_X(X) \to \A^{2,0}(X)$ and the wedge power is calculated in $\A^*(X,\C)$.
 \end{enumeratei}
\end{proposition}

\begin{proof}
 There exists a $2r$\hyphen dimensional holomorphic chart $z$ on $X$ at $p$. Set $U:=\dom(z)$ and let $A$ be the unique alternating $2r\times 2r$-matrix over $\O_X(U)$ such that $\sigma|U = \sum_{i<j} A_{ij}\cdot \dd z^i\wedge\dd z^j$. Then $\sigma'(p) = \sum_{i<j} A_{ij}(p)\cdot \dd z^i(p)\wedge \dd z^j(p)$. Thus by Proposition \ref{p:pfaffian}, we have $(\sigma'(p))^{\wedge r} \neq 0$ in $\wedge^{2r}_\C(\T^*_{\C,p}(X))$ if and only if $A(p) \in \GL_{2r}(\C)$. As $(\sigma'^{\wedge r})(p) = (\sigma'(p))^{\wedge r}$, Proposition \ref{p:nondegchart} implies that (i) and (iii) are equivalent. Now let $\lambda \in \O_X(U)$ such that $\sigma^{\wedge r}|U=\lambda\cdot\dd z^0\wedge\dots\wedge \dd z^{2r-1}$. Then (ii) holds if and only if $\lambda(p)\neq 0$. Yet $\sigma'^{\wedge r}|U = [\lambda]\cdot\dd z^0\wedge\dots\wedge\dd z^{2r-1}$. Therefore, (iii) holds if and only if $[\lambda](p)\neq 0$. Since $[\lambda](p)=\lambda(p)$ per definitionem, (ii) and (iii) are equivalent and we are finished.
\end{proof}

\begin{definition}[Symplecticity I]
 \label{d:symp0}
 Let $X$ be a complex manifold.
 \begin{enumerate}
  \item $\sigma$ is called a \emph{(generically) symplectic structure} on $X$ when $\sigma\in\Omega^2_X(X)$ such that $\sigma$ is (generically) nondegenerate on $X$ and $\dd^2_X \colon \Omega^2_X \to \Omega^3_X$ sends $\sigma$ to the zero of $\Omega^3_X(X)$.
  \item $X$ is called \emph{(generically) symplectic} when there exists $\sigma$ such that $\sigma$ is a (generically) symplectic structure on $X$.
  \item $w$ is called \emph{(generically) symplectic class} on $X$ when there exists a (generically) symplectic structure $\sigma$ on $X$ such that $w$ is the class of $\sigma$ in $\H^2(X,\C)$; note that it makes sense to speak of ``the class of $\sigma$ in $\H^2(X,\C)$'' given that $\sigma$ is a closed $2$\hyphen differential on $X$ by a).
 \end{enumerate}
\end{definition}

\begin{remark}
 \label{r:symp0}
 For us, an interesting feature of the class of generically symplectic complex manifolds---as opposed to the (strictly smaller) class of symplectic complex manifolds---is presented by the fact that the former is stable under modifications, precisely: When $W$ and $X$ are complex manifolds, $f\colon W\to X$ is a modification, and $\sigma$ is a generically symplectic structure on $X$, then the image of $\sigma$ under the pullback of Kähler differentials $\Omega^2_X(X) \to \Omega^2_W(W)$ which is induced by $f$ is a generically symplectic structure on $W$. The proof is straightforward. In consequence, when $f\colon W\to X$ is a modification such that $W$ is smooth and $X$ is a generically symplectic complex manifold, then $W$ is a generically symplectic complex manifold too.
\end{remark}

In order to define what a symplectic structure on a complex space is (\cf Definition \ref{d:symp} below) we need to talk about extensions of Kähler $2$\hyphen differentials defined over the regular locus of some complex space with respect to resolutions of singularities. As we will encounter the phenomenon of extension of differentials more often than once, let us introduce the following convention of speech right away.

\begin{definition}
 \label{d:extdiff}
 Let $f\colon W\to X$ be a resolution of singularities, $p\in\N$, and $\beta \in \Omega^p_X(X_\reg)$. Then $\alpha$ is called an \emph{extension as $p$\hyphen differential} of $\beta$ with respect to $f$ when $\alpha \in \Omega^p_W(W)$ such that the restriction of $\alpha$ to $f^{-1}(X_\reg)$ within the presheaf $\Omega^p_W$ equals the image of $\beta$ under the pullback of $p$\hyphen differentials mapping $\Omega^p_X(X_\reg) \to \Omega^p_W(f^{-1}(X_\reg))$ induced by $f$.

 As the case where $p=2$ is of primary---maybe even exclusive---interest for us, we agree on using the word ``extension'' as a synonym for ``extension as $2$\hyphen differential'' in the above construction.
\end{definition}

\begin{proposition}
 \label{p:extdiff}
 Let $X$ be a complex space, $p\in\N$, and $\beta \in \Omega^p_X(X_\reg)$.
 \begin{enumerate}
  \item Let $f\colon W\to X$ be a resolution of singularities and $\alpha$ and $\alpha'$ extensions as $p$\hyphen differentials of $\beta$ with respect to $f$. Then $\alpha = \alpha'$.
  \item When $X$ is a reduced complex space, the following are equivalent:
  \begin{enumerate}
   \item there exists a resolution of singularities $f_0\colon W_0\to X$ and $\alpha_0$ such that $\alpha_0$ is an extension as $p$\hyphen differential of $\beta$ with respect to $f_0$;
   \item for all resolutions of singularities $f\colon W\to X$ there exists $\alpha$ such that $\alpha$ is an extension as $p$\hyphen differential of $\beta$ with respect to $f$.
  \end{enumerate}
 \end{enumerate}
\end{proposition}

\begin{proof}
 a). Since $f$ is a resolution of singularities, $W\setminus f^{-1}(X_\reg)$ is a closed thin subset of $W$, and $W$ is a complex manifold. Hence the restriction mapping $\Omega^p_W(W) \to \Omega^p_W(f^{-1}(X_\reg))$ is one-to-one. As both $\alpha$ and $\alpha'$ restrict to the pullback of $\beta$ along $f$, \cf Definition \ref{d:extdiff}, we obtain $\alpha = \alpha'$.

 b). Assume (i). Let $f\colon W\to X$ be a resolution of singularities. Then there exists a complex manifold $V$ as well as proper modifications $g_0\colon V\to W_0$ and $g\colon V\to W$ such that $f_0\circ g_0 = f\circ g =: h$. Since $g$ is a proper modification between complex manifolds, the pullback function $g^* \colon \Omega^p_W(W) \to \Omega^p_V(V)$ is a bijection. Define $\alpha$ to be the inverse image under $g^*$ of the image of $\alpha_0$ under the function $\Omega^p_{W_0}(W_0) \to \Omega^p_V(V)$. Then it is easily verified that $\alpha$ restricts to the pullback of $\beta$ under $f$ within the presheaf $\Omega^p_W$. In fact, this is true on $V$ so that it suffices to note that the function $\Omega^p_W(f^{-1}(X_\reg)) \to \Omega^p_V(h^{-1}(X_\reg))$ is injective. Therefore $\alpha$ is an extension as $p$\hyphen differential of $\beta$ with respect to $f$, and we have proven (ii). Conversely, when one assumes (ii), (i) follows instantly; one simply has note that there exists a resolution of singularities $f_0\colon W_0\to X$ since $X$ is a reduced complex space.
\end{proof}

\begin{definition}[Symplecticity II]
 \label{d:symp}
 Let $X$ be a complex space.
 \begin{enumerate}
  \item \label{d:symp-str} $\sigma$ is called a \emph{symplectic structure} on $X$ when $\sigma \in \Omega^2_X(X_\reg)$ such that:
  \begin{enumerate}
   \item \label{d:symp-str-mfd} The image of $\sigma$ under the pullback function $\Omega^2_X(X_\reg) \to \Omega^2_{X_\reg}(X_\reg)$ induced by the canonical morphism of complex spaces $X_\reg \to X$ is a symplectic structure on $X_\reg$ in the sense of Definition \ref{d:symp0}.
   \item \label{d:symp-str-ext} For all resolutions of singularities $f\colon W\to X$, there exists $\rho$ such that $\rho$ is an extension as $2$\hyphen differential of $\sigma$ with respect to $f$.
  \end{enumerate}
  \item \label{d:symp-symp} $X$ is called \emph{symplectic} when $X$ is normal and there exists $\sigma$ such that $\sigma$ is a symplectic structure on $X$.
 \end{enumerate}
\end{definition}

\begin{proposition}
 \label{p:sympres}
 Let $X$ be a symplectic complex space and $f\colon W\to X$ a resolution of singularities. Then:
 \begin{enumerate}
  \item When $\sigma$ is a symplectic structure on $X$ and $\rho$ an extension as $2$\hyphen differential of $\sigma$ with respect to $f$, then $\rho$ is a generically symplectic structure on $W$. In particular, $\dd^2_W \colon \Omega^2_W \to \Omega^3_W$ sends $\rho$ to the zero of $\Omega^3_W(W)$.
  \item $W$ is a generically symplectic complex manifold.
 \end{enumerate}
\end{proposition}

\begin{proof}
 a). Let $\sigma$ and $\rho$ be as proposed. Then $\rho \in \Omega^2_W(W)$ by Definition \ref{d:extdiff}. As $f$ is a resolution of singularities, there exist closed thin subsets $A$ and $B$ of $W$ and $X$, respectively, such that $f$ induces an isomorphism $W\setminus A \to X\setminus B$. As $W\setminus A$ is a complex manifold, we have $X\setminus B \subset X_\reg$. Hence the pullback of $\sigma$ along the inclusion morphism $X\setminus B \to X$ is nondegenerate on $X\setminus B$, and therefore, for all $p\in W\setminus A$, $\rho$ is nondegenerate on $W$ at $p$. By Definition \ref{d:nondeg} c), $\rho$ generically nondegenerate on $W$. As $\rho$ is an extension as $2$\hyphen differential of $\sigma$ with respect to $f$, the restriction of $\rho$ to $f^{-1}(X_\reg)$ within the presheaf $\Omega^2_W$ equals the pullback of $\sigma$ along $f$. As $\sigma$ is a symplectic structure on $X$, we have $(\dd^2_X)_{X_\reg}(\sigma) = 0$ in $\Omega^3_X(X_\reg)$. Thus $(\dd^2_W)_{f^{-1}(X_\reg)}(\rho|f^{-1}(X_\reg)) = 0$ in $\Omega^3_W(f^{-1}(X_\reg))$. As $W\setminus f^{-1}(X_\reg) \subset A$, we see that $W\setminus f^{-1}(X_\reg)$ is thin in $W$. In consequence, the restriction mapping $\Omega^3_W(W) \to \Omega^3_W(f^{-1}(X_\reg))$ is certainly one-to-one. So, $(\dd^2_W)_W(\rho) = 0$ in $\Omega^3_W(W)$. By Definition \ref{d:symp0} a), $\rho$ is a generically symplectic structure on $W$.

 b). As $X$ is a symplectic complex space, there exists a symplectic structure $\sigma$ on $X$. By (ii) of Definition \ref{d:symp} a), there exists an extension as $2$\hyphen differential $\rho$ of $\sigma$ with respect to $f$. Now by a), $\rho$ is a generically symplectic structure on $W$. Hence $W$ is a generically symplectic complex manifold by Definition \ref{d:symp0} b).
\end{proof}

\begin{proposition}
 \label{p:sympdim}
 Let $X$ be a symplectic complex space.
 \begin{enumerate}
  \item Let $p\in X$. Then $\dim_p(X)$ is an even natural number.
  \item When $X$ is nonempty and finite dimensional, then $\dim(X)$ is an even natural number.
 \end{enumerate}
\end{proposition}

\begin{proof}
 a). As $X$ is symplectic, $X$ is normal, whence locally pure dimensional by Proposition \ref{p:normalpuredim}. So, there exists a neighbohood $U$ of $p$ in $X$ such that, for all $x\in U$, we have $\dim_x(X) = \dim_p(X)$. As $X$ is reduced, there exists $q \in U\cap X_\reg$. Since $X$ is symplectic, there exists a symplectic structure $\sigma$ on $X$. By Definition \ref{d:symp} a), the pullback of $\sigma$ to $X_\reg$ is a symplectic structure on $X_\reg$ in the sense of Definition \ref{d:symp0} a). Specifically, by Corollary \ref{c:nondegdim}, $\dim_q(X_\reg)$ is an even natural number. Since $\dim_q(X_\reg) = \dim_q(X) = \dim_p(X)$, we obtain our claim.

 b). When $X$ is nonempty and of finite dimension, there exists $p\in X$ such that $\dim_p(X) = \dim(X)$. Hence $\dim(X)$ is an even natural number by a).
\end{proof}

We would like to get a somewhat better understanding of conditions \eqref{d:symp-str-mfd} and \eqref{d:symp-str-ext} of Definition \ref{d:symp} \ref{d:symp-str}) for an element $\sigma$ of $\Omega^2_X(X_\reg)$ to be a symplectic structure on a complex space $X$. Concerning condition \eqref{d:symp-str-ext}, we recall below a result of H.\ Flenner on the existence of extensions of $p$\hyphen differentials with respect to resolutions of singularities which in turn yields a criterion for condition \eqref{d:symp-str-ext} of Definition \ref{d:symp} \ref{d:symp-str}) to come for free.

\begin{theorem}
 \label{t:flenner}
 Let $X$ be a normal complex space, $f\colon W\to X$ a resolution of singularities, $p\in\N$ such that $p+1<\codim(\Sing(X),X)$, and $\beta \in \Omega^p_X(X_\reg)$. Then there exists $\alpha$ such that $\alpha$ is an extension as $p$\hyphen differential of $\beta$ with respect to $f$.
\end{theorem}

\begin{proof}
 This follows from \cite[Theorem]{Fl88} by working locally on $X$.
\end{proof}

\begin{corollary}
 \label{c:flennersymp}
 Let $X$ be a normal complex space and $\sigma \in \Omega^2_X(X_\reg)$. Assume that $\codim(\Sing(X),X) \geq 4$. Then condition \eqref{d:symp-str-ext} of Definition \ref{d:symp} \ref{d:symp-str}) holds.
\end{corollary}

\begin{proof}
 The assertion is an immediate consequence of Theorem \ref{t:flenner} taking into account that $2+1 < 4 \leq \codim(\Sing(X),X)$.
\end{proof}

The upshot of Corollary \ref{c:flennersymp} is that for normal complex spaces $X$ with singular loci of codimension $\geq 4$, symplectic structures on $X$ are nothing but symplectic structures (in the sense of Definition \ref{d:symp0}) on the regular locus $X_\reg$ of $X$ via the canonical mapping $\Omega^2_X(X_\reg) \to \Omega^2_{X_\reg}(X_\reg)$, which is of course a bijection.

Quite generally, when the singularities of a complex space are ``mild'', one might expect $p$\hyphen differentials to extend with respect to resolutions of singularities. In Theorem \ref{t:flenner} the mildness of the singularities of the complex space $X$ comes (next to the normality of $X$) from the codimension of the singular locus. We would like to hint at another form of mildness of singularities which plays a role in the theory of symplectic spaces due to works of A.\ Beauville and Y.\ Namikawa.

\begin{proposition}
 \label{p:ratgor}
 Let $X$ be a symplectic complex space. Then $X$ is Gorenstein and has rational singularities.
\end{proposition}

\begin{proof}
 This follows from \cite[Proposition 1.3]{Be00}.
\end{proof}

Inspired by Proposition \ref{p:ratgor} one might ask the following:

\begin{question}
 \label{q:ratgorext}
 Let $X$ be a Gorenstein complex space which has rational singularities. Is it true then that, for all $\sigma \in \Omega^2_X(X_\reg)$ (\resp all $\sigma \in \Omega^2_X(X_\reg)$ such that \eqref{d:symp-str-mfd} of Definition \ref{d:symp} \ref{d:symp-str}) holds) and all resolutions of singularities $f\colon W\to X$, there exists an extension of $\sigma$ with respect to $f$?
\end{question}

If the answer to (any of the two versions of) Question \ref{q:ratgorext} were positive, a complex space $X$ would be symplectic if and only if it was Gorenstein, had rational singularities, and $X_\reg$ was a symplectic complex manifold. As it turns out, Y.\ Namikawa was able to give a partial (positive) answer to Question \ref{q:ratgorext}, namely:

\begin{theorem}
 \label{t:namikawaext}
 Let $X$ be a projective, Gorenstein complex space having rational singularities, $f\colon W\to X$ a resolution of singularities, and $\sigma \in \Omega^2_X(X_\reg)$. Then there exists an extension as $2$\hyphen differential of $\sigma$ with respect to $f$.
\end{theorem}

\begin{proof}
 This is a consequence of \cite[Theorem 4]{Na01a}.
\end{proof}

We move on to the investigation of condition (i) of Definition \ref{d:symp} a). Looking at Definition \ref{d:symp0} a), we see that symplecticity is made up of two components, namely nondegeneracy and closedness. We observe that for spaces of Fujiki class $\sC$ for which extension of $2$-differentials holds the closedness part is automatic.

\begin{proposition}
 \label{p:sympclexassc}
 Let $X$ be a complex space of Fujiki class $\sC$ and $\sigma \in \Omega^2_X(X_\reg)$. Assume that condition \eqref{d:symp-str-ext} of Definition \ref{d:symp} \ref{d:symp-str}) holds (for $X$ and $\sigma$). Then $\sigma$ is sent to the zero of $\Omega^3_X(X_\reg)$ by the differential $\dd^2_X \colon \Omega^2_X \to \Omega^3_X$.
\end{proposition}

\begin{proof}
 As $X$ is of Fujiki class $\sC$, there exists a proper modification $f\colon W\to X$ such that $W$ is a compact complex manifold of Kähler type. As condition \eqref{d:symp-str-ext} of Definition \ref{d:symp} \ref{d:symp-str}) holds, there exists $\rho \in \Omega^2_W(W)$ restricting to the pullback $\sigma'$ of $\sigma$ within the presheaf $\Omega^2_W$. As $W$ is a compact complex manifold of Kähler type, $(\dd^2_W)_W(\rho) = 0$ in $\Omega^3_W(W)$. Hence, $(\dd^2_W)_{f^{-1}(X_\reg)}(\sigma') = 0$ in $\Omega^3_W(f^{-1}(X_\reg))$. As the pulling back of differential forms commutes with the respective algebraic de Rham differentials, we see that $(\dd^2_X)_{X_\reg}(\sigma)$ is mapped to $0$ by the pullback function $\Omega^3_X(X_\reg) \to \Omega^3_W(f^{-1}(X_\reg))$. As the latter function is one-to-one, we infer that $(\dd^2_X)_{X_\reg}(\sigma) = 0$ in $\Omega^3_{X}(X_\reg)$.
\end{proof}

\begin{proposition}
 \label{p:symph2}
 Let $X$ be a symplectic complex space of Fujiki class $\sC$. Then the mixed Hodge structure $\H^2(X)$ is pure of weight $2$. In particular, we have
 \begin{equation}
  \label{e:symph2}
  \H^2(X,\C) = \H^{0,2}(X)\oplus \H^{1,1}(X)\oplus \H^{2,0}(X),
 \end{equation}
 where $\H^{p,q}(X) := \F^p\H^2(X)\cap \oF^q\H^2(X)$.
\end{proposition}

\begin{proof}
 By Proposition \ref{p:ratgor}, $X$ has rational singularities. Therefore, the mixed Hodge structure $\H^2(X)$ is pure of weight $2$ by Corollary \ref{c:h2pure}. \eqref{e:symph2} is a formal consequence of the purity of the mixed Hodge structure $\H^2(X)$ given that $\H^2(X)_\C = \H^2(X,\C)$ (by definition of $\H^2(X)$).
\end{proof}

When $X$ is a (generically) symplectic complex manifold, every (generically) symplectic structure on $X$ gives naturally rise to an element of $\H^2(X,\C)$ (via de Rham cohomology). Such an element is what we have decided on calling a ``(generically) symplectic class'', \cf Definition \ref{d:symp0} c). Hence, when $X$ is a symplectic complex space, every symplectic structure $\sigma$ on $X$ gives naturally rise to an element of $\H^2(X_\reg,\C)$ since $\sigma$ is mapped to a symplectic structure on the complex manifold $X_\reg$ by the evident function $\Omega^2_X(X_\reg) \to \Omega^2_{X_\reg}(X_\reg)$. However, this is somewhat unsatisfactory as, for reasons that will become clear in \S\ref{s:bbform}, we would like $\sigma$ to already correspond to an element in $\H^2(X,\C)$ rather than only to an element in $\H^2(X_\reg,\C)$---in the sense that any element of $\H^2(X,\C)$ automatically procures an element of $\H^2(X_\reg,\C)$ via the function
\[
 i^*\colon \H^2(X,\C) \to \H^2(X_\reg,\C)
\]
which is induced by the inclusion $i\colon X_\reg \to X$. A priori it is not clear whether there exists (a unique) $w$ such that $w$ is sent to the class of $\sigma$ in $\H^2(X_\reg,\C)$ by $i^*$. This observation motivates:

\begin{definition}[Symplectic classes]
 \label{d:sympcl}
 Let $X$ be a symplectic complex space.
 \begin{enumerate}
  \item Let $\sigma$ be a symplectic structure on $X$. $w$ is called \emph{symplectic class} of $\sigma$ on $X$ when, for all resolutions of singularities $f\colon W\to X$, the function
  \[
   f^*\colon \H^2(X,\C) \to \H^2(W,\C)
  \]
  induced by $f$ maps $w$ to the class of the extension of $\sigma$ with respect to $f$. Observe that it makes sense to speak about ``the class of the extension of $\sigma$'' here since by Proposition \ref{p:sympres} we have $(\dd^2_W)_W(\rho)=0$ in $\Omega^3_W(W)$ when $\rho$ denotes the extension of $\sigma$ with respect to $f$.
  \item $w$ is called \emph{symplectic class} on $X$ when there exists a symplectic structure $\sigma$ on $X$ such that $w$ is the symplectic class of $\sigma$ on $X$.
 \end{enumerate} 
\end{definition}

\begin{proposition}
 \label{p:sympclbis}
 Let $X$ be a symplectic complex space and $\sigma$ a symplectic structure on $X$. Then, for all $w$, the following are equivalent:
 \begin{enumeratei}
  \item $w$ is a symplectic class of $\sigma$ on $X$;
  \item there exists a resolution of singularities $f\colon W\to X$ such that $f^*(w)$ is the class of an extension of $\sigma$ with respect to $f$.
 \end{enumeratei}
\end{proposition}

\begin{proof}
 Since the complex space $X$ is reduced, there exists a resolution of singularities $f\colon W\to X$. Hence (i) implies (ii). Now suppose that (ii) holds. Let $f'\colon W'\to X$ be a resolution of singularities. Then there exist a complex space $V$ as well as two morphisms of complex spaces $g\colon V\to W$ and $g'\colon V\to W'$ such that $g$ and $g'$ both are resolutions of singularities and $f\circ g = f'\circ g'=:h$. By assumption there exists a closed global $2$\hyphen differential $\rho$ on $W$ such that $f^*(w)$ is the class of $\rho$ and $\rho$ is an extension of $\sigma$ with respect to $f$. Hence the image $\pi$ of $\rho$ under the canonical function $\Omega^2_W(W) \to \Omega^2_V(V)$ is an extension of $\sigma$ with respect to $h$. Similarly, as $\sigma$ is a symplectic structure on $X$, there exists $\rho'$ such that $\rho'$ is an extension of $\sigma$ with respect to $f'$. By Proposition \ref{p:sympres}, $\rho'$ is closed. Moreover, the image $\pi'$ of $\rho'$ under the canonical mapping $\Omega^2_{W'}(W') \to \Omega^2_V(V)$ is an extension of $\sigma$ with respect to $h$. As $h$ is a resolution of singularities, we see that $\pi = \pi'$. Denoting by $v$, $v'$, and $u$ the class of $\rho$, $\rho'$, and $\pi$, respectively, we obtain:
 \[
  g'^*(f'^*(w)) = (f'\circ g')^*(w) = (f\circ g)^*(w) = g^*(f^*(w)) = g^*(v) = u = g'^*(v').
 \]
 Given that the function $g'^*\colon \H^2(W',\C) \to \H^2(V,\C)$ is one-to-one, we infer that $f'^*(w) = v'$. As $f'$ was an arbitrary resolution of singularities of $X$, this shows that $w$ is a symplectic class of $\sigma$ on $X$, \iev (i).
\end{proof}

\begin{remark}
 \label{r:sympcl}
 Let $X$ be a symplectic complex space, $w$ a symplectic class on $X$, and $f\colon W\to X$ a resolution of singularities. Then $f^*(w)$ is a generically symplectic class on $W$. This is because by Definition \ref{d:sympcl} there exists a symplectic structure $\sigma$ on $X$ such that $f^*(w)$ is the class of $\rho$ on $W$, $\rho$ denoting the extension of $\sigma$ with respect to $f$, and by Proposition \ref{p:sympres}, $\rho$ is a generically symplectic structure on $W$.
\end{remark}

\begin{proposition}
 \label{p:sympclex}
 Let $X$ be a symplectic complex space of Fujiki class $\sC$.
 \begin{enumerate}
  \item For all symplectic structures $\sigma$ on $X$ there exists one, and only one, $w$ such that $w$ is a symplectic class of $\sigma$ on $X$.
  \item There exists $w$ which is a symplectic class on $X$.
 \end{enumerate}
\end{proposition}

\begin{proof}
 a). Let $\sigma$ be a symplectic structure on $X$. There exists a resolution of singularities $f\colon W\to X$. As $\sigma$ is symplectic structure on $X$, there exists $\rho$ such that $\rho$ is an extension of $\sigma$ with respect to $f$, \cf Definition \ref{d:symp} a), condition (ii). By Proposition \ref{p:sympres}, $\rho$ is a closed Kähler $2$\hyphen differential on $W$. Define $v$ to be the class of $\rho$. Then $v \in \F^2\H^2(W)$; note that $W$ is of Fujiki class $\sC$ so that it makes sense to speak of $\F^2\H^2(W)$ in the first place. By Proposition \ref{p:ratgor}, the complex space $X$ has rational singularities. Thus by Proposition \ref{p:f2h2}, the function
 \[
  f^*\colon \H^2(X,\C) \to \H^2(W,\C)
 \]
 induces a bijection
 \[
  f^*|\F^2\H^2(X) \colon \F^2\H^2(X) \to \F^2\H^2(W).
 \]
 In particular, there exists $w$ such that $f^*(w) = v$. Therefore, employing Proposition \ref{p:sympclbis}, we see that $w$ is a symplectic class of $\sigma$ on $X$. To see that $w$ is unique, let $w'$ be another symplectic class of $\sigma$ on $X$. Then $f^*(w') = v$ since the extension of $\sigma$ with respect to $f$ is unique. From this it follows that $w'=w$ as, by Proposition \ref{p:resh2inj}, the function $f^*$ is one-to-one.

 b). As $X$ is symplectic, there exists a symplectic structure on $X$, whence the assertion is a consequence of a).
\end{proof}

\section{The Beauville-Bogomolov form}
\label{s:bbform}

In \cite[p.~772]{Be83}, A.\ Beauville introduced a certain complex quadratic form on the complex vector space $\H^2(X,\C)$, where $X$ is an irreducible symplectic complex manifold. This quadratic form is nowadays customarily called the Beauville-Bogomolov form of $X$ (\cf \cite[Abschnitt 1.9]{Hu99} for instance). In what follows, we generalize the concept of the Beauville-Bogomolov form in two directions, namely that of compact, connected
\begin{enumeratei}
 \item generically symplectic complex manifolds and
 \item symplectic complex spaces,
\end{enumeratei}
where the symplectic structures are, in both cases, unique up to scaling. We would like to point out that throughout Chapter \ref{ch:symp} we aim to study potentially singular symplectic complex spaces. In that respect, we view the concept of generically symplectic complex manifolds, and thus generalization (i) above, as an auxiliary tool. We will not revisit the notion of ``generic symplecticity'' in later sections.

\begin{notation}
 \label{not:bbform}
 Let $(X,w)$ be an ordered pair such that $X$ is a compact, irreducible reduced complex space of strictly positive, even dimension and $w$ is an element of $\H^2(X,\C)$. Then we write $q_{(X,w)}$ for the unique function from $\H^2(X,\C)$ to $\C$ such that, for all $a\in\H^2(X,\C)$, we have\footnote{We slightly deviate from Beauville's original formula by writing $w^{r-1}\bar{w}^{r-1}$ instead of $(w\bar w)^{r-1}$, \cf \cite[p.~772]{Be83}, as we feel this is more natural to work with in calculations.}:
 \begin{equation} \label{e:bbform}
  q_{(X,w)}(a) := \frac{r}{2}\int_X\left(w^{r-1}\bar{w}^{r-1}a^2\right)+(r-1)\int_X\left(w^{r-1}\bar w^ra\right)\int_X\left(w^r\bar w^{r-1}a\right),
 \end{equation}
 where $r$ denotes the unique natural number such that $2r=\dim(X)$.
\end{notation}

\begin{remark}
 \label{r:bbformscale}
 Let $(X,w)$ be as in Notation \ref{not:bbform} and $\mu$ a complex number of absolute value $1$, \iev $|\mu|^2=\mu\bar\mu=1$. Then $q_{(X,\mu w)} = q_{(X,w)}$ as our readers will readily deduce from formula \eqref{e:bbform}.
\end{remark}

\begin{lemma}
 \label{l:integral}
 Let $n$ be an even natural number, $X$ a pure $n$\hyphen dimensional complex manifold, and $\alpha \in \A^{n,0}_\cp(X)$. Then the complex number $I:=\int_X(\alpha\wedge\bar\alpha)$ is real and $\geq 0$, and we have $I=0$ if and only if $\alpha$ is the trivial differential $n$-form on $X$.
\end{lemma}

\begin{proof}
 As $n$ is even, there is a natural number $r$ such that $2r=n$. Let $z\colon U\to \C^n$ be a (holomorphic) chart on $X$. Then there exists a $\sC^\infty$ function $f\colon U\to \C$ such that
 \[
  \alpha|U = f\cdot \dd z^1 \wedge \dots \wedge \dd z^n.
 \]
 Therefore:
 \begin{align*}
  (\alpha \wedge \bar\alpha)|U & = f\bar f \cdot \dd z^1 \wedge \dots \wedge \dd z^n \wedge \dd\bar z^1 \wedge \dots \wedge \dd\bar z^n \\
  & = |f|^2(-1)^{\frac{n(n-1)}{2}} \cdot (\dd z^1 \wedge \dd\bar z^1) \wedge \dots \wedge (\dd z^n \wedge \dd\bar z^n) \\
  & = |f|^2(-1)^{\frac{n(n-1)}{2}}\frac{1}{(-2i)^n} \cdot (\dd x^1 \wedge \dd y^1) \wedge \dots \wedge (\dd x^n \wedge \dd y^n) \\
  & = |f|^2\frac{1}{4^r} \cdot (\dd x^1 \wedge \dd y^1) \wedge \dots \wedge (\dd x^n \wedge \dd y^n),
 \end{align*}
 where $z^i = (x^i,y^i)$ for $i=1,\dots,n$. This calculation shows that the $2n$-form $\alpha\wedge\bar\alpha$ is real and nonnegative (with respect to the canonical orientation of $X$) at every point of $X$. Hence $I$ is real and $\geq 0$.

 Assume that $I=0$. Then $\alpha \wedge \bar\alpha$ has to be trivial at each point of $X$ (otherwise $\alpha \wedge \bar\alpha$ would be strictly positive on a nonempty open subset of $X$, which would imply $I>0$). By the above calculation, $\alpha \wedge \bar\alpha$ is trivial at a point of $X$ (if and) only if $\alpha$ is trivial at that point. So, $I=0$ implies that $\alpha$ is trivial. On the other hand, clearly, when $\alpha$ is trivial, then $I=0$.
\end{proof}

\begin{proposition}
 \label{p:sympvol+}
 Let $X$ be a nonempty, compact, connected complex manifold and $w$ a generically symplectic class on $X$. Then the complex number $\int_X\left(w^r\bar w^r\right)$, where $r$ denotes half the dimension of $X$, is real and strictly positive.
\end{proposition}

\begin{proof}
 As $w$ is a generically symplectic class on $X$, there exists a generically symplectic structure $\sigma$ on $X$ such that $w$ is the class of $\sigma$ (see Definition \ref{d:symp0}). Abusing notation, we symbolize the image of $\sigma$ under the canonical mapping $\Omega^2_X(X) \to \A^{2,0}(X)$ again by $\sigma$. Set $n:=\dim(X)$ and $\alpha := \sigma^{\wedge r}$. Then $\alpha \in \A^{2r,0}(X) = \A^{n,0}_\cp(X)$, and $\alpha$ is not the trivial $n$-form on $X$ as $\sigma$ is generically nondegenerate on $X$, \cf Proposition \ref{p:nondeg}. Moreover, $n$ is even and $X$ is pure $n$\hyphen dimensional (specifically as $X$ is connected). Thus applying Lemma \ref{l:integral}, we see that $\int_X(\alpha \wedge \bar\alpha) > 0$. By the definition of the integral on cohomology,
 \[
  \int_X\left(w^r\bar w^r\right) = \int_X(\alpha\wedge\bar\alpha)
 \]
 since $w^r\bar w^r$ is clearly the class of $\alpha\wedge\bar\alpha$ in $\H^{2n}(X,\C)$.
\end{proof}

\begin{definition}[Normed classes, I]
 \label{d:normedcl+}
 Let $X$ be a nonempty, compact, connected complex manifold. $w$ is called \emph{normed generically symplectic class} on $X$ when $w$ is a generically symplectic class on $X$ such that
 \[
  \int_X w^r\bar w^r = 1,
 \]
 where $r$ denotes half the dimension of $X$.
\end{definition}

\begin{remark}
 \label{r:normedcl+}
 Let $X$ be a nonempty, compact, connected, generically symplectic complex manifold. Then, according to Definition \ref{d:symp0}, there exists a generically symplectic structure $\sigma_0$ on $X$. Let $w_0$ be the class of $\sigma_0$. Then $w_0$ is a generically symplectic class on $X$, and, by Proposition \ref{p:sympvol+}, the complex number $I:=\int_X\left(w_0^r\bar{w_0}^r\right)$, where $r:=\nicefrac{1}{2}\dim(X)$, is real an strictly positive. Denote by $\lambda$ the ordinary (positive) $2r$-th real root of $I$. Set $w:=\lambda^{-1}w_0$. Then $w$ is the class of a generically symplectic structure on $X$, namely of $\lambda^{-1}\sigma_0$, and we have: 
 \begin{align*}
  \int_X w^r\bar w^r & = \int_X (\lambda^{-1}w_0)^r(\bar{\lambda^{-1}w_0})^r = \int_X(\lambda^{-1}w_0)^r(\lambda^{-1}\bar{w_0})^r = (\lambda^{2r})^{-1}\int_X w_0^r\bar{w_0}^r \\
  & = I^{-1}I = 1.
 \end{align*}
 The upshot is that, for all $X$ as above, there exists a normed generically symplectic class on $X$. In fact, we can always rescale a given generically symplectic class (by a strictly positive real number) to procure a normed generically symplectic class. 
\end{remark}

\begin{definition}[Beauville-Bogomolov form, I]
 \label{d:bbformgen}
 Let $X$ be a compact, connected, generically symplectic complex manifold such that $\dim_\C(\Omega^2_X(X))=1$. We claim there exists a unique function
 \[
  q \colon \H^2(X,\C) \to \C
 \]
 such that, for all normed generically symplectic classes $w$ on $X$, we have
 \[
  q_{(X,w)} = q.
 \]
 Note that the expression ``$q_{(X,w)}$'' makes sense here since $X$ is a compact, connected complex manifold whose dimension is a strictly positive, even natural number and $w\in\H^2(X,\C)$, \cf Notation \ref{not:bbform}.
 
 In fact, by Remark \ref{r:normedcl+}, there exists a normed generically symplectic class $w_1$ on $X$. Obviously, $q_{(X,w_1)}$ is a function from $\H^2(X,\C)$ to $\C$. Let $w$ be any normed generically symplectic class on $X$. Then there is a generically symplectic structure $\sigma$ on $X$ such that $w$ is the class of $\sigma$. Besides, there is a generically symplectic structure $\sigma_1$ on $X$ such that $w_1$ is the class of $\sigma_1$. Now as $\dim_\C(\Omega^2_X(X))=1$, the dimension of $X$ is strictly positive so that $\sigma_1 \neq 0$ in $\Omega^2_X(X)$. Thus $\sigma_1$ generates $\Omega^2_X(X)$ as a complex vector space. In particular, there exists a complex number $\mu$ such that $\sigma = \mu\sigma_1$. It follows that $w = \mu w_1$ and in turn, setting $r:=\nicefrac{1}{2}\dim(X)$,
 \[
  |\mu|^{2r} = |\mu|^{2r}\int_X w_1^r\bar{w_1}^r = \int_X (\mu w_1)^r(\bar{\mu w_1})^r = \int_X w^r\bar w^r = 1.
 \]
 As $2r$ is a natural number $\neq 0$, we infer that $|\mu|=1$. Thus
 \[
  q_{(X,w)} = q_{(X,\mu w_1)} = q_{(X,w_1)}.
 \]
 This proves, on the one hand, the existence of $q$. On the other hand, the uniqueness of $q$ is evident by the fact that there exists a normed generically symplectic class $w_1$ on $X$ (any $q$ has to agree with $q_{(X,w_1)}$).
 
 In what follows, we refer to the unique $q$ satisfying the condition stated above as the \emph{Beauville-Bogomolov form} of $X$. The Beauville-Bogomolov form of $X$ will be denoted $q_X$.
\end{definition}

\begin{lemma}
 \label{l:integralmod}
 Let $n$ be a natural number and $f\colon W\to X$ a proper modification such that $W$ and $X$ are reduced complex spaces of pure dimension $n$. Then we have
 \begin{equation} \label{e:integralmod}
  \int_W f^*(c) = \int_X c
 \end{equation}
 for all $c\in \H^{2n}_\cp(X,\C)$, where
 \[
  f^* \colon \H^{2n}_\cp(X,\C) \to \H^{2n}_\cp(W,\C)
 \]
 signifies the morphism induced by $f$ on complex cohomology with support.
\end{lemma}

\begin{proof}
 Since $f\colon W\to X$ is a proper modification, there exist thin, closed analytic subsets $A$ and $B$ of $W$ and $X$, respectively, such that $f$ induces an isomorphism of complex spaces
 \[
  W\setminus A \to X\setminus B
 \]
 by restriction. Define
 \[
  X' := X \setminus (B \cup \Sing(X)) \quad \text{and} \quad W' := W|(f^{-1}(X') \setminus A),
 \]
 write
 \[
  i \colon W' \to W \quad \text{and} \quad j \colon X' \to X
 \]
 for the canonical morphisms and
 \[
  f' \colon W' \to X'
 \]
 for the restriction of $f$. Then the diagram
 \[
  \xysquare{W'}{W}{X'}{X}{i}{f'}{f}{j}
 \]
 commutes in the category of complex spaces, whence the diagram
 \[
  \xysquare{\H^{2n}_\cp(X',\C)}{\H^{2n}_\cp(X,\C)}{\H^{2n}_\cp(W',\C)}{\H^{2n}_\cp(W,\C)}{j_*}{f'^*}{f^*}{i_*}
 \]
 commutes in the category of complex vector spaces. Observe that $X' \subset X_\reg$. So, $j_*$ factors as follows:
 \[
  \H^{2n}_\cp(X',\C) \xrightarrow{(j_0)_*} \H^{2n}_\cp(X_\reg,\C) \xrightarrow{(j_1)_*} \H^{2n}_\cp(X,\C).
 \]
 Realizing the pushforward morphism
 \[
  (j_0)_* \colon \H^{2n}_\cp(X',\C) \to \H^{2n}_\cp(X_\reg,\C)
 \]
 as well as the integrals on $X'$ and $X_\reg$ by means of $\sC^\infty$ differential $2n$-forms (via de Rham's theorem), we obtain that
 \[
  \int_{X'} c' = \int_{X_\reg} (j_0)_*(c') = \int_X (j_1)_*((j_0)_*(c')) = \int_X j_*(c')
 \]
 for all $c' \in \H^{2n}_\cp(X',\C)$. Similarly, we have
 \[
  \int_{W'} b' = \int_W i_*(b')
 \]
 for all $b' \in \H^{2n}_\cp(W',\C)$. Furthermore, we have
 \[
  \int_{W'} f'^*(c') = \int_{X'} c'
 \]
 for all $c' \in \H^{2n}_\cp(X',\C)$ since $f'$ is an isomorphism. By the long exact sequence in complex cohomology with compact support, we see that there exists an exact sequence of complex vector spaces
 \[
  \H^{2n}_\cp(X',\C) \overset{j_*}\to \H^{2n}_\cp(X,\C) \to \H^{2n}_\cp(B \cup \Sing(X),\C).
 \]
 Since $X$ is reduced and pure dimensional and $B$ is thin in $X$, we have
 \[
  \dim(B \cup \Sing(X)) < \dim(X) = n
 \]
 and thus
 \[
  \H^{2n}_\cp(B \cup \Sing(X),\C) \iso 0.
 \]
 Now let $c \in \H^{2n}_\cp(X,\C)$. Then there exists an element $c'$ of $\H^{2n}_\cp(X',\C)$ such that $j_*(c') = c$, and we deduce \eqref{e:integralmod} from the already established identities.
\end{proof}

\begin{proposition}
 \label{p:bbformmod}
 Let $X$ be a nonempty, compact, connected, generically symplectic complex manifold and $f\colon X'\to X$ a proper modification such that $X'$ is a complex manifold.
 \begin{enumerate}
  \item \label{p:bbformmod-mfd} $X'$ is a nonempty, compact, connected, generically symplectic complex manifold.
  \item \label{p:bbformmod-class} When $w$ is a normed generically symplectic class on $X$, then $f^*(w)$ is a normed generically symplectic class on $W$.
  \item \label{p:bbformmod-bbform} When $\dim_\C(\Omega^2_X(X))=1$, then $\dim_\C(\Omega^2_{X'}(X'))=1$, and we have
  \[
   q_X = q_{X'} \circ f^*,
  \]
  where
  \[
   f^* \colon \H^2(X,\C) \to \H^2(X',\C)
  \]
  signifies the morphism induced by $f$ on second complex cohomology.
 \end{enumerate}
\end{proposition}

\begin{proof}
 \ref{p:bbformmod-mfd}). Clearly, $X'$ is nonempty, compact, and connected. $X'$ is generically symplectic by means of Remark \ref{r:symp0}.

 \ref{p:bbformmod-class}). Let $w$ be a normed generically symplectic class on $X$. Then $f^*(w)$ is a generically symplectic class on $X'$ by Remark \ref{r:symp0}. $f^*(w)$ is a normed generically symplectic class on $X'$ since by Lemma \ref{l:integralmod} we have:
 \[
  \int_{X'}f^*(w)^r(\bar{f^*(w)})^r = \int_{X'} f^*(w^r\bar w^r) = \int_X w^r\bar w^r = 1.
 \]

 \ref{p:bbformmod-bbform}). As $f$ is a proper modification between complex manifolds, the pullback of Kähler differentials
 \[
  \Omega^2_X(X) \to \Omega^2_{X'}(X')
 \]
 induced by $f$ furnishes an isomorphism of complex vector spaces. Specifically, when $\dim_\C(\Omega^2_X(X))=1$, then $\dim_\C(\Omega^2_{X'}(X'))=1$.

 Looking at formula \eqref{e:bbform}, Lemma \ref{l:integralmod} implies that, for any $w\in\H^2(X,\C)$, we have
 \[
  q_{(X,w)} = q_{(X',f^*(w))} \circ f^*;
 \]
 observe that $f^*$ is compatible with the respective multiplications and conjugations on $\H^*(X,\C)$ and $\H^*(X',\C)$ and that $\dim(X') = \dim(X)$. By Remark \ref{r:normedcl+}, there exists a normed generically symplectic class $w$ on $X$. Thus we deduce
 \[
  q_X = q_{(X,w)} = q_{(X',f^*(w))} \circ f^* = q_{X'} \circ f^*
 \]
 from \ref{p:bbformmod-class}) recalling Definition \ref{d:bbformgen}.
\end{proof}

\begin{proposition}
 \label{p:unisympres}
 Let $X$ be a (compact, connected) symplectic complex space such that $\dim_\C(\Omega^2_X(X_\reg))=1$. Then, for all resolutions of singularities $f\colon W\to X$, the space $W$ is a  (compact, connected) generically symplectic complex manifold with $\Omega^2_W(W)$ of dimension $1$ over the field of complex numbers.
\end{proposition}

\begin{proof}
 Let $f\colon W\to X$ be a resolution of singularities. Then $W$ is a generically symplectic complex manifold according to Proposition \ref{p:sympres}. The restriction mapping $\Omega^2_W(W) \to \Omega^2_W(f^{-1}(X_\reg))$ is surely one-to-one. The pullback function $\Omega^2_X(X_\reg) \to \Omega^2_W(f^{-1}(X_\reg))$ is a bijection. As the complex space $X$ is symplectic, there exists a symplectic structure $\sigma$ on $X$. Since $\Omega^2_X(X_\reg)$ is $1$-dimensional, $\sigma$ generates $\Omega^2_X(X_\reg)$ as a complex vector space. Thus as $\sigma$ has an extension $\rho$ with respect to $f$, we see that the restriction mapping $\Omega^2_W(W) \to \Omega^2_W(f^{-1}(X_\reg))$ is, in addition to being one-to-one, onto. Therefore we have $\dim_\C(\Omega^2_W(W))=1$. Of course, when $X$ is compact and connected, $W$ is compact and conected.
\end{proof}

\begin{definition}[Beauville-Bogomolov form, II]
 \label{d:bbform}
 Let $X$ be a compact, and connected, and symplectic complex space such that $\dim_\C(\Omega^2_X(X_\reg))=1$. We claim there exists a unique function
 \[
  q \colon \H^2(X,\C) \to \C
 \]
 such that, for all resolutions of singularities $f\colon W\to X$, we have
 \[
  q = q_W \circ f^*,
 \]
 where
 \[
  f^* \colon \H^2(X,\C) \to \H^2(W,\C)
 \]
 signifies the pullback function induced by $f$ (or better, $f_\top$) on second complex cohomology. Note that it makes sense to write ``$q_W$'' above as by Proposition \ref{p:unisympres}, for all resolutions of singularities $f\colon W\to X$, the space $W$ is a compact, connected, generically symplectic complex manifold with $1$-dimensional $\Omega^2_W(W)$, whence Definition \ref{d:bbformgen} tells what is to be understood by the Beauville-Bogomolov form of $W$.
 
 As there exists a resolution of singularities $f\colon W\to X$, we see that $q$ is uniquely determined. Now for the existence of $q$ it suffices to show that for any two resolutions of singularities $f\colon W\to X$ and $f'\colon W'\to X$, we have
 \[
  q_W \circ f^* = q_{W'} \circ (f')^*.
 \]
 Given such $f$ and $f'$, there exist a complex manifold $V$ as well as proper modifications $g\colon V\to W$ and $g'\colon V\to W'$ such that the following diagram commutes in the category of complex spaces:
 \[
  \xysquare{V}{W}{W'}{X}{g}{g'}{f}{f'}
 \]
 Therefore, by Proposition \ref{p:bbformmod} \ref{p:bbformmod-bbform}), we have:
 \[
  q_W\circ f^* = q_V\circ g^*\circ f^* = q_V\circ (g')^*\circ (f')^* = q_{W'}\circ (f')^*.
 \]
 
 The unique $q$ satisfying the condition stated above will be called the \emph{Beauville-Bogomolov form} of $X$. We denote the Beauville-Bogomolov form of $X$ by $q_X$.
\end{definition}

\begin{remark}[Ambiguity]
 \label{r:bbform}
 In case $X$ is a compact, connected, symplectic complex manifold with $1$-dimensional $\Omega^2_X(X)$ both Definition \ref{d:bbformgen} and Definition \ref{d:bbform} are applicable in order to tell what the Beauville-Bogomolov form of $X$ is. Gladly, employing Proposition \ref{p:bbformmod}, one infers that the Beauville-Bogomolov form of $X$ in the sense of Definition \ref{d:bbformgen} satisfies the condition given for $q$ in Definition \ref{d:bbform}, hence is the Beauville-Bogomolov form of $X$ in the sense of Definition \ref{d:bbform}.
\end{remark}

Our philosophy in defining the Beauville-Bogomolov form on possibly singular complex spaces $X$ (\cf Definition \ref{d:bbform}) is to make use of the Beauville-Bogomolov form for generically symplectic complex manifolds (\cf Definition \ref{d:bbformgen}) together with a resolution of singularities.

An alternative approach might be to employ the formula \eqref{e:bbform} directly on $X$ rather than first passing to a resolution. Then, of course, for $w$ we should plug a (suitably normed) symplectic class on $X$ into \eqref{e:bbform}. Unfortunately, as we have already noticed in \S\ref{s:symp}, it is not clear whether on a given arbitrary compact, symplectic complex space $X$, there exists one (and only one) symplectic class for every symplectic structure $\sigma$ on $X$ (\cf Proposition \ref{p:sympclex}). Therefore we cannot pursue this alternative in general. However, if we are lucky and there do exist symplectic classes on $X$, calculating the Beauville-Bogomolov form on $X$ is as good as calculating it on a resolution. We briefly explain the details.

\begin{definition}[Normed classes, II]
 \label{d:normedcl}
 Let $X$ be a nonempty, compact, and connected, and symplectic complex space. Then $w$ is called \emph{normed symplectic class} on $X$ when $w$ is a symplectic class on $X$, \cf Definition \ref{d:sympcl}, and we have
 \[
  \int_X w^r\bar w^r = 1,
 \]
 where $r$ is short for half the dimension of $X$.
\end{definition}

\begin{proposition}
 \label{p:normedclres}
 Let $X$ be a nonempty, compact, connected, and symplectic complex space, $w$ a normed symplectic class on $X$, and $f\colon W\to X$ a resolution of singularities. Then $f^*(w)$ is a normed generically symplectic class on $W$.
\end{proposition}

\begin{proof}
 By Remark \ref{r:sympcl}, $f^*(w)$ is a generically symplectic class on $W$. Moreover, $W$ is a nonempty, compact, connected complex manifold. Set $r:=\nicefrac{1}{2}\dim(W)$. Then by means of Lemma \ref{l:integralmod}, we obtain:
 \[
  \int_W (f^*(w))^r(\bar{f^*(w)})^r = \int_W f^*(w^r\bar w^r) = \int_X w^r\bar w^r = 1.
 \]
 The very last equality holds since $w$ is a normed symplectic class on $X$ and we have $\nicefrac{1}{2}\dim(X)=\nicefrac{1}{2}\dim(W)=r$.
\end{proof}

\begin{proposition}
 \label{p:bbformres}
 Let $X$ be a compact, connected, symplectic complex space such that $\dim_\C(\Omega^2_X(X_\reg))=1$. Let $w$ be a normed symplectic class on $X$. Then $q_X = q_{(X,w)}$.
\end{proposition}

\begin{proof}
 There exists a resolution of singularities $f\colon \tilde X\to X$. Set $\tilde w:=f^*(w)$. By Definition \ref{d:bbform}, we have $q_X = q_{\tilde X} \circ f^*$. By Definition \ref{d:bbformgen}, we have $q_{\tilde X} = q_{(\tilde X,\tilde w)}$ since by Proposition \ref{p:normedclres}, $\tilde w$ is a normed generically symplectic class on $\tilde X$. Employing Lemma \ref{l:integralmod} (three times) as well as the fact that $\dim(X)=\dim(\tilde X)$, one easily deduces $q_{(X,w)} = q_{(\tilde X,\tilde w)} \circ f^*$ from \eqref{e:bbform}. So, $q_X = q_{(X,w)}$.
\end{proof}

Having introduced the notion of a Beauville-Bogomolov form for two (overlapping) classes of complex spaces, we are now going to establish, essentially in Proposition \ref{p:bbformhdg} and Proposition \ref{p:bbformtopint} below, two formulae for $q_X$. These formulae are classically due to Beauville, \cf \cite[Th\'eor\`eme 5, D\'emostration be (b)]{Be83}; the proofs are pretty much straightforward.

An essential point is that we require the complex spaces in question to be of Fujiki class $\sC$ so that their cohomologies carry mixed Hodge structures.

\begin{proposition}
 \label{p:bbformhdg}
 Let $X$ be a compact, connected, and generically symplectic complex manifold of Fujiki class $\sC$ such that $\Omega^2_X(X)$ is $1$-dimensional over $\C$. Let $w$ be a normed generically symplectic class on $X$, $c \in \H^{1,1}(X)$, and $\lambda,\lambda' \in \C$. Then, setting
 \[
  a := \lambda w + c + \lambda'\bar w,
 \]
 where we calculate in $\H^2(X,\C)$, and setting $r:=\nicefrac{1}{2}\dim(X)$, we have:
 \begin{equation}
  \label{e:bbformhdg}
  q_X(a) = \frac{r}{2}\int_X\left(w^{r-1}\bar w^{r-1}c^2\right) + \lambda\lambda'.
 \end{equation}
\end{proposition}

\begin{proof}
 As $w$ is the class of a (closed) holomorphic $2$-form on $X$, we have $w \in \F^2\H^2(X)$ by the definition of the Hodge structure $\H^2(X)$. Thus, for all $d \in \F^1\H^2(X)$, we have $w^rd \in \F^{2r+1}\H^{2r+2}(X)$ by the compatibility of the Hodge filtrations with the cup product on $\H^*(X,\C)$. Since $2r+1 > 2r = \dim(X)$, we know that $\F^{2r+1}\H^{2r+2}(X)$ is the trivial vector subspace of $\H^{2r+2}(X,\C)$. Hence, $w^rd=0$ in $\H^*(X,\C)$ for all $d \in \F^1\H^2(X)$. It follows that $\bar w^rd'=0$ in $\H^*(X,\C)$ for all $d' \in \oF^1\H^2(X)$. In particular, we have
 \[
  w^{r+1} = \bar w^{r+1} = w^rc = \bar w^rc = 0
 \]
 in $\H^*(X,\C)$ as $w \in \F^1\H^2(X)$, and $c \in \F^1\H^2(X) \cap \oF^1\H^2(X)$.

 Note that the subring $\H^{2*}(X,\C)$ of $\H^*(X,\C)$ is commutative. Exploiting the above vanishings, we obtain:
 \begin{align*}
  w^{r-1}\bar w^ra & = w^{r-1}\bar w^r(\lambda w+c+\lambda'\bar w) \\
  & = (w^{r-1}\bar w^r)(\lambda w) + (w^{r-1}\bar w^r)c + (w^{r-1}\bar w^r)(\lambda'\bar w) \\
  & = \lambda w^r\bar w^r + w^{r-1}(\bar w^rc) + \lambda'w^{r-1}\bar w^{r+1} \\
  & = \lambda w^r\bar w^r.
 \end{align*}
 That is,
 \begin{equation} \label{e:bbformhdg-1}
  \int_X\left(w^{r-1}\bar w^ra\right) = \int_X\left(\lambda w^r\bar w^r\right) = \lambda \int_X\left(w^r\bar w^r\right) = \lambda1 = \lambda,
 \end{equation}
 specifically since $\int_X w^r\bar w^r = 1$ given that $w$ is a normed generically symplectic class on $X$ (\cf Definition \ref{d:normedcl+}). Likewise, one shows that
 \begin{equation} \label{e:bbformhdg-2}
  \int_X\left(w^r\bar w^{r-1}a\right) = \lambda'.
 \end{equation}
 Further on, we calculate:
 \begin{align*}
  w^{r-1}\bar w^{r-1}a^2 & = w^{r-1}\bar w^{r-1}(\lambda w + c + \lambda'\bar w)^2 \\
  & = w^{r-1}\bar w^{r-1}\left((\lambda w)^2 + c^2 + (\lambda'\bar w)^2 + 2(\lambda w)b + 2(\lambda w)(\lambda'\bar w) + 2c(\lambda'\bar w)\right) \\
  & = \lambda^2(w^{r+1}\bar w^{r-1}) + w^{r-1}\bar w^{r-1}c^2 + \lambda'^2(w^{r-1}\bar w^{r+1}) + 2\lambda((w^rc)\bar w^{r-1}) \\
  & \qquad + 2\lambda\lambda'(w^r\bar w^r) + 2\lambda'(w^{r-1}\bar w^rc) \\
  & = w^{r-1}\bar w^{r-1}c^2 + 2\lambda\lambda'(w^r\bar w^r).
 \end{align*}
 That is:
 \begin{equation} \label{e:bbformhdg-3}
  \int_X w^{r-1}\bar w^{r-1}a^2 = \int_X\left(w^{r-1}\bar w^{r-1}c^2 + 2\lambda\lambda'(w^r\bar w^r)\right) = \int_X\left(w^{r-1}\bar w^{r-1}c^2\right) + 2\lambda\lambda'.
 \end{equation}
 By the definition of the Beauville-Bogomolov form of $X$ (\cf Definition \ref{d:bbformgen}), we have $q_X(a) = q_{(X,w)}(a)$. So, plugging identities \eqref{e:bbformhdg-1}, \eqref{e:bbformhdg-2}, and \eqref{e:bbformhdg-3} into formula \eqref{e:bbform}, we infer:
 \[
  q_X(a) = \frac{r}{2}\left(\int_X\left(w^{r-1}\bar w^{r-1}c^2\right)+2\lambda\lambda'\right) + (1-r)\lambda\lambda' = \frac{r}{2}\int_X\left(w^{r-1}\bar w^{r-1}c^2\right) + \lambda\lambda',
 \]
 which is nothing but \eqref{e:bbformhdg}.
\end{proof}

\begin{corollary}
 \label{c:bbformhdg}
 Let $X$ be a compact, connected, symplectic complex space of Fujiki class $\sC$ such that $\Omega^2_X(X_\reg)$ is of dimension $1$ over the field of complex numbers. Let $w$ a normed symplectic class on $X$, $c \in \H^{1,1}(X)$, and $\lambda,\lambda' \in \C$. Define $a$ and $r$ as before. Then \eqref{e:bbformhdg} holds.
\end{corollary}

\begin{proof}
 There exists a resolution of singularities $f\colon \tilde X\to X$. By Proposition \ref{p:unisympres}, $\tilde X$ is a compact, connected, generically symplectic complex manifold with $\Omega^2_{\tilde X}(\tilde X)$ of dimension $1$ over the field of complex numbers. As $X$ is of Fujiki class $\sC$, $\tilde X$ is of Fujiki class $\sC$. Set $\tilde w:=f^*(w)$ and $\tilde c:=f^*(c)$. Then $\tilde w$ is a normed generically symplectic class on $\tilde X$, and $\tilde c \in \H^{1,1}(\tilde X)$ since $f^*$ respects the Hodge filtrations. By Definition \ref{d:bbform}, we have $q_X = q_{\tilde X} \circ f^*$. Thus as $\nicefrac{1}{2}\dim(\tilde X) = \nicefrac{1}{2}\dim(X) = r$, we obtain, invoking Proposition \ref{p:bbformhdg} in particular:
 \begin{align*}
  q_X(a) & = q_{\tilde X}(f^*(a)) = q_{\tilde X}(f^*(\lambda w + c + \lambda'\bar w)) = q_{\tilde X}(\lambda\tilde w + \tilde c + \lambda'\bar{\tilde w}) \\ & = \frac{r}{2}\int_{\tilde X}\left({\tilde w}^{r-1}\bar{\tilde w}^{r-1}{\tilde c}^2\right) + \lambda\lambda' = \frac{r}{2}\int_X\left(w^{r-1}\bar w^{r-1}c^2\right) + \lambda\lambda';
 \end{align*}
 for the very last equality we use Lemma \ref{l:integralmod} and the fact that
 \[
  f^*(w^{r-1}\bar w^{r-1}c^2) = {\tilde w}^{r-1}\bar{\tilde w}^{r-1}{\tilde c}^2.
 \]
 Evidently, we have established \eqref{e:bbformhdg}.
\end{proof}

\begin{proposition}
 \label{p:bbformtopint}
 Let $X$ and $w$ be as in Proposition \ref{p:bbformhdg}. Furthermore, let $a \in \H^2(X,\C)$ and $\lambda \in \C$ such that $a^{(2,0)}=\lambda w$. Then, setting $r:=\nicefrac{1}{2}\dim(X)$, the following identity holds:
 \begin{equation}
  \label{e:bbformtopint}
  \int_X\left(a^{r+1}\bar w^{r-1}\right)=(r+1)\lambda^{r-1}q_X(a).
 \end{equation}
\end{proposition}

\begin{proof}
 We know that $w$ generates $\F^2\H^2(X)$ as a $\C$-vector space. Hence $\bar w$ generates $\oF^2\H^2(X)$ as a $\C$-vector space. Thus there exists a complex number $\lambda'$ such that $a^{(0,2)}=\lambda'\bar w$. Setting $b:=a^{(1,1)}$, we have $a = \lambda w + b + \lambda'\bar w$. As the subring $\H^{2*}(X,\C)$ of $\H^*(X,\C)$ is commutative, we may calculate as follows employing the ``trinomial formula'':
 \begin{align} \label{e:bbformtopint-1}
  a^{r+1}\bar w^{r-1} & = (\lambda w + b + \lambda'\bar w)^{r+1}\bar w^{r-1} \\
  & = \sum_{\substack{(i,j,k)\in\N^3 \\ i+j+k=r+1}}\binom{r+1}{i,j,k}(\lambda w)^ib^j(\lambda'\bar w)^k\bar w^{r-1}. \nonumber
 \end{align}
 Since the product on $\H^*(X,\C)$ is ``filtered'' with respect to the Hodge filtrations on the graded pieces, we have
 \[
  \bar w^{r+1} = b\bar w^r = b^3\bar w^{r-1} = w^{r+1} = w^rb = 0
 \]
 in $\H^*(X,\C)$. Therefore, $w^ib^j\bar w^{k+r-1}=0$ in $\H^*(X,\C)$ for all $(i,j,k)\in \N^3$ such that either $k>1$, or $k=1$ and $j>0$, or $k=0$ and $j>2$. Moreover, when $(i,j)\in\N^2$ such that $i+j=r+1$ and $j<2$, we have $w^ib^j=0$ in $\H^*(X,\C)$. Thus from \eqref{e:bbformtopint-1} we deduce:
 \begin{align*}
  a^{r+1}\bar w^{r-1} & =  \binom{r+1}{r-1,2,0}(\lambda w)^{r-1}b^2\bar w^{r-1} + \binom{r+1}{r,0,1}(\lambda w)^r(\lambda'\bar w)\bar w^{r-1} \\
  & = \frac{r(r+1)}{2}\lambda^{r-1}w^{r-1}\bar w^{r-1}b^2 + (r+1)\lambda^r\lambda' w^r\bar w^r  \\
  & = (r+1)\lambda^{r-1}\left(\frac{r}{2}(w^{r-1}\bar w^{r-1}b^2) + \lambda\lambda'(w^r\bar w^r)\right).
 \end{align*}
 So,
 \[
  \int_X\left(a^{r+1}\bar w^{r-1}\right) = (r+1)\lambda^{r-1}\left(\frac{r}{2}\int_X\left(w^{r-1}\bar w^{r-1}b^2\right) + \lambda\lambda'\right) = (r+1)\lambda^{r-1}q_X(a),
 \]
 where we eventually plug in $\int_X\left(w^r\bar w^r\right)=1$ as well as \eqref{e:bbformhdg}.
\end{proof}

\begin{corollary}
 \label{c:bbformtopint}
 Let $X$ be a compact, connected, symplectic complex space of Fujiki class $\sC$ such that $\dim_\C(\Omega^2_X(X_\reg))=1$. Furthermore, let $w$ be a normed symplectic class on $X$, $a\in\H^2(X,\C)$, and $\lambda\in\C$ such that $a^{(2,0)}=\lambda w$. Then, setting $r:=\nicefrac{1}{2}\dim(X)$, \eqref{e:bbformtopint} holds.
\end{corollary}

\begin{proof}
 There exists a resolution of singularities $f\colon \tilde X \to X$. Put $\tilde w:=f^*(w)$ and $\tilde a:=f^*(a)$, where $f^*$ denotes the pullback on second complex cohomology induced by $f$. As $f^*$ preserves Hodge types, we have
 \[
  \tilde a^{(2,0)}=f^*(a^{(2,0)})=f^*(\lambda w)=\lambda\tilde w.
 \]
 Moreover, $\tilde X$ is a compact, connected complex manifold of Fujiki class $\sC$ with $\Omega^2_{\tilde X}(\tilde X)$ of dimension $1$ over the field of complex number and, according to Remark \ref{r:sympcl}, $\tilde w$ is a normed generically symplectic class on $\tilde X$. Thus as $r=\nicefrac{1}{2}\dim(\tilde X)$ we obtain, using Lemma \ref{l:integralmod} and Proposition \ref{p:bbformtopint}:
 \[
  \int_X\left(a^{r+1}\bar w^{r-1}\right) = \int_{\tilde X}\left({\tilde a}^{r+1}\bar{\tilde w}^{r-1}\right) = (r+1)\lambda^{r-1}q_{\tilde X}(\tilde a) = (r+1)\lambda^{r-1}q_X(a).
 \]
 Observe that the very last equality holds by definition of the Beauville-Bogomolov form on $X$ (\cf Definition \ref{d:bbform}).
\end{proof}

\begin{remark}[Quadratic forms]
 \label{r:quadform}
 We review the definition of a quadratic form on a module, \cf \egv \cite[(2.1) Definition a)]{Kn02}: Let $R$ be a ring and $M$ an $R$-module. Then $q$ is called an \emph{$R$-quadratic form} on $M$ when $q$ is a function from $M$ to $R$ such that:
 \begin{enumeratei}
  \item \label{i:quadform-quad} for all $\lambda \in R$ and all $x\in M$, we have $q(\lambda\cdot x)=\lambda^2\cdot q(x)$;
  \item \label{i:quadform-bil} there exists an $R$-bilinear form $b$ on $M$ such that, for all $x,y\in M$, we have
  \begin{equation} \label{e:quadform}
   q(x+y) = q(x) + q(y) + b(x,y).
  \end{equation}
 \end{enumeratei}
 In case $R$ equals the ring of complex numbers (\resp real numbers, \resp rational numbers, \resp integers), we use the term \emph{complex} (\resp \emph{real}, \resp \emph{rational}, \resp \emph{integral}) \emph{quadratic form} as a synonym for the term ``$R$-quadratic form''.
 
 Observe that when $2\neq0$ in $R$ and $2$ is not a zero divisor in $R$, then, for all functions $q\colon M\to R$, \eqref{i:quadform-bil} implies \eqref{i:quadform-quad} above, \iev $q$ is an $R$-quadratic form on $M$ if and only if \eqref{i:quadform-bil} is satisfied.
 
 Given an $R$-quadratic form $q$ on $M$, there is one, and only one, $R$-bilinear form $b$ on $M$ such that \eqref{e:quadform} holds for all $x,y\in M$. We call this so uniquely determined $b$ the \emph{$R$-bilinear form on $M$ associated to $q$}. Note that the bilinear form $b$ is always symmetric. When $R$ is a field and $M$ is a finite dimensional $R$-vector space, we have, for any $R$-bilinear form $b$ on $M$, a well-defined concept of an $R$-rank of $b$ on $M$ given, for instance, as the $R$-rank of any matrix associated with $b$ relative to an ordered $R$-basis of $M$. In this context, we define the \emph{$R$-rank of $q$ on $M$} as the $R$-rank of $b$ on $M$, where $b$ is the $R$-bilinear form on $M$ associated to $q$.
\end{remark}

\begin{proposition}
 \label{p:bbformquad} \strut
 \begin{enumerate}
  \item \label{p:bbformquad-0} Let $(X,w)$ be as in Notation \ref{not:bbform}. Then $q_{(X,w)}$ is a complex quadratic form on $\H^2(X,\C)$.
  \item \label{p:bbformquad-gen} Let $X$ be a compact, connected, generically symplectic complex manifold such that $\Omega^2_X(X)$ is $1$-dimensional. Then $q_X$ is a complex quadratic form on $\H^2(X,\C)$.
  \item \label{p:bbformquad-sp} Let $X$ be a compact, connected, symplectic complex space such that $\Omega^2_X(X_\reg)$ is $1$-dimensional. Then $q_X$ is a complex quadratic form on $\H^2(X,\C)$.
 \end{enumerate}
\end{proposition}

\begin{proof}
 \ref{p:bbformquad-0}). Denote by $r$ the unique natural number such that $\dim(X)=2r$. Then $r\neq 0$ since $\dim(X) > 0$. Define
 \[
  s \colon \H^2(X,\C) \times \H^2(X,\C) \to \C
 \]
 to be the function given by:
 \begin{align*}
   & s(a,b) := r\int_X\left(w^{r-1}\bar{w}^{r-1}ab\right) \\ & + (r-1)\left(\int_X\left(w^{r-1}\bar w^ra\right)\int_X\left(w^r\bar w^{r-1}b\right)
  + \int_X\left(w^{r-1}\bar w^rb\right)\int_X\left(w^r\bar w^{r-1}a\right)\right).
 \end{align*}
 Then $s$ surely is a $\C$-bilinear form on $\H^2(X,\C)$, and, for all $a,b\in \H^2(X,\C)$, we have
 \[
  q_{(X,w)}(a+b) = q_{(X,w)}(a) + q_{(X,w)}(b) + s(a,b)
 \]
 as one easily verifies looking at \eqref{e:bbform}. Hence $q_{(X,w)}$ is a complex quadratic form on $\H^2(X,\C)$ according to Remark \ref{r:quadform}.

 \ref{p:bbformquad-gen}). There exists a normed generically symplectic class $w$ on $X$ by Remark \ref{r:normedcl+}. Now by Definition \ref{d:bbform}, we have $q_X = q_{(X,w)}$. Thus $q_X$ is a complex quadratic form on $\H^2(X,\C)$ by a).

 \ref{p:bbformquad-sp}). There exists a resolution of singularities $f\colon W\to X$. By Proposition \ref{p:unisympres}, $W$ is a compact, connected, generically symplectic complex manifold with $\Omega^2_W(W)$ of dimension $1$ over the field of complex numbers. By Definition \ref{d:bbform}, we have $q_X = q_W \circ f^*$. So, $q_X$ is a complex quadratic form on $\H^2(X,\C)$ since $q_W$ is a complex quadratic form on $\H^2(W,\C)$ by b) and $f^*$ is a homomorphism of complex vector spaces from $\H^2(X,\C)$ to $\H^2(W,\C)$ (it is a general fact that quadratic forms pull back to quadratic form under module homomorphisms).
\end{proof}

\begin{proposition}
 \label{p:bbformpol}
 Let $X$ be a compact, connected, symplectic complex space such that $\Omega^2_X(X_\reg)$ is of dimension $1$ over $\C$. Let $w$ be a normed symplectic class on $X$ and denote by $b$ the $\C$-bilinear form on $\H^2(X,\C)$ associated to $q_X$ (\cf Remark \ref{r:quadform}).
 \begin{enumerate}
  \item Setting $r:=\nicefrac{1}{2}\dim(X)$, we have, for all $c,d\in\H^{1,1}(X)$ and all $\lambda,\lambda',\mu,\mu'\in\C$:
  \begin{equation}
   \label{e:bbformpol}
   b(\lambda w+c+\lambda'\bar w,\mu w+d+\mu'\bar w) = r \int_X\left(w^{r-1}\bar w^{r-1}cd\right) + (\lambda\mu'+\mu\lambda').
  \end{equation}
  \item $w$ and $\bar w$ are perpendicular to $\H^{1,1}(X)$ in $(\H^2(X,\C),q_X)$.
  \item $b(w,w)=b(\bar w,\bar w)=0$ and $b(w,\bar w)=b(\bar w,w)=1$.
 \end{enumerate}
\end{proposition}

\begin{proof}
 Assertion a) is an immediate consequence of Corollary \ref{c:bbformhdg} (applied three times) and the fact that, for all $x,y \in \H^2(X,\C)$, we have
 \[
  b(x,y) = q_X(x+y) - q_X(x) - q_X(y).
 \]
 Both b) and c) are immediate corollaries of a).
\end{proof}

\begin{definition}[Beauville-Bogomolov quadric]
 \label{d:bbquad}
 Let $X$ be a compact, and connected, and symplectic complex space satisfying $\dim_\C(\Omega^2_X(X_\reg))=1$. We set
 \[
  Q_X := \{p\in\P(\H^2(X,\C)):(\forall c\in p)q_X(c)=0\}
 \]
 and call $Q_X$ the \emph{Beauville-Bogomolov quadric} of $X$.
 
 Since $X$ is compact, $\H^2(X,\C)$ is a finite dimensional complex vector space, so that we may view $\P(\H^2(X,\C))$ as a complex space. Obviously, as $q_X$ is a complex quadratic form on $\H^2(X,\C)$ by Proposition \ref{p:bbformquad} \ref{p:bbformquad-sp}), $Q_X$ is a closed analytic subset of $\P(\H^2(X,\C))$. We abuse notation and signify the closed complex subspace of $\P(\H^2(X,\C))$ induced on $Q_X$ again by $Q_X$. Besides, the latter $Q_X$ (complex space) will go by the name of \emph{Beauville-Bogomolov quadric} of $X$ too. We hope this ambivalent terminology will not irritate our readers.
\end{definition}

In order to prove in \S\ref{s:fr} that certain---consult Theorem \ref{t:fr} for the precise statement---compact, connected, symplectic complex spaces of Kähler type satisfy the so-called ``Fujiki relation'' (\cf Definition \ref{d:fr}), we need to know a priori that, for these $X$, the Beauville-Bogomolov quadric $Q_X$ is an \emph{irreducible} closed analytic subset of $\P(\H^2(X,\C))$ (strictly speaking, the irreducibility of $Q_X$ is exploited in the proof of Lemma \ref{l:frfam0}). Hence, we set out to investigate the rank of the quadratic form $q_X$.

\begin{proposition}
 \label{p:hrbr}
 Let $(V,g,I) = ((V,g),I)$ be a finite dimensional real inner product space endowed with a compatible (\iev orthogonal) almost complex structure $I$. Let $p$ and $q$ be natural numbers such that
 \[
  k := p+q \leq n := \nicefrac{1}{2}\dim_\RR(V).
 \]
 Denote $\omega$ the complexified fundamental form of $(V,g,I)$. Then, for all primitive forms $\alpha$ of type $(p,q)$ on $V$, we have
 \begin{equation} \label{e:hrbr}
  i^{p-q}(-1)^{\frac{k(k-1)}{2}}\cdot \alpha\wedge \bar\alpha \wedge \omega^{n-k} \geq 0
 \end{equation}
 in $\wedge^{2n}_\C(V_\C)^\vee$ with equality holding if and only if $\alpha$ is the trivial $k$-form on $V_\C$.
\end{proposition}

\begin{proof}
 See \cite[Corollary 1.2.36]{Hu05}.
\end{proof}

\begin{proposition}
 \label{p:bbformkc}
 Let $X$ be a compact, connected, symplectic complex space of Kähler type such that $\dim_\C(\Omega^2_X(X_\reg))=1$, and let $c$ be the image of a Kähler class on $X$ under the canonical mapping $\H^2(X,\RR) \to \H^2(X,\C)$. Then we have $q_X(c)>0$ (in the sense that $q_X(c)$ is in particular real).
\end{proposition}

\begin{proof}
 There exists a resolution of singularities $f\colon W\to X$. In particular, there are thin closed subsets $A$ and $B$ of $W$ and $X$ respectively such that $f$ induces by restriction an isomorphism of complex spaces $W\setminus A \to X\setminus B$. Since $c$ is (the image in $\H^2(X,\C)$ of) a Kähler class on $X$, there exists $\omega \in \A^{1,1}(W)$ such that $\omega$ is a de Rham representative of $f^*(c)$ and the restriction of $\omega$ (as a differential $2$-form) to $W \setminus A$ is (the complexification of) a Kähler form on $W \setminus A$. We know there exists a normed generically symplectic class $v$ on $W$. Thus there exists a generically symplectic structure $\rho$ on $W$ such that $v$ is the class of $\rho$. Denote the image of $\rho$ under the canonical mapping $\Omega^2_W(W) \to \A^{2,0}(W)$ again by $\rho$. Set $r:=\nicefrac{1}{2}\dim(W)$ and $\alpha := \rho^{\wedge(r-1)}$ (calculated in $\A^*(W,\C)$). Then $\alpha$ is a differential form of type $(2r-2,0)$ on $W$, whence in particular a primitive form. Thus by Proposition \ref{p:hrbr}, we see that the differential $2n$-form $\alpha \wedge \bar\alpha \wedge \omega^{\wedge 2}$ on $W$ is, for all $p\in W\setminus A$, strictly positive in $p$. As $W\setminus A$ is a nonempty, open, and dense subset of $W$, it follows that
 \[
  \int_W \alpha \wedge \bar\alpha \wedge \omega^{\wedge 2} > 0.
 \]
 Now obviously, $\alpha \wedge \bar\alpha \wedge \omega^{\wedge 2}$ is a de Rham representative of $v^{r-1}\bar v^{r-1}(f^*(c))^2$. Since $f^*(c) \in \H^{1,1}(W)$, we obtain by means of Definition \ref{d:bbform} and Proposition \ref{p:bbformhdg}:
 \[
  q_X(c) = q_W(f^*(c)) = \frac{r}{2}\int_W v^{r-1}\bar v^{r-1}(f^*(c))^2 = \frac{r}{2} \int_W \alpha \wedge \bar\alpha \wedge \omega^{\wedge 2} > 0.
 \]
 Yet this was just our claim.
\end{proof}

\begin{corollary}
 \label{c:bbquadirred}
 Let $X$ be as in Proposition \ref{p:bbformkc}. Then:
 \begin{enumerate}
  \item The $\C$-rank of the quadratic form $q_X$ on $\H^2(X,\C)$ is at least $3$.
  \item $Q_X$ is an irreducible closed analytic subset of $\P(\H^2(X,\C))$.
 \end{enumerate}
\end{corollary}

\begin{proof}
 There exists a normed symplectic class $w$ on $X$. Moreover, as $X$ is of Kähler type, there exists a Kähler class on $X$; denote by $c$ the image of this Kähler class under the canonical mapping $\H^2(X,\RR) \to \H^2(X,\C)$. Set $V := c^\perp \cap \H^{1,1}(X)$, where $c^\perp$ signifies the orthogonal complement of $c$ in $(\H^2(X,\C),q_X)$, and let $\basis v$ be an ordered $\C$-basis of $V$. We claim that the tuple $\basis b$ obtained by concatenating $(w,\bar w,c)$ and $\basis v$ is an ordered $\C$-basis of $\H^2(X,\C)$. This is because $w$ is a basis for $\H^{2,0}(X)$, $\bar w$ is a basis for $\H^{0,2}(X) = \bar{\H^{2,0}(X)}$, $c$ is a basis for $\C c$ (\iev $c\neq 0$), and we have $\H^{1,1}(X) = \C c \oplus V$ (since $q_X(c) \neq 0$) as well as
 \[
  \H^2(X,\C)=\H^{2,0}(X) \oplus \H^{0,2}(X) \oplus \H^{1,1}(X).
 \]
 Let $b$ be the $\C$-bilinear form on $\H^2(X,\C)$ associated to $q_X$. As
 \[
  (\H^{2,0}(X)+\H^{0,2}(X)) \perp \H^{1,1}(X),
 \]
 the matrix $M$ associated with $b$ relative to the basis $\basis b$ looks as follows:
 \[
  \begin{pmatrix}
   0 & 1 & 0 & 0 & \cdots & 0 \\
   1 & 0 & 0 & 0 & \cdots & 0 \\
   0 & 0 & 2q_X(c) & 0 & \cdots & 0 \\
   0 & 0 & 0 & * & \cdots & * \\
   \vdots & \vdots & \vdots & \vdots & \ddots & \vdots \\
   0 & 0 & 0 & * & \cdots & *
  \end{pmatrix}
 \]
 Clearly, taking into account that $2q_X(c)\neq 0$, the $\C$-rank of $M$, which equals (by definition) the rank of the $\C$-bilinear form $b$ on $\H^2(X,\C)$, is $\geq 3$. This proves a).
 
 Assertion b) follows from a) by means of the general fact that given a finite dimensional complex vector space $V$ and a $\C$-quadratic form $q$ of rank $\geq 3$ on $V$, the zero set defined by $q$ in $\P(V)$ is an irreducible closed analytic subset of $\P(V)$.
\end{proof}

\section{Deformation theory of symplectic complex spaces}
\label{s:defsymp}

In what follows, we prove that the quality of a complex space to be connected, symplectic, of Kähler type, and with a singular locus of codimension not deceeding $4$ is stable under small proper and flat deformation, \cf Theorem \ref{t:stabsymp} as well as Corollary \ref{c:stabsymp}. This result is originally due to Y.\ Namikawa, \cf \cite[Theorem 7']{Na01a}. We include its proof here for two reasons: Firstly, we felt that several points of Namikawa's exposition \loccit, in particular the essence of what we are going to say in Lemma \ref{l:stabsymp}, were not quite so clear. Secondly, the proof blends in nicely with the remainder of our presentation. Observe that Theorem \ref{t:stabsymp} used in conjunction with the likewise crucial Theorem \ref{t:defsm}, which we recall below, makes up a key ingredient for proving, in \S\ref{s:fr}, that the Fujiki relation holds for compact, connected, symplectic complex spaces $X$ with $1$-dimensional $\Omega^2_X(X_\reg)$ and a singular locus of codimension $\geq 4$ (\cf Theorem \ref{t:fr}).

Before delving into the stability of symplecticity for complex spaces, we need to review some preliminary stability results. We start by making a general

\begin{definition}[Stability under deformation]
 \label{d:stab}
 To speak about ``stability'' of certain properties of complex spaces (or similar geometrical objects) under deformation is pretty much folklore in the field. To the day we have, however, not seen a rigorous definition of the concept of stability in the literature. Therefore we move forward and suggest a definition---mainly for conceptual purposes.

 To that end, let $\cat C$ be a class (preferrably, yet not necessarily, a subclass of the class of complex spaces or the class of compact complex spaces). Then we say that $\cat C$ is \emph{stable under small proper and flat deformation} when, for all proper and flat morphisms of complex spaces $f\colon X\to S$ and all $t\in S$ such that $X_t \in \cat C$, there exists a neighborhood $V$ of $t$ in $S$ such that $X_s\in\cat C$ for all $s\in V$. Similarly, when $\phi=\phi(v)$ is a property (\iev $\phi$ is a formula in the language of set theory with one free variable $v$), we say that $\phi$ is \emph{stable under small proper and flat deformation} if the class $\{v:\phi(v)\}$ is so.

 In the same spirit, one may define a local variant of the notion of stability: Given a class $\cat C$ (\resp a property $\phi = \phi(v_0,v_1)$) we say that $\cat C$ (\resp $\phi$) is \emph{stable under small flat deformation} when, for all flat morphisms of complex spaces $f\colon X\to S$ and all $p\in X$ such that $(X_{f(p)},p) \in \cat C$ (\resp such that $\phi(X_{f(p)},p)$ holds), there exists a neighborhood $U$ of $p$ in $X$ such that $(X_{f(x)},x) \in \cat C$ (\resp such that $\phi(X_{f(x)},x)$ holds) for all $x\in U$.
\end{definition}

\begin{remark}
 \label{r:stab}
 Let $\cat C$ be any class (imagine $\cat C$ to be a subclass of the class of pointed complex spaces), and define $\cat C'$ to be the class containing precisely the complex spaces $X$ such that, for all $p\in |X|$, we have $(X,p) \in \cat C$. Speaking in terms of properties, this means that $\cat C$ reflects a local or pointwise property $\phi$ of a complex space whereas $\cat C'$ stands for property that a complex space satisfies $\phi$ at each of its points. Assume that $\cat C$ is stable under small flat deformation. Then it is easy to see that $\cat C'$ is stable under small proper and flat deformation.
\end{remark}

Here goes an overview of classical stable properties.

\begin{theorem}
 \label{t:stab}
 For a complex space $X$ and $p\in X$, let $\phi = \phi(X,p)$ signify one of the following properties:
 \begin{enumeratei}
  \item \label{i:stab-red} $X$ is reduced in $p$.
  \item \label{i:stab-nor} $X$ is normal in $p$.
  \item \label{i:stab-cm} $X$ is Cohen-Macaulay in $p$.
  \item \label{i:stab-gor} $X$ is Gorenstein in $p$.
  \item \label{i:stab-rtl} $X$ has a rational singularity in $p$.
 \end{enumeratei}
 Then $\phi$ is stable under small flat deformation.
\end{theorem}

\begin{proof}
 Cases \eqref{i:stab-red} and \eqref{i:stab-nor} can be deduced from \cite[Thm.~1.101 (2)]{GrLoSh07} by first passing respectively to the reduction or normalization of the base of the deformation in question. Case \eqref{i:stab-cm} follows from \cite[V, Theorem 2.8]{BaSt76}. Case \eqref{i:stab-gor} follows from case \eqref{i:stab-cm} by considering the relative dualizing sheaf. Case \eqref{i:stab-rtl} is treated by transferring R.~Elkik's proof of \cite[Théorème 4]{El78} to the analytic category (Elkik proves the same assertion for finite type $k$-schemes, where $k$ is a field of characteristic zero).
\end{proof}

\begin{corollary}
 \label{c:stab}
 For a complex space $X$, let $\phi = \phi(X)$ denote one of the following properties:
 \begin{enumeratei}
  \item \label{c:stab-red} $X$ is reduced.
  \item \label{c:stab-nor} $X$ is normal.
  \item \label{c:stab-cm} $X$ is Cohen-Macaulay.
  \item \label{c:stab-gor} $X$ is Gorenstein.
  \item \label{c:stab-rtl} $X$ has rational singularities.
 \end{enumeratei}
 Then $\phi$ is stable under small proper and flat deformation.
\end{corollary}

\begin{proof}
 This is immediate from Theorem \ref{t:stab} and the local to global principle outlined in Remark \ref{r:stab}.
\end{proof}

We move on to a property which is more delicate as concerns stability, namely that of ``Kählerity'' of a complex space. In general, small proper and flat deformations of complex spaces of Kähler type need not be of Kähler type, even if one assumes them to be, for instance, normal in addition (\cf \cite[Section 2]{Moi75}). Yet, we have the following result due to B.~Moishezon and J.~Bingener.

\begin{theorem}
 \label{t:bingener}
 Let $f\colon X\to S$ be a proper, flat morphism of complex spaces and $t\in S$. Assume that $X_t$ is of Kähler type and that the function
 \begin{equation} \label{e:bingener}
  \H^2(X_t,\RR)\to \H^2(X_t,\O_{X_t})
 \end{equation}
 induced by the canonical sheaf map $\RR_{X_t} \to \O_{X_t}$ on $X_t$ is a surjection. Then $f$ is weakly Kähler at $t$, \iev there exists an open neighborhood $V$ of $t$ in $S$ such that $f_V$ is weakly Kähler.
\end{theorem}

\begin{proof}
 This is a consequence of \cite[Theorem (6.3)]{Bi83}. As a matter of fact, the precise statement is given in the introduction of \loccit on p.~506.
\end{proof}

\begin{corollary}
 \label{c:stabkähler} \strut
 \begin{enumerate}
  \item \label{c:stabkähler-k} Let $f\colon X\to S$ be a proper and flat morphism of complex spaces and $t\in S$ such that $X_t$ is of Kähler type and has rational singularities. Then there exists an open neighborhood $V$ of $t$ in $S$ such that $X_s$ is of Kähler type for all $s\in V$.
  \item \label{c:stabkähler-stab} The class of Kähler complex spaces with rational singularities is stable under small proper and flat deformation.
 \end{enumerate} 
\end{corollary}

\begin{proof}
 \ref{c:stabkähler-k}). As $X_t$ is a compact complex space of Kähler type having rational singularities, the canonical mapping \eqref{e:bingener} is a surjection by Proposition \ref{p:bingenercrit}. Hence by Theorem \ref{t:bingener}, there exists an open neighborhood $V$ of $t$ in $S$ such that the morphism $f_V\colon X_V \to S|V$ is weakly Kähler. In consequence, for all $s\in V$, the complex space $(X_V)_s$ is Kähler; thus $X_s$ is Kähler as we have $(X_V)_s \iso X_s$ in $\An$.

 Assertion \ref{c:stabkähler-stab}) follows from \ref{c:stabkähler-k}) coupled with case \eqref{c:stab-rtl} of Corollary \ref{c:stab}.
\end{proof}

Next, we discuss the property that codimension of the singular locus of a complex space does not drop below a given fixed number.

\begin{notation}
 \label{not:codimcl}
 Let $c \in \N \cup \{\omega\}$. We introduce the follwing classes:
 \begin{align*}
  \cat C_c & := \left\{(X,p):\parbox{7cm}{$(X,p)$ is a pointed complex space such that $c \leq \codim_p(\Sing(X),X)$} \right\}. \\
  \cat C_c' & := \{X:X\text{ is a complex space such that } c \leq \codim(\Sing(X),X)\}. 
 \end{align*}
 Note that $\cat C_c'$ is the ``globalization'' of $\cat C_c$ in the sense of Remark \ref{r:stab}, \iev for any complex space $X$, we have $X \in \cat C_c'$ if and only if, for all $p\in X$, we have $(X,p) \in \cat C_c$.
\end{notation}

We ask whether the class $\cat C_c$ (\resp $\cat C_c'$) is stable under small flat deformation (\resp small proper and flat deformation). In fact, we will briefly sketch how to deduce that the intersection of $\cat C_c$ with the class of normal pointed complex spaces is stable under small flat deformation.

\begin{proposition}
 \label{p:fiberdimusc}
 For all morphisms of complex spaces $f\colon X\to Y$ and all $p\in X$ there exists a neighborhood $U$ of $p$ in $X$ such that, for all $x\in U$, we have
 \[
  \dim_x(X_{f(x)}) \leq \dim_p(X_{f(p)}).
 \]
\end{proposition}

\begin{proof}
 See \cite[Proposition in 3.4]{Fi76}.
\end{proof}

\begin{definition}[Equidimensionality]
 \label{d:equidim}
 Let $f\colon X\to Y$ be a morphism of complex spaces.
 \begin{enumerate}
  \item \label{d:equidim-pt} Let $p\in X$. We say that $f$ is \emph{locally equidimensional} in $p$ when there exists a neighbohood $U$ of $p$ in $X$ such that, for all $x\in U$, we have
  \[
   \dim_x(X_{f(x)}) = \dim_p(X_{f(p)}).
  \]
  \item \label{d:equidim-total} We say that $f$ is \emph{locally equidimensional} when $f$ is locally equidimensional in $p$ for all $p\in X$.
 \end{enumerate}
\end{definition}

\begin{proposition}
 \label{p:codimlsc}
 Let $f\colon X\to S$ be a morphism of complex spaces, $A$ a closed analytic subset of $X$, and $p\in X$. Suppose that $f$ is locally equidimensional in $p$. Then there exists a neighborhood $U$ of $p$ in $X$ such that, for all $x\in U$, we have:
 \begin{equation} \label{e:codimlsc}
  \codim_p(A \cap |X_{f(p)}|,X_{f(p)}) \leq \codim_x(A \cap |X_{f(x)}|,X_{f(x)}).
 \end{equation}
\end{proposition}

\begin{proof}
 When $p\notin A$, we put $U := |X|\setminus A$. Then $U$ is open in $X$ and $p\in U$. Moreover, for all $x\in U$, we have $A\cap |X_{f(x)}| = \emptyset$ and thus $\codim_x(A \cap |X_{f(x)}|,X_{f(x)}) = \omega$. Hence, for all $x\in U$, \eqref{e:codimlsc} holds.
 
 Now, assume that $p\in A$. Denote by $Y$ the closed complex subspace of $X$ induced on $A$; denote by $i\colon Y\to X$ the corresponding inclusion morphism. Set $g:=f\circ i$. By Proposition \ref{p:fiberdimusc} (applied to $g$), there is a neighborhood $V$ of $p$ in $Y$ such that, for all $y\in V$, we have
 \[
  \dim_y(Y_{g(y)}) \leq \dim_p(Y_{g(p)}).
 \]
 By the definition of the subspace topology, there exists a neighborhood $\tilde V$ of $p$ in $X$ such that $\tilde V\cap A \subset V$.  As $f$ is locally equidimensional in $p$, there exists a neighborhood $U'$ of $p$ in $X$ such that, for all $x\in U'$, we have
 \[
  \dim_x(X_{f(x)}) = \dim_p(X_{f(p)}),
 \]
 \cf Definition \ref{d:equidim}. Set $U:=U' \cap \tilde V$. Then $U$ is a neighborhood of $p$ in $X$ and, for all $x\in U\cap A$, we have:
 \begin{align*}
  & \codim_p(A \cap |X_{f(p)}|,X_{f(p)}) = \dim_p(X_{f(p)}) - \dim_p(Y_{g(p)}) \\
  & \leq \dim_x(X_{f(x)}) - \dim_x(Y_{g(x)}) = \codim_x(A \cap |X_{f(x)}|,X_{f(x)}),
 \end{align*}
 where we use that, for all $s\in S$, the complex subspace of $X_s$ induced on $A\cap |X_s|$ is isomorphic in $\An$ to the complex subspace of $Y$ induced on $|Y_s|$. For all $x\in U\setminus A$, \eqref{e:codimlsc} holds since $A\cap |X_{f(x)}| = \emptyset$ and thus $\codim_x(A \cap |X_{f(x)}|,X_{f(x)}) = \omega$.
\end{proof}

Looking at Proposition \ref{p:codimlsc}, we wish to find criteria for a (possibly flat) morphism of complex spaces to be locally equidimensional in a certain point of its source space. We content ourselves with treating the case where fiber passing through the given point is normal.

\begin{theorem}
 \label{t:dimformula}
 Let $f\colon X\to Y$ be a morphism of complex spaces and $p\in X$. Assume that $f$ is flat in $p$. Then:
 \begin{equation} \label{e:dimformula}
  \dim_p(X) = \dim_p(X_{f(p)}) + \dim_{f(p)}(Y).
 \end{equation}
\end{theorem}

\begin{proof}
 See \cite[Lemma in 3.19]{Fi76}.
\end{proof}

\begin{theorem}
 \label{t:normalred}
 Let $f\colon X\to Y$ be a flat morphism of complex spaces and $p\in X$. When $X_{f(p)}$ and $Y$ are normal (\resp reduced) in $p$ and $f(p)$, respectively, then $X$ is normal (\resp reduced) in $p$.
\end{theorem}

\begin{proof}
 See \cite[Thm.~1.101 (2)]{GrLoSh07}.
\end{proof}

\begin{proposition}
 \label{p:normalpuredim}
 Let $X$ be a complex space and $p\in X$. When $X$ is normal in $p$, then $X$ is pure dimensional at $p$.
\end{proposition}

\begin{proof}
 See \cite[Kapitel 6, \S4, Abschnitt 2]{CAS}. 
\end{proof}

\begin{proposition}
 \label{p:flatequidim}
 Let $f\colon X\to S$ be a flat morphism of complex spaces and $p\in X$. Assume that $X_{f(p)}$ and $S$ are normal in $p$ and $f(p)$, respectively. Then $f$ is locally equidimensional in $p$.
\end{proposition}

\begin{proof}
 By Proposition \ref{p:normalpuredim}, $S$ is pure dimensional at $f(p)$, \iev there exists a neighborhood $V$ of $f(p)$ in $S$ such that, for all $s\in V$, we have $\dim_s(S) = \dim_{f(p)}(S)$. By Theorem \ref{t:normalred}, $X$ is normal in $p$, whence again by means of Proposition \ref{p:normalpuredim}, $X$ is pure dimensional at $p$. Accordingly, there exists a neighborhood $U'$ of $p$ in $X$ such that, for all $x\in U'$, $\dim_x(X) = \dim_p(X)$. Set $U := f^{-1}(V) \cap U'$. Then clearly, $U$ is a neighborhood of $p$ in $X$. Moreover, by Theorem \ref{t:dimformula}, we have for all $x\in U$:
 \begin{equation*}
  \dim_x(X_{f(x)}) = \dim_x(X) - \dim_{f(x)}(S) = \dim_p(X) - \dim_{f(p)}(S) = \dim_p(X_{f(p)}).
 \end{equation*}
 In consequence, $f$ is locally equidimensional in $p$, \cf Definition \ref{d:equidim}.
\end{proof}

\begin{corollary} \label{c:stabcodim} \strut
 \begin{enumerate}
  \item \label{c:stabcodim-pt} Let $f\colon X\to S$ be a flat morphism of complex spaces and $p\in X$. Suppose that $X_{f(p)}$ and $S$ are normal in $p$ and $f(p)$, respectively. Then there exists a neighborhood $U$ of $p$ in $X$ such that, for all $x\in U$, we have:
  \begin{equation} \label{e:stabcodim}
   \codim_p(\Sing(X_{f(p)}),X_{f(p)}) \leq \codim_x(\Sing(X_{f(x)}),X_{f(x)}). 
  \end{equation}
  \item \label{c:stabcodim-cl} For all $c\in \N\cup\{\omega\}$, the class $\{(X,p)\in\cat C_c:(X,p)\text{ is normal}\}$ is stable under small flat deformation.
  \item \label{c:stabcodim-cl'} For all $c\in \N\cup\{\omega\}$, the class $\{X\in\cat C_c':X\text{ is normal}\}$ is stable under small proper and flat deformation.
 \end{enumerate}
\end{corollary}

\begin{proof}
 \ref{c:stabcodim-pt}). Set $A:=\Sing(f)$. Then $A$ is a closed analytic subset of $X$. By Proposition \ref{p:flatequidim}, $f$ is locally equidimensional in $p$. Therefore, by Proposition \ref{p:codimlsc}, there exists a neighborhood $U$ of $p$ in $X$ such that, for all $x\in U$, \eqref{e:codimlsc} holds. Due to the flatness of $f$, we have $A\cap |X_s| = \Sing(X_s)$ for all $s\in S$. Thus, for all $x\in U$, we have \eqref{e:stabcodim}.
 
 \ref{c:stabcodim-cl}). This is an immediate consequence of \ref{c:stabcodim-pt}), at least in case the base space of the deformation is normal in its basepoint. For arbitrary (\iev not necessarily normal) base spaces the result traced back to the result for normal base spaces by pulling the deformation back along the normalization of the base space. We refrain from explaining the details of this argument as we will apply the result only in case the base of the deformation is normal.
 
 \ref{c:stabcodim-cl'}). This follows from \ref{c:stabcodim-cl}) by means of Remark \ref{r:stab}.
\end{proof}

The remainder of \S\ref{s:defsymp} deals with stability of symplecticity.

\begin{lemma}
 \label{l:stabsymp}
 Let $f\colon X\to S$ be a proper, flat morphism of complex spaces with normal (or else Cohen-Macaulay) base and normal, Gorenstein fibers, let $t\in S$ and $\sigma \in \Omega^2_f(U)$, where
 \[
  U := |X| \setminus \Sing(f).
 \]
 For any $s\in S$, denote $i_s \colon X_s \to X$ the inclusion of the $f$-fiber over $s$, denote
 \[
  \phi_s \colon \Omega^2_f \to {i_s}_*(\Omega^2_{X_s})
 \]
 the pullback of $2$-differentials induced by $i_s$, and set
 \[
  \sigma_s := (\phi_s)_U(\sigma) \in \Omega^2_{X_s}((X_s)_\reg).
 \]
 Assume that $\sigma_t$ is nondegenerate on $(X_t)_\reg$. Then there exists a neighborhood $V$ of $t$ in $S$ such that $\sigma_s$ is nondegenerate on $(X_s)_\reg$ for all $s\in V$.
\end{lemma}

\begin{proof}
 Let us assume that the space $X$ is connected and $X_t \neq \emptyset$---the general case can be traced back easily to this special case by restricting $f$ to the connected components of $X$ having nonempty intersection with $f^{-1}(\{t\})$. Since $X_t \neq \emptyset$ and $X_t$ is normal, there exists $p_0 \in (X_t)_\reg$. Furthermore, since $\sigma_t$ is nondegenerate on $(X_t)_\reg$ at $p_0$, there exists a natural number $r$ such that $\dim_{p_0}(X_t) = 2r$ by Corollary \ref{c:nondegdim}. Obviously, the spaces $X$ and $S$ are locally pure dimensional so that the morphism $f$ is locally equidimensional by Proposition \ref{p:equidim}. For $X$ is connected, the morphism $f$ is yet equidimensional, and we have
 \[
  \dim_x(X_{f(x)}) = \dim_{p_0}(X_{f(p_0)}) = 2r
 \]
 for all $x\in X$. Denote $\omega_f$ the relative dualizing sheaf for $f$. Then since the morphism $f$ is submersive in $U$ with fibers pure of dimension $2r$, we have a canonical isomorphism
 \[
  \psi \colon \omega_f|U \to \Omega^{2r}_f|U
 \]
 of modules on $X|U$. Set $A := \Sing(f)$. Then $A$ is a closed analytic subset of $X$, and due to the flatness of $f$ we have
 \[
  \Sing(X_s) = A \cap |X_s|
 \]
 for all $s \in S$. As the fibers of $f$ are normal, we obtain
 \[
  2 \leq \codim(\Sing(X_s),X_s) = \codim(A \cap |X_s|,X_s)
 \]
 for all $s\in S$ and thus
 \[
  2 \leq \codim(A,f) \leq \codim(A,X)
 \]
 by Proposition \ref{p:codim} \ref{p:codim-global}). Since the fibers of $f$ are Gorenstein, we know that $\omega_f$ is locally free of rank $1$ on $X$ (\cf \egv \cite[Theorem 3.5.1]{Con00}). So, by Riemann's extension theorem, the restriction map
 \[
  \omega_f(X) \to \omega_f(X \setminus A) = \omega_f(U)
 \]
 is bijective. In particular, there exists one (and only one) $\alpha \in \omega_f(X)$ which restricts to $\sigma^{\wedge r} \in \Omega^{2r}_f(U)$ via $\psi_U \colon \omega_f(U) \to \Omega^{2r}_f(U)$.
 
 Let $s\in S$ and $p\in (X_s)_\reg$. Then, by Proposition \ref{p:nondeg}, $\sigma_s$ is nondegenerate on $(X_s)_\reg$ at $p$ if and only if we have $(\sigma_s^{\wedge r})(p) \neq 0$ in 
 \[
  \Omega^{2r}_{X_s}(p) := \C \otimes_{\O_{X_s,p}} \Omega^{2r}_{X_s,p}.
 \]
 The pullback of differentials
 \[
  \Omega^{2r}_f \to {i_s}_*(\Omega^{2r}_{X_s})
 \]
 induces an isomorphism
 \[
  \O_{X_s,p} \otimes_{\O_{X,p}} \Omega^{2r}_{f,p} \to \Omega^{2r}_{X_s,p},
 \]
 whence an isomorphism
 \[
  \C \otimes_{\O_{X,p}} \Omega^{2r}_{f,p} \to \C \otimes_{\O_{X_s,p}} \Omega^{2r}_{X_s,p},
 \]
 under which $(\sigma^{\wedge r})(p)$ is mapped to $(\sigma_s^{\wedge r})(p)$. Moreover, $\psi$ gives rise to an isomorphism
 \[
  \psi(p) \colon \omega_f(p) \to \Omega^{2r}_f(p)
 \]
 which sends $\alpha(p)$ to $(\sigma^{\wedge r})(p)$. In conclusion, we see that $\sigma_s$ is nondegenerate on $(X_s)_\reg$ at $p$ if and only if $\alpha(p) \neq 0$ in $\omega_f(p)$.
 
 Set
 \[
  Z := \{x \in X: \alpha(x) = 0 \text{ in } \omega_f(x)\}.
 \]
 Then as $\sigma_t$ is nondegenerate on $(X_t)_\reg$ by assumption, we have $Z \cap (X_t)_\reg = \emptyset$. In other words, setting $Z_t := Z \cap |X_t|$, we have $Z_t \subset \Sing(X_t)$. So, for all $p\in Z_t$,
 \[
  2 \leq \codim(\Sing(X_t),X_t) \leq \codim_p(\Sing(X_t),X_t) \leq \codim_p(Z_t,X_t).
 \]
 Clearly, $Z_t$ is the zero locus of the image of $\alpha$ under the canonical mapping
 \[
  \omega_f(X) \to (i_t^*(\omega_f))(X_t)
 \]
 in the module $i_t^*(\omega_f)$ on $X_t$. As $\omega_f$ is locally free of rank $1$ on $X$, we see that $i_t^*(\omega_f)$ is locally free of rank $1$ on $X_t$. Hence, for all $p\in Z_t$, we have
 \[
  \codim_p(Z_t,X_t) \leq 1.
 \]
 Therefore, $Z_t = \emptyset$. This implies that $|X_t| \subset |X| \setminus Z$. As $Z$ is a closed subset of $X$, we infer, exploiting the properness of $f$, that there exists a neighborhood $V$ of $t$ in $S$ such that $f^{-1}(V) \subset |X| \setminus Z$. From what we have noticed earlier, it follows that $\sigma_s$ is nondegenerate on $X_s$ for all $s\in V$.
\end{proof}

\begin{theorem}
 \label{t:stabsymp}
 Let $f\colon X\to S$ be a proper, flat morphism of complex spaces with smooth base and $t\in S$. Assume that $X_t$ is symplectic, Kähler, and in $\cat C_4'$. Then there exists a neighborhood $V$ of $t$ in $S$ such that $X_s$ is symplectic, Kähler, and in $\cat C_4'$ for all $s\in V$.
\end{theorem}

\begin{proof}
 As the space $X_t$ is symplectic, we know that $X_t$ is Gorenstein and has rational singularities by Proposition \ref{p:ratgor}. By Theorem \ref{t:stab} cases \eqref{i:stab-nor} and \eqref{i:stab-gor}, Corollary \ref{c:stabkähler} \ref{c:stabkähler-k}), and Corollary \ref{c:stabcodim} \ref{c:stabcodim-cl'}) there exists an open neighborhood $V'$ of $t$ in $S$ such that $X_s$ is normal, Gorenstein, Kähler, and in $\cat C_4'$ for all $s \in V'$. Therefore, without loss of generality, we may assume that the fibers of $f$ are altogether normal, Gorenstein, Kähler, and in $\cat C_4'$ to begin with.
 
 Set $A := \Sing(f)$ and define $g \colon Y \to S$ to be the composition of the inclusion $Y := X \setminus A \to X$ and $f$. Then due to the flatness of $f$, we have
 \[
  A \cap |X_s| = \Sing(X_s)
 \]
 for all $s\in S$ and thus
 \[
  (2+0)+2 = 4 \leq \codim(A,f).
 \]
 Therefore, by Proposition \ref{p:hdglffbc}, the module $\sH^{2,0}(g)$ is a locally finite free on $S$ in $t$ and the base change map
 \begin{equation} \label{e:stabsymp-bc}
  \C \otimes_{\O_{S,t}} (\sH^{2,0}(g))_t \to \sH^{2,0}(Y_t)
 \end{equation}
 is an isomorphism of complex vector spaces.
 
 Since $X_t$ is symplectic, there exists a symplectic structure $\sigma_t$ on $X_t$. By the surjectivity of the base change map \eqref{e:stabsymp-bc}, there exists an open neighborhood $V''$ of $t$ in $S$ and an element $\sigma \in \Omega^2_f(f^{-1}(V'') \setminus A)$ such that $\sigma$ is mapped to $\sigma_t$ by the pullback of Kähler $2$-differentials
 \[
  \Omega^2_f \to {i_t}_*(\Omega^2_{X_t})
 \]
 induced by the inclusion $i_t \colon X_t \to X$. By passing from $f\colon X\to S$ to
 \[
  f_{V''} \colon X|f^{-1}(V'') \to S|V''
 \]
 we may assume, again without loss of generality, that $\sigma \in \Omega^2_f(|X| \setminus A)$.
 
 By Lemma \ref{l:stabsymp}, there exists a neighborhood $V$ of $t$ in $S$ such that $\sigma_s$, which is to be defined as in the formulation of the lemma, is nondegenerate on $(X_s)_\reg$ for all $s\in V$. Let $s\in V$. Then by Proposition \ref{c:flennersymp}, we know that, for all resolutions of singularities $h \colon W \to X_s$, there exists $\rho$ such that $\rho$ is an extension as $2$-differential of $\sigma_s$ with respect to $h$, \iev condition \eqref{d:symp-str-ext} of Definition \ref{d:symp} \ref{d:symp-str}) holds (for $X_s$ and $\sigma_s$ in place of $X$ and $\sigma$, respectively). Since $X_s$ is a reduced, compact, and Kähler complex space, $X_s$ is of Fujiki class $\sC$ so that Proposition \ref{p:sympclexassc} implies that $\sigma_s$ induces a closed $2$\hyphen differential on $(X_s)_\reg$. Thus $\sigma_s$ is a symplectic structure on $X_s$ and $X_s$ is symplectic.
\end{proof}

\begin{theorem}
 \label{t:defsm}
 Let $f\colon X\to S$ be a proper, flat morphism of complex spaces and $t\in S$ such that $f$ is semi-universal in $t$ and $X_t$ is a symplectic complex space of Kähler type such that $\codim(\Sing(X_t),X_t)\geq4$. Then the complex space $S$ is smooth at $t$.
\end{theorem}

\begin{proof}
 In case the complex space $X_t$ is projective, the statement follows from \cite[Theorem (2.5)]{Na01}. However, the proof given \loccit~remains valid without requiring $X_t$ to be projective by means of Theorem \ref{t:ohsawa}.
\end{proof}

\begin{corollary}
 \label{c:stabsymp}
 The class
 \[
  \cat C := \left\{X:\parbox{7cm}{$X$ is a symplectic, Kähler complex space such that $\codim(\Sing(X),X)\geq4$}\right\}
 \]
 is stable under small proper and flat deformation.
\end{corollary}

\begin{proof}
 Let $f\colon \cX\to S$ be a proper, flat morphism of complex spaces and $t\in S$ such that $\cX_t \in \cat C$. Then by Theorem \ref{t:kuranishi}, there exists a proper, flat morphism $f' \colon \cX'\to S'$ and an element $t'\in S'$ such that $\cX_t \iso \cX'_{t'}$ and $f'$ is semi-universal in $t'$. By Theorem \ref{t:defsm}, the complex space $S'$ is smooth at $t'$. In consequence, by Theorem \ref{t:stabsymp}, there exists an open neighborhood $V'$ of $t'$ in $S'$ such that, for all $s\in V'$, we have $\cX'_s \in \cat C$. Without loss of generality, we may assume that $\cX'_s \in \cat C$ for all $s\in S'$. By the semi-universality of $f'$ in $t'$, there exists an open neighborhood $V$ of $t$ in $S$ and morphisms of complex spaces
 \[
  b \colon S|V \to S' \quad \text{and} \quad i \colon \cX_V \to \cX'
 \]
 such that $b(t) = t'$ and
 \[
  \xysquare{\cX_V}{\cX'}{S|V}{S'}{i}{f_V}{f'}{b}
 \]
 is a pullback square in the category of complex spaces. In particular, for all $s\in V$, the morphism $i$ induces an isomorphism $\cX_s \iso \cX'_s$, whence $\cX_s \in \cat C$.
\end{proof}

\section{The local Torelli theorem}
\label{s:lt}

First and foremost, we state and prove here our version of a local Torelli theorem for compact, connected, symplectic complex spaces $X$ of Kähler type such that $\Omega^2_X(X_\reg)$ is $1$-dimensional and the codimension of the singular locus of $X$ is $\geq 4$ (\cf Theorem \ref{t:lt}). Note that Y.\ Namikawa has proposed a local Torelli theorem for a slightly smaller class of spaces in \cite[Theorem 8]{Na01a}. Our statement is more general in the following respects: Firstly, our spaces need neither be projective nor $\Q$-factorial (or similar, in the nonprojective case). Secondly, we do not require $\H^1(X,\O_X)$ to be trivial. Moreover, we feel that we are after all the first to prove a local Torelli type statement in the context of singular symplectic complex spaces since, in \loccit, Namikawa contents himself with referring to Beauville's work \cite{Be83} as a proof for the decisive point (3) of his theorem.

Beauville's proof of the local Torelli theorem for irreducible symplectic manifolds certainly provides the basis for our line of reasoning below. Nonetheless, we would like to point out that in guise of Theorem \ref{t:pmclassic} and Corollary \ref{c:froertlsing}, the upshots of our entire Chapters \ref{ch:peri} and \ref{ch:froe} enter the proof of Theorem \ref{t:lt}.

\begin{proposition}
 \label{p:cmext}
 Let $X$ be a Cohen-Macaulay complex space and $A$ a closed analytic subset of $X$ such that $\Sing(X)\subset A$ and $\codim(A,X)\geq 3$. Put $Y:=X\setminus A$. Then the evident restriction mapping
 \[
  \Ext^1(\Omega^1_X,\O_X) \to \Ext^1(\Omega^1_Y,\O_Y)
 \]
 is bijective.
\end{proposition}

\begin{proof}
 This follows from the analytic counterpart of \cite[Lemma (12.5.6)]{KoMo92}.
\end{proof}

\begin{corollary}
 \label{c:ksiso}
 Let $f\colon X\to S$ be a proper, flat morphism of complex spaces, $A$ a closed analytic subset of $X$ such that $\Sing(f)\subset A$, and $t\in S$. Assume that $S$ is smooth, $f$ is semi-universal in $t$, $X_t$ is Cohen-Macaulay, and $\codim(A\cap|X_t|,X_t)\geq 3$. Set $Y:=X\setminus A$, and define $g$ to be the composition of the inclusion $Y\to X$ and $f$. Then the Kodaira-Spencer map of $g$ at $t$,
 \begin{equation} \label{e:ksiso-0}
  \KS_{g,t} \colon \T_S(t) \to \H^1(Y_t,\Theta_{Y_t})
 \end{equation}
 (\cf Notation \ref{not:ksm}), is an isomorphism of complex vector spaces.
\end{corollary}

\begin{proof}
 Denote $\T^1_{X_t}$ the first (co-)tangent cohomology of $X_t$ and
 \begin{equation} \label{e:ksiso-1}
  D_{f,t} \colon \T_S(t) \to \T^1_{X_t}
 \end{equation}
 the ``Kodaira-Spencer map'' for $f$ in $t$ (\cf \cite{Pal76}, \cite{Pal90}, \cite{Ill71}). The construction of this generalized Kodaira-Spencer map is functorial so that the inclusion $Y \to X$ gives rise to a commutative diagram of complex vector spaces
 \[
  \xymatrix{
   & \T_S(t) \ar[ld]_{D_{f,t}} \ar[dr]^{D_{g,t}} \\
   \T^1_{X_t} \ar[rr] && \T^1_{Y_t}
  }
 \]
 where the lower horizontal arrow is the morphism induced on (co-)tangent cohomology by $Y_t \to X_t$.

 Since the complex spaces $X_t$ and $Y_t$ are reduced, the canonical maps
 \begin{align*}
  & \Ext^1(\Omega^1_{X_t},\O_{X_t}) \to \T^1_{X_t}, \\
  & \Ext^1(\Omega^1_{Y_t},\O_{Y_t}) \to \T^1_{Y_t}
 \end{align*}
 are isomorphisms (\cf \egv \cite[(III.3.1), (iv) and (v)]{Ser06}). Moreover, the following diagram commutes:
 \[
  \xysquare{\Ext^1(\Omega^1_{X_t},\O_{X_t})}{\T^1_{X_t}}{\Ext^1(\Omega^1_{Y_t},\O_{Y_t})}{\T^1_{Y_t}}{}{}{}{}
 \]
 
 Since the morphism $f\colon X\to S$ is semi-universal in $t$, the Kodaira-Spencer map \eqref{e:ksiso-1} is a bijection. Therefore, by Proposition \ref{p:cmext} and the commutativity of the above diagrams we see that $D_{g,t}$ is an isomorphism. However, through the composition of canonical maps
 \[
  \Ext^1(\Omega^1_{Y_t},\O_{Y_t}) \to \Ext^1(\O_{Y_t},\Theta_{Y_t}) \to \H^1(Y_t,\Theta_{Y_t})
 \]
 the morphism $D_{g,t}$ is isomorphic to the ordinary Kodaira-Spencer map \eqref{e:ksiso-0}, whence \eqref{e:ksiso-0} is an isomorphism as claimed.
\end{proof}

\begin{lemma}
 \label{l:sympcc}
 Let $X$ be a symplectic complex manifold. Then the (adjoint) cup and contraction 
 \[
  \gamma := \gamma^{2,0}_X \colon \H^1(X,\Theta_X) \to \Hom(\sH^{2,0}(X),\sH^{1,1}(X))
 \]
 (\cf Notation \ref{not:cc'}) is injective. Moreover, when $\sH^{2,0}(X)$ is $1$-dimensional over the field of complex numbers, then $\gamma$ is an isomorphism.
\end{lemma}

\begin{proof}
 Let $c$ be an element of $\H^1(X,\Theta_X)$ such that $\gamma(c) = 0$. Observe that, by definition, $\gamma$ arises from the ``ordinary'' cup and contraction
 \[
  \gamma' \colon \H^1(X,\Theta_X) \otimes_\C \H^0(X,\Omega^2_X) \to \H^1(X,\Omega^1_X)
 \]
 through tensor-hom adjunction over $\C$ (\cf Notation \ref{not:cc'}). In particular, for all $d \in \H^0(X,\Omega^2_X)$, we have
 \[
  \gamma'(c \otimes d) = (\gamma(c))(d) = 0.
 \]

 As $X$ is a symplectic complex manifold, there exists a symplectic structure $\sigma$ on $X$ (\cf Definition \ref{d:symp0}). Denote $s$ the image of $\sigma$ under the canonical map
 \[
  \Omega^2_X(X) \to \H^0(X,\Omega^2_X).
 \]
 Moreover, denote $\phi$ the composition
 \[
  \Theta_X \to \Theta_X\otimes_X\O_X \xrightarrow{\id\otimes\sigma} \Theta_X\otimes_X\Omega^2_X \to \Omega^1_X
 \]
 of morphisms of modules on $X$, where the first arrow stands for the inverse of the right tensor unit for $\Theta_X$ on $X$, the $\sigma$ stands, by abuse of notation, for the unique morphism $\O_X\to\Omega^2_X$ of modules on $X$ which sends the $1$ of $\O_X(X)$ to the actual $\sigma$, and the last arrows stands for the sheaf-theoretic contraction morphism (\cf Notation \ref{not:cont})
 \[
  \gamma^2_X(\Omega^1_X) \colon \Theta_X \otimes_X \Omega^2_X \to \Omega^1_X.
 \]
 Then, since $\sigma$ is nondegenerate on $X$,
 \[
  \phi \colon \Theta_X \to \Omega^1_X
 \]
 is an isomorphism of modules on $X$ (\cf Remarks \ref{r:nondeg} \ref{r:nondeg-global})), whence
 \[
  \H^1(X,\phi) \colon \H^1(X,\Theta_X) \to \H^1(X,\Omega^1_X)
 \]
 is an isomorphism of complex vector spaces by functoriality. By the definition of the cup product, we have
 \[
  (\H^1(X,\phi))(c) = \gamma'(c \otimes s).
 \]
 Thus $(\H^1(X,\phi))(c) = 0$, and, in consequence, $c = 0$. This proves the injectivity of $\gamma$.

 Suppose that $\sH^{2,0}(X)$ is $1$\hyphen dimensional and let
 \[
  f \in \Hom(\sH^{2,0}(X),\sH^{1,1}(X))
 \]
 be an arbitrary element. By the surjectivity of $\H^1(X,\phi)$, there exists $c \in \H^1(X,\Theta_X)$ such that
 \[
  (\gamma(c))(s) = (\H^1(X,\phi))(c) = f(s).
 \]
 As $\sH^{2,0}(X)$ is nontrivial, we have $\dim(X)>0$, whence $s\neq 0$ in $\sH^{2,0}(X)$. Accordingly, as $\sH^{2,0}(X)$ is $1$-dimensional, $s$ generates $\sH^{2,0}(X)$ over $\C$. Thus $\gamma(c) = f$, which proves that $\gamma$ is surjective.
\end{proof}

\begin{theorem}[Local Torelli]
 \label{t:lt}
 Let $f\colon X\to S$ be a proper, flat morphism of complex spaces such that $S$ is simply connected and smooth and the fibers of $f$ have rational singularities, are of Kähler type, and have singular loci of codimension $\geq 4$. Furthermore, let $t\in S$ such that $f$ is semi-universal in $t$ and $X_t$ is symplectic and connected with $\Omega^2_{X_t}((X_t)_\reg)$ of dimension $1$ over the field of complex numbers. Define $g\colon Y\to S$ to be the submersive share of $f$. Then the period mapping
 \[
  \cP^{2,2}_t(g) \colon S \to \Gr(\sH^2(Y_t))
 \]
 (\cf Notation \ref{not:pm} \ref{not:pm-bc}) is well-defined and the tangent map
 \begin{equation} \label{e:lt-0}
  \T_t(\cP^{2,2}_t(g)) \colon \T_S(t) \to \T_{\Gr(\sH^2(Y_t))}(\F^2\sH^2(Y_t))
 \end{equation}
 is an injection with $1$-dimensional cokernel.
\end{theorem}

\begin{proof}
 By Corollary \ref{c:froertlsing}, we know that the Frölicher spectral sequence of $g$ degenerates in entries
 \[
  I := \{(\nu,\mu) \in \Z\times\Z : \nu+\mu = 2\}
 \]
 at sheet $1$ in $\Mod(S)$; moreover, for all $(p,q) \in I$, the Hodge module $\sH^{p,q}(g)$ is locally finite free on $S$ and compatible with base change in the sense that, for all $s \in S$, the Hodge base change map
 \[
  \beta^{p,q}_{g,s} \colon (\sH^{p,q}(g))(s) \to \sH^{p,q}(Y_s)
 \]
 is an isomorphism of complex vector spaces. Let $s\in S$. Then as $X_s$ has rational singularities, $X_s$ is normal, whence locally pure dimensional according to Proposition \ref{p:normalpuredim}. Thus, by Theorem \ref{t:ohsawa}, the Frölicher spectral sequence of $X_s \setminus \Sing(X_s)$ degenerates in entries $I$ at sheet $1$ in $\Mod(\C)$. Since the morphism $f$ is flat, the inclusion $Y \to X$ induces an isormophism of complex spaces
 \[
  Y_s \to X_s \setminus \Sing(X_s).
 \]
 Therefore, the Frölicher spectral sequence of $Y_s$ degenerates in entries $I$ at sheet $1$ in $\Mod(\C)$.
 
 By Theorem \ref{t:pmclassic} (applied to $g$ in place of $f$ and $n=2$), the period mapping $\cP^{2,2}_t(g)$ is well-defined (this is implicit in the theorem) and there exists a sequence $\tilde\psi = (\tilde\psi^\nu)_{\nu\in\Z}$ such that firstly, for all $\nu\in\Z$,
 \[
  \tilde\psi^\nu \colon \F^\nu\sH^2(Y_t)/\F^{\nu+1}\sH^2(Y_t) \to \sH^{\nu,2-\nu}(Y_t)
 \]
 is an isomorphism in $\Mod(\C)$ and secondly, setting
 \begin{align*}
  \alpha & := \tilde\psi^2 \circ \coker(\iota^2_{Y_t}(2,3)) && \colon \F^2\sH^2(Y_t) \to \sH^{2,0}(Y_t), \\
  \beta & := (\iota^2_{Y_t}(1)/\F^2\sH^2(Y_t)) \circ (\tilde\psi^1)^{-1} && \colon \sH^{1,1}(Y_t) \to \sH^2(Y_t)/\F^2\sH^2(Y_t),
 \end{align*}
 the following diagram commutes in $\Mod(\C)$:
 \begin{equation} \label{e:lt-1}
  \xymatrix{
   \T_S(t) \ar[r]^{\KS_{g,t}} \ar[dd]_{\T_t(\cP^{2,2}_t(g))} & \H^1(Y_t,\Theta_{Y_t}) \ar[d]^{\gamma^{2,0}_{Y_t}} \\
   & \Hom(\sH^{2,0}(Y_t),\sH^{1,1}(Y_t)) \ar[d]^{\Hom(\alpha,\beta)} \\
   \T_{\Gr(\sH^2(Y_t))}(\F^2\sH^2(Y_t)) \ar[r] \ar@{}@<-1ex>[r]_{\theta(\sH^2(Y_t),\F^2\sH^2(Y_t))} & \Hom(\F^2\sH^2(Y_t),\sH^2(Y_t)/\F^2\sH^2(Y_t))
  }
 \end{equation}
 commutes in $\Mod(\C)$ (for the definition of $\theta$, see Notation \ref{not:theta}).
  
 Since $X_t$ has rational singularities, $X_t$ is Cohen-Macaulay. Put $A := \Sing(f)$. Then clearly $A$ is a closed analytic subset of $X$ and, due to the flatness of $f$, we have $A\cap |X_t| = \Sing(X_t)$, whence
 \[
  3 < 4 \leq \codim(\Sing(X_t),X_t) = \codim(A \cap |X_t|,X_t).
 \]
 Applying Corollary \ref{c:ksiso}, we see that the Kodaira-Spencer map
 \[
  \KS_{g,t} \colon \T_S(t) \to \H^1(Y_t,\Theta_{Y_t})
 \]
 is an isomorphism in $\Mod(\C)$. As $X_t$ is a symplectic complex space, $(X_t)_\reg$ is a symplectic complex manifold. As $\Omega^2_{X_t}((X_t)_\reg)$ is a $1$-dimensional complex vector space, $\Omega^2_{(X_t)_\reg}((X_t)_\reg)$ is a $1$-dimensional complex vector space. Since $Y_t \iso (X_t)_\reg$ (\cf above), the same assertions hold for $Y_t$ instead of $(X_t)_\reg$. Therefore, by Lemma \ref{l:sympcc}, the cup and contraction
 \[
  \gamma^{2,0}_{Y_t} \colon \H^1(Y_t,\Theta_{Y_t}) \to \Hom(\sH^{2,0}(Y_t),\sH^{1,1}(Y_t))
 \]
 is an isomorphism in $\Mod(\C)$. Since $\F^3\sH^2(Y_t) \iso 0$, we know that the cokernel of the inclusion
 \[
  \iota^2_{Y_t}(2,3) \colon \F^3\sH^2(Y_t) \to \F^2\sH^2(Y_t)
 \]
 is an isomorphism. Thus, $\alpha$ is an isomorphism. The morphism $\theta$ in \eqref{e:lt-1} is an isomorphism anyway. By the definition of $\beta$, $\beta$ is certainly injective. So, $\Hom(\alpha,\beta)$ is injective, and exploiting the commutativity of \eqref{e:lt-1}, we conclude that the tangent map \eqref{e:lt-0} is injective.
 
 Moreover, the dimension of the cokernel of the injection \eqref{e:lt-0} equals the dimension of the cokernel of the injection $\Hom(\alpha,\beta)$, which in turn equals the dimension of the cokernel of $\beta$ since $\sH^{2,0}(Y_t)$, and likewise $\F^2\sH^2(Y_t)$, is $1$-dimensional. Now the cokernel of $\beta$ is obviously isomorphic to $\sH^2(Y_t)/\F^1\sH^2(Y_t)$, and via $\tilde\psi^0$, \iev by the degeneration of the Frölicher spectral sequence of $Y_t$ in entry $(0,2)$ at sheet $1$, we have
 \[
  \sH^2(Y_t)/\F^1\sH^2(Y_t) \iso \sH^{0,2}(Y_t).
 \]
 According to Theorem \ref{t:ohsawa}, we have
 \[
  \sH^{0,2}(Y_t) \iso \sH^{2,0}(Y_t).
 \]
 Thus, we deduce that the cokernel of the tangent map \eqref{e:lt-0} is $1$-dimensional.
\end{proof}

Reviewing the statement of Theorem \ref{t:lt}, we would like to draw our reader's attention to the fact that the period mapping $\cP := \cP^{2,2}_t(g)$ depends exclusively on the submersive share $g$ of the original family $f$. Note that the submersive share $g$ of $f$ is nothing but the family of regular loci of the fibers of $f$. Further on, note that the local system with respect to which $\cP$ is defined assigns to a point $s$ in $S$ in principle the second complex cohomology $\H^2((X_s)_\reg,\C)$ of $(X_s)_\reg$. Thus $\cP$ encompasses Hodge theoretic information of the $(X_s)_\reg$, not however a priori of the $X_s$ themselves.

The upcoming series of results is aimed at sheding a little light on the relationship between $X$ and $X_\reg$, for $X$ a compact symplectic complex space of Kähler type, in terms of Hodge structures on second cohomologies. We also discuss ramifications of this relationship for the deformations theory of $X$. To begin with, we introduce terminology capturing the variation of mixed Hodge structure in a family of spaces of Fujiki class $\sC$.

\begin{construction}
 \label{con:repcoh}
 Let $f\colon X\to S$ be a morphism of topological spaces, $A$ a ring, $n$ an integer. Our aim is to construct, given that $f$ satisfies certain conditions (\cf below), an $A$-representation of the fundamental groupoid of $S$ which parametrizes the $A$-valued cohomology modules in degree $n$ of the fibers of $f$.

 Recall that we denote
 \[
  f^A \colon (X,A_X) \to (S,A_S)
 \]
 the canonical morphism of ringed spaces derived from $f$. Moreover, set
 \[
  \H^n(f,A) := \R^n(f^A)_*(A_X).
 \]
 Thus $\H^n(f,A)$ is a sheaf of $A_S$-modules on $S$. Assume that $\H^n(f,A)$ is a locally constant sheaf on $S$. Then by means of Remark \ref{r:locsysrep}, the sheaf $\H^n(f,A)$ induces an $A$-representation
 \[
  \rho' \colon \Pi(S) \to \Mod(A).
 \]
 Assume that, for all $s\in S$, the evident base change map
 \[
  \beta_s \colon \left(\H^n(f,A)\right)_s \to \H^n(X_s,A)
 \]
 is a bijection. Then define
 \[
  \rho^n(f,A) \colon \Pi(S) \to \Mod(A)
 \]
 to be the unique functor such that, for all $s\in S$, we have
 \[
  (\rho^n(f,A))_0(s) = \H^n(X_s,A)
 \]
 and, for all $(s,t)\in S\times S$ and all morphisms $a \colon s \to t$ in $\Pi(S)$, we have
 \[
  \left((\rho^n(f,A))_1(s,t)\right)(a) = \beta_t \circ ((\rho')_1(s,t))(a) \circ (\beta_s)^{-1}.
 \]
 As a matter of fact, it would be more accurate to simply define $\rho_0$ and $\rho_1$ as indicated above, then set $\rho^n(f,A) := (\rho_0,\rho_1)$ and assert that the so defined $\rho$ is a functor from $\Pi(S)$ to $\Mod(S)$.

 When $f\colon X\to S$ is not a morphism of topological spaces but a morphism of complex spaces (\resp a morphism of ringed spaces), we agree on writing $\rho^n(f,A)$ for $\rho^n(f_\top,A)$ in case this makes sense.
\end{construction}

\begin{definition}[Period mappings for MHS]
 \label{d:pmmhs}
 Let $f\colon X\to S$ be a proper morphism of complex spaces such that $X_s$ is of Fujiki class $\sC$ for all $s\in S$. Let $n$ and $p$ be integers. We define $\F^p\H^n(f)$ to be the unique function on $|S|$ such that, for all $s\in |S|$, we have
 \[
  (\F^p\H^n(f))(s) = \F^p\H^n(X_s).
 \]
 We call $\F^p\H^n(f)$ the \emph{system of Hodge filtered pieces in degree $p$ on $n$-th cohomology} associated to $f$.

 Assume further that $S_\top$ is simply connected and $\H^n(f,\C)$ is a locally constant sheaf on $S_\top$. Then, for any $t\in S$, we set (\cf Construction \ref{con:pmrep}):
 \[
  \cP^{p,n}_t(f)_\MHS := \cP^\C_t(S_\top,\rho^n(f,\C),\F^p\H^n(f)).
 \]
 Note that this definition makes sense. In fact, $\rho^n(f,\C)$ is a well-defined complex representation of $\Pi(S_\top)$ and $\F^p\H^n(f)$ is a complex distribution in $\rho:=\rho^n(f,\C)$ as clearly, $(\F^p\H^n(f))(s) = \F^p\H^n(X_s)$ is a complex vector subspace of $\rho_0(s)=\H^n(X_s,\C)$ for all $s\in S$.
\end{definition}

\begin{lemma}
 \label{l:f2h2compare}
 Let $X$ be a complex space of Fujiki class $\sC$ having rational singularities and satisfying $4\leq\codim(\Sing(X),X)$.
 \begin{enumerate}
  \item \label{l:f2h2compare-h2} The mapping
  \[
   j^* \colon \H^2(X,\C) \to \H^2(X_\reg,\C)
  \]
  induced by the inclusion morphism $j\colon X_\reg \to X$ is one-to-one.
  \item \label{l:f2h2compare-f2h2} The composition
  \[
   \H^2(X,\C) \to \H^2(X_\reg,\C) \to \sH^2(X_\reg)
  \]
  restricts to a bijection
  \[
   \F^2\H^2(X) \to \F^2\sH^2(X_\reg).
  \]
 \end{enumerate}
\end{lemma}

\begin{proof}
 \ref{l:f2h2compare-h2}). Let $f \colon W\to X$ be a resolution of singularities such that the exceptional locus $E$ of $f$ is a simple normal crossing divisor in $W$ and $f$ induces by restriction an isomorphism
 \[
  f' \colon W\setminus E \to X_\reg
 \]
 of complex spaces. Then, by Proposition \ref{p:resh2inj}, since $X$ has rational singularities, we know that $f$ induces a one-to-one map
 \[
  f^* \colon \H^2(X,\C) \to \H^2(W,\C)
 \]
 on complex cohomology. Denote
 \begin{align*}
  & i \colon W\setminus E \to W, \\
  & i' \colon (W,\emptyset) \to (W,W\setminus E)
 \end{align*}
 the respective inclusion morphisms. Then the sequence
 \[
  \xymatrix{
   \H^2(W,W\setminus E;\C) \ar[r]^-{i'^*} & \H^2(W,\C) \ar[r]^-{i^*} & \H^2(W\setminus E,\C)
  }
 \]
 is exact in $\Mod(\C)$. Thus the kernel of $i^*$ is precisely the $\C$-linear span of the fundamental cohomology classes $[E_\nu]$ of the irreducible components $E_\nu$ of $E$. So, since
 \[
  f^*[\H^2(X,\C)] \cap \C\langle [E_\nu] \rangle = \{0\},
 \]
 we see that
 \[
  (f \circ i)^* = i^* \circ f^* \colon \H^2(X,\C) \to \H^2(W\setminus E,\C)
 \]
 is one-to-one. Therefore, $j^*$ is one-to-one taking into account that $f'$ furnishes an isomorphism $f\circ i \to j$ in the overcategory $\An_{/X}$.

 \ref{l:f2h2compare-f2h2}). By Proposition \ref{p:f2h2}, we know that
 \[
  f^*|\F^2\H^2(X) \colon \F^2\H^2(X) \to \F^2\H^2(W)
 \]
 is a bijection. By the functoriality of base change maps, we know that the following diagram, where the vertical arrows denote the respective inclusions, commutes in $\Mod(\C)$:
 \begin{equation} \label{e:f2h2compare-1}
  \xysquare{\F^2\sH^2(W)}{\F^2\sH^2(W\setminus E)}{\sH^2(W)}{\sH^2(W\setminus E)}{}{}{}{}
 \end{equation}
 
 Let $p \in \{1,2\}$ and $c \in \sH^{p,0}(W\setminus E)$. Then by Theorem \ref{t:flenner}, there exists one, and only one, element $b \in \sH^{p,0}(W)$ such that $b$ is sent to $c$ by the restriction mapping
 \[
  \sH^{p,0}(W) \to \sH^{p,0}(W\setminus E).
 \]
 Since $W$ is a complex manifold of Fujiki class $\sC$, the Frölicher spectral sequene of $W$ degenerates at sheet $1$, whence, specifically, $b$ corresponds to a closed Kähler $p$-differential on $W$. In consequence, $c$ corresponds to a closed Kähler $p$-differential on $W\setminus E$. Varying $c$ and $p$, we deduce that the Frölicher spectral sequence of $W\setminus E$ degenerates in entries $(2,0)$ at sheet $1$. Thus, there exist isomorphisms such that the diagram
 \[
  \xymatrix{
   \F^2\sH^2(W) \ar[r] \ar@{.>}[d]_\sim & \F^2\sH^2(W\setminus E) \ar@{.>}[d]^\sim \\
   \sH^{2,0}(W) \ar[r] & \sH^{2,0}(W\setminus E)
  }
 \]
 commutes in $\Mod(\C)$.
 
 As we have already noticed, the Hodge base change
 \[
  \sH^{2,0}(W) \to \sH^{2,0}(W\setminus E)
 \]
 is an isomorphism. Thus from the commutativity of the diagram in \eqref{e:f2h2compare-1} we deduce that the composition of morphisms
 \[
  \H^2(W,\C) \to \sH^2(W) \to \sH^2(W\setminus E)
 \]
 restricts to an isomorphism
 \[
  \F^2\H^2(W) \to \F^2\sH^2(W\setminus E);
 \]
 note here that by the definition of the Hodge structure $\H^2(W)$, the Hodge filtered piece $\F^2\H^2(W)$ is the inverse image of $\F^2\sH^2(W)$ under the canonical map
 \[
  \H^2(W,\C) \to \sH^2(W).
 \]
 
 Finally, we observe that since $f' \colon W\setminus W \to X_\reg$ is an isomorphism of complex spaces, $f'$ induces an isomorphism
 \[
  f'^* \colon \F^2\sH^2(X_\reg) \to \F^2\sH^2(W\setminus E)
 \]
 of complex vector spaces. Thus the commutativity in $\Mod(\C)$ of the diagram
 \[
  \xymatrix{
   \H^2(X,\C) \ar[r] \ar[d]_{f^*} & \H^2(X_\reg,\C) \ar[r] \ar[d]_{f'^*} & \sH^2(X_\reg) \ar[d]^{f'^*} \\
   \H^2(W,\C) \ar[r] & \H^2(W\setminus E,\C) \ar[r] & \sH^2(W\setminus W)
  }
 \]
 yields our claim.
\end{proof}

\begin{proposition}
 \label{p:pmcompare}
 Let $f\colon X\to S$ be a proper, flat morphism of complex spaces such that $S$ is smooth and simply connected, $\H^2(f,\C)$ is a locally constant sheaf on $S_\top$, and the fibers of $f$ have rational singularities, are of Kähler type, and have singular loci of codimension $\geq 4$. Furthermore, let $t\in S$. Define $f'\colon X'\to S$ to be the submersive share of $f$, set $\cP:=\cP^{2,2}_t(f)_\MHS$ and $\cP':=\cP^{2,2}_t(f')$, and denote
 \[
  \phi_t \colon \H^2(X_t,\C) \to \H^2(X'_t,\C) \to \sH^2(X'_t)
 \]
 the composition of canonical mappings. Then, for all $s\in S$, we have
 \[
  \cP'(s) = \phi_t[\cP(s)].
 \]
\end{proposition}

\begin{proof}
 Set $\rho := \rho^2(f,\C)$ (\cf Construction \ref{con:repcoh}); note that this makes sense as $\H^2(f,\C)$ is a locally constant sheaf on $S_\top$ and $f$ is proper. By Proposition \ref{c:froertlsing} \ref{c:froertlsing-lff}) and \ref{c:froertlsing-degen}), the algebraic de Rahm module $\sH^2(f')$ is locally finite free on $S$. Let $H'$ be the module of horizontal sections of
 \[
  \nabla^2_\GM(f') \colon \sH^2(f') \to \Omega^1_S \otimes_S \sH^2(f')
 \]
 on $S$ (\cf Notations \ref{not:gm} and \ref{not:hor}). Then by Proposition \ref{p:gmconn} \ref{p:gmconn-f}) and Proposition \ref{p:rhc} \ref{p:rhc-locsys}), $H'$ is a locally constant sheaf on $S$. Define $\rho''$ to be $\C$-representation of $\Pi(S)$ associated $H'$ (\cf Construction \ref{con:locsysrep} and Remark \ref{r:locsysrep}). For any $s\in S$, define $\beta'_s$ to be the composition of the following morphisms
 \[
  (H')_s \to (\sH^2(f'))_s \to (\sH^2(f'))(s) \to \sH^2(X'_s)
 \]
 in $\Mod(\C)$, where the first arrow stands for the stalk-at-$s$ morphism on $S_\top$ associated to the inclusion morphism $H' \to \sH^2(f')$, the second arrow is the evident ``quotient morphism'', and the third arrow stands for the de Rham base change in degree $2$ for $f'$ at $s$. Define $(\rho')_0$ to be the function on $|S|$ given by the assignment
 \[
  s \mto \sH^2(X'_s).
 \]
 Define $(\rho')_1$ to be the unique function on $|S|\times |S|$ such that, for all $(r,s)\in |S|\times |S|$, $(\rho')_1(r,s)$ is the unique function on $(\Pi(S))_1(r,s)$ satisfying, for all $a \in (\Pi(S))_1(r,s)$:
 \[
  ((\rho')_1(r,s))(a) = \beta'_s \circ ((\rho'')_1(r,s))(a) \circ (\beta'_r)^{-1}.
 \]
 Set $\rho' := ((\rho')_0,(\rho')_1)$. Then $\rho'$ is functor from $\Pi(S)$ to $\Mod(\C)$.
 
 Define $\psi$ to be the composition
 \[
  \H^2(f,\C) \to \H^2(f',\C) \to H'
 \]
 of morphisms of sheaves of $\C_S$-modules on $S$, where the first arrow signifies the evident base change map and the second arrow denotes the unique morphism from $\H^2(f',\C)$ to $H'$ which factors the canonical morphism $\H^2(f',\C) \to \sH^2(f')$ through the inclusion $H' \to \sH^2(f')$. Let $\phi = (\phi_s)_{s\in S}$ be the family of morphisms
 \[
  \phi_s \colon \H^2(X_s,\C) \to \H^2(X'_s,\C) \to \sH^2(X'_s);
 \]
 note that this fits with the notation ``$\phi_t$'' introduced in the statement of the proposition. Then by the functoriality of base change maps, the following diagram commutes in $\Mod(\C)$ for all $s\in S$:
 \[
  \xymatrix{
   (\H^2(f,\C))_s \ar[r]^{\psi_s} \ar[d]_{\beta_s} & (H')_s \ar[d]^{\beta'_s} \\
   \H^2(X_s,\C) \ar[r]_{\phi_s} & \sH^2((X')_s)
  }
 \]
 
 In consequence,
 \[
  \phi \colon \rho \to \rho'
 \]
 is a natural transformation of functors from $\Pi(S)$ to $\Mod(\C)$. Thus employing Lemma \ref{l:f2h2compare}, we see that, for all $s\in S$, letting $a$ stand for the unique element of $(\Pi(S))_1(s,t)$, we have
 \begin{align*}
  \phi_t[\cP(s)] & = \phi_t[\rho_1(s,t)(a)[\F^2\H^2(X_s)]] = (\phi_t\circ \rho_1(s,t)(a))[\F^2\H^2(X_s)] \\
  & = (\rho'_1(s,t)(a)\circ \phi_s)[\F^2\H^2(X_s)] = \rho'_1(s,t)(a)[\phi_s[\F^2\H^2(X_s)]] \\
  & = \rho'_1(s,t)(a)[\F^2\sH^2(X'_s)] = \cP'(s),
 \end{align*}
 which is what had to be proven.
\end{proof}

\begin{corollary}
 \label{c:pmmhsholo}
 Let $f$ and $t$ be as in Proposition \ref{p:pmcompare}. Then there exists a unique morphism of complex spaces
 \[
  \cP^+ \colon S \to G := \Gr(\H^2(X_t,\C))
 \]
 such that the underlying function of $\cP^+$ is precisely $\cP := \cP^{2,2}_t(f)_\MHS$; in particular, $\cP$ is a holomorphic mapping from $S$ to $G$. Moreover, the diagram
 \[
  \xymatrix{
   & S \ar[ld]_{\cP^+} \ar[dr]^{\cP'} \\
   \Gr(\H^2(X_t,\C)) \ar[rr]_{(\phi_t)_*} && \Gr(\sH^2(X'_t))
  }
 \]
 where $f'$, $\cP'$ and $\phi_t$ have the same meaning as in Proposition \ref{p:pmcompare} and $(\phi_t)_* := \Gr(\phi_t)$, commutes in the category of complex spaces.
\end{corollary}

\begin{proof}
 By Proposition \ref{p:pmcompare}, we know that $|\cP'| = |(\phi_t)_*| \circ \cP$ (in the plain set-theoretic sense) since, for all $l \in \Gr(\H^2(X_t,\C))$, we have $|(\phi_t)_*|(l) = \phi_t[l]$ by the definition of $(\phi_t)_*$. Therefore, $\cP'(S) \subset (\phi_t)_*[G]$. As $\phi_t$ is a monomorphism of complex vector spaces, $(\phi_t)_*$ is a closed embedding of complex spaces. Hence, given that the complex space $S$ is reduced (for it is smooth by assumption), there exists a morphism of complex spaces $\cP^+\colon S\to G$ such that $(\phi_t)_* \circ \cP^+ = \cP'$. From this we obtain
 \[
  |(\phi_t)_*| \circ |\cP^+| = |\cP'| = |(\phi_t)_*| \circ \cP,
 \]
 which implies $|\cP^+| = \cP$ as $|(\phi_t)_*|$ is a one-to-one function. This proves the existence of $\cP^+$.
 
 When $\cP^+_1$ is another morphism of complex spaces from $S$ to $G$ such that $|\cP^+_1| = \cP$, we have $|\cP^+_1| = |\cP^+|$ and thus $\cP^+_1 = \cP^+$ by the reducedness of $S$. This shows the uniqueness of $\cP^+$.
\end{proof}

\begin{proposition}
 \label{p:ltquad}
 Let $f\colon X\to S$ be a proper, flat morphism of complex spaces and $t\in S$ such that $X_t$ is connected and symplectic, $\Omega^2_{X_t}((X_t)_\reg)$ is $1$-dimensional, $S$ is simply connected, $\H^2(f,\C)$ is a locally constant sheaf on $S$, and $f$ is fiberwise of Fujiki class $\sC$.
 \begin{enumerate}
  \item \label{p:ltquad-cont} When $\cP:=\cP^{2,2}_t(f)_\MHS$ is a continuous mapping from $S$ to
  \[
   G_1 := \Gr(1,\H^2(X_t,\C)),
  \]
  there exists a neighborhood $V$ of $t$ in $S$ such that $\cP(V)\subset Q_{X_t}$.
  \item \label{p:ltquad-hol} When $S$ is smooth and $\cP$ is a holomorphic mapping from $S$ to $G_1$, we have $\cP(S)\subset Q_{X_t}$.
 \end{enumerate}
\end{proposition}

\begin{proof}
 \ref{p:ltquad-cont}). Let $r$ be half the dimension of $X_t$. As $X_t$ is nonempty, connected, symplectic, and of Fujiki class $\sC$, there exists a normed symplectic class $w$ on $X_t$ (\cf Proposition \ref{p:sympclex} and Definition \ref{d:normedcl}). We know that $w$ is a nonzero element of $\F^2\H^2(X_t)$. Since $\cP$ has image lying in the Grassmannian of $1$-dimensional subspaces of $\H^2(X_t,\C)$ and $\cP(t) = \F^2\H^2(X_t)$, we see that $\F^2\H^2(X_t)$ is $1$-dimensional, whence generated by $w$. By Proposition \ref{p:symph2}, we have:
 \[
  \H^2(X_t,\C) = \H^{0,2}(X_t) \oplus \H^{1,1}(X_t) \oplus \H^{2,0}(X_t).
 \]
 Define $E$ to be the hyperplane in $G_1$ which is spanned by $\H^{0,2}(X_t) \oplus \H^{1,1}(X_t)$. Moreover, put $V:=\cP^{-1}(G_1\setminus E)$.
 
 Let $s\in V$ be arbitrary. Since $\F^2\H^2(X_s)$ is $1$-dimensional, there exists a nonzero element $v \in \F^2\H^2(X_s)$. As $H := \H^2(f,\C)$ is a locally constant sheaf on $S$ and $S$ is simply connected, the stalk map $H(S)\to H_s$ is one-to-one and onto. As $f$ is proper, the base change map $H_s\to \H^2(X_s,\C)$ is one-to-one and onto, too. Thus there exists a unique $\gamma\in H(S)$ which is sent to $v$ by the composition of the latter two functions. Write $a$ for the image of $\gamma$ in $\H^2(X_t,\C)$. Then, by the definition of $\cP$, we have
 \[
  \cP(s) = \C a \subset \H^2(X_t,\C).
 \]
 Thus, by the definition of $Q_{X_t}$ (\cf Definition \ref{d:bbquad}), we have $\cP(s)\in Q_{X_t}$ if and only if $q_{X_t}(a)=0$.
 
 According to the Hodge decomposition on $\H^2(X_t,\C)$, there exist complex numbers $\lambda$ and $\lambda'$ as well as an element $b\in\H^{1,1}(X_t)$ such that
 \[
  a = \lambda w+b+\lambda'\bar w.
 \]
 By Proposition \ref{p:bbformtopint}, the following identity holds:
 \begin{equation} \label{e:ltquad-0}
  \int_{X_t}(a^{r+1}{\bar w}^{r-1})=(r+1)\lambda^{r-1}q_{X_t}(a).  
 \end{equation}
 Denote $\delta$ the unique lift of $\bar w$ with respect to the function $H(S)\to \H^2(X_t,\C)$. Denote $c$ the image of $\delta$ under $H(S)\to \H^2(X_s,\C)$. As $\dim(X_s)=2r$ (due to the flatness of $f$) and $v \in \F^2\H^2(X_s)$, we see that $v^{r+1}=0$ in $\H^*(X_s,\C)$. In consequence, we have $v^{r+1}c^{r-1}=0$ in $\H^*(X_s,\C)$. As the mapping
 \[
  (\H^*(f,\C))(S) \to \H^*(X_s,\C)
 \]
 is a morphism of rings, it sends $\gamma^{r+1}\delta^{r-1}$ to $v^{r+1}c^{r-1}$. Therefore, as
 \[
  (\H^{4r}(f,\C))(S) \to \H^{4r}(X_s,\C)
 \]
 is one-to-one, we see that $\gamma^{r+1}\delta^{r-1}=0$ in $(\H^*(f,\C))(S)$. But then, $a^{r+1}{\bar w}^{r-1}=0$ in $\H^*(X_t,\C)$ as this is the image of $\gamma^{r+1}\delta^{r-1}$ under
 \[
  (\H^*(f,\C))(S) \to \H^*(X_t,\C).
 \]
 Thus the left hand side of equation \eqref{e:ltquad-0} equals zero. As $\cP(s)=\C a\notin E$, we have $\lambda\neq0$. So, we obtain $q_{X_t}(a)=0$, whence $\cP(s) \in Q_{X_t}$. As $s$ was an arbitrary element of $V$, we deduce $\cP(V) \subset Q_{X_t}$.
 
 \ref{p:ltquad-hol}). We apply \ref{p:ltquad-cont}) and conclude by means of the Identitätssatz for holomorphic functions.
\end{proof}

\begin{theorem}[Local Torelli, II]
 \label{t:ltfull}
 Let $f\colon X\to S$ be a proper, flat morphism of complex spaces such that $S$ is smooth and simply connected and the fibers of $f$ have rational singularities, are of Kähler type, and have singular loci of codimension $\geq 4$. Furthermore, let $t\in S$ and suppose that $X_t$ is connected and symplectic with $\Omega^2_{X_t}((X_t)_\reg)$ of dimension $1$ over $\C$. Define $g\colon Y\to S$ to be the submersive share of $f$, set $\cP':=\cP^{2,2}_t(g)$, and assume that the tangent map
 \[
  \T_t(\cP') \colon \T_S(t) \to \T_{\Gr(\sH^2(Y_t))}(\F^2\sH^2(Y_t))
 \]
 is an injection with $1$-dimensional cokernel. Moreover, assume that $\H^2(f,\C)$ is a locally constant sheaf on $S_\top$, and set $\cP:=\cP^{2,2}_t(f)_\MHS$.
 \begin{enumerate}
  \item \label{t:ltfull-hol} There exists one, and only one, morphism of complex spaces
  \[
   \cP^+\colon S \to G:=\Gr(\H^2(X_t,\C))
  \]
  such that $|\cP^+| = \cP$.
  \item \label{t:ltfull-quad} There exists one, and only one, morphism of complex spaces
  \[
   \bar\cP \colon S \to Q_{X_t}
  \]
  such that $j \circ \bar\cP = \cP^+$, where $j \colon Q_{X_t} \to G$ denotes the inclusion morphism.
  \item \label{t:ltfull-iso} $\bar\cP$ is locally biholomorphic at $t$.
  \item \label{t:ltfull-tgt} The tangent map
  \[
   \T_t(\cP^+) \colon \T_S(t) \to \T_G(\F^2\H^2(X_t))
  \]
  is an injection with $1$-dimensional cokernel.
  \item \label{t:ltfull-h2} The mapping
  \[
   \H^2(X_t,\C) \to \H^2((X_t)_\reg,\C)
  \]
  induced by the inclusion $(X_t)_\reg \to X_t$ is a bijection.
 \end{enumerate}
\end{theorem}

\begin{proof}
 Assertion \ref{t:ltfull-hol}) is an immediate consequence of Corollary \ref{c:pmmhsholo}.
 
 Let us write $\phi_t$ for the composition of the following canonical morphisms in $\Mod(\C)$:
 \[
  \H^2(X_t,\C) \to \H^2(Y_t,\C) \overset\sim\to \sH^2(Y_t).
 \]
 Then by Lemma \ref{l:f2h2compare} \ref{l:f2h2compare-h2}), noting that, due to the flatness of $f$, the morphism $i_t \colon Y_t \to X_t$ induces an isomorphism $Y_t \to (X_t)_\reg$, we infer that $\phi_t$ is a monomorphism of complex vector spaces. Denote
 \[
  (\phi_t)_* \colon G = \Gr(\H^2(X_t,\C)) \to \Gr(\sH^2(Y_t)) =: G'.
 \]
 the induced morphism of Grassmannians. By Corollary \ref{c:pmmhsholo}, we have $\cP' = (\phi_t)_* \circ \cP^+$. Since by assumption $\Omega^2_{X_t}((X_t)_\reg)$ is of dimension $1$ over $\C$, we see that $\F^2\sH^2(Y_t)$ is of dimension $1$ over $\C$. Thus the range of $\cP'$ is a subset of the Grassmannian of lines in $\sH^2(Y_t)$. In consequence, the range of $\cP^+$ is a subset of the Grassmannian $G_1$ of lines in $\H^2(X_t,\C)$, whence $\cP = |\cP^+|$ is a holomorphic map from $S$ to $G_1$. As the fibers of $f$ are compact, reduced, and of Kähler type, the fibers of $f$ are of Fujiki class $\sC$. Therefore, by Proposition \ref{p:ltquad}, we have assertion \ref{t:ltfull-quad}).

 From \ref{t:ltfull-quad}), we deduce that $\cP' = (\phi_t)_* \circ j \circ \bar\cP$. Thus by the functoriality of tangent maps, we obtain:
 \[
  \T_t(\cP') = \T_{\cP(t)}((\phi_t)_*) \circ \T_{\bar\cP(t)}(j) \circ \T_t(\bar\cP).
 \]
 By assumption, $\T_t(\cP')$ is a monomorphism with cokernel of dimension $1$. Therefore, $\T_t(\bar\cP)$ is certainly a monomorphism. Besides, $\T_{\bar\cP(t)}(j)$ and $\T_{\cP(t)}((\phi_t)_*)$ are monomorphisms since $j$ and $(\phi_t)_*$ are closed immersions. Since the quadric $Q_{X_t}$ is smooth at $\cP(t)=\bar\cP(t)$, the cokernel of $\T_{\bar\cP(t)}(j)$ has dimension $1$. Thus both $\T_t(\bar\cP)$ and $\T_{\cP(t)}((\phi_t)_*)$ are isomorphisms. Since $S$ is smooth at $t$, we deduce \ref{t:ltfull-iso}).

 As
 \[
  \T_t(\cP^+) = \T_{\bar\cP(t)}(j) \circ \T_t(\bar\cP),
 \]
 we see that $\T_t(\cP)$ is an injection with $1$-dimensional cokernel, which proves \ref{t:ltfull-tgt}).

 Furthermore, $\T_{\cP(t)}((\phi_t)_*)$ is an isomorphism (if and) only if $\phi_t$ is an isomorphism. Therefore, as $\H^2(Y_t,\C) \to \sH^2(Y_t)$ is an isomorphism anyway, we see that
 \[
  i_t^* \colon \H^2(X_t,\C) \to \H^2(Y_t,\C)
 \]
 is an isomorphism. As we have already pointed out, the morphism $i_t$ is isomorphic to the inclusion morphism $(X_t)_\reg \to X_t$ in the overcategory $\An_{/X_t}$, whence we deduce \ref{t:ltfull-h2}) by means of the functoriality of $\H^2(-,\C)$.
\end{proof}

\section{The Fujiki relation}
\label{s:fr}

Let $X$ be a nonempty, compact, irreducible reduced complex space. Then $n:=\dim(X)$ is a natural number, and we may define a function $t_X$ on $\H^2(X,\C)$ by means of the assignment
\[
 a \mapsto \int_X a^n.
\]
The Fujiki relation reveals how, in case $X$ is symplectic and with $\Omega^2_X(X_\reg)$ of complex dimension $1$, the Beauville-Bogomolov form of $X$ (\cf Definition \ref{d:bbform}) relates to the function $t_X$. We will introduce some appropriate rigorous terminology in Definition \ref{d:fr} below. Note that the fact that any irreducible symplectic complex manifold satisfies the Fujiki relation is due to A.~Fujiki (\cf \cite{Fu87}), hence the name. The main result of this section is Theorem \ref{t:fr} which generalizes Fujiki's result to a wider class of certain, possibly nonsmooth, symplectic complex spaces.

Let us mention that D.\ Matsushita has advertised a result similar to ours in form of \cite[Theorem 1.2]{Ma01}. However, we feel that several points of Matsushita's line of reasoning in \loccit~are much harder to establish than his exposition makes the reader believe. In addition to, the proof of \cite[Theorem 1.2]{Ma01} does unfortunately rely on Y.\ Namikawa's \cite[Theorem 8]{Na01a}, which we reprehend as explained at the very beginning of \S\ref{s:lt}. We emphasize that we do not think that Matsushita's intermediate conclusions are wrong; on the contrary, we reckon that his argument can be amended by occasionally invoking techniques of \cite{Na06}. Anyhow, our Theorem \ref{t:fr} is stronger than Matsushita's in the respect that we do not require our space in question to be projective, nor do we require it to be $\Q$-factorial (or similar).
 
\begin{definition}
 \label{d:fr}
 Let $X$ be a compact, connected, symplectic complex space such that $\Omega^2_X(X_\reg)$ is of dimension $1$ over the field of complex numbers. We say that $X$ \emph{satisfies the Fujiki relation} when, for all $a\in\H^2(X,\C)$, we have:
 \begin{equation}
  \label{e:fr}
  \int_X a^{2r} = \binom{2r}{r}(q_X(a))^r,
 \end{equation}
 where $r:=\nicefrac{1}{2}\dim(X)$.
\end{definition}

\begin{proposition}
 \label{p:frequiv}
 Let $X$ be a compact, connected, symplectic complex space of Kähler type such that $\Omega^2_X(X_\reg)$ is of dimension $1$ over $\C$. Then the following are equivalent:
 \begin{enumeratei}
  \item $X$ satisfies the Fujiki relation.
  \item There exists $\lambda\in\C^*$ such that, for all $a\in\H^2(X,\C)$, $\int_X a^{2r} = \lambda(q_X(a))^r$, where $r:=\nicefrac{1}{2}\dim(X)$.
  \item $Q_X=\{l \in \Gr(1,\H^2(X,\C)):(\forall v\in l)v^{\dim(X)}=0 \textup{ in } \H^*(X,\C)\}$.
 \end{enumeratei}
\end{proposition}

\begin{proof}
 (i) implies (iii). On the one hand, let $l\in Q_X$. Then $l \in \Gr(1,\H^2(X,\C))$ and $q_X(v) = 0$ for all $v\in l$. As $X$ satisfies the Fujiki relation, it follows that $\int_X v^{2r} = 0$ for all $v \in l$ (note that $r\neq 0$), \cf Definition \ref{d:fr}. Thus $v^{\dim(X)}=v^{2r}=0$ in $\H^*(X,\C)$ since as $X$ is a compact, irreducible reduced complex space of dimension $2r$, the function
 \[
  \int_X \colon \H^{2\dim(X)}(X,\C) \to \C
 \]
 is one-to-one. On the other hand, when $l\in \Gr(1,\H^2(X,\C))$ is an element such that, for all $v\in l$, $v^{\dim(X)} = 0$ in $\H^*(X,\C)$, we have $\int_X v^{2r} = 0$ and hence $q_X(v) = 0$ for all $v\in l$ by means of \eqref{e:fr}.

 (iii) implies (ii). As $X$ is compact, $\H^2(X,\C)$ is finite dimensional. Thus there exists $d\in\N$ as well as a $\C$-linear isomorphism $\phi \colon \C^d \to \H^2(X,\C)$. Furthermore, there exist complex polynomials $f$ and $g$ in $d$ variables such that $f_\C(x) = q_X(\phi(x))$ and $g_\C(x) = \int_X(\phi(x))^{\dim(X)}$ for all $x\in \C^d$, where $f_\C$ and $g_\C$ denote the polynomial functions on $\C^d$ associated to $f$ and $g$, respectively. Clearly, $f$ and $g$ are homogeneous of degrees $2$ and $\dim(X)$, respectively. By (iii), we have $\rZ(f) = \rZ(g)$ as, for all $v\in\H^2(X,\C)$, $v^{\dim(X)}=0$ in $\H^*(X,\C)$ if and only if $\int_X v^{\dim(X)}=0$. Thus $\sqrt{(f)} = \sqrt{(g)}$ by elemetary algebraic geometry, whence $f^j = hg$. By Corollary \ref{c:bbquadirred}, $f$ is irreducible in the polynomial ring. Therefore, there exists $\lambda \in \C^*$ and $i$ such that $g = \lambda f^i$. By comparison of degrees, $i = r$. Thus, for all $x\in \C^d$, we have
 \[
  g_\C(x) = (\lambda f^r)_\C(x) = \lambda (f_\C(x))^r.
 \]
 Plugging in $\phi^{-1}$, we obtain (ii).

 (ii) implies (i). As $X$ is a compact, reduced complex space of Kähler type, $X$ is of Fujiki class $\sC$. Hence, $X$ being in addition symplectic, there exists a normed symplectic class $w$ on $X$ by Proposition \ref{p:sympclex}, Lemma \ref{l:integralmod}, Proposition \ref{p:normedclres}, and Proposition \ref{p:sympvol+}. By Corollary \ref{c:bbformhdg}, we have $q_X(w+\bar w) = 1$. There exists a resolution of singularities $f\colon \tilde X \to X$. Set $\tilde w := f^*(w)$. Then
 \[
  f^*((w+\bar w)^{2r}) = (f^*(w+\bar w))^{2r} = (\tilde w+\bar{\tilde w})^{2r}.
 \]
 By Proposition \ref{p:normedclres}, $\tilde w$ is a normed generically symplectic class on $\tilde X$. In particular, $\tilde w$ is the class of a global $2$\hyphen differential on $\tilde X$. Thus we have ${\tilde w}^i = 0$ in $\H^*(\tilde X,\C)$ for all $i\in\N$ such that $i>r$ and likewise $\bar{\tilde w}^j = 0$ for all $j\in\N$ such that $j>r$. In consequence, since the subring $\H^{2*}(\tilde X,\C)$ of $\H^*(\tilde X,\C)$ is commutative, calculating with the aid of the binomial formula yields:
 \[
  (\tilde w+\bar{\tilde w})^{2r} = \sum_{j=0}^{2r}\binom{2r}{j}{\tilde w}^{2r-j}\bar{\tilde w}^j = \binom{2r}{r}{\tilde w}^r\bar{\tilde w}^r
 \]
 in $\H^*(\tilde X,\C)$. Applying Lemma \ref{l:integralmod}, we obtain (recall that $\tilde w$ is normed on $\tilde X$):
 \[
  \int_X (w+\bar w)^{2r} = \int_{\tilde X} (\tilde w+\bar{\tilde w})^{2r} = \binom{2r}{r} \int_{\tilde X}{\tilde w}^r\bar{\tilde w}^r = \binom{2r}{r}.
 \]
 Now,
 \[
  \lambda = \lambda (q_X(w+\bar w))^r = \int_X (w+\bar w)^{2r} = \binom{2r}{r}.
 \]
 In turn, for all $a\in\H^2(X,\C)$, \eqref{e:fr} holds, which then implies (i).
\end{proof}

\begin{lemma}
 \label{l:critlines}
 Let $V$ be a finite dimensional complex vector space, $Q$ a quadric in $\P(V)$, and $p\in\P(V)\setminus Q$. Denote $G_p$ the set of all lines in $\P(V)$ passing through $p$. Then there exist a hyperplane $H$ in $\P(V)$ and a quadric $Q'$ in $H$ such that $p\notin H$, $\rk(Q')=\rk(Q)-1$, and
 \begin{equation} \label{e:critlines}
  \bigcup\{L\in G_p:|L\cap Q|\neq 2\} = K,
 \end{equation}
 where $K$ denotes the cone in $\P(V)$ with base $Q'$ and vertex $p$.
\end{lemma}

\begin{proof}
 As $p \in \P(V) \setminus Q$, we have $Q \neq \P(V)$. Thus $\dim_\C(V) \geq 2$. Put $N := \dim_\C(V) - 1$. We fix a coordinate system in $\P(V)$, \iev an ordered $\C$-basis of $V$, such that $p = [1:0:\dots:0]$. Let $f \in \C[X_0,\dots,X_N]$ be homogeneous of degree $2$ such that:
 \[
  Q = \{x \in \P(V): f(x) = 0\}.
 \]
 There are $a_0,\dots,a_N \in \C$ and $b_{ij} \in \C$, $(i,j) \in \{0,\dots,N\}^2$, $i<j$, such that:
 \[
  f = a_0X_0^2 + \dots + a_NX_N^2 + 2\sum_{0\leq i<j\leq N}b_{ij}X_iX_j
 \]
 in $\C[X_0,\dots,X_N]$. We define $H := \{x\in \P(V):x_0=0\}$ and
 \begin{equation} \label{e:critlines-1}
  f' := (b_{01}^2-a_0a_1)X_1^2+\dots+(b_{0N}^2-a_0a_N)X_N^2 + 2\sum_{1\leq i<j\leq N}(b_{0i}b_{0j}-a_0b_{ij})X_iX_j
 \end{equation}
 in $\C[X_1,\dots,X_N]$. Further, we set $Q' := \{x\in H:f'(x)=0\}$. Clearly, $H$ is a hyperplane in $\P(V)$, $Q'$ is a quadric in $H$, and $p\notin H$.  We claim that, for all $q\in H$, the cardinality of $\overline{pq} \cap Q$ is different from $2$ if and only if $q \in Q'$.

 For that matter, let $q\in H$ be an arbitrary element. Set $L:=\overline{pq}$. Let $(0,q_1,\dots,q_N)$ be a representative of $q$ (written in our fixed basis of $V$). Then $(1,0,\dots,0)$ and $(0,q_1,\dots,q_N)$ make up a coordinate system in $l$. The coordinates corresponding to the points $p$ and $q$ will be denoted by $Y_0$ and $Y_1$, respectively. Thus the point $[y_0:y_1]$ of $L$ equals the point $[y_0:y_1q_1:\dots:y_1q_N]$ of $\P(V)$. That is, in coordinates of $L$, the set $L\cap Q$ is the zero set of the polynomial
 \begin{align*}
  g & = f(Y_0,q_1Y_1,\dots,q_NY_1) \\
  & = a_0Y_0^2 + \sum_{i=1}^N a_i(q_iY_1)^2 + 2\sum_{j=1}^N b_{0j}Y_0(q_jY_1) + 2\sum_{1\leq i<j\leq N} b_{ij}(q_iY_1)(q_jY_1) \\
  & = a_0Y_0^2 + \left(2\sum_{j=1}^N b_{0j}q_j\right)Y_0Y_1 + \left(\sum_{i=1}^N a_iq_i^2 + 2\sum_{1\leq i<j\leq N}b_{ij}q_iq_j\right)Y_1^2
 \end{align*}
 in $\C[Y_0,Y_1]_{(2)}$. Hence, we have $|L\cap Q|\neq 2$ if and only if the discriminant of $g$ vanishes, in symbols:
 \begin{equation} \label{e:critlines-2}
  \left(\sum_{j=1}^N b_{0j}q_j\right)^2 - a_0\left(\sum_{i=1}^N a_iq_i^2 + 2\sum_{1\leq i<j\leq N} b_{ij}q_iq_j\right) = 0.
 \end{equation}
 Recalling \eqref{e:critlines-1}, we see that the left hand side of \eqref{e:critlines-2} equals $f'(q_1,\dots,q_N)$. Therefore, we have established our claim.

 Let us deduce \eqref{e:critlines} from the claim. $K$ is to denote the cone in $\P(V)$ with base $Q'$ and vertex $p$, as in the formulation of the lemma. Let $L\in G_p$ such that $|L\cap Q|\neq 2$. There exists $q\in H$ such that $L=\overline{pq}$. The claim implies $q\in Q'$. Thus $L\subset K$. For the other inclusion let $x\in K$ be arbitrary. By definition of the cone $K$, there exists $q \in Q'$ such that $x \in \overline{pq}$, whence $|\overline{pq}\cap Q|\neq 2$ again by the claim. Therefore, $x$ is contained in the left hand side of \eqref{e:critlines}.

 It remains to show that $\rk(f') = \rk(f) - 1$. Denote by $A$ and $A'$ the symmetric coefficient matrices associated to $f$ and $f'$, respectively. The rows of $A$ are numbered starting with $0$. When first multiplying each of the rows $1$ to $N$ of $A$ by $-a_0$ and then adding, for $i=1,\dots,N$, $b_{0i}$-times the $0$-th row to the $i$-th row, one obtains the matrix
 \[
  B := \begin{pmatrix} a_0 & b_{0*} \\ \underline{0} & A' \end{pmatrix}.
 \]
 As $p=[1:0:\dots:0] \neq Q$, we have $a_0\neq0$. In consequence, $\rk(A) = \rk(B)$, and $\rk(B) = \rk(A') + 1$. Noting that $\rk(f)=\rk(A)$ and $\rk(f')=\rk(A)'$, we are finished.
\end{proof}

\begin{corollary}
 \label{c:goodlines}
 Let $V$ be a finite dimensional complex vector space, $Q$ a quadric of rank $\geq 2$ in $\P(V)$, and $p\in\P(V)\setminus Q$. Then there exists a nonempty, Zariski-open subset $U$ of $\P(V)$ such that $p\notin U$ and, for all $q\in U$, we have $|L\cap Q|=2$, where $L$ denotes the line in $\P(V)$ joining $p$ and $q$.
\end{corollary}

\begin{proof}
 By Lemma \ref{l:critlines} there exist a hyperplane $H$ in $\P(V)$ and a quadric $Q'$ in $H$ such that $p\notin H$, $\rk(Q') = \rk(Q) - 1$, and \eqref{e:critlines} holds, where $G_p$ and $K$ stand for the set of all lines in $\P(V)$ passing through $p$ and the cone with base $Q'$ and apex $p$, respectively. Set $U := \P(V) \setminus K$. Then as $K$ is an algebraic set in $\P(V)$, $U$ is Zariski-open in $\P(V)$. Since $\rk(Q) \geq 2$ by assumption, we have $\rk(Q') = \rk(Q) - 1 \geq 1$, whence $H \setminus Q' \neq \emptyset$. Thus $U \neq \emptyset$ as $H\setminus Q' \subset U$. Since $Q' \neq \emptyset$, $p\in K$; thus $p\notin U$. Now let $q$ be an arbitrary element of $U$. Let $L$ be the line in $\P(V)$ joining $p$ and $q$. Then if $|L\cap Q| \neq 2$, we would have $L\subset K$ by \eqref{e:critlines} and therefore $q\in K$, which is clearly not the case looking at the definition of $U$. In conclusion, $|L\cap Q| = 2$.
\end{proof}

\begin{lemma}
 \label{l:fr}
 Let $A$ be a commutative $\C$-algebra, $V$ a finite dimensional $\C$-vector subspace of $A$, $r\in\N\setminus\{0\}$, and $Q$ a quadric of rank $\geq2$ in $\P(V)$. Assume that there exists $c\in V$ such that $c^{2r}\neq0$ in $A$. Moreover, assume that
 \[
  Q\subset R:=\{p\in\P(V):(\forall x\in p)x^{r+1}=0 \textup{ in }A\}.
 \]
 Then we have:
 \[
  Q=R=S:=\{p\in\P(V):(\forall x\in p)x^{2r}=0 \textup{ in }A\}.
 \]
\end{lemma}

\begin{proof}
 Obviously we have $Q\subset R\subset S$ since $r+1 \leq 2r$ (thus for all $x\in A$, $x^{r+1} = 0$ in $A$ implies $x^{2r} = 0$ in $A$). Hence, it suffices to show that $S\subset Q$. For that matter, let $p$ be an arbitrary element of $S$. Assume that $p$ is not an element of $Q$. Then by Corollary \ref{c:goodlines}, there exists a nonempty, Zariski-open subset $U$ of $\P(V)$ such that $p\notin U$ and, for all $q\in U$, one has $|L\cap Q| = 2$, where $L$ stands for the line in $\P(V)$ joining $p$ and $q$. Observe that $S$ is an algebraic set in $\P(V)$. Since there exists $c\in V$ such that $c^{2r} \neq 0$ in $A$, $\P(V)\setminus S\neq \emptyset$. Thus $U\setminus S \neq \emptyset$ and there exists $q\in U\setminus S$. Denote the line in $\P(V)$ joining $p$ and $q$ by $L$. Then $|L\cap Q| = 2$, \iev there exist $p'$ and $q'$ such that $p'\neq q'$ and $L\cap Q = \{p',q'\}$. Since $Q\subset R$, we have $p',q'\in R$, whence there exist $x\in p'$ and $y\in q'$ such that $x^{r+1} = y^{r+1} = 0$ in $A$. As the multiplication of $A$ is commutative, we infer using the binomial theorem that, for all $\lambda,\mu\in\C$:
 \begin{equation} \label{e:frbinom}
  (\lambda x+\mu y)^{2r} = \sum_{j=0}^{2r}\binom{2r}{j}(\lambda x)^{2r-j}(\mu y)^j = \binom{2r}{r}(\lambda^r\mu^r)x^ry^r.
 \end{equation}
 Let $v\in p$ and $w\in q$. Since $x$ and $y$ span the line $L$, there are $\lambda,\mu,\lambda_1,\mu_1\in \C$ such that $v = \lambda x + \mu y$ and $w = \lambda_1 x + \mu_1 y$. Since $q\notin S$, we have $w^{2r}\neq 0$ in $A$. Therefore, substituting $\lambda_1$ and $\mu_1$ for $\lambda$ and $\mu$, respectively, in \eqref{e:frbinom}, we see that $x^ry^r\neq 0$ in $A$. As $p\in S$, we have $v^{2r} = 0$. Thus now, \eqref{e:frbinom} yields that $\lambda^r\mu^r=0$. In turn, $\lambda=0$ and consequently $p=q'$ or else $\mu=0$ and consequently $p=p'$. Either way, $p\in Q$. This argument shows that indeed, for all $p\in S$, we have $p\in Q$. In other words, $S\subset Q$, quod erat demonstrandum.
\end{proof}

\begin{proposition}
 \label{p:frcrit}
 Let $X$ be a compact, connected, symplectic complex space of Kähler type such that $\Omega^2_X(X_\reg)$ is of dimension $1$ over $\C$. Assume that
 \begin{equation} \label{e:frcrit}
  Q_X\subset \{l\in\Gr(1,\H^2(X,\C)):(\forall v\in l)v^{r+1}=0 \textup{ in } \H^*(X,\C)\},
 \end{equation}
 where $r:=\nicefrac{1}{2}\dim(X)$. Then $X$ satisfies the Fujiki relation.
\end{proposition}

\begin{proof}
 Define $A$ to be the complex subalgebra $\H^{2*}(X,\C)$ of $\H^*(X,\C)$ and $V$ to be the $\C$-submodule of $A$ which is given by the canonical image of $\H^2(X,\C)$ in $A$. Then $A$ is a commutative $\C$-algebra and $V$ is a finite dimensional $\C$-submodule of $A$ (as $X$ is compact). Evidently, $r$ is a natural number different from $0$. Set
 \[
  Q := \{p\in\P(V):(\forall x\in p)q_X(\psi(x))=0\},
 \]
 where $\psi$ denotes the inverse function of the canonical isomorphism $\H^2(X,\C) \to V$. Then $Q$ is a quadric of rank $\geq2$ in $\P(V)$ by Corollary \ref{c:bbquadirred}. By Proposition \ref{p:sympclex}, there exists a symplectic class $w$ on $X$. We identify $w$ with its image in $V$. As in the proof of Proposition \ref{p:frequiv}, we calculate in $A$ (or $\H^*(X,\C)$):
 \[
  (w + \bar w)^{2r} = \binom{2r}{r} w^r\bar w^r.
 \]
 By Proposition \ref{p:sympvol+}, $\int_X w^r\bar w^r > 0$; in particular, $(w + \bar w)^{2r} \neq 0$ in $A$. By \eqref{e:frcrit}, we see that
 \[
  Q \subset \{p\in\P(V):(\forall x\in p)x^{r+1}=0 \text{ in }A\}.
 \]
 Hence by Lemma \ref{l:fr},
 \[
  Q=\{p\in\P(V):(\forall x\in p)x^{2r}=0 \text{ in }A\}.
 \]
 This implies that (iii) of Proposition \ref{p:frequiv} holds. Applying Proposition \ref{p:frequiv}, one infers that $X$ satisfies the Fujiki relation.
\end{proof}

\begin{lemma}
 \label{l:frfam0}
 Let $f\colon X\to S$ be a proper, equidimensional morphism of complex spaces and $t\in S$ such that $X_t$ is connected and symplectic, $\Omega^2_{X_t}((X_t)_\reg)$ is of dimension $1$ over the field of complex numbers, $S$ is simply connected, and the fibers of $f$ are of Fujiki class $\sC$. Moreover, assume that, for all $i\in\N$, the sheaf $\H^i(f,\C)$ (\cf \eqref{ss:Avalcoh}) is locally constant on $S_\top$. Set $\cP:=\cP^{2,2}_t(f)_\MHS$. Suppose that the mapping
 \[
  \cP \colon S \to G := \Gr(1,\H^2(X_t,\C))
 \]
 is holomorphic and there exists a neighborhood $W$ of $\cP(t)$ in $G$ such that
 \[
  Q_{X_t}\cap W \subset \cP(S).
 \]
 Then $X_t$ satisfies the Fujiki relation.
\end{lemma}

\begin{proof}
 Define $r$ to be the unique natural number such that $2r=\dim(X_t)$. Further on, introduce the following notation, where we exponentiate in the ring $\H^*(X_t,\C)$, into which $\H^2(X_t,\C)$ embeds canonically, and the ``$0$'' shall be considered the zero element of the ring $\H^*(X_t,\C)$:
 \begin{align*}
  R_t & := \{l\in |G|:(\forall v\in l)v^{r+1}=0\}, \\
  S_t & := \{l\in |G|:(\forall v\in l)v^{2r}=0\}.
 \end{align*}
 As $f$ is locally topologically trivial at $t$ and the topological space $S_\top$ is locally connected, there exists a connected open neighborhood $V$ of $t$ in $S_\top$ such that the sheaf $\H^*(f,\C)$ is constant on $S_\top$ over $V$. In particular, for all $s\in V$, the canonical maps
 \[
  (\H^*(f,\C))(V) \to (\H^*(f,\C))_s
 \]
 are bijective.

 Let now $s\in V$ be arbitrary and $v\in \cP(s)$. Put $\rho := \rho^2(f,\C)$. Since $S_\top$ is simply connected, there exists a unique element $a$ in $(\Pi(S_\top))_1(s,t)$. Put $\phi_{s,t} := (\rho_1(s,t))(a)$ and recall that thus
 \[
  \phi_{s,t} \colon \H^2(X_s,\C) \to \H^2(X_t,\C)
 \]
 is an isomorphism of complex vector spaces. Therefore, there exists a unique element $u \in \H^2(X_s,\C)$ such that $\phi_{s,t}(u) = v$; note that $\cP(s)\subset \H^2(X_t,\C)$, whence $v \in \H^2(X_t,\C)$. By the definition of $\cP^{2,2}_t(f)_\MHS$, we have
 \[
  \cP(s) = \phi_{s,t}[\F^2\H^2(X_s)],
 \]
 which tells us that $u\in \F^2\H^2(X_s)$. Since $f$ is equidimensional, we have $\dim(X_s) = \dim(X_t) = 2r$. Consequently, we have $u^{r+1}=0$ in $\H^*(X_s,\C)$ since the cohomological cup product is filtered with respect to the Hodge filtrations on the cohomology of $X_s$. As the morphism $f$ is proper, the base change map
 \[
  (\H^2(f,\C))_t \to \H^2(X_t,\C)
 \]
 is a bijection. Therefore, there exists one, and only one, $\gamma \in (\H^2(f,\C))(V)$ which is sent to $v$ by the composition of functions:
 \begin{equation} \label{e:frfam-stalk}
  (\H^2(f,\C))(V) \to (\H^2(f,\C))_t \to \H^2(X_t,\C).
 \end{equation}
 By the definition of $\rho^2(f,\C)$, we know that $\gamma$ is sent to $u$ by the composition of functions \eqref{e:frfam-stalk}, where we replace $t$ by $s$. Clearly, the following diagram of sets and functions commutes:
 \begin{equation} \label{e:frfam-0}
  \xymatrix{
   (\H^2(f,\C))(V) \ar[r] \ar[d] & (\H^2(f,\C))_s \ar[r] \ar[d] & \H^2(X_s,\C) \ar[d] \\
   (\H^*(f,\C))(V) \ar[r] & (\H^*(f,\C))_s \ar[r] & \H^*(X_s,\C) \\   
  }
 \end{equation}
 Thus the canonical image of $\gamma$ in $(\H^*(f,\C))(V)$ is mapped to the canonical image of $u$ in $\H^*(X_s,\C)$ by the composition of the two functions in the bottom row of \eqref{e:frfam-0}. In addition, the composition of the two functions in the bottom row of \eqref{e:frfam-0} is a morphism of rings from $(\H^*(f,\C))(V)$ to $\H^*(X_s,\C)$. Thus the composition of the two functions in the bottom row of \eqref{e:frfam-0} sends $\gamma^{r+1}$ to $u^{r+1}$, where $\gamma$ and $u$ denote the canonical images of $\gamma$ and $u$ in $(\H^*(f,\C))(V)$ and $\H^*(X_s,\C)$, respectively. Employing the fact that the composition of the two functions in the bottom row of \eqref{e:frfam-0} is one-to-one (it is a bijection actually), we further infer that $\gamma^{r+1}=0$ in $(\H^*(f,\C))(V)$. Playing the same game as before with $t$ in place of $s$, we see that $\gamma^{r+1}$ is sent to $v^{r+1}$ by the composition of functions
 \[
  (\H^*(f,\C))(V) \to (\H^*(f,\C))_t \to \H^*(X_t,\C),
 \]
 whence $v^{r+1}=0$ in $\H^*(X_t,\C)$. As $v$ was an arbitrary element of $\cP(s)$ and $s$ was an arbitrary element of $V$, we obtain that $\cP(V)\subset R_t$.

 As $\cP\colon S\to G$ is a holomorphic mapping, $S$ is a smooth and connected complex space, $R_t$ is a closed analytic subset of $G$, and $V$ is an open subset of $S$, the identity theorem for holomorphic mappings yields $\cP(S)\subset R_t$. We know there exists an open neighborhood $W$ of $\cP(t)$ in $G$ such that $Q_{X_t}\cap W\subset \cP(S)$. In consequence, $Q_{X_t}\cap W\subset R_t$. As $\cP(t)\in Q_{X_t}$, we know that $Q_{X_t}\cap W$ is nonempty. Thus since $Q_{X_t}$ is an irreducible closed analytic subset of $G$ and $R_t$ is a closed analytic subset of $G$, we have $Q_{X_t}\subset R_t$. Consulting Proposition \ref{p:frcrit}, it follows that $X_t$ satisfies the Fujiki relation.
\end{proof}

\begin{definition}
 \label{d:loctoptriv}
 \begin{enumerate}
  \item  Let $f\colon E\to B$ be a morphism of topological spaces and $b\in B$. $f$ is said to be a \emph{(topological) fiber bundle at} $t$ when there exist an open neighborhood $V$ of $b$ in $B$, a topological space $F$, and a homeomorphism
  \[
   h \colon X|f^{-1}(V) \to (B|V)\times F
  \]
  such that $f = \pr_0\circ h$, where $\pr_0$ denotes the projection to $B|V$.
  \item Let $f\colon X\to S$ be a morphism of ringed spaces (or else complex spaces) and $t\in S$. Then $f$ is said to be \emph{locally topologically trivial at} $t$ when $f_\top\colon X_\top \to S_\top$ is a fiber bundle at $t$.
 \end{enumerate}
\end{definition}

\begin{lemma}
 \label{l:frloctriv}
 Let $f\colon X\to S$ be a proper, flat morphism of complex spaces and $t\in S$ such that $X_t$ is connected and symplectic, $\Omega^2_{X_t}((X_t)_\reg)$ is of dimension $1$ over the field of complex numbers, $S$ is smooth and simply connected, and the fibers of $f$ have rational singularities, are of Kähler type, and have singular loci of codimension $\geq 4$. Denote $g\colon Y\to S$ the submersive share of $f$ and set $\cP':=\cP^{2,2}_t(g)$. Assume that the tangent map
 \[
  \T_t(\cP') \colon \T_S(t) \to \T_{\Gr(\sH^2(Y_t))}(\F^2\sH^2(Y_t))
 \]
 is an injection with $1$-dimensional cokernel and $f$ is locally topologically trivial at $t$. Then $X_t$ satisfies the Fujiki relation.
\end{lemma}

\begin{proof}
 As $S$ is smooth (at $t$) and $f$ is locally topologically trivial at $t$, there exists an open neighborhood $V$ of $t$ in $S$ such that $S|V$ is isomorphic to the $d$\hyphen dimensional complex unit ball for some $d\in\N$ and $(f_V)_\top$ is isomorphic to the projection $(S|V)_\top\times F \to (S|V)_\top$ for some topological space $F$. This last condition implies that, for all $i\in\N$, the sheaf $\H^i(f_V,\C)$ is a constant sheaf on $S|V$. Moreover, $f_V\colon X_V\to S|V$ is a proper, flat morphism of complex spaces, $S|V$ is smooth and simply connected, and the fibers of $f_V$ have rational singularities, are of Kähler type, and have singular loci of codimension $\geq 4$; $(X_V)_t$ is connected and symplectic and
 \[
  \dim_\C \left(\Omega^2_{(X_V)_t}(((X_V)_t)_\reg) \right) = 1.
 \]

 Let $g_V\colon Y_V\to S|V$ be the submersive share of $f_V$ and set $\cP'_V := \cP^{2,2}_t(g_V)$. Let $i_t\colon (Y_V)_t \to Y_t$ be the canonical morphism of complex spaces, which is induced by the fact that $Y_V$ is an open complex subspace of $Y$. Write
 \[
  i_t^* \colon \sH^2(Y_t) \to \sH^2((Y_V)_t)
 \]
 for the associated morphism taking algebraic de Rham cohomology. As $i_t$ is an isomorphism of complex spaces, $i_t^*$ is an isomorphism in $\Mod(\C)$. Hence, $i_t^*$ gives rise to an isomorphism of complex spaces
 \[
  \Gr(\sH^2(Y_t)) \to \Gr(\sH^2((Y_V)_t))
 \]
 that we sloppily denote $i_t^*$, too. Given this notation we find that
 \[
  \cP'_V = i_t^* \circ \cP'|V.
 \]
 Thus as $\T_t(\cP')$ is an injection with $1$-dimensional cokernel, $\T_t(\cP'_V)$ is an injection with $1$-dimensional cokernel. Define $\cP:=\cP^{2,2}_t(f_V)_\MHS$ in the sense of
 \[
  \cP \colon S|V \to G := \Gr(\H^2((X_V)_t,\C))
 \]
 being a morphism of complex spaces. Then by Theorem \ref{t:ltfull}, there exists a morphism of complex spaces $\bar\cP\colon S|V\to Q_{(X_V)_t}$ such that $\cP=i\circ \bar\cP$, where $i$ denotes the inclusion morphism $Q_{(X_V)_t} \to G$, and $\bar\cP$ is biholomorphic at $t$. In particular, there exist open neighborhoods $V'$ of $t$ in $S|V$ and $W'$ of $\bar\cP(t)$ in $Q_{(X_V)_t}$ such that $\bar\cP$ induces an isomorphism from $(S|V)|V'$ to $Q_{(X_V)_t}|W'$. Specifically, we have $\bar\cP(V')=W'$. Since $(Q_{(X_V)_t})_\top$ is a topological subspace of $G_\top$, there exists an open subset $W$ of $G_\top$ such that $W'=Q_{(X_V)_t}\cap W$. Therefore, exploiting the fact that $\cP$ and $\bar\cP$ agree set-theoretically, $W$ is an open neighborhood of $\cP(t)$ in $G_\top$ and
 \[
  Q_{(X_V)_t} \cap W = \cP(V') \subset \cP(S).
 \]
 Now from Lemma \ref{l:frfam0} (applied to $f_V$ and $t$), we infer that $(X_V)_t$ satisfies the Fujiki relation. As $(X_V)_t \iso X_t$ in $\An$, we deduce that $X_t$ satisfies the Fujiki relation.
\end{proof}

\begin{theorem}
 \label{t:kuranishi}
 For all compact complex spaces $X$ there exist a proper, flat morphism of complex spaces $f\colon \cX \to S$ and $t\in S$ such that $X \iso \cX_t$ in $\An$ and the morphism $f$ is semi-universal in $t$ and complete in $s$ for all $s\in S$.
\end{theorem}

\begin{proof}
 The statement is proven as ``Hauptsatz'' in \cite[\S5]{Gr74}\footnote{One should note that in \cite{Gr74} Grauert calls ``versell'' what we call ``semi-universal''; for the modern reader this might be confusing given that today people use the (English) word ``versal'' as a synonym for ``complete'', which is a condition strictly weaker than that of semi-universality. So, Grauert's (German) ``versell'' is not equivalent to but strictly stronger than the contemporary (English) ``versal''.}.
\end{proof}

\begin{theorem}
 \label{t:fr}
 Let $X$ be a compact, connected, symplectic complex space of Kähler type such that $\dim_\C(\Omega^2_X(X_\reg))=1$ and $\codim(\Sing(X),X) \geq 4$. Then $X$ satisfies the Fujiki relation.
\end{theorem}

\begin{proof}
 By Theorem \ref{t:kuranishi}, as $X$ is compact complex space, there exist a proper, flat morphism of complex spaces $f_0 \colon \cX_0 \to S_0$ and $t\in S_0$ such that $f_0$ is semi-universal in $t$ and $X \iso (\cX_0)_t$ in $\An$. By Theorem \ref{t:defsm}, the complex space $S_0$ is smooth at $t$. By Corollary \ref{c:stabsymp} (or else Theorem \ref{t:stabsymp}), there exists a neighborhood $V_0$ of $t$ in $S_0$ such that, for all $s\in V_0$, the complex space $(\cX_0)_s$ is connected, symplectic, of Kähler type, and with codimension of its singular locus not deceeding $4$. By the smoothness of $S_0$ at $t$, there exists an open neighborhood $V_0'$ of $t$ in $S_0$ such that $V_0' \subset V_0$ and $S_0|V_0'$ is isomorphic in $\An$ to some complex unit disk. Define $f_1\colon \cX_1\to S_1$ to be the morphism of complex spaces obtained from $f_0$ by shrinking the base to $V_0'$. Then by Theorem \ref{t:lt}, the tangent map $\T_t(\cP_1')$ is an injection with cokernel of dimension $1$, where $\cP_1' := \cP^{2,2}_t(g_1)$ and $g_1$ denotes the submersive share of $f_1$. Since $\cP_1'$ is a morphism between two complex manifolds, there exists an open neighborhood $V_1$ of $t$ in $S_1$ such that $\T_s(\cP_1')$ is an injection with cokernel of dimension $1$ for all $s\in V_1$. Let $V_1'$ be an open neighborhood of $t$ in $S_1$ such that $V_1' \subset V_1$ and $S_1|V_1'$ is isomorphic in $\An$ to some complex unit disk, and define $f\colon \cX\to S$ to be the morphism of complex spaces obtained from $f_1$ by shrinking the base to $V_1'$.
 
 Define $g\colon \cY\to S$ to be the submersive share of $f$ and set $\cP' := \cP^{2,2}_t(g)$. Then $g$ equals the morphism obtained from $g_1$ by shrinking the base to $V_1'$. In consequence, $\cP'$ is isomorphic to $\cP_1'|V_1'$ as an arrow under $S = S_1|V_1'$ in $\An$. In particular, for all $s\in |S| = V_1'$, as $\T_s(\cP_1')$ is an injection with cokernel of dimension $1$, the map $\T_s(\cP_1'|V_1')$ is an injection with cokernel of dimension $1$, hence $\T_s(\cP')$ is an injection with cokernel of dimension $1$. Note that $\cP^{2,2}_s(g)$ is isomorphic to $\cP^{2,2}_t(g)$ as an arrow under $S$ in $\An$ for all $s\in S$. Thus, $\T_s(\cP^{2,2}_s(g))$ is an injection with cokernel of dimension $1$ for all $s\in S$. Note that on account of Proposition \ref{p:ratgor}, for all $s\in S$, as the complex space $\cX_s \iso (\cX_0)_s$ is symplectic, $\cX_s$ has rational singularities. Therefore, by Corollary \ref{c:froertlsing}, the module $\sH^{2,0}(g)$ is locally finite free on $S$ and, for all $s\in S$, the Hodge base change
 \[
  (\sH^{2,0}(g))(s) \to \sH^{2,0}(\cY_s)
 \]
 is an isomorphism in $\Mod(\C)$. Observe that, for all $s\in S$, we have $\cY_s \iso (\cX_s)_\reg$ due to the flatness of $f$; thus:
 \[
  \sH^{2,0}(\cY_s) \iso \sH^{2,0}((\cX_s)_\reg) \iso \Omega^2_{(\cX_s)_\reg}((\cX_s)_\reg) \iso \Omega^2_{\cX_s}((\cX_s)_\reg)
 \]
 in $\Mod(\C)$. Since $S$ is connected and $\Omega^2_{\cX_t}((\cX_t)_\reg)$ is $1$-dimensional (given $\cX_t \iso (\cX_0)_t \iso X$), we infer that the module $\sH^{2,0}(g)$ is locally free of rank $1$ on $S$ and that, for all $s\in S$, the complex vector space $\Omega^2_{\cX_s}((\cX_s)_\reg)$ is $1$-dimensional. By stratification theory and Thom's First Isotopy Lemma, for $f$ being proper, there exists a connected open subset $T$ of $S$ such that $t \in \bar T$ and $f$ is locally topologically trivial at $s$ for all $s\in T$. Employing Lemma \ref{l:frloctriv}, we obtain that, for all $s\in T$, the space $\cX_s$ satisfies the Fujiki relation.
 
 Now let $a\in\H^2(\cX_t,\C)$. Since $f$ is proper, the base change map
 \[
  (\H^2(f,\C))_t \to \H^2(\cX_t,\C)
 \]
 is a surjection (a bijection, in fact). Therefore, there exists a connected open neighborhood $V$ of $t$ in $S$ and $\tilde a \in (\H^2(f,\C))(V)$ such that $\tilde a$ is sent to $a$ by the composition of canonical functions
 \begin{equation} \label{e:h2vtocoh}
  (\H^2(f,\C))(V) \to (\H^2(f,\C))_t \to \H^2(\cX_t,\C).
 \end{equation}
 As $t \in \bar T$ and the topological space $S_\top$ is first-countable at $t$, there exists an $\N$-indexed sequence $(s_\alpha)$ of elements of $T \cap V$ such that $(s_\alpha)$ converges to $t$ in $S_\top$. When $s\in V$, let us write $a_s$ for the image of $\tilde a$ under the composition of functions \eqref{e:h2vtocoh}, where we replace $t$ by $s$. Set $r:=\nicefrac{1}{2}\dim(\cX_t)$ (which makes sense as the complex space $\cX_t$ is nonempty, connected, and symplectic) and define:
 \[
  \phi := \{(s,\int_{\cX_s}(a_s)^{2r}):s\in V\} \quad \text{and} \quad \psi :=\{(s,q_{\cX_s}(a_s)):s\in V\}.
 \]
 Then $\phi$ is a locally constant function from $S_\top|V$ to $\C$. In particular, we have $\lim(\phi(s_\alpha)) = \phi(t)$. In a separate proof below we show that $\lim(\psi(s_\alpha)) = \psi(t)$. For all $\alpha \in \N$, since $s_\alpha \in T$, the space $\cX_{s_\alpha}$ satisfies the Fujiki relation and thus, since $\nicefrac{1}{2}\dim(\cX_{s_\alpha}) = \nicefrac{1}{2}\dim(\cX_t) = r$, we have:
 \[
  \phi(s_\alpha) = \binom{2r}{r}(\psi(s_\alpha))^r.
 \]
 Accordingly, remarking that by the definition of $a_t$ we have $a = a_t$, we obtain:
 \begin{align} \label{e:fr-xt}
  \int_{\cX_t}a^{2r} & = \int_{\cX_t}(a_t)^{2r} = \phi(t) = \lim(\phi(s_\alpha)) = \lim\left(\binom{2r}{r}(\psi(s_\alpha))^r\right) \\
  & = \binom{2r}{r}(\psi(t))^r = \binom{2r}{r}(q_{\cX_t}(a_t))^r = \binom{2r}{r}(q_{\cX_t}(a))^r. \notag
 \end{align}
 As in the deduction of \eqref{e:fr-xt} $a$ was an arbitrary element of $\H^2(\cX_t,\C)$, we see that $\cX_t$ satisfies the Fujiki relation. In consequence, as $X \iso (\cX_0)_t \iso \cX_t$ in $\An$, we see that $X$ satisfies the Fujiki relation.
\end{proof}

\begin{proof}[Proof that $\lim(\psi(s_\alpha)) = \psi(t)$]
 Assume that $a=0$ in $\H^2(\cX_t,\C)$. Then we have $(\tilde a)_t=0$ in the stalk $(\H^2(f,\C))_t$. Thus there exists an open neighborhood $W$ of $t$ in $S$ such that $W \subset V$ and $\tilde a$ restricts to the zero element of $(\H^2(f,\C))(W)$ within the presheaf $\H^2(f,\C)$. This implies that, for all $s\in W$, we have $a_s=0$ in $\H^2(\cX_s,\C)$ and consequently $q_{\cX_s}(a_s)=0$. Therefore clearly, $\lim(\psi(s_\alpha)) = \psi(t)$.

 Now, assume that $a\neq 0$ in $\H^2(\cX_t,\C)$. Then there exists $k\in\N_{>0}$ as well as a $k$-tuple $\basis b$ such that $\basis b$ is an ordered $\C$-basis for $\H^2(\cX_t,\C)$ and $\basis b_0=a$. For any $s\in S$, denote by $i_s$ the inclusion morphism complex spaces $(\cX_s)_\reg \to \cX_s$ and set $i_s^*:=\H^2(i_s,\C)$. Since $i_t^*$ is injective, there exists $l\in\N$ with $k\leq l$ and an $l$-tuple $\basis b'$ such that $\basis b'$ is an ordered $\C$-basis for $\H^2((\cX_t)_\reg,\C)$ and $\basis b'|k = i_t^* \circ \basis b$. Define $H'$ to be the module of horizontal sections of $\nabla^2_\GM(g)$. By Proposition \ref{p:rhc}, $H'$ is a locally constant sheaf on $S$. Since $S$ is simply connected, $H'$ is a constant sheaf on $S$ and the stalk map $H'(S) \to (H')_t$ is a bijection. Therefore, there exists a unique $l$-tuple $\basis v$ of elements of $H'(S)$ which is pushed forward to $\basis b'$ by the composition of evident functions:
 \begin{equation} \label{e:h'tocoh}
  H'(S) \to (H')_t \to b_t^*(\sH^2(g)) \to \sH^2((\cX_t)_\reg) \to \H^2((\cX_t)_\reg,\C).
 \end{equation}
 By Corollary \ref{c:froertlsing}, we have $\sH^{2,0}(g) \iso \F^2\sH^2(g)$ in $\Mod(S)$. Thus as the module $\sH^{2,0}(g)$ is locally free of rank $1$ on $S$, the module $\F^2\sH^2(g)$ is locally free of rank $1$ on $S$, and as $S$ is contractible, we even have $\O_S \iso \F^2\sH^2(g)$ in $\Mod(S)$. Since $\basis v$ is an ordered $\O_S(S)$-basis for $(\sH^2(g))(S)$, we may write the image of $1 \in \O_S(S)$ under the composition
 \[
  \O_S \to \F^2\sH^2(g) \to \sH^2(g)
 \]
 as $\sum_{i\in l}\lambda_i\basis v_i$ for some $l$-tuple $\lambda$ with values in $\O_S(S)$. For the time being, fix some $s\in S$. Define $\basis b_s'$ to be the pushforward of $\basis v$ under the composition of functions \eqref{e:h'tocoh}, where one replaces $t$ by $s$. Set:
 \[
  v'_s := \lambda_0(s)(\basis b'_s)_0 + \dots + \lambda_{l-1}(s)(\basis b'_s)_{l-1}.
 \]
 Then $v'_s$ is a symplectic class on $(\cX_s)_\reg$. Hence there exists one, and only one, $v_s \in \H^2(\cX_s,\C)$ such that $i_s^*(v_s) = v'_s$; moreover, $v_s$ is a symplectic class on $\cX_s$. Therefore, $I_s := \int_{\cX_s}v_s^r\bar{v_s}^r > 0$ (note that $\nicefrac{1}{2}\dim(\cX_s)=r$ since $\dim(\cX_s) = \dim(\cX_t)$). Assume that $s\in T$. Then $i_s^*$ is an isomorphism of $\C$-vector spaces; thus there exists one, and only one, $l$-tuple $\basis b_s$ of elements of $\H^2(\cX_s,\C)$ such that $i_s^*\circ \basis b_s = \basis b_s'$. We have $v_s = \sum_{i\in l}\lambda_i(s)(\basis b_s)_i$. Therefore,
 \[
  I_s = \sum_{\substack{\nu,\nu'\in\N^l \\ |\nu|=|\nu'|=r}}\binom{r}{\nu}\binom{r}{\nu'}(\lambda(s))^\nu (\bar{\lambda(s)})^{\nu'}\int_{\cX_s}(\basis b_s)^\nu(\bar{\basis b_s})^{\nu'}.
 \]
 Since, for any $i\in l$, the $(\basis b_s)_i$'s are induced by a section in $\H^2(f,\C)$ over $T$, we know that, for any $\nu,\nu'\in\N^l$, the function assigning $\int_{\cX_s}(\basis b_s)^\nu(\bar{\basis b_s})^{\nu'}$ to $s\in T$, is locally constant on $S$; as $T$ is a connected open subset of $S$, the latter functions are even constant. As $T\neq\emptyset$, there exists, for all $\nu,\nu'\in\N^l$, a unique complex number $C_{\nu,\nu'}$ such that, for all $s\in T$, we have $\int_{\cX_s}(\basis b_s)^\nu(\bar{\basis b_s})^{\nu'} = C_{\nu,\nu'}$.
 
 Since $\basis b$ is a basis for $\H^2(\cX_t,\C)$, we may write $v_t$ as linear combination of the $(\basis b_i)$, $i\in k$. Applying $i_t^*$ to this linear combination, comparing with the definition of $v'_t$ and noting that $\basis b'_t = \basis b'$, we see that
 \[
  v_t = \lambda_0(t)\basis b_0 + \dots + \lambda_{k-1}(t)\basis b_{k-1}
 \]
 and that
 \[
  \lambda_k(t) = \dots = \lambda_{l-1}(t) = 0.
 \]
 Since the base change map
 \[
  (\H^2(f,\C))_t \to \H^2(\cX_t,\C)
 \]
 is a surjection (indeed, it is even a bijection), there exists a connected open neighborhood $W$ of $t$ in $S$ as well as a $k$-tuple $\basis u$ of elements of $(\H^2(f,\C))(W)$ such that the pushforward of $\basis u$ under the composition of functions
 \begin{equation} \label{e:h2tocoh}
  (\H^2(f,\C))(W) \to (\H^2(f,\C))_t \to \H^2(\cX_t,\C)
 \end{equation}
 equals $\basis b$. Since $t \in \bar T$, there exists $s^* \in W \cap T$. It is an easy matter to verify that composing $\basis u$ with the composition of functions \eqref{e:h2tocoh}, where one replaces $t$ by $s^*$, we obtain $\basis b_{s^*}|k$. Therefore, for any $\zeta,\eta\in\N^k$ such that $|\zeta|+|\eta|=2r$, we have:
 \[
  \int_{\cX_t}{\basis b}^\zeta\bar{\basis b}^\eta = \int_{\cX_{s^*}}(\basis b_{s^*}|k)^\zeta (\bar{\basis b_{s^*}|k})^\eta = \int_{\cX_{s^*}}(\basis b_{s^*})^\zeta (\bar{\basis b_{s^*}})^\eta = C_{\zeta,\eta}.
 \]
 Thus:
 \[
  I_t = \sum_{\substack{\nu,\nu'\in\N^k \\ |\nu|=|\nu'|=r}}\binom{r}{\nu}\binom{r}{\nu'}(\lambda(t))^\nu (\bar{\lambda(t)})^{\nu'} C_{\nu,\nu'}.
 \]
 For any $i\in l$, the function assigning $\lambda_i(s)$ to $s\in S$, is a holomorphic function on $S$; in particular, it is a continuous function from $S_\top$ to $\C$, and $\lim(\lambda_i(s_\alpha)) = \lambda_i(t)$. In consequence, we see that: 
 \begin{align*}
  \lim(I_{s_\alpha}) & = \sum_{\substack{\nu,\nu'\in\N^l \\ |\nu|=|\nu'|=r}}\binom{r}{\nu}\binom{r}{\nu'}(\lambda(t))^\nu (\bar{\lambda(t)})^{\nu'}C_{\nu,\nu'} \\
  & = \sum_{\substack{\nu,\nu'\in\N^k \\ |\nu|=|\nu'|=r}}\binom{r}{\nu}\binom{r}{\nu'}(\lambda(t))^\nu (\bar{\lambda(t)})^{\nu'}C_{\nu,\nu'} = I_t.
 \end{align*}
 
 Now, for any $s\in S$, set $w_s := (\sqrt[2r]{I_s})^{-1}v_s$ in $\H^2(\cX_s,\C)$. Then clearly,  for all $s\in S$, $w_s$ is a normed symplectic class on $\cX_s$. Define $\mu$ to be the unique function on $S$ which sends $s\in S$ to the $l$-tuple $\mu_s$ such that, for all $i\in l$, we have:
 \[
  (\mu_s)_i = (\sqrt[2r]{I_s})^{-1}\lambda_i(s).
 \]
 Then, for all $s\in T$,
 \[
  w_s = (\mu_s)_0(\basis b_s)_0 + \dots + (\mu_s)_{l-1}(\basis b_s)_{l-1},
 \]
 and
 \[
  w_t = (\mu_t)_0\basis b_0 + \dots + (\mu_t)_{k-1}\basis b_{k-1}.
 \]
 Since $a = \basis b_0$ and $V$ is a connected open subset of $S$, it is easy to verify that, for all $s\in V \cap T$, we have $a_s = (\basis b_s)_0$. Hence, applying Proposition \ref{p:bbformres}, we infer that, for all $s\in V\cap T$:
 \begin{align*}
  q_{\cX_s}(a_s) & = \frac{r}{2} \int_{\cX_s}w_s^{r-1}\bar{w_s}^{r-1}a_s^2 + (1-r)\int_{\cX_s}w_s^{r-1}\bar{w_s}^ra_s \int_{\cX_s}w_s^r\bar{w_s}^{r-1}a_s \\
  & = \frac{r}{2} \sum_{\substack{\nu,\nu'\in\N^l \\ |\nu|=|\nu'|=r-1}} \binom{r-1}{\nu} \binom{r-1}{\nu'} (\mu_s)^\nu(\bar{\mu_s})^{\nu'}C_{\nu+(2,0,\dots,0),\nu'} \\
  & \qquad + (1-r)\left(\sum_{\substack{\nu,\nu'\in\N^l \\ |\nu|=r-1,|\nu'|=r}} \binom{r-1}{\nu} \binom{r}{\nu'}(\mu_s)^\nu(\bar{\mu_s})^{\nu'} C_{\nu+(1,0,\dots,0),\nu'}\right) \\
  & \qquad \cdot \left(\sum_{\substack{\nu,\nu'\in\N^l \\ |\nu|=r,|\nu'|=r-1}}\binom{r}{\nu}\binom{r-1}{\nu'}(\mu_s)^\nu(\bar{\mu_s})^{\nu'} C_{\nu+(1,0,\dots,0),\nu'}\right)
 \end{align*}
 Therefore, since we have $\lim(\mu_{s_\alpha}) = \mu_t$:
 \begin{align*}
  \lim(q_{\cX_{s_\alpha}}(a_{s_\alpha})) & = \frac{r}{2}\sum_{\substack{\nu,\nu'\in\N^l \\ |\nu|=|\nu'|=r-1}} \binom{r-1}{\nu} \binom{r-1}{\nu'}(\mu_t)^\nu(\bar{\mu_t})^{\nu'}C_{\nu+(2,0,\dots,0),\nu'} \\
  & + (1-r)\left(\sum_{\substack{\nu,\nu'\in\N^l \\ |\nu|=r-1,|\nu'|=r}}\binom{r-1}{\nu}\binom{r}{\nu'}(\mu_t)^\nu(\bar{\mu_t})^{\nu'} C_{\nu+(1,0,\dots,0),\nu'}\right) \\
  & \cdot\left(\sum_{\substack{\nu,\nu'\in\N^l \\ |\nu|=r,|\nu'|=r-1}}\binom{r}{\nu}\binom{r-1}{\nu'}(\mu_t)^\nu(\bar{\mu_t})^{\nu'} C_{\nu+(1,0,\dots,0),\nu'}\right) \\
  & = \frac{r}{2}\sum_{\substack{\nu,\nu'\in\N^k \\ |\nu|=|\nu'|=r-1}}\binom{r-1}{\nu}\binom{r-1}{\nu'}(\mu_t)^\nu(\bar{\mu_t})^{\nu'} C_{\nu+(2,0,\dots,0),\nu'} \\
  & + (1-r)\left(\sum_{\substack{\nu,\nu'\in\N^k \\ |\nu|=r-1,|\nu'|=r}}\binom{r-1}{\nu}\binom{r}{\nu'}(\mu_t)^\nu(\bar{\mu_t})^{\nu'} C_{\nu+(1,0,\dots,0),\nu'}\right) \\
  & \cdot\left(\sum_{\substack{\nu,\nu'\in\N^k \\ |\nu|=r,|\nu'|=r-1}}\binom{r}{\nu}\binom{r-1}{\nu'}(\mu_t)^\nu(\bar{\mu_t})^{\nu'} C_{\nu+(1,0,\dots,0),\nu'}\right) \\
  & = q_{\cX_t}(a_t).
 \end{align*}
 This shows that $\lim(\psi(s_\alpha)) = \psi(t)$.
\end{proof}

\appendix
\chapter{Foundations and conventions}
\label{ch:foun}

This appendix serves the purpose of increasing the rigor, the consistency, and the comprehensibility of the bulk of our text, \iev of Chapters \ref{ch:peri}, \ref{ch:froe}, and \ref{ch:symp}. More concretely, in what follows, we will lay out certain conventions of speech, certain notational conventions, as well as certain conventions concerning the axiomatic foundations.

Disclaimer: We do not intend to explain (in a textbook sort of way) any of the concepts that we are going to address in the subsequent. Quite on the contrary, will we assume the reader's familiarity with the mentioned concepts so that we merely clarify our points of view. For instance, when, in the realm of Definition \ref{d:cat}, we define what a category is, our intention is not to explain the notion of a category to a reader who has not heard of what a category is beforehand; rather would we like to tell the reader who is familiar with the concept of categories which of the various possible definitions (a single collection of morphisms versus a collection of hom-sets, small versus large, etc.) we adopt.

\section{Set Theory}
\label{s:st}

\subsection{}

Our exposition is formally based upon the Zermelo-Fraenkel axiomatic set theory with axiom of choice, ``ZFC'' (\cf \egv \cite[Chapter I.1]{Je02}). All statements we make can be formulated in ZFC. All of our proofs can be executed in ZFC.

\subsection{}

Note that occasionally we do work with classes. For instance, do we consider the large categories $\Mod(X)$ of modules on a given ringed space $X$ or the large category $\An$ of complex analytic spaces. We feel that our reader has two choices of how to deal with the occurrences of these concepts: The first possibility is to strictly stick to ZFC and interpret classes not as objects of the theory but as metaobjects, \iev at each occurrence of a class, the reader replaces the class by a formula in the language of ZFC which describes the class. The second possibility would be to use a conservative extension of ZFC, such as the von Neumann-Bernay-Gödel set theory (NBG), which can deal with classes as an additional type of the theory, as an overall foundation for our text. This second approach has the advantage that classes can be quantified over, so that, for instance, ``for all (large) abelian categories the five lemma holds'' or ``for all (large) abelian categories $C$, $D$, and $E$, all functors $f\colon C\to D$ and $g\colon D\to E$, and all objects $X$ of $C$ such that (\ldots) there exists a Grothendieck spectral sequence'' become valid statements. Some sources we refer to make (implicitly or explicitly) use of statements where classes are quantified over.

We point out that we do not presuppose any sort of universe axiom like \egv the existence of Grothendieck universes. Mind that, even though many authors in modern algebraic geometry seem to have forgotten about this, the standard construction of a (bounded) derived category fails for arbitrary large abelian categories since hom-sets, and not even hom-classes, do exist. Therefore, we cannot (and won't) talk about the $\D^+$ of $\Mod(X)$, for instance. Yet, we do fine without it.

\subsection{}

In our set theoretic notation (like
\[
 \{x:\phi(x)\}, \quad \subset, \quad \times, \quad \dom, \quad \ran, \quad X^Y, \quad \inf, \quad \sup,
\]
etc.) and terminology (like ``function'', ``domain'', ``range'', etc.) we follow \cite{Je02}. Occasionally, we will use the words \emph{mapping} and \emph{map} as synonyms for the word function. When $f$ is a function and $A$ is any set, we will occasionally write the image of $A$ under $f$ with square brackets instead of round brackets for sake of better readability in formulas that already contain several round brackets, \iev we write $f[A]$ instead of $f(A)$. The inverse image of $A$ under $f$ will be written $f^{-1}(A)$ as opposed to Jech's suggestion of $f_{-1}(A)$ in \loccit

\subsection{}

We denote $\N$ the \emph{set of natural numbers}. $\N$ is, by definition, the least nonzero limit ordinal \cite[Part I, Definition 2.13]{Je02}. Equivalently, $\N$ is the smallest inductive set, which is in turn nothing but the intersection of the class of inductive sets. The elements of $\N$ are also called \emph{finite ordinals}. For the first finite ordinals we have
\[
 0 = \emptyset, \quad 1 = 0\cup\{0\}, \quad 2 = 1\cup\{1\}, \quad 3 = 2\cup\{2\}, \quad \ldots
\]
In that respect, we note that for ordinal numbers, the $<$-relation coincides (by definition) with the $\in$-relation.

Sometimes, when we would like to stress the interpretation of $\N$ as the first infinite ordinal number, we will write $\omega$ instead of $\N$. For instance, when a nonempty complex space $X$ is not of finite dimension, we have $\dim(X) = \omega$; frequently, authors write this statement as $\dim(X) = \infty$ or $\dim(X) = +\infty$, which is of course not language immanent.

\subsection{}

By an \emph{ordered pair} we mean a Kuratowski pair; this is written $(x,y)$. We define (ordered) \emph{triples}, \emph{quadruples}, etc.\ by iteration of Kuratowski pairs, \iev
\[
 (a,b,c) := ((a,b),c), \quad (a,b,c,d) := ((a,b,c),d), \quad \ldots
\]
In contrast, when $r$ is a set (\egv a natural number, \cf above), an \emph{$r$-tuple} is nothing but a function with domain of definition equal to $r$; in this context we occasionally use the words \emph{family} or \emph{sequence} as synonyms for the word \emph{tuple}. Note that we obtain the competing notions of ``$2$-tuple'' versus ``ordered pair'', ``$3$-tuple'' versus ``ordered triple'', etc.

\subsection{}

Sometimes, especially when working with ``large'' categories and functors (\cf \S\ref{s:cat}), one is in need of notions of ordered pairs, triples, quadruples, etc.\ for classes rather than for sets. Let $I$ be any class or set. Then an \emph{$I$-tuple of classes} or \emph{$I$-family of classes} or \emph{$I$-tuple in the class sense} is a class $F$ with the property that every element of $F$ can be written in the form $(i,x)$, where $i$ is an element of $I$. When $F$ is an $I$-tuple of classes and $i \in I$, then the \emph{$i$(-th) component of $F$} is defined to be the class
\[
 F_i := \{x: (i,x) \in F\}.
\]
Now we use the words \emph{ordered pair}, \emph{triple}, \emph{quadruple}, etc.\ \emph{of classes} as synonyms for a $2$-tuple, $3$-tuple, $4$-tuple, etc.\ of classes, respectively. Mind the differences between the notions of an ordered pair (in the ordinary sense), a $2$-tuple (in the ordinary sense), and a $2$-tuple in the class sense.

\subsection{}

We use the symbols $\Z$, $\Q$, $\RR$, and $\C$ to denote the \emph{set of integers}, the \emph{set of rational numbers}, the \emph{set of real numbers}, and the \emph{set of complex numbers}, respectively, where we follow the constructions given in \cite{Zahlen}.

As is customary, we freely interpret any one of the canonical injections
\[
 \N \to \Z \to \Q \to \RR \to \C
\]
as an actual inclusion.

By abuse of notation, we denote $\Z$, $\Q$, $\RR$, and $\C$ also the \emph{ring of integers}, the \emph{field of rational numbers}, the \emph{field of real numbers}, and the \emph{field of complex numbers}, respectively.

\section{Categories}
\label{s:cat}

Our primary view on categories is the one which uses a family of hom-sets (\cf \cite[Definition 1.2.1]{KaSc06}, \cite[Section 1.1.1]{Lur09}).

\begin{definition}
 \label{d:cat}
 A set (\resp class) $C$ is a \emph{category} when there exist sets (\resp classes) $O$, $H$, $I$, and $V$ such that the following assertions hold:
 \begin{enumerate}
  \item $C = (O,H,I,V)$ (where the ordered quadruple possibly needs to be interpreted in the class sense).
  \item $H$ is a function on $O \times O$.
  \item $I$ is a function on $O$ such that, for all $x\in O$, we have $I(x)\in H(x,x)$.
  \item $V$ is a function on $O\times O\times O$ such that, for all $(x,y,z)\in O\times O\times O$, $V(x,y,z)$ is a function
  \[
   V(x,y,z) \colon H(y,z) \times H(x,y) \to H(x,z).
  \]
  \item (Units) For all $x,y\in O$ and all $f\in H(x,y)$, we have:
  \begin{align*}
   (V(x,x,y))(f,I(x)) = f, \\
   (V(x,y,y))(I(y),f) = f.
  \end{align*}
  \item (Associativity) For all $x,y,z,w\in O$ and all
  \[
   (h,g,f)\in H(z,w)\times H(y,z)\times H(x,y),
  \]
  we have:
  \[
   (V(x,y,w))\bigl((V(y,z,w))(h,g),f\bigr) = (V(x,z,w))\bigl(h,(V(x,y,z))(g,f)\bigr).
  \]
 \end{enumerate}
 From time to time, when we would like to stress the fact that a given category is either a set or a class, we use the following terminology: By a \emph{small category} we always mean a set $C$ which is a category; by a \emph{large category} we mean a class---which might or might not be a set---$C$ which is a category.
\end{definition}

\subsection{}
\label{ss:cat-not}

When $C$ is a category, then the four components $C_0$, $C_1$, $C_2$, and $C_3$ are well-defined. Occasionally, we use the alternative denominations
\begin{align*}
 \ob(C) & := C_0, & \id_C & := C_2, \\
 \hom_C & := C_1, & \circ_C & := C_3.
\end{align*}
$\ob(C)$ is called the \emph{set} (\resp \emph{class}) \emph{of objects of $C$}. $\hom_C$ is called the \emph{family of hom-sets of $C$}. $\id_C$ and $\circ_C$ are called the \emph{identity function} and \emph{composition function of $C$}, respectively.

Sometimes we write sloppily $x \in C$ instead of $x \in C_0$. Sometimes we write sloppily $C(x,y)$ instead of $C_1(x,y)$. When we feel that confusion is unlikely to arise, we might omit the index referring to $C$ in the expressions $\hom_C$, $\id_C$, and $\circ_C$. We use the customary infix notation
\[
 g \circ f := g \circ_C f := (\circ_C(x,y,z))(g,f).
\]

We say that \emph{$x$ is an object of/in $C$} when $x \in C_0$. We say that \emph{$f$ is a morphism from $x$ to $y$ in $C$} when $x$ and $y$ are objects of $C$ and $f \in C_1(x,y)$. Sometimes we paraphrase the latter statement by writing that ``\emph{$f\colon x\to y$ is a morphism in $C$}''.

\subsection{The language of diagrams}
\label{ss:diag}

Very frequently throughout the text we will paraphrase certain (set-theoretic, categorical, or algebraic) statements by saying that a ``diagram commutes'' in a given category. We exemplify this by looking at the picture (or ``diagram''):
\begin{equation} \label{e:cat-diag}
 \xysquare{x}{y}{x'}{y'}{f}{g}{g'}{f'}
\end{equation}
Note that we do not formalize what a diagram is. Rather, we content ourselves with saying that a diagram is a picture (like \eqref{e:cat-diag}, resembling a directed graph), in which one should be able to recognize a finite number of \emph{vertices} as well as a finite number of directed edges (or \emph{arrows}) which are drawn between the vertices, \iev for each arrow in the picture, it should be clear which of the vertices is the starting point and which of the vertices is the end point of the arrow. At each of the vertices in the picture we draw a symbolic expression corresponding to a term in our usual language. The arrows may or may not be labeled by similar symbolic expressions.

Let $C$ be a category (class or set). Then we say that \emph{the diagram \eqref{e:cat-diag} commutes in $C$} when the following assertions hold:
\begin{enumerate}
 \item \label{i:diag-v} $x$, $y$, $x'$, $y'$ are objects of $C$.
 \item \label{i:diag-e} We have $f \in C_1(x,y)$, $g \in C_1(x,x')$, $g' \in C_1(y,y')$, and $f' \in C_1(x',y')$.
 \item \label{i:diag-c} We have
 \[
  (\circ_C(x,y,y'))(g',f) = (\circ_C(x,x',y'))(f',g).
 \]
\end{enumerate}
From this we hope it is clear how to translate the phrase ``the diagram \ldots commutes in $C$'' into a valid formula of our set theoretic language for an arbitrary diagrammatic picture in place of the dots. Notice that the above assertions \ref{i:diag-v}) and \ref{i:diag-e}) correspond respectively to the facts that (the labels at) the vertices of the diagram are objects of $C$ and that, for every ordered pair of vertices of the diagram, every (label at an) arrow drawn between the given vertices is a matching morphism in $C$. So to speak, assertions \ref{i:diag-v}) and \ref{i:diag-e}) taken together say that the picture \eqref{e:cat-diag} \emph{is a diagram in $C$}. The final assertion \ref{i:diag-c}) says that the diagram is actually \emph{commutative in $C$}.

Observe that, strictly speaking, the phrase ``the diagram \ldots commutes in $C$'' makes sense only if all the arrows in the given diagram are actually labeled. Therefore, when we say the phrase referring to a diagram with some arrows unlabeled, we ask our readers to kindly guess the missing arrow labels from the individual context. Usually, in this regard, an unlabeled arrow corresponds to some sort of canonical morphism. At times, we will stress this point by labeling the arrow with a ``$\can$''.

Sometimes the concrete arrow labels (or corresponding morphisms) rendering a given diagram commutative in a category are irrelevant. In these cases we use the phrase ``\emph{there exists a commutative diagram \ldots in $C$}''. For instance, when we say there exists a commutative diagram
\[
 \xymatrix{ x \ar[r] \ar@/^1pc/[rr]^f & y \ar[r] & z }
\]
in $C$, we mean that there exist $f'$ and $f''$ such that the diagram
\[
 \xymatrix{ x \ar[r]_{f'} \ar@/^1pc/[rr]^f & y \ar[r]_{f''} & z }
\]
commutes in $C$ (in particular, we see that $x$, $y$, and $z$ need to be objects of $C$ and $f'$ needs to be an element of $C_1(x,y)$, etc.).

In some cases, where we want to draw attention to specific arrows in a diagram---mostly this happens when we assert that the morphisms corresponding to the arrows exist or exist in a unique way---, we print these arrows dotted as in:
\[
 \xymatrix{ x \ar@{.>}[r]^f & y }
\]
Statementwise, the dotted style of an arrow has no impact whatsoever.

Occasionally, instead of an ordinary arrow we will draw a stylized, elongated equality sign as, \egv in:
\[
 \xymatrix{x \ar@{=}[r] & y }
\]
These ``equality sign arrows'' go without label. Then, to the interpretation of the discussed commutativity statements, we have to add the requirement that $x = y$; moreover, in order to check the actual commutativity of the diagram (\cf assertion \ref{i:diag-c}) above), one simply merges (successively) any two vertices in the diagram which are connected by an equality sign arrow into a single vertex until one ends up with an equality sign free diagram. As an alternative, one replaces each occurrence of an equality sign arrow with an ordinary arrow, choosing the direction at will, and attaches the label $\id_C(*)$ to it, where $*$ is to be replaced by the label at the chosen starting point. So, the previous diagram becomes either
\[
 \xymatrix@C=3pc{x \ar[r]^{\id_C(x)} & y} \qquad \text{or} \qquad \xymatrix@C=3pc{x & y \ar[l]_{\id_C(y)}}
\]

Last but not least, we occasionally draw ``$\sim$'' signs at arrows in diagrams (possibly in addition to already existing labels), like, for instance, in:
\[
 \xymatrix{x \ar[r]_f^\sim & y}
\]
In these cases, we add the requirement that $f$ be an isomorphism from $x$ to $y$ in $C$ to the interpretation of any commutativity statement.

\subsection{}
\label{ss:cat-stdcat}

We use the following notation for the ``standard categories'', where the respective rigorous definitions are to be modeled after Definition \ref{d:cat}:
\begin{table}[ht]
\centering
\begin{tabular}{ll}
 $\Set$ & the category of sets \\
 $\Top$ & the category of topological spaces \\
 $\Sp$  & the category of ringed spaces \\
 $\An$  & the category of complex analytic spaces
\end{tabular}
% \caption{Standard categories}
\end{table}

\subsection{}
\label{ss:proset}

Let $(X,\leq)$ be a preordered set. Put $C_0:=X$, let
\[
 C_1\colon X\times X\to\{0,1\}
\]
be the unique function such that, for all $x,y\in X$, we have $C_1(x,y)=1$ if and only if $x \leq y$; let $\id_C := X \times \{0\}$, and let
\[
 \circ_C \colon X\times X\times X \to \{0,(1\times1)\times1\}
\]
be the unique function such that, for all $x,y,z\in X$, we have $\circ_C(x,y,z) = (1\times1)\times1$ if and only if $x \leq y$ and $y \leq z$. Then $C = (C_0,C_1,\id_C,\circ_C)$ is a small category, which we call the \emph{preorder category associated to $(X,\leq)$}.

Let $r$ be a natural number (or, more generally, any ordinal number). Then the $\in$-relation induces a preorder on $r$, so that we obtain a canonical preorder category whose set of objects is simply $r$ from the preceding construction. For the natural numbers $0$, $1$, $2$, $3$, etc.\ we denote the so associated preorder categories by $\mathbf{0}$, $\mathbf{1}$, $\mathbf{2}$, $\mathbf{3}$, etc.

\subsection{}
\label{ss:cat-op}

When $C$ is a category, we denote $C^\op$ the \emph{opposite category} of $C$.

\subsection{}
\label{ss:fun}

In the spirit of Definition \ref{d:cat}, when $C$ and $D$ are categories (both either classes or sets), for us, a \emph{functor from $C$ to $D$} is formally an ordered pair $F = (F_0,F_1)$ (possibly in the class sense), where
\[
 F_0 \colon C_0 \to D_0
\]
is a function and $F_1$ is a function defined on $C_0 \times C_0 = \dom(C_1)$ such that, for all $(x,y) \in C_0 \times C_0$,
\[
 F_1(x,y) \colon C_1(x,y) \to D_1(F_0(x),F_0(y))
\]
is a function such that the well-known compatibilities with the identities and compositions of $C$ and $D$ are fulfilled. We will also use the notation
\[
 F \colon C \to D
\]
for the fact that $F$ is a functor from $C$ to $D$.

When $F$ is a functor from $C$ to $D$, we denote the uniquely determined components of $F$ by $F_0$ and $F_1$, respectively. We call $F_0$ the \emph{object function} of $F$. We call $F_1$ the \emph{morphism function} of $F$ or the \emph{family of morphism functions} of $F$.

Most of the time, we will sloppily write $F(x)$ instead of $F_0(x)$ and $F(x,y)$ instead of $F_1(x,y)$. When $x$ and $y$ are objects of $C$ and $f\colon x\to y$ is a morphism in $C$, \iev $f \in C_1(x,y)$, we will sloppily write $F(f)$ instead of $(F(x,y))(f)$. When $x$ and $y$ are objects of $C$ such that there exists a unique morphism $f\colon x\to y$ in $C$, we will sloppily write $F(x,y)$ instead of $(F(x,y))(f)$.

\subsection{}
\label{ss:nat}

Let $C$ and $D$ be two categories. Let $F$ and $G$ be two functors from $C$ to $D$. Then $\alpha$ is a \emph{natural transformation of functors from $C$ to $D$ from $F$ to $G$} when $\alpha$ is a function on $C_0$ such that, for all $x,y\in C_0$ and all $f\in C_1(x,y)$, the following diagram commutes in $C$:
\[
 \xysquare{F_0(x)}{G_0(x)}{F_0(y)}{G_0(y)}{\alpha(x)}{(F_1(x,y))(f)}{(G_1(x,y))(f)}{\alpha(y)}
\]
In particular, for all $x\in C_0$, we need to have
\[
 \alpha(x) \in D_1(F_0(x),G_0(x)).
\]
We will also write ``\emph{$\alpha\colon F\to G$ is a natural transformation of functors from $C$ to $D$}'' for the fact that $\alpha$ is a natural transformation of functors from $C$ to $D$ from $F$ to $G$.

When $C$ is a small category, any functor from $C$ to $D$ is a set; moreover, for two fixed functors $F$ and $G$, any natural transformation of functors from $C$ to $D$ from $F$ to $G$ is a set and the class of all natural transformations of functors from $C$ to $D$ from $F$ to $G$ is a set, too. Thus we can consider the \emph{functor category}, denoted $\Fun(C,D)$ or $D^C$ (we are confident that the reader can model the precise definition of the functor category after Definition \ref{d:cat} himself).

\subsection{}
\label{ss:triples}

Let $C$ be a category. Then the functor category $C^{\mathbf3} = \Fun(\mathbf3,C)$ is called the \emph{category of triples in $C$}. A \emph{triple in $C$} is an object of $C^{\mathbf3}$, \iev a functor
\[
 t \colon \mathbf3 \to C.
\]
Sometimes we will write a triple $t$ in $C$ in the form $t \colon x \to y \to z$; by this we mean that $t(0) = x$, $t(1) = y$, and $t(2) = z$.

\subsection{The absolute view}

The alternative to our definition of a category is the definition with a single set or class of morphisms together with a domain function and a codomain function (\cf \egv \cite{MLa98}); we call this the ``absolute view''. It is a standard technique to switch from our point of view to the absolute view. Namely, given a category $C$, one defines the set (or class) of absolute morphisms of $C$ to be
\[
 \{(x,f,y) : x,y \in C_0 \text{ and } f \in C_1(x,y)\}.
\]
The domain function and the codomain function are the obvious ones. We call a triple $(x,f,y)$ as above an \emph{absolute morphism} in $C$.

Throughout our text, the trick of switching to the absolute view will be ubiquitous. Let us illustrate this with two prototypical examples. For one, consider two topological spaces $X$ and $Y$ as well as a continuous map, \iev morphism of topological spaces, $f\colon X\to Y$. To this situation we have associated a direct image functor
\[
 f_* \colon \Sh(X) \to \Sh(Y).
\]
Consequently, when $F$ is a sheaf (of sets) on $X$, we denote $f_*(F)$ the direct image sheaf under $f$. The catch is that the notation $f_*(F)$ (or $f_*$ alone) is abusive since $f_*(F)$ certainly depends on the topology of $Y$, which is, however, implied neither by $f$ nor by $F$. The solution to this problem lies in switching to the absolute view: instead of regarding $f$ as a mere function between the sets underlying $X$ and $Y$, we regard $f$ as a triple $(X,f,Y)$. Then the notation for the direct image functor is certainly justified---in fact, the $*$ symbol can itself be seen as a ``metafunctor'' from the category of topological spaces $\Top$ to the metacategory of all (possibly large) categories.

For another, let $X$ and $Y$ be two complex spaces and $f\colon X\to Y$ a morphism of complex spaces, so that $f$ is, by the naive definition, an ordered pair $(\psi,\theta)$ such that $\psi$ is a continuous map between the topological spaces underlying $X$ and $Y$ and $\theta$ is a certain morphism of sheaves on $Y$. Then one considers the associated sheaf of Kähler $1$-differentials, denoted $\Omega^1_{X/Y}$ or $\Omega^1_f$. Both notations are obviously abusive as the first one lacks the reference to $f$, whereas the second one lacks the reference to $X$ and $Y$. However, the latter notation $\Omega^1_f$ can be amended by regarding $f$ as absolute morphism, \iev we regard $f$ as $(X,f,Y)$ and are fine.

\section{Ringed spaces}
\label{s:rsp}

\subsection{}

Let $X$ be a topological space. Then we write $|X|$ for the underlying set of $X$. We write $\Sh(X)$ for the category of sheaves of sets on $X$ and $\Ab(X)$ for the category of abelian sheaves on $X$.

\subsection{}

Let $Y$ be another topological space and $f\colon X\to Y$ a continuous function. Then, viewing $f$ as an absolute morphism in $\Top$ (\cf \S\ref{s:cat}), we dispose of the direct image functor
\[
 f_* \colon \Sh(X) \to \Sh(Y)
\]
the usual way. For the inverse image functor
\[
 f^* \colon \Sh(Y) \to \Sh(X)
\]
we follow the convention that when $G$ is a sheaf of sets on $Y$, its inverse image sheaf $f^*(G)$ is the sheaf on $X$ associated to the presheaf $F'$ on $X$ given by
\[
 F'(U) = \colim \bigl( G \downarrow f(U) \bigr)
\]
for all open subsets $U$ of $X$, where we denote
\[
 G \downarrow f(U) \colon \Op(Y)^\op \downarrow f(U) \to \Set
\]
the restriction of the functor $G$ to the full subcategory of $\Op(Y)^\op$ induced on the set of open subsets $V$ of $Y$ satisfying $f(U) \subset V$, together with the canonical maps
\[
 F'(U,U') \colon F'(U) \to F'(U')
\]
induced by the colimits for all open subsets $U$ and $U'$ of $X$ such that $U' \subset U$. Notice that when $U$ is an open subset of $X$ such that $f(U)$ is open in $Y$, then $f(U)$ is a terminal object in the category on which $G \downarrow f(U)$ is defined, so that we have
\[
 \colim(G \downarrow f(U)) = G(f(U)).
\]
In particular, when the morphism $f\colon X\to Y$ is open, then $f^*(G)$ is the sheafification on $X$ of the composition of functors
\[
 \xymatrix{ \Op(X)^\op \ar[r]^-{f^\op} & \Op(Y)^\op \ar[r]^-G & \Set };
\]
in case the latter composition already is a sheaf on $X$ (\egv when $f\colon X\to Y$ is isomorphic to the inclusion of an open subspace), the process of sheafification is furthermore obsolete. When $B$ is a subset of $Y$, we write $G|B$ for $f^*(G)$, where $f\colon Y|B \to Y$ denotes the inclusion morphism.

For abelian sheaves we obtain the functors
\begin{align*}
 f_* & \colon \Ab(X) \to \Ab(Y), \\
 f^* & \colon \Ab(Y) \to \Ab(X),
\end{align*}
which are defined so that they commute with the forgetful functors to the categories of sheaves of sets on $X$ and $Y$, respectively. The very same way, we obtain direct and inverse image functors $f_*$ and $f^*$ for sheaves of rings.

\subsection{}

Let $X$ be a ringed space (\cf \cite[chap.~0, (4.1.1)]{EGA1}). Then we denote $X_\top$ and $\O_X$ the underlying topological space and the structure sheaf of $X$, respectively. We set $|X| := |X_\top|$.

When $Y$ is another ringed space and $f\colon X\to Y$ is a morphism of ringed spaces viewed as an absolute morphism in the category of ringed spaces, we write
\[
 f_\top \colon X_\top \to Y_\top
\]
for the induced (absolute) morphism in the category of topological spaces. We agree on writing $f_*$ as a shorthand for any of the direct image functors $(f_\top)_*$ introduced above (for sheaves of sets, abelian sheaves, and sheaves of rings). In contrast, we write $f^{-1}$ for any of the introduced inverse image functors $(f_\top)^*$.

We write
\[
 f^\sharp \colon \O_Y \to f_*(\O_X)
\]
for the morphism of sheaves of rings on $Y_\top$ coming about with $f$. By abuse of notation, occasionally, we also write
\[
 f^\sharp \colon f^{-1}(\O_Y) \to \O_X
\]
for the morphism of sheaves of rings on $X_\top$ derived from the previous $f^\sharp$ by means of adjunction between the functors $f^{-1}$ and $f_*$.

\subsection{}

By a \emph{module on $X$} or \emph{sheaf of modules on $X$} we mean a sheaf of left $\O_X$-modules on $X_\top$. By a \emph{submodule of \ldots on $X$} or a \emph{subsheaf of modules of \ldots on $X$} we mean a subsheaf of $\O_X$-modules of \ldots on $X_\top$. By an \emph{ideal on $X$} or \emph{sheaf of ideals on $X$} we mean a sheaf of left $\O_X$-ideals on $X_\top$. Analogously, we speak of a \emph{morphism of modules on $X$} or a \emph{morphism of sheaves of modules on $X$} in place of a morphism of sheaves of left $\O_X$-modules on $X_\top$.

\subsection{}

We denote $\Mod(X)$ the \emph{category of modules on $X$} and
\[
 \Gamma(X,-) \colon \Mod(X) \to \Mod(\Z)
\]
the \emph{global section functor} for $X$. Note that some authors choose $\Mod(\O_X(X))$ as their target for the global section functor, which, though richer in algebraic structure, seems functorially less apt to our taste (\cf below). We denote
\begin{align*}
 f_* & \colon \Mod(X) \to \Mod(Y), \\
 f^* & \colon \Mod(Y) \to \Mod(X)
\end{align*}
the direct and inverse image functors for modules, respectively. Thereby, our generic choice for the inverse image functor is
\[
 f^*(G) := \O_X \otimes_{f^{-1}(\O_Y)} f^{-1}(G).
\]
Under special circumstances, however, tensoring with $\O_X$ seems superfluous. Concreteley, in case $f^\sharp \colon f^{-1}(\O_Y) \to \O_X$ is an isomorphism of sheaves of rings on $X_\top$, we set
\[
 f^*(G) := f^{-1}(G),
\]
where we only change the module structure from $f^{-1}(\O_Y)$ to $\O_X$ by means of the inverse of the isomorphism $f^\sharp$.

\subsection{}

We denote $\sHom_X$ the \emph{sheaf hom} or \emph{internal hom} on $X$, which we regard as a functor
\[
 \sHom_X \colon \Mod(X)^\op \times \Mod(X) \to \Mod(X).
\]
Moreover, we set
\[
 \Hom_X := \Gamma(X,-) \circ \sHom_X \colon \Mod(X)^\op \times \Mod(X) \to \Mod(\Z).
\]

\subsection{}

For the moment, assume that the ringed space $X$ be commutative. Then, we denote $\otimes_X$ the \emph{tensor product} on $X$. Note that people usually tend to signify the tensor product on a ringed space $X$ by $\otimes_{\O_X}$ (\cf \egv \cite[chap.~0, \S4]{EGA1}), which we feel is bad as the topological space $X_\top$ certainly enters the contruction of the tensor product. Of course, one might argue that the sheaf $\O_X$ actually determines the topological space $X_\top$, yet, nonetheless, the notation $\otimes_X$ is shorter as well as more comprehensive. We regard $\otimes_X$ as functor
\[
 \otimes_X \colon \Mod(X) \times \Mod(X) \to \Mod(X),
\]
though we will mostly write the tensor product ``infix'' as is customary. Observe that in principle for any set $I$, one may define a tensor product
\[
 \otimes_X \colon (\Mod(X))^I \to \Mod(X)
\]
for $I$-families of modules on $X$; here we agree on suppressing the indexing set $I$ in the notation.

\subsection{}

For any natural number $n$, we denote
\begin{align*}
 \T^n_X & \colon \Mod(X) \to \Mod(X), \\
 \wedge^n_X & \colon \Mod(X) \to \Mod(X)
\end{align*}
the \emph{$n$-th tensor power} and \emph{$n$-th wedge power} functors on $X$, respectively, which are defined by sheafifying the corresponding ordinary linear algebra concepts. Specifically, as for any commutative ring $A$ and any (left) $A$-module $M$, we set
\begin{align*}
 & \T^0_A(M) := \wedge^0_A(M) := A, \\
 & \T^1_A(M) := \wedge^1_A(M) := M,
\end{align*}
we obtain that $\T^0_X$ and $\wedge^0_X$ both equal the constant functor with value $\O_X$ (viewed as a module on $X$) and $\T^1_X$ and $\wedge^1_X$ both equal the identity functor on $\Mod(X)$. We extend the tensor and wedge power operators to negative powers by defining $\T^n_X$ and $\wedge^n_X$ to be the constant functor with value the distinguished zero module on $X$ for all integers $n<0$.

\subsection{}

We like the idea of unifying the view on rings and ringed spaces, that is to say, we will freely regard any ring $A$ as a ringed space by defining $A_\top$ to be the \emph{distinguished one-point topological space} (\iev $|A_\top| = 1 = \{0\} = \{\emptyset\}$ with its uniquely determined topology) and defining $\O_A$ by
\begin{align*}
 \O_A(\emptyset) & := 1 = \{0\}, \\
 \O_A(1) & := A,
\end{align*}
where the ring structure and the restriction maps of $\O_A$ are the obvious ones. Specifically, we will denote $\Z$, $\Q$, $\RR$, and $\C$ also ringed spaces obtained from the rings $\Z$, $\Q$, $\RR$, and $\C$ by the described construction. Furthermore, when $A$ is a given ring, we may convert any $A$-module (in the ordinary sense) to a sheaf of modules on (the ringed space) $A$. Conversely, when $X$ is a one-point ringed space (that is to say, the underlying set is in bijection with the set $1$), we pass to an ordinary ring by considering $\O_X(X)$; modules on $X$ can be converted to $\O_X(X)$-modules also by taking global sections on $X$. Observing that standard constructions and properties (like forming quotients, direct sums, tensor products, speaking of submodules, etc.) are invariant under this translation machinery, we are at liberty to switch perspectives at will.

\subsection{}

The ringed space $\Z$ is terminal in the category of ringed spaces so that for all ringed spaces $X$ there exists a unique morphism of ringed spaces
\[
 a_X \colon X \to \Z,
\]
which we will mostly view as an absolute morphism in the category. Note that the induced direct image functor
\[
 (a_X)_* \colon \Mod(X) \to \Mod(\Z)
\]
corresponds to the global section functor $\Gamma(X,-)$ up to viewing $\Z$ as an ordinary ring or a ringed space. Due to the terminality of $\Z$, we see that for all morphisms of ringed spaces $f\colon X\to Y$, the diagram
\[
 \xymatrix{
  X \ar[rr]^{f} \ar[dr]_{a_X} && Y \ar[ld]^{a_Y} \\
  & \Z
 }
\]
commutes in the category of ringed spaces. Thus we obtain
\[
 (a_X)_* = (a_Y)_* \circ f_*
\]
or else
\[
 \Gamma(X,-) = \Gamma(Y,-) \circ f_*.
\]

\subsection{}
\label{ss:Avalcoh}

Let $X$ be a topological space, $A$ a ring. Then we denote $A_X$ the \emph{constant sheaf of rings with value $A$ on $X$}. We obtain a ringed space by setting $X^A := (X,A_X)$. Every morphism $f\colon X\to Y$ of topological spaces gives rise to a morphism of ringed spaces
\[
 f^A \colon X^A \to Y^A.
\]
In fact, these constructions yield a functor
\[
 -^A \colon \Top \to \Sp.
\]

Moreover, we have a global section functor
\[
 \Gamma^A(X,-) \colon \Mod(X^A) \to \Mod(A).
\]
For any integer $n$, we define the \emph{$n$-th cohomology of $X$ with values in $A$} by 
\[
 \H^n(X,A) := \R^n(\Gamma^A(X,-))(A_X),
\]
\cf \S\ref{s:der}. Note that viewing $A$ as a ringed space, the global section functor $\Gamma^A(X,-)$ corresponds to the direct image functor $(f^A)_*$ when $f\colon X\to Y$ denotes the unique morphism of topological spaces from $X$ to the distinguished one-point topological space. In this regard, occasionally, we will use the (admittedly awkward) notation
\[
 \H^n(f,A) := \R^n(f^A)_*(A_X).
\]

\section{Complex spaces}
\label{s:an}

\subsection{}

By a \emph{complex analytic space} we understand an object of the overcategory $\Sp_{/\C}$---without imposing any topological restrictions---which is locally isomorphic (in the category $\Sp_{/\C}$) to a complex analytic model space (\cf \cite[Définition 2.1]{SHC13.1.9}). We reserve the term \emph{complex space} for a complex analytic space whose underlying topological space is Hausdorff (\cf \cite[chap.~I, \S3]{GrPeRe94}, \cite[\S1, Definition 2]{Gr60}).

Note the following subtlety: most authors formally take complex (analytic) spaces to be ``$\C$-ringed'' or ``$\C$-algebraized'' spaces, which are topological spaces equipped with a sheaf of $\C$-algebras. However, our convention is to work exclusively with sheaves of rings, \iev with ringed spaces, so that we express the $\C$-algebraized nature of a ringed space $X$ by giving a morphism of ringed spaces $X \to \C$.

\subsection{}

We denote $\An$ the large \emph{category of complex analytic spaces}, which is, in fact, the full subcategory of the category $\Sp_{/\C}$ whose class of objects comprises precisely all complex analytic space. The \emph{category complex spaces} is a full subcategory of $\An$ for which we do not dispose of a separate name.

\subsection{}

When $X$ is a complex analytic space, then $X$ can be written uniquely as an ordered pair $(Y,f)$, where $Y$ is a ringed space and $f\colon Y\to \C$ is a morphism of ringed spaces. We call $Y$ the \emph{underlying ringed space} of $X$ and $f$ the \emph{structural morphism} of $X$. We follow the convention that when $X$ denotes a complex analytic space, we denote the underlying ringed space of $X$ by $X$ also. The structural morphism of an analytic space $X$ will be denoted $a_X$.

\subsection{}

Observe that the ordered pair consisting of the ringed space $\C$ together with the identity morphism $\C \to \C$ of ringed spaces is a complex space, which we call the \emph{distinguished one-point complex space}. By abuse of notation, we denote the distinguished one-point complex space by $\C$. Notice that $\C$ is a terminal object of $\An$. Besides, notice that for all complex analytic spaces $X$, the structural morphism $a_X$ is the unique morphism from $X$ to $\C$ in $\An$.

\subsection{}

Let $f\colon X\to S$ be a morphism in $\An$ (viewed in the absolute sense). Let $p \in X$. Then we say that $f$ is \emph{submersive in $p$} when, locally around $p$, $f$ is isomorphic to a product with fiber an open complex analytic subspace of $\C^n$ for some natural number $n$ (\cf \cite[(2.18)]{Fi76}). $f$ is called \emph{submersive} when, for all $x\in X$, the morphism $f$ is submersive in $x$.

We say $X$ is \emph{smooth in} or \emph{at $p$} when there exists an open neighborhood $U$ of $p$ in $X$ such that the open complex analytic subspace of $X$ induced on $U$ is isomorphic to an open complex analytic subspace of $\C^n$ for some natural number $n$. $X$ is called \emph{smooth} when $X$ is smooth in all of its points.

By a \emph{complex manifold} we mean a smooth complex space.

\subsection{}

Let $f\colon X\to S$ be a morphism in $\An$ (viewed in the absolute sense). Then we write $\Omega^1_f$ or (less frequently) $\Omega^1_{X/S}$ for the \emph{sheaf of (relative) (Kähler) $1$-differentials for $f$} (\cf \cite[\S2, p.~8]{SHC13.2.14}; note also the scheme theoretic analogue \cite[(16.3.1)]{EGA4.4}). For any natural number (or else integer) $p$, we set
\[
 \Omega^p_f := \wedge^p_X(\Omega^1_f).
\]
Since $\wedge^1_X$ is the identity functor on $\Mod(X)$, there is no contradiction. Moreover, we denote
\[
 \dd^p_f \colon \Omega^p_f \to \Omega^{p+1}_f
\]
the canonical differential sheaf map. We denote $\Omega^\kdot_f$ the \emph{complex of Kähler differentials for $f$}, i.e.:
\[
 \Omega^\kdot_f := \Bigl((\Omega^p_f)_{p\in\Z},(\dd^p_f)_{p\in\Z}\Bigr).
\]
Beware that $\Omega^\kdot_f$ is not a complex of modules on $X$, even though $\Omega^p_f$ is a module on $X$ for all $p\in\Z$ (the problem is that the differentials $\dd^p_f$ are not $\O_X$-linear).

We write $\Omega^p_X$, and $\Omega^\kdot_X$, and $\dd^p_X$ for $\Omega^p_{a_X}$, and $\Omega^\kdot_{a_X}$, and $\dd^p_{a_X}$, respectively.

\subsection{}
\label{ss:difffun}

Any commutative square
\begin{equation} \label{e:difffun-sq}
 \xysquare{X'}{X}{S'}{S}{u}{f'}{f}{w}
\end{equation}
in the category of complex analytic spaces induces a morphism
\[
 \phi^1 \colon \Omega^1_f \to u_*(\Omega^1_{f'})
\]
of sheaves of modules on $X$ (\cf \cite[\S2]{SHC13.2.14}), which we call the \emph{pullback of (relative) (Kähler) $1$-differentials} associated to the square in \eqref{e:difffun-sq}. The latter morphism induces in turn, for any integer $p$, a morphism
\[
 \phi^p \colon \Omega^p_f \to u_*(\Omega^p_{f'})
\]
of modules on $X$, which we call the \emph{pullback of $p$-differentials} associated to \eqref{e:difffun-sq}.

The pullbacks of differentials are functorial in the following sense. Let
\begin{equation} \label{e:difffun-sq'}
 \xysquare{X''}{X'}{S''}{S'}{u'}{f''}{f'}{w'}
\end{equation}
another commutative square in $\An$. Denote
\[
 \phi'^p \colon \Omega^p_{f'} \to {u'}_*(\Omega^p_{f''})
\]
the pullback of $p$-differentials associated to \eqref{e:difffun-sq'}. Then the composition
\[
 u_*(\phi'^p) \circ \phi^p \colon \Omega^p_f \to u_*({u'}_*(\Omega^p_{f''})) = (u\circ u')_*(\Omega^p_{f''})
\]
equals the pullback of $p$-differentials associated to the square obtained by composing \eqref{e:difffun-sq'} and \eqref{e:difffun-sq} horizontally.

\subsection{}

When $f\colon X\to S$ is a morphism of complex analytic spaces, we denote $\Theta_f$ or $\Theta_{X/S}$ the \emph{(relative) tangent sheaf} of $f$, \iev
\[
 \Theta_f := \sHom_X(\Omega^1_f,\O_X).
\]
Moreover, we set $\Theta_X := \Theta_{a_X}$.

\subsection{}

Let $X$ be a complex analytic space, $F$ a module on $X$, and $p \in X$. Observe that there exists a unique morphism $i \colon \C \to X$ of complex analytic spaces which sends the unique element $0 \in |\C|$ to $p$. Occasionally, we will use the notation
\[
 F(p) := i^*(F) \iso \C \otimes_{\O_{X,p}} F_p.
\]

\subsection{}

When $X$ is a complex analytic space and $p\in X$, we define
\[
 \T_X(p) := \Hom_\C(\Omega^1_X(p),\C)
\]
to be the \emph{tangent space of $X$ in $p$}. When $f\colon X\to Y$ is a morphism in $\An$, we write
\[
 \T_p(f) \colon \T_X(p) \to \T_Y(f(p))
\]
for the \emph{tangent map for $f$ in $p$}.

\subsection{}

We write $\dim_p(X)$ for the \emph{dimension of $X$ in $p$}. As usual, we set
\[
 \dim(X) := \sup\{\dim_x(X):x\in X\},
\]
where we agree to take the supremum with respect to the set
\[
 \hat\Z := \{-\omega\} \cup \Z \cup \{\omega\}
\]
equipped with its canonical ordering so that we have $-\omega \leq a$ and $a \leq \omega$ for all $a \in \hat\Z$. Specifically, we obtain
\[
 \dim(\emptyset) = -\omega \leq a
\]
for all integers $a$.

\subsection{}
\label{ss:res}

By a \emph{resolution of singularities} we mean a proper modification $f\colon W\to X$ of complex spaces such that $W$ is a complex manifold.

\subsection{Mixed Hodge structures}
\label{ss:mhs}

We feel that, over Deligne's original \cite[Définition (2.3.1)]{De71}, it has certain technical advantages to define a \emph{mixed Hodge structure} to be an ordered septuple
\begin{equation} \label{e:mhs}
 H = (H_\Z,H_\Q,W,\alpha,H_\C,F,\beta)
\end{equation}
such that $H_\Z$, $H_\Q$, and $H_\C$ are finite type $\Z$-, $\Q$-, and $\C$-modules, respectively, $W$ is a finite increasing filtration of $H_\Q$ by rational vector subspaces, $F$ is a finite decreasing filtration of $H_\C$ by complex vector subspaces, and $\alpha \colon H_\Z \to H_\Q$ and $\beta \colon H_\Q \to H_\C$ are mappings inducing isomorphisms $\Q\otimes_\Z H_\Z \to H_\Q$ and $\C\otimes_\Q H_\Q \to H_\C$, respectively, such that the triple $(W_\C,F,\bar F)$ is a triple of ``opposed'' filtrations on $H_\C$ by complex vector subspaces (\cf \loccit).

The gist is that Deligne's definition forces the $\Q$-vector space $H_\Q$ to be $\Q\otimes_\Z H_\Z$ just as it forces the $\C$-vector space $H_\C$ to be $\C\otimes_\Z H_\Z$, whereas our definition grants us some freedom there.

When $H$ is a mixed Hodge structure \eqref{e:mhs}, we set
\[
 \F^pH := F^p, \qquad \oF^pH := \bar{F^p}, \qquad \W_iH := W_i
\]
for all integers $p$ and $i$.

We will make use of Fujiki's idea (\cf \cite[(1.4)]{Fu80}) that Deligne's construction of a mixed Hodge theory for complex algebraic varieties in \cite{De71,De74} carries over naturally to the category of complex spaces $X$ endowed with a compactification $X \subset X^*$ such that $X^*$ is of class $\sC$; the morphisms in this category are morphisms $f\colon X\to Y$ of complex spaces extending to meromorphic maps $X^* \to Y^*$ between the respective given compactifications. Concretely, when $X$ is as above an $n$ is an integer, we denote $\H^n(X)$ the \emph{mixed Hodge structure of cohomology in degree $n$ of $X$}. In view of \eqref{e:mhs}, we have
\[
 \H^n(X)_\Z = \H^n(X,\Z), \qquad \H^n(X)_\Q = \H^n(X,\Q), \qquad \H^n(X)_\C = \H^n(X,\C);
\]
the mappings $\alpha$ and $\beta$ are induced respectively by the canonical injections $\Z\to\Q$ and $\Q\to\C$.

\section{Spectral sequences}
\label{s:homalg}

\subsection{}
\label{d:ss}

In our understanding of spectral sequences, we basically follow \cite[Definition 2.2]{MCl01}. Let $C$ be a (possibly large) abelian category. Then $E$ is a \emph{spectral sequence (with values) in $C$} when $E = (E_r)$ is a sequence indexed over $\N_{\geq r_0}$ for some natural number $r_0$ such that, for all $r \in \N_{\geq r_0}$, $E_r$ is a differential bigraded object of $C$ with differential of bidegree $(r,1-r)$ such that one has
\[
 \H^{p,q}(E_r) \iso E_{r+1}^{p,q}
\]
in $C$ for all $(p,q) \in \Z \times \Z$.

Note that many authors tend to incorporate additional data into a spectral sequence, like, for instance, a choice of isomorphisms
\[
 \phi_r^{p,q} \colon \H^{p,q}(E_r) \to E_{r+1}^{p,q}
\]
in $C$ (\cf \egv \cite[III.7.3]{GelMan03}, \cite[Definition 5.2.1]{Wei94}, \cite[XI.1]{MLa63}), which makes sense of course only if $C$ is equipped with a canonical (co-)homology functor $\H$. For our purposes, however, the incorporating of additional data is unnecessary.

\begin{definition}[Degeneration]
 \label{d:degen}
 Let $C$ be an abelian category, $E$ a spectral sequence with values in $C$. Denote $r_0$ the starting term of $E$, \iev the smallest element of the domain of definition of the sequence $E$, and let $r_1$ be a natural number $\geq r_0$.
 \begin{enumerate}
  \item Let $(p,q) \in \Z \times \Z$. Then we say that \emph{$E$ degenerates from behind} (\resp \emph{forwards}) \emph{in (the entry) $(p,q)$ at sheet $r_1$ in $C$} when, for all natural numbers $r \geq r_1$, the differential $d_r^{p-r,q+r-1}$ (\resp $d_r^{p,q}$) equals the zero morphism from $E_r^{p-r,q+r-1}$ to $E_r^{p,q}$ (\resp from $E_r^{p,q}$ to $E_r^{p+r,q-r+1}$) in $C$. We say that \emph{$E$ degenerates in (the entry) $(p,q)$ at sheet $r_1$ in $C$} when $E$ degenerates both from behind and forwards in $(p,q)$ at sheet $r_1$ in $C$.
  \item Let $I$ be a subset of $\Z \times \Z$. Then we say that $E$ \emph{degenerates} (\resp \emph{degenerates from behind}, \resp \emph{degenerates forwards}) \emph{in entries $I$ at sheet $r_1$ in $C$} when, for all $(p,q)\in I$, the spectral sequence $E$ degenerates (\resp degenerates from behind, \resp degenerates forwards) in the entry $(p,q)$ at sheet $r_1$ in $C$.
 \end{enumerate}
\end{definition}

\section{Derived functors}
\label{s:der}

We refrain from considering derived categories as we refrain from anticipating any sort of universe axiom (\cf \S\ref{s:st}). Nevertheless, we would like to talk about derived functors. Let us outline the conventions that we follow throughout the text in this regard.

\subsection{}
\label{ss:caninjres}

Let $X$ be an arbitrary ringed space. Moreover, let $F$ be an object of $\Com^+(X)$. We will make use of the device of fabricating a \emph{canonical injective resolution} of $F$ on $X$ (here, by a resolution we simply  mean a quasi-isomorphism to another complex of modules on $X$). First of all, we refer to \cite[Chapter II, 3.5]{Bre97}, where it is explained how to construct for a given sheaf of modules on $X$, \iev a sheaf of $\O_X$-modules on $X_\top$, call it $G$, a canonical monomorphism
\[
 \alpha \colon G \to I
\]
to an injective sheaf of modules on $X$. As a matter of fact, in \loccit, it is further explained how to construct a canonical injective resolution (in the naive sense, \iev not working with quasi-isomorphisms of complexes) for $G$. However, as we would like to work strictly with complexes, we do the following:

As the complex $F$ is bounded below, we know there exists an integer $m$ such that $F^i \iso 0$ in $\Mod(X)$ for all integers $i<m$. When the complex $F$ is completely trivial (\iev when $F^i \iso 0$ holds for all integers $i$), then we define $I$ to be the canonical zero complex of modules on $X$ and $\rho \colon F \to I$ to be the unique morphism of complexes. When $F$ is not entirely trivial, we choose $m$ maximal with the mentioned property. For all integers $i<m$, we define $I^i$ to be the canonical zero module on $X$ and $\rho_i \colon F^i \to I^i$ to be the zero morphism. From $i=m$ onwards, we define $I^i$ together with the differential $d^{i-1} \colon I^i \to I^{i+1}$ and the morphism $\rho_i\colon F^i \to I^i$ by the inductive procedure layed out in the proof of \cite[I, Theorem 6.1]{Iv86}, where we plug in our canonical monomorphism to an injective module at the appropriate stage. That way, we obtain a bounded below complex $I$ of injective modules on $X$ together with a morphism
\[
 \rho \colon F \to I
\]
of complexes of modules on $X$ which is a quasi-isomorphism of complexes of modules on $X$.

The so defined $\rho$ (viewed as an absolute morphism in $\Com^+(X)$), as well as its canonical image in $\K^+(X)$, are called the canonical injective resolution of $F$ of $X$.

\begin{lemma}
 \label{l:iversen}
 Let $C$ be an abelian category, $\gamma \colon F \to G$ a quasi-isomorphism in $\K(C)$, and $I$ a bounded below complex of injective objects of $C$. Then, for all morphisms $\alpha \colon F \to I$ in $\K(C)$, there exists one, and only one, morphism $\beta \colon G \to I$ in $\K(C)$ such that we have $\beta \circ \gamma = \alpha$, \iev such that the diagram
 \[
  \xymatrix{
   F \ar[r]^\gamma \ar[dr]_\alpha & G \ar@{.>}[d]^\beta \\
   & I
  }
 \]
 commutes in $\K(C)$. Equivalently, the function
 \[
  \hom_{\K(C)}(\gamma,I) \colon \hom_{\K(C)}(G,I) \to \hom_{\K(C)}(F,I)
 \]
 is a bijection.
\end{lemma}

\begin{proof}
 The equivalence of the two statements is clear. For the proof of any one of them, see \cite[I, Theorem 6.2]{Iv86}.
\end{proof}

\subsection{}
\label{ss:injresfun}

With the help of Lemma \ref{l:iversen}, we construct an \emph{injective resolution functor} on $X$,
\[
 I \colon \K^+(X) \to \K^+(X),
\]
together with a natural transformation
\[
 \rho \colon \id_{\K^+(X)} \to I
\]
of endofunctors on $\K^+(X)$ as follows: As to the object function of $I$, for a given object $F$ of $\K^+(X)$, we define $I(F)$ to be the canonical injective resolution of $F$ on $X$ as constructed above. We define
\[
 \rho(F) \colon F \to I(F)
\]
to be the corresponding resolving morphism. As to the family of morphism functions of $I$, when $F$ and $G$ are objects of $\K^+(X)$ and $\alpha \colon F \to G$ is a morphism, Lemma \ref{l:iversen} implies that there exists a unique $\beta$ such that the diagram 
\[
 \xymatrix{
  F \ar[r]^{\rho(F)} \ar[d]_\alpha & I(F) \ar@{.>}[d]^\beta \\
  G \ar[r]_{\rho(G)} & I(G)
 }
\]
commutes in $\K^+(X)$ (or equivalently, in $\K(X)$). We set
\[
 (I(F,G))(\alpha) := \beta.
\]

\subsection{}
\label{ss:derfun}

Let $X$ and $Y$ be ringed spaces (\resp complex spaces),
\[
 T \colon \Mod(X) \to \Mod(Y)
\]
an additive functor. Then we define the \emph{(bounded below) right derived functor} of $T$ with respect to $X$ and $Y$ as the composition of functors
\[
 \xymatrix@C=3pc{ \K^+(X) \ar[r]^I & \K^+(X) \ar[r]^{\K^+(T)} & \K^+(Y) },
\]
where $I$ denotes the injective resolution functor on $X$. As is customary, we denote the right derived functor of $T$ by $\R(T)$ or $\R T$ (even though one better incorporate the references to $X$ and $Y$ into the notation and write something like $\R_{X,Y}(T)$; moreover, writing $\R^+(T)$ instead of $\R(T)$ would probably be more conceptual).

\subsection{}

Composing the natural transformation $\rho \colon \id \to I$ with the functor
\[
\K^+(T) \colon \K^+(X) \to \K^+(Y),
\]
we obtain a natural transformation
\[
 \K^+(T) \circ \rho \colon \K^+(T) \to \R(T)
\]
of functors from $\K^+(X)$ to $\K^+(Y)$.

\subsection{}

Precomposing $\R(T)$ with the (composition of) canonical morphism(s)
\[
 \bigl( \Mod(X) \to \bigr) \Com^+(X) \to \K^+(X),
\]
we establish two variants of the right derived functor of $T$ that will go under the same denomination of $\R(T)$.

\subsection{}

Furthermore, for any integer $n$, we define
\[
 \R^n(T) := \H^n \circ \R(T) \colon \K^+(X) \to \Mod(Y),
\]
where
\[
 \H^n \colon \K^+(Y) \to \Mod(Y)
\]
denotes the evident cohomology in degree $n$ functor on $Y$. Just as before, we fabricate the two variants
\begin{align*}
 \R^n(T) & \colon \Com^+(X) \to \Mod(Y), \\
 \R^n(T) & \colon \Mod(X) \to \Mod(Y)
\end{align*}
which go by the same denomination.

\subsection{}

In the situation of \eqref{ss:derfun}, one may define, for any integer $n$, a \emph{connecting homomorphism in degree $n$}, denoted $\delta^n_T$ or simply $\delta^n$, which we would like to view as a function defined either on the class of short exact triples in $\Com^+(X)$ or the class of short exact triples in $\Mod(X)$ without making a notational distinction.

\chapter{Tools}
\label{ch:tool}

\section{Base change maps}
\label{s:bc}

\begin{construction}
 \label{con:dercomp}
 Let $X$, $Y$, and $Z$ be ringed spaces and
 \begin{align*}
  S & \colon \Mod(X) \to \Mod(Y), \\
  T & \colon \Mod(Y) \to \Mod(Z)
 \end{align*}
 additive functors. Define the bounded below right derived functors $\R S$, $\R T$, and $\R(T \circ S)$ as explained in \eqref{ss:derfun}. Then we dispose of a natural transformation
 \begin{equation} \label{e:dercomp}
  \phi \colon \R(T\circ S) \to \R T \circ \R S
 \end{equation}
 of functors from $\K^+(X)$ to $\K^+(Z)$. The definition is in fact straightforward. Denote $I_X$ and $I_Y$ the canonical injective resolution functors on $X$ and $Y$, respectively, as given in \eqref{ss:injresfun}; moreover, denote
 \[
  \rho_Y \colon \id_{\K^+(Y)} \to I_Y
 \]
 the accompanying resolving natural transformation on $Y$. Then set
 \[
  \phi := (T \circ \rho_Y) \circ (S \circ I_X) \colon (T \circ \id_{\K^+(Y)}) \circ (S \circ I_X) \to (T \circ I_Y) \circ (S \circ I_X),
 \]
 where we compose natural transformations of functors with functors the obvious way. Observe that we thus obtain a natural transformation \eqref{e:dercomp} since
 \[
  \R(T \circ S) = (T \circ S) \circ I_X = (T \circ \id_{\K^+(Y)}) \circ (S \circ I_X)
 \]
 and
 \[
  \R T \circ \R S = (T \circ I_Y) \circ (S \circ I_X)
 \]
 by the definition of the derived functors.
\end{construction}

The following proposition is a variant of either \cite[Proposition 13.3.13]{KaSc06} or \cite[III.7.1]{GelMan03}.

\begin{proposition}
 \label{p:dercomp}
 In the situation of Construction \ref{con:dercomp}, assume that the functor $T$ is left exact and $S$ takes injective modules on $X$ to right $T$-acyclic modules on $Y$. Then \eqref{e:dercomp} is a natural quasi-equivalence of functors from $\K^+(X)$ to $\K^+(Z)$.
\end{proposition}

\begin{proof}
 This follows from the general fact that when $A$ and $I$ are bounded below complexes of right $T$-acyclic and injective modules on $Y$, respectively, and $\rho \colon A \to I$ is a quasi-isomorphism of complexes of modules on $Y$, then
 \[
  T(\rho) \colon T(A) \to T(I)
 \]
 is a quasi-isomorphism of complexes of modules on $Z$.
\end{proof}

\begin{construction}
 \label{con:cohswap}
 Let $X$ and $S$ be ringed spaces,
 \[
  T \colon \Mod(X) \to \Mod(Y)
 \]
 a left exact additive functor, $n$ an integer. Let $F$ be a complex of modules on $X$. Denote $\rZ^n(F) \subset F^n$ and $\rZ^n(TF) \subset TF^n$ the null spaces of the differential maps $\dd^n_F \colon F^n \to F^{n+1}$ and $\dd^n_{TF} \colon TF^n \to TF^{n+1}$ of the complexes $F$ and $TF$, respectively. Then due to the universal property of kernels, there exists one, and only one, $\alpha$ such that the diagram
 \[
  \xymatrix@C1pc{
   T\rZ^n(F) \ar@{.>}[rr]^\alpha \ar[dr]_{T(\subset)} && \rZ^n(TF) \ar[ld]^{\subset} \\
   & TF^n
  }
 \]
 commutes in $\Mod(Y)$. Moreover, since the functor $T$ is left exact, $\alpha$ is an isomorphism. Thus by the universal property of cokernels, there exists one, and only one, $\beta(F)$ such that the diagram
 \[
  \xymatrix{
   \rZ^n(TF) \ar[r]^{\pi^n_{TF}} \ar[d]_{\alpha^{-1}} & \H^n(TF) \ar@{.>}[d]^{\beta(F)} \\
   T\rZ^n(F) \ar[r]_{T\pi^n_F} & T\H^n(F)
  }
 \]
 commutes in $\Mod(Y)$, where $\pi^n_F$ and $\pi^n_{TF}$ denote the respective quotient morphisms to cohomology. An easy argument shows that the family $\beta = (\beta(F))_F$ constitutes a natural transformation
 \[
  \beta \colon \H^n \circ T \to T \circ \H^n
 \]
 of functors from $\K(X)$ to $\Mod(Y)$. Denote $\beta^+$ the restriction of $\beta$ to $\K^+(X)$.

 Denote $I_X$ the canonical injective resolution functor on $X$ \eqref{ss:injresfun} and
 \[
  \rho_X \colon \id_{\K^+(X)} \to I_X
 \]
 the accompanying resolving natural transformation. We know that $\rho_X$ is a natural quasi-equivalence of endofunctors on $\K^+(X)$, so that
 \[
  \H^n \circ \rho_X \colon \H^n \to \H^n \circ I_X
 \]
 is a natural equivalence of functors from $\K^+(X)$ to $\Mod(X)$. In turn, we obtain a natural transformation
 \[
  \gamma := (\beta^+ \ast I_X) \circ (T \ast (\H^n \ast \rho_X)^{-1}) \colon \R^nT = \H^n \circ T \circ I_X \to T \circ \H^n
 \]
 of functors from $\K^+(X)$ to $\Mod(Y)$, where we have denoted the horizontal composition of a natural transformation an a functor by ``$\ast$'' in order to distinguish it from the vertical composition ``$\circ$'' of natural transformations.
\end{construction}

\begin{construction}
 \label{con:derfib}
 We fix a commutative square in the category of ringed spaces $\Sp$:
 \begin{equation} \label{e:derfib-sq}
  \xysquare{X'}{X}{S'}{S}{u}{f'}{f}{w}
 \end{equation}
 Let $n$ be an integer. Let $F$ and $F'$ be bounded below complexes of modules on $X$ and $X'$, respectively. Finally, let $\alpha\colon F\to F'$ be a $u$-morphism of complexes modules modulo homotopy, \iev a morphism $F\to u_*(F')$ in $\K^+(X)$, where, as usual, we agree on writing $u_*$ instead of $\K^+(u_*)$.
 
 Denote
 \[
  \tau \colon u_* \to \R u_*
 \]
 the natural transformation of functors from $\K^+(X')$ to $\K^+(X)$ which comes about with the construction of the right derived functor as explained in \S\ref{s:der}. Composing $\alpha$ with $\tau(F')$ in $\K^+(X)$ yields a morphism
 \[
  F \to \R u_*(F')
 \]
 in $\K^+(X)$. Applying the functor $\R f_*$, we obtain a morphism
 \[
  \R f_*(F) \to \R f_*(\R u_*(F'))
 \]
 in $\K^+(S)$. Let
 \begin{align*}
  \phi & \colon \R(f_*\circ u_*) \to \R f_* \circ \R u_*, \\
  \psi & \colon \R(w_*\circ f'_*) \to \R w_*\circ\R f'_*
 \end{align*}
 be the natural transformations of functors from $\K^+(X')$ to $\K^+(S)$ which are associated to the triples of categories and functors
 \begin{align*}
  & \xymatrix{\Mod(X') \ar[r]^{u_*} & \Mod(X) \ar[r]^{f_*} & \Mod(S)}, \\
  & \xymatrix{\Mod(X') \ar[r]^{f'_*} & \Mod(S') \ar[r]^{w_*} & \Mod(S)},
 \end{align*}
 respectively, by means of Construction \ref{con:dercomp}. Since
 \[
  f_* \circ u_* = (f \circ u)_* = (w \circ f')_* = w_* \circ f'_*,
 \]
 we obtain the following diagram of morphisms in $\K^+(S)$:
 \[
  \xymatrix{
   \R f_*(\R u_*(F')) & \R(f_*u_*)(F') \ar@{=}[r] \ar[l]_-{\phi(F')} & \R(w_*f'_*)(F') \ar[r]^-{\psi(F')} & \R w_*(\R f'_*(F'))
  }
 \]
 By Proposition \ref{p:dercomp}, we know that $\phi(F')$ is a quasi-isomorphism of complexes of modules on $S$. In particular,
 \[
  \H^n(\phi(F')) \colon \H^n\R(f_*u_*)(F') \to \H^n\R f_*(\R u_*(F'))
 \]
 is an isomorphism in $\Mod(S)$; write $\H^n(\phi(F'))^{-1}$ for its inverse. Then we have the following composition of morphisms in $\Mod(S)$:
 \begin{equation} \label{e:derfib-1}
  \H^n(\psi(F')) \circ \H^n(\phi(F'))^{-1} \circ \R^nf_*(\tau(F') \circ \alpha) \colon \R^nf_*(F) \to \H^n\R w_*(\R f'_*(F')).
 \end{equation}
 Let
 \[
  \gamma \colon \H^n \circ \R w_* \to w_* \circ \H^n
 \]
 be the natural transformation of functors from $\K^+(S')$ to $\Mod(S)$ given by Construction \ref{con:cohswap}. Then composing the morphisms \eqref{e:derfib-1} and  $\gamma(\R f'_*(F'))$ in $\Mod(S)$, we arrive at a morphism
 \[
  \beta^n \colon \R^nf_*(F) \to w_*(\R^nf'_*(F'))
 \]
 in $\Mod(S)$, \iev at a $w$-morphism of modules $\R^nf_*(F)\to\R^nf'_*(F')$.
\end{construction}

\begin{proposition}
 \label{p:derfibid}
 In the situation of Construction \ref{con:derfib}, assume that $X' = X$, $S' = S$, $f' = f$, $u = \id_X$, and $w = \id_S$. Then $\alpha \colon F \to F'$ is an ordinary morphism in $\K^+(X)$ and we have
 \[
  \beta^n = \R^nf_*(\alpha) \colon \R^nf_*(F) \to \R^nf_*(F').
 \]
\end{proposition}

\begin{proof}
 Omitted.
\end{proof}

\begin{proposition}[Functoriality, I]
 \label{p:derfibfun}
 Let
 \begin{equation} \label{e:bcfun-diag}
  \xymatrix{
   X'' \ar[r]^{u'} \ar[d]_{f''} & X' \ar[r]^u \ar[d]_{f'} & X \ar[d]^f \\
   S'' \ar[r]_{w'} & S' \ar[r]_w & S
  }
 \end{equation}
 be a commutative diagram in the category of ringed spaces and $n$ an integer. Let $F$, $F'$, and $F''$ be objects of $\K^+(X)$, $\K^+(X')$, and $\K^+(X'')$, respectively. Let $\alpha\colon F\to F'$ and $\alpha'\colon F'\to F''$ be a $u$- and $u'$-morphism of modules, respectively. Write $\alpha''$ for the composition $\alpha$ and $\alpha'$, \iev set
 \[
  \alpha'' := u_*(\alpha') \circ \alpha
 \]
 in $\K^+(X)$, and note that consequently $\alpha''$ is a $(uu')$-morphism of modules $F\to F''$. Denote
 \begin{align*}
  \beta^n & \colon \R^nf_*(F) \to w_*(\R^nf'_*(F')), \\
  \beta'^n & \colon \R^nf'_*(F') \to w'_*(\R^nf''_*(F'')), \\
  \beta''^n & \colon \R^nf_*(F) \to (ww')_*(\R^nf''_*(F''))
 \end{align*}
 the morphisms in $\Mod(S)$, $\Mod(S')$, and $\Mod(S)$ associated to $\alpha$, $\alpha'$, and $\alpha''$ with respect to the right, left, and outer subsquares of the diagram in \eqref{e:bcfun-diag}, respectively, by means of Construction \ref{con:derfib}. Then $\beta''^n$ is the composition of $\beta^n$ and $\beta'^n$ in the sense that we have
 \[
  \beta''^n = w_*(\beta'^n) \circ \beta^n
 \]
 in $\Mod(S)$.
\end{proposition}

\begin{proof}
 Omitted.
\end{proof}

\begin{proposition}[Functoriality, II]
 \label{p:derfibfun2}
 Let \eqref{e:derfib-sq} be a commutative square in the category of ringed spaces, $n$ an integer. Let $F,G\in\K^+(X)$ and $F',G'\in\K^+(X')$, let $\alpha\colon F\to F'$ and $\gamma\colon G\to G'$ be $u$-morphisms of complexes of modules modulo homotopy, and let $\phi\colon F\to G$ and $\phi'\colon F'\to G'$ be morphisms in $\K^+(X)$ and $\K^+(X')$, respectively, such that the diagram
 \[
  \xysquare{F}{u_*(F')}{G}{u_*(G')}{\alpha}{\phi}{\phi'}{\gamma}
 \]
 commutes in $\K^+(X)$. Denote
 \begin{align*}
  \beta^n & \colon \R^nf_*(F) \to w_*(\R^nf'_*(F')), \\
  \delta^n & \colon \R^nf_*(G) \to w_*(\R^nf'_*(G'))
 \end{align*}
 the morphisms in $\Mod(S)$ obtained from $\alpha$ and $\beta$, respectively, by means of Construction \ref{con:derfib}. Then the following diagram commutes in $\Mod(S)$:
 \[
  \xysquare{\R^nf_*(F)}{w_*(\R^nf'_*(F'))}{\R^nf_*(G)}{w_*(\R^nf'_*(G'))}{\beta^n}{\R^nf_*(\phi)}{w_*(\R^nf'_*(\phi'))}{\delta^n}
 \]
\end{proposition}

\begin{proof}
 This is an immediate consequence of Propositions \ref{p:derfibid} and \ref{p:derfibfun}.
\end{proof}

The following construction shows how to derive the classical base change maps for higher direct image sheaves (\cf \egv \cite[XII (4.2)]{SGA4.3}) from Construction \ref{con:derfib}.

\begin{construction}[Base change maps]
 \label{con:bc}
 Assume we are given a commutative square \eqref{e:derfib-sq} in the category of ringed spaces as well as an integer $n$. Let $F$ be a bounded below complex of modules on $X$. Then the unit of the adjunction between the functors $u^*$ and $u_*$ yields a $u$-morphism
 \[
  \alpha \colon F \to u^*(F)
 \]
 of complexes of modules modulo homotopy. Applying Construction \ref{con:derfib}, we obtain in turn a $w$-morphism of modules
 \[
  \beta^n \colon \R^nf_*(F) \to \R^nf'_*(u^*(F)),
 \]
 which yields a morphism of modules on $S'$,
 \[
  \tilde\beta^n(F)\colon w^*(\R^nf_*(F)) \to \R^nf'_*(u^*(F)),
 \]
 by means of adjunction between the functors $w^*$ and $w_*$.
 
 We claim that the family $\tilde\beta^n := (\tilde\beta^n(F))_F$ constitutes a natural transformation
 \[
  \tilde\beta^n \colon w^* \circ \R^nf_* \to \R^nf'_* \circ u^*
 \]
 of functors from $\K^+(X)$ to $\Mod(S')$. For that matter, let $G$ be another bounded below complex of modules on $X$ and $\phi \colon F \to G$ a morphism in $\K^+(X)$. Denote
 \[
  \gamma \colon G \to u^*(G)
 \]
 the $u$-morphism of complexes of modules modulo homotopy given by the unit of the adjunction between the functors $u^*$ and $u_*$. Then, due to the naturality of the adjunction unit, the diagram
 \[
  \xymatrix{
   F \ar[r]^-\alpha \ar[d]_\phi & u_*u^*(F) \ar[d]^{u_*u^*(\phi)} \\
   G \ar[r]_-\gamma & u_*u^*(G)
  }
 \]
 commutes in the category of modules on $X$. Thus by means Proposition \ref{p:derfibfun2}, taking into account the naturality of the adjunction between the functors $w^*$ and $w_*$, we see that the diagram
 \[
  \xysquare{w^*(\R^nf_*(F))}{\R^nf'_*(u^*(F))}{w^*(\R^nf_*(G))}{\R^nf'_*(u^*(G))}{\tilde\beta^n(F)}{w^*(\R^nf_*(\phi))}{\R^nf'_*(u^*(\phi))}{\tilde\beta^n(G)}
 \]
 commutes in the category of modules on $S'$.
 
 We call $\tilde\beta^n$ the \emph{base change natural transformation in degree $n$} associated to the square \eqref{e:derfib-sq}. If you prefer the term ``morphism of functors'' over the term ``natural transformation'', you might want to call $\tilde\beta^n$ the base change morphism in degree $n$ associated to \eqref{e:derfib-sq}, yet we reserve the term \emph{base change morphism in degree $n$ for $F$} associated to \eqref{e:derfib-sq} for the individual $\tilde\beta^n(F)$; we also say \emph{base change map} in the latter context.
\end{construction}

\begin{construction}[Hodge base change]
 \label{con:hodgebc}
 Assume we are given a commutative square
 \begin{equation} \label{e:hodgebc-sq}
  \xysquare{X'}{X}{S'}{S}{u}{f'}{f}{w}
 \end{equation}
 in the category of complex spaces $\An$. Let $p$ and $q$ be integers. As explained in \S\ref{s:an}, the square \eqref{e:hodgebc-sq} induces a pullback of relative Kähler $p$-differentials, which is a $u$-morphism of modules
 \[
  \alpha^p \colon \Omega^p_f \to \Omega^p_{f'}.
 \]
 By means of Construction \ref{con:derfib}, viewing \eqref{e:hodgebc-sq} as a square in the category of ringed spaces, the morphism $\alpha^p$ induces a $w$-morphism of modules
 \begin{equation} \label{e:hodgebc}
  \beta^{p,q} \colon \R^qf_*(\Omega^p_f) \to \R^qf'_*(\Omega^p_{f'}).
 \end{equation}
 We call $\beta^{p,q}$ the \emph{Hodge base change in bidegree $(p,q)$} associated to \eqref{e:hodgebc-sq}. By slight abuse of terminology, we call the morphism
 \[
  \tilde\beta^{p,q} \colon w^*(\R^qf_*(\Omega^p_f)) \to \R^qf'_*(\Omega^p_{f'})
 \]
 of modules on $S'$ which is obtained from $\beta^{p,q}$ by means of adjunction between the functors $w^*$ and $w_*$ also the \emph{Hodge base change in bidegree $(p,q)$} associated to \eqref{e:hodgebc-sq}; probably, one should better call it the ``adjoint Hodge base change'' or similar instead.
\end{construction}

\begin{proposition}
 \label{p:hodgebcfun}
 Let
 \begin{equation} \label{e:hodgebcfun}
  \xymatrix{
   X'' \ar[r]^{u'} \ar[d]_{f''} & X' \ar[r]^u \ar[d]_{f'} & X \ar[d]^f \\
   S'' \ar[r]_{w'} & S' \ar[r]_w & S
  }
 \end{equation}
 be a commutative diagram in the category of complex spaces and $p$ and $q$ integers. Denote $\beta^{p,q}$, $\beta'^{p,q}$, and $\beta''^{p,q}$ the Hodge base changes in bidegree $(p,q)$ associated to the right, left, and outer subsquares of the diagram in \eqref{e:hodgebcfun}, respectively. Then $\beta''^{p,q}$ is the composition of $\beta^{p,q}$ and $\beta'^{p,q}$, \iev we have:
 \[
  \beta''^{p,q} = w_*(\beta'^{p,q}) \circ \beta^{p,q}.
 \]
\end{proposition}

\begin{proof}
 Denote $\alpha^p$, $\alpha'^p$, and $\alpha''^p$ the pullbacks of $p$-differentials associated to the right, left, and outer subsquares of the diagram in \eqref{e:hodgebcfun}, respectively. Then we have
 \[
  \alpha''^p = u_*(\alpha'^p) \circ \alpha^p
 \]
 (\cf \S\ref{s:an}). Thus the claim follows from Proposition \ref{p:derfibfun}.
\end{proof}

\begin{construction}
 \label{con:barx}
 Let $f\colon X\to S$ be a morphism of ringed spaces (viewed in the absolute sense, \cf \S\ref{s:cat}). Define
 \[
  \bar X := (X_\top,f^{-1}(\O_S))
 \]
 and
 \[
  \bar f := (f,\eta(\O_S) \colon \O_S \to f_*f^{-1}(\O_S)),
 \]
 where $\eta$ denotes the natural transformation of adjunction between the functors $f^{-1}$ and $f_*$ for sheaves of rings. Then $\bar X$ is a ringed space and
 \[
  \bar f \colon \bar X \to S
 \]
 is a morphism of ringed spaces.

 Assume that we are given a commutative square
 \[
  \xysquare{X'}{X}{S'}{S}{u}{f'}{f}{w}
 \]
 of ringed spaces. Define
 \[
  \bar{f'} \colon \bar{X'} \to S'
 \]
 for $f'$ as we have defined $\bar f$ for $f$. Moreover, define
 \[
  \bar u := (u,\bar\theta \colon f^{-1}\O_S \to u_*f'^{-1}(\O_{S'})),
 \]
 where $\bar\theta$ is obtained from the composition
 \[
  \xymatrix@C3.5pc{ \O_S \ar[r]^-{w^\sharp} & w_*(\O_{S'}) \ar[r]^-{w_*(\eta'(\O_{S'}))} & w_*f'_*f'^{-1}(\O_{S'}) \ar@{=}[r] & f_*u_*f'^{-1}(\O_{S'}) }
 \]
 by means of adjunction between the functors $f^{-1}$ and $f_*$; here, $\eta'$ denotes the natural transformation associated to the adjunction between the functors $f'^{-1}$ and $f'_*$. Then the following diagram commutes in the category of ringed spaces:
 \begin{equation} \label{e:barx-sq}
  \xysquare{\bar{X'}}{\bar X}{S'}{S}{\bar u}{\bar{f'}}{\bar f}{w}
 \end{equation}
 In fact, the described construction may be interpreted as an endofunctor
 \[
  \Sp^{\mathbf2} \to \Sp^{\mathbf2}
 \]
 on the arrow category of the category of ringed spaces $\Sp$. As one will notice, the latter endofunctor features the property of commutating with the functor
 \[
  \pr_1 \colon \Sp^{\mathbf2} \to \Sp
 \]
 which projects to the target of an arrow.
\end{construction}

\begin{construction}[De Rham base change]
 \label{con:drbc}
 Assume we are given a commutative square \eqref{e:hodgebc-sq} in the category of complex spaces, which we equally view as a commutative square in the category of ringed spaces. Then Construction \ref{con:barx} yields the commutative square of ringed spaces \eqref{e:barx-sq}. According to \S\ref{s:an}, the pullback of Kähler differentials associated to the square of complex spaces \eqref{e:hodgebc-sq} gives rise to a $\bar u$-morphism of complexes of modules
 \[
  \alpha \colon \bar\Omega^\kdot_f \to \bar\Omega^\kdot_{f'}.
 \]
 Let $n$ be an integer. Then Construction \ref{con:derfib} yields a $w$-morphism of modules
 \[
  \beta^n \colon \R^n\bar f_*(\bar\Omega^\kdot_f) \to \R^n\bar{f'}_*(\bar\Omega^\kdot_{f'}),
 \]
 which we call the \emph{de Rahm base change in degree $n$} associated to \eqref{e:hodgebc-sq}. By slight abuse of terminology, we call the morphism
 \[
  \tilde\beta^n \colon w^*(\R^n\bar f_*(\bar\Omega^\kdot_f)) \to \R^n\bar{f'}_*(\bar\Omega^\kdot_{f'})
 \]
 of modules on $S'$ which is obtained from $\beta^n$ by means of adjunction between the functors $w^*$ and $w_*$ the \emph{de Rahm base change in degree $n$} associated to \eqref{e:hodgebc-sq}, too (\cf Construction \ref{con:hodgebc}).
\end{construction}

\begin{proposition}
 \label{p:drbcfun}
 Let \eqref{e:hodgebcfun} be a commutative diagram in the category of complex spaces, $n$ an integer. Denote $\beta^n$, $\beta'^n$, and $\beta''^n$ the de Rahm base changes in degree $n$ associated to the right, left, and outer subsquares of the diagram in \eqref{e:hodgebcfun}, respectively. Then $\beta''^n$ is the composition of $\beta^n$ and $\beta'^n$, \iev we have:
 \[
  \beta''^n = w_*(\beta'^n) \circ \beta^n.
 \]
\end{proposition}

\begin{proof}
 First of all, by the functoriality of Construction \ref{con:barx}, we see that the diagram
 \[
  \xymatrix{
   \bar{X''} \ar[r]^{\bar{u'}} \ar[d]_{\bar{f''}} & \bar{X'} \ar[r]^{\bar u} \ar[d]_{\bar{f'}} & \bar{X} \ar[d]^{\bar f} \\
   S'' \ar[r]_{w'} & S' \ar[r]_w & S
  }
 \]
 commutes in the category of ringed spaces. Denote $\alpha$, $\alpha'$, and $\alpha''$ the pullbacks of differentials associated to the right, left, and outer subsquares of the diagram in \eqref{e:hodgebcfun}, respectively. Then we have
 \[
  \alpha'' = {\bar u}_*(\alpha') \circ \alpha
 \]
 (\cf \S\ref{s:an}). Thus our claim follows from Proposition \ref{p:derfibfun}.
\end{proof}

\begin{proposition}
 \label{p:filtdrbc}
 In the situation of Construction \ref{con:drbc}, let $p$ be another integer. Then the de Rahm base change $\beta^n$ restricts to a $w$-morphism of modules
 \[
  \beta^{p,n} \colon \F^p\sH^n(f) \to \F^p\sH^n(f').
 \]
 Recall here from Notation \ref{not:hodgefilt} that we have
 \[
  \F^p\sH^n(f) \subset \sH^n(f) = \R^n\bar f_*(\bar\Omega^\kdot_f),
 \]
 and likewise for $f'$ in place of $f$.
\end{proposition}

\begin{proof}
 This follows from Proposition \ref{p:derfibfun2} by considering the commutative diagram
 \[
  \xymatrix{
   \sigma^{\geq p}\bar\Omega^\kdot_f \ar[r]^{\sigma^{\geq p}\alpha} \ar[d]_{i^{\geq p}(\bar\Omega^\kdot_f)} & \sigma^{\geq p}\bar\Omega^\kdot_{f'} \ar[d]^{i^{\geq p}(\bar\Omega^\kdot_{f'})} \\
   \bar\Omega^\kdot_f \ar[r]_\alpha & \bar\Omega^\kdot_{f'}
  }
 \]
 of modules. We omit the details.
\end{proof}

\section{Hodge theory of rational singularities}
\label{s:rtlsing}

Below we review a couple of Hodge theoretic properties of complex spaces with rational singularities. The following lemma, which is a variation on the theme of \cite[Lemma (1.2)]{Na01}, will be fundamental.

\begin{lemma}
 \label{l:rtlsingres}
 Let $X$ be a complex space having a rational singularity at $p\in X$ and let $f\colon W\to X$ be a resolution of singularities such that $E := f^{-1}(p)$ is a complex space of Fujiki class $\sC$ whose underlying set is a simple normal crossing divisor in $W$. Then we have
 \[
  \gr^0_\F(\H^n(E)) := \F^0\H^n(E)/\F^1\H^n(E) \iso 0
 \]
 for all natural numbers $n>0$.
\end{lemma}

\begin{proof}
 Fix a natural number $n>0$. By stratification theory, there exists an open neighborhood $U$ of $E$ in $W$ such that $E$ is a deformation retract of $U$. Thus, by means of shrinking the base of $f$ around $p$, we may assume that the complex space $X$ is Stein and has rational singularities and that the function
 \[
  i^* \colon \H^n(W,\C) \to \H^n(E,\C)
 \]
 induced by the inclusion $i\colon E\to W$ is a surjection.

 Denote $F = (F^p)_{p\in\Z}$ the filtration on $\H^n(W,\C)$ induced by the stupid filtration of the algebraic de Rham complex $\Omega^\kdot_W$ via the canonical isomorphism $\H^n(W,\C) \to \H^n(W,\Omega^\kdot_W)$. Then the morphism $i^*$ is filtered with respect to $F$ and the Hodge filtration $(\F^p\H^n(E))_{p\in\Z}$ of the mixed Hodge structure $\H^n(E)$. In particular, $i^*$ induces a surjective mapping
 \[
  F^0/F^1 \to \F^0\H^n(E)/\F^1\H^n(E).
 \]
 Looking at the Frölicher spectral sequence of $W$, we see that there exists a monomorphism
 \[
  F^0/F^1 \to \H^n(W,\O_W),
 \]
 yet
 \[
  \H^n(W,\O_W) \iso \H^n(X,\O_X) \iso 0
 \]
 since $X$ is Stein and has rational singularities, whence we conclude that $F^0/F^1 \iso 0$ and thus $\gr^0_\F(\H^n(E)) \iso 0$.
\end{proof}

\begin{corollary}
 \label{c:rtlsingh1}
 Under the hypotheses of Lemma \ref{l:rtlsingres}, we have $\H^1(E,\C) \iso 0$.
\end{corollary}

\begin{proof}
 By Lemma \ref{l:rtlsingres}, we know that
 \[
  \gr^0_F(\H^1(E)) \iso 0.
 \]
 Thus $\gr_0^\W(\H^1(E))$ is certainly trivial---and so is $\gr_1^\W(\H^1(E))$ as a result of Hodge symmetry. The remaining weights (\iev those in $\Z \setminus \{0,1\}$) of the mixed Hodge structure $\H^1(E)$ are trivial from the start, hence our claim.
\end{proof}

\begin{proposition}
 \label{p:rtlsingr1}
 Let $X$ be a complex space having rational singularities, $f\colon W\to X$ a resolution of singularities. Then we have $\R^1f_*(\C_W) \iso 0$.
\end{proposition}

\begin{proof}
 Let $p\in X$ be arbitrary. Since the morphism $f$ is proper, we have
 \[
  (\R^1f_*(\C_W))_p \iso \H^1(W_p,\C),
 \]
 so that it suffices to show that $\H^1(W_p,\C) \iso 0$. Since any resolution of singularities of $X$ can be dominated by a projective one, we may assume that $f$ is projective from the start. We know there exists a projective embedded resolution $g\colon W'\to W$ for $W_p$ such that $g^{-1}(W_p) =: E$ is a simple normal crossing divisor in $W'$. Thus, by Corollary \ref{c:rtlsingh1} (applied to $f' := f \circ g$ in place of $f$), we have $\H^1(E,\C) \iso 0$. As the Leray spectral sequence for $g|E \colon E\to W_p$ implies that the pullback function
 \[
  \H^1(W_p,\C) \to \H^1(E,\C)
 \]
 is one-to-one, we deduce that $\H^1(W_p,\C) \iso 0$.
\end{proof}

\begin{proposition}
 \label{p:resh2inj}
 Let $X$ be a complex space having rational singularities, $f\colon W\to X$ a resolution of singularities. Then the function
 \begin{equation} \label{e:resh2inj}
  f^*\colon \H^2(X,\C) \to \H^2(W,\C)
 \end{equation}
 is one-to-one.
\end{proposition}

\begin{proof}
 Denote $E$ the Grothendieck spectral sequence (or, more specifically, the Leray spectral sequence) associated to the triple
 \[
  \xymatrix{ \Mod(W,\C_W) \ar[r]^-{f_*} & \Mod(X,\C_X) \ar[r]^-{\Gamma(X,-)} & \Mod(\C) }
 \]
 of categories and functors. Then, since we have $\R^1f_*(\C_W) \iso 0$ by Proposition \ref{p:rtlsingr1}, the spectral sequence $E$ degenerates in entry $(2,0)$ at sheet $2$ in $\Mod(\C)$, whence the canonical map
 \begin{equation} \label{e:resh2inj-1}
  \H^2(X,f_*\C_W) \to \H^2(W,\C_W)
 \end{equation}
 is one-to-one. As the function \eqref{e:resh2inj} factors through \eqref{e:resh2inj-1} via the $\H^2(X,-)$ of the canonical isomorphism $\C_X \to f_*\C_W$ of sheaves on $X_\top$, we deduce that \eqref{e:resh2inj} is one-to-one.
\end{proof}

\begin{corollary}
 \label{c:h2pure}
 Let $X$ be a complex space of Fujiki class $\sC$ having rational singularities. Then the mixed Hodge structure $\H^2(X)$ is pure of weight $2$.
\end{corollary}

\begin{proof}
 Let $f\colon W\to X$ be a projective resolution of singularities. Then by Proposition \ref{p:resh2inj}, the pullback
 \[
  f^* \colon \H^2(X,\Q) \to \H^2(W,\Q)
 \]
 is a monomorphism of $\Q$-vector spaces. Note that $f^*$ is filtered with respect to the weight filtrations $(\W_n\H^2(X))_{n\in\Z}$ and $(\W_n\H^2(W))_{n\in\Z}$ of the mixed Hodge structures $\H^2(X)$ and $\H^2(W)$, respectively, \iev we have
 \[
  f^*[\W_n\H^2(X)] \subset \W_n\H^2(W)
 \]
 for all integers $n$. Since $W$ is a complex manifold, we know that
 \[
  \W_n\H^2(W) =
  \begin{cases}
   \{0\} & \text{for all } n<2, \\
   \H^2(W,\Q) & \text{for all } n\geq 2.
  \end{cases}
 \]
 Thus, using the injectivity of $f^*$, we deduce that
 \[
  \W_n\H^2(X) = \{0\}
 \]
 for all integers $n<2$. Since $f^*$ is even strictly compatible with the weight filtrations, we further deduce that
 \[
  f^*[\W_2\H^2(X)] = f^*[\H^2(X,\Q)] \cap \W_2\H^2(W) = f^*[\H^2(X,\Q)],
 \]
 whence
 \[
  \W_2\H^2(X) = \W_3\H^2(X) = \W_4\H^2(X) = \dots = \H^2(X,\Q).
 \]
 Thus the mixed Hodge structure $\H^2(X)$ is pure of weight $2$.
\end{proof}

\begin{proposition}
 \label{p:f2h2}
 Let $X$ be a complex space of Fujiki class $\sC$ having rational singularities, $f\colon W\to X$ a resolution of singularities. Then the pullback function \eqref{e:resh2inj} restricts to a bijection
 \begin{equation} \label{e:f2h2}
  f^*|\F^2\H^2(X) \colon \F^2\H^2(X) \to \F^2\H^2(W).
 \end{equation}
\end{proposition}

\begin{proof}
 First of all, we know that the pullback \eqref{e:resh2inj} is filtered with respect to the Hodge filtrations $(\F^p\H^2(X))_{p\in\Z}$ and $(\F^p\H^2(W))_{p\in\Z}$ of the mixed Hodge structures $\H^2(X)$ and $\H^2(W)$, respectively, so that $f^*$ certainly restricts to yield a function \eqref{e:f2h2}. According to Proposition \ref{p:resh2inj}, the function \eqref{e:f2h2} is injective. It remains to show that \eqref{e:f2h2} is surjective. For that matter, let $c$ be an arbitrary element of $\F^2\H^2(W)$. We claim that $c$ is sent to zero by the canonical morphism
 \[
  \alpha \colon \H^2(W,\C) \to \H^0(X,\R^2f_*(\C_W)).
 \]
 
 So, fix an element $p\in X$. Let $g\colon W'\to W$ be an embedded resolution of $W_p \subset W$ such that $E := g^{-1}(W_p)$ is a simple normal crossing divisor in $W'$. Then, Lemma \ref{l:rtlsingres} implies that $\gr^0_\F(\H^2(E)) \iso 0$. Given that $\gr_n^\W(\H^2(E)) \iso 0$ for all $n>2$, we deduce that
 \[
  \F^2\H^2(E) \iso 0
 \]
  exploiting the Hodge symmetry of the weight-$2$ Hodge structure $\gr_2^\W(\H^2(E))$. Since the composition of mappings
 \[
  \H^2(W,\C) \to \H^2(W_p,\C) \to \H^2(E,\C)
 \]
 is filtered with respect to the Hodge filtrations of the mixed Hodge structures $\H^2(W)$ and $\H^2(E)$, respectively, it sends $c$ to zero. As $W$ has rational singularities, Proposition \ref{p:rtlsingr1} implies that $\R^1g_*(\C_{W'}) \iso 0$. In consequence, the pullback mapping
 \[
  \H^2(W_p,\C) \to \H^2(E,\C)
 \]
 is one-to-one, whence $c$ is sent to zero by the restriction
 \[
  \H^2(W,\C) \to \H^2(W_p,\C)
 \]
 already. Since the morphism $f$ is proper, the topological base change
 \[
  (\R^2f_*(\C_W))_p \to \H^2(W_p,\C)
 \]
 is a bijection (and in particular one-to-one), so that the canonical image of $\alpha(c)$ in the stalk $(\R^2f_*(\C_W))_p$ vanishes. As $p\in X$ was arbitrary, we see that $\alpha(c) = 0$.

 Now, since $\R^1f_*(\C_W) \iso 0$ by Proposition \ref{p:rtlsingr1}, the Leray spectral sequence for $f$ yields that
 \[
  \xymatrix{ \H^2(X,\C) \ar[r]^-{f^*} & \H^2(W,\C) \ar[r]^-{\alpha} & \H^0(X,\R^2f_*(\C_W)) }
 \]
 is an exact sequence of complex vector spaces. Hence, there exists an element $d \in \H^2(X,\C)$ such that $f^*(d) = c$. Accordingly, since $f^*$ is strictly compatible with the Hodge filtrations, there exists an element $d' \in \F^2\H^2(X)$ such that $f^*(d') = c$ (in fact, we have $d = d'$ as $f^*$ is one-to-one). This proves the surjectivity of \eqref{e:f2h2}.
\end{proof}

\begin{proposition}
 \label{p:bingenercrit}
 Let $X$ be a complex space of Fujiki class $\sC$ having rational singularities. Then the mapping
 \begin{equation} \label{e:bingenercrit}
  \H^2(X,\RR) \to \H^2(X,\O_X)
 \end{equation}
 which is induced by the canonical morphism $\RR_X \to \O_X$ of sheaves on $X_\top$ is a surjection.
\end{proposition}

\begin{proof}
 Let $f\colon W\to X$ be a resolution of singularities. Then (essentially by Proposition \ref{p:derfibfun2}) the diagram
 \begin{equation} \label{e:bingenercrit-1}
  \xysquare{\H^2(X,\C)}{\H^2(W,\C)}{\H^2(X,\O_X)}{\H^2(W,\O_W)}{f^*}{\beta}{\alpha}{f^*}
 \end{equation}
 commutes in the category of complex vector spaces, where $\alpha$ and $\beta$ are induced by the structural sheaf maps $\C_W \to \O_W$ and $\C_X \to \O_X$, respectively. Restricting the domains of definition of the functions corresponding to the vertical arrows in \eqref{e:bingenercrit-1}, we obtain the following diagram, which commutes in the category of complex vector spaces too:
 \begin{equation} \label{e:bingenercrit-2}
  \xysquare{\oF^2\H^2(X)}{\oF^2\H^2(W)}{\H^2(X,\O_X)}{\H^2(W,\O_W)}{f^*}{\beta'}{\alpha'}{f^*}
 \end{equation}
 By Proposition \ref{p:f2h2}, the upper horizontal arrow in \eqref{e:bingenercrit-2} is an isomorphism. Since $W$ is a complex manifold, $\alpha'$ is an isomorphism by classical Hodge theory. The lower horizontal arrow in \eqref{e:bingenercrit-2} is an isomorphism since the complex space $X$ has rational singularities. Thus, $\beta'$ is an isomorphism by the commutativity of the diagram in \eqref{e:bingenercrit-2}.

 Furthermore, we have
 \[
  \F^1\H^2(X) \subset \ker(\beta)
 \]
 since, for all $c \in \F^1\H^2(X)$, we certainly have $f^*(c) \in \F^1\H^2(W)$, whence $\alpha(f^*(c)) = 0$ looking at the Frölicher spectral sequence of $W$; therefore $\beta(c) = 0$ due to the injectivity of
 \[
  f^* \colon \H^2(X,\O_X) \to \H^2(W,\O_W).
 \]

 Now, be $d \in \H^2(X,\O_X)$ arbitrary. Then there exists an element $c \in \oF^2\H^2(X)$ such that $\beta(c) = \beta'(c) = d$. Since $\bar c \in \F^2\H^2(X)$ and $\F^2\H^2(X) \subset \F^1\H^2(X)$, we deduce that
 \[
  \beta(c + \bar c) = d.
 \]
 Clearly, $c+\bar c$ is real in $\H^2(X,\C)$, \iev there exists an element $c' \in \H^2(X,\RR)$ which is mapped to $c + \bar c$ by the function
 \[
  \H^2(X,\RR) \to \H^2(X,\C)
 \]
 induced by the canonical sheaf map $\RR_X \to \C_X$ on $X_\top$. In consequence, $c'$ is mapped to $d$ by the function \eqref{e:bingenercrit}. This shows that the function \eqref{e:bingenercrit} is a surjection as $d$ was an arbitrary element of $\H^2(X,\O_X)$.
\end{proof}

\backmatter
\printbibliography
\end{document}